\documentclass[reqno]{amsart}
\usepackage{tikz,genyoungtabtikz,enumitem,rotating,latexsym,bm,stmaryrd}
 \usetikzlibrary{positioning,intersections}
  \usepackage{amsmath,amsthm,amsfonts,amssymb,mathrsfs,pb-diagram}
 \usepackage{caption,cases}
\usepackage{lipsum,wasysym}
\usepackage{mathtools}
\usetikzlibrary{arrows,calc,through,backgrounds,matrix,decorations.pathmorphing}

\usepackage[a4paper,
            left=1in,right=1in,top=1in,bottom=0.86in]{geometry}

\usepackage{cleveref}
\newtheorem{thm}{Theorem}[section]
\newtheorem{cor}[thm]{Corollary}

\newtheorem{lem}[thm]{Lemma}
\newtheorem{prop}[thm]{Proposition}
\newtheorem*{prop*}{Proposition}
\newtheorem*{thm*}{Main Theorem}

\newtheorem*{cor*}{Corollary}
\newtheorem*{conj*}{Conjecture}

\newtheorem*{move1'}{Move $\mathbf{1^*}$}

\newtheorem*{clcopieri}{Classical co-Pieri rule}

\theoremstyle{remark}

\newtheorem{rmk}[thm]{Remark}

\newtheorem*{Acknowledgements*}{Acknowledgements}

\theoremstyle{definition}
\newtheorem{defn}[thm]{Definition}
\newtheorem{eg}[thm]{Example}

\crefname{defn}{Definition}{Definitions}
\crefname{thm}{Theorem}{Theorems}
\crefname{prop}{Proposition}{Propositions}
\crefname{lem}{Lemma}{Lemmas}
\crefname{cor}{Corollary}{Corollaries}
\crefname{conj}{Conjecture}{Conjectures}
\crefname{section}{Section}{Sections}
\crefname{subsection}{Subsection}{Subsections}
\crefname{eg}{Example}{Examples}
\crefname{figure}{Figure}{Figures}
\crefname{rem}{Remark}{Remarks}
\crefname{rmk}{Remark}{Remarks}
\crefname{equation}{equation}{equation}

\Crefname{defn}{Definition}{Definitions}
\Crefname{thm}{Theorem}{Theorems}
\Crefname{prop}{Proposition}{Propositions}
\Crefname{lem}{Lemma}{Lemmas}
\Crefname{cor}{Corollary}{Corollaries}
\Crefname{conj}{Conjecture}{Conjectures}
\Crefname{section}{Section}{Sections}
\Crefname{subsection}{Subsection}{Subsections}
\Crefname{eg}{Example}{Examples}
\Crefname{figure}{Figure}{Figures}
\Crefname{rem}{Remark}{Remarks}
\Crefname{rmk}{Remark}{Remarks}

\numberwithin{equation}{section}

\newcommand{\flux}{\mathsf{t}}  %standard index
\newcommand{\stw}{\mathsf{w}}  %standard index
\newcommand{\sts}{\mathsf{s}}  %standard index
\newcommand{\stt}{\mathsf{t}}  %standard index
\newcommand{\stu}{\mathsf{u}}  %standard index
\newcommand{\stv}{\mathsf{v}}  %standard index
 \newcommand{\SSTS}{\mathsf{S}}  %semistandard index
\newcommand{\SSTT}{\mathsf{T}}  %semistandard index
\newcommand{\SSTU}{\mathsf{U}}  %semistandard index

\renewcommand{\Vec}{V}  %semistandard index

\newcommand{\down}{m{\downarrow}}
\newcommand{\up}{m{\uparrow}}
\newcommand{\Hom}{\operatorname{Hom}}

\newcommand{\suchthat}{\;\ifnum\currentgrouptype=16 \middle\fi|\;} 
\newcommand{\ZZ}{{\mathbb Z}}
\newcommand{\NN}{{{\mathbb Z}_{\geq0}}}

\newcommand{\Specht}{\mathbf{ \textbf  \rm S}}

\newcommand{\CC}{\mathbb{Q}}

\newcommand{\minmax}{\operatorname{minmax}}

\newcommand{\eleven}{11}

\newcommand\mptn[1]{\mathscr{P}_{#1}}

\usetikzlibrary{decorations.pathreplacing}

\usepackage{amsmath}
\mathchardef\mhyphen="2D

\DeclareMathOperator{\node}{node}

\newcommand{\half}{\frac{1}{2}}
\newcommand{\thalf}{\textstyle\frac{1}{2}}

\newcommand{\Std}{\mathrm{Std}}

\newcommand{\SStd}{\mathrm{SStd}}
\newcommand{\Latt}{\mathrm{Latt}}

\let\originalleft\left
\let\originalright\right
\def\left#1{\mathopen{}\originalleft#1}
\def\right#1{\originalright#1\mathclose{}}

\renewcommand{\geq}{\geqslant}

\renewcommand{\leq}{\leqslant}

\def\ignore#1{\relax}

\def\ignore#1{\relax}

 \newcommand{\ik}{k}

 \title{{The co-Pieri rule for stable Kronecker coefficients}}
\author{C. Bowman}
\author{M. De Visscher}
\author{J. Enyang}

  \makeatletter
\def\@maketitle{%
  \newpage
  \begin{center}%
  \let \footnote \thanks
    {\LARGE \@title \par}%
    \vskip 1.5em%
    {\large
      \lineskip .5em%
      \begin{tabular}[t]{c}%
        \@author C. Bowman, M. De Visscher, and J. Enyang
      \end{tabular}\par}%
    \vskip 1em%
    {\large \@date}%
  \end{center}%
  \par
  \vskip 1.5em}
\makeatother

\usepackage{scalefnt}
\begin{document}

\!\!\!\!\!\!\!\!\!\!\!\!\!\!\!\!\!\!\!
 \maketitle

 \maketitle

%\vskip-10pt
\!\!\!\!\!\!\!\!\!\!\!\! \section*{Introduction}

Perhaps the last major open problem in  the  complex representation theory of symmetric groups %$\mathfrak{S}_n$   
is to describe the decomposition of a tensor product of two simple representations.  
  The  coefficients describing the decomposition of these tensor products  are known %in the literature 
   as the {\sf Kronecker coefficients} and they have been described as 
    `perhaps the most challenging, deep and mysterious objects in algebraic combinatorics' \cite{PP1}. %; whilst the problem  is now over 80 years old, 
%  `frustratingly little is known'   about these coefficients \cite{Bur}.     
%Geometric complexity theory that translate the arithmetic version of 
% celebrated  P versus NP problem   \cite{MR3338303,MS} and 
% is founded on
 More recently, these coefficients have provided the centrepiece of Geometric Complexity Theory (GCT),  a { ``new hope"} \cite{newhope} for settling the {\sf P} versus {\sf NP} problem \cite{MR2927658}.  It was recently shown that GCT requires not only to understand the positivity, but also   precise information on the explicit values of these coefficients \cite{zeroes}. 
   The positivity of Kronecker coefficients is equivalent to the existence of certain quantum systems  \cite{ky,MR2197548,MR2276458}  and  they have been used to understand entanglement entropy \cite{MR3748296}.    Much recent progress has  focussed on   the stability properties enjoyed by  Kronecker coefficients \cite{BDO15,stab3,MR3461556,stab1,stem}.

  Whilst a complete understanding of the  Kronecker coefficients seems out of reach, 
the purpose of this paper is to attempt to understand
 the  {\em stable} Kronecker coefficients in terms of oscillating tableaux.
 Oscillating tableaux  hold a distinguished position in the study of 
  tensor product decompositions   \cite{MR1035496,MR3090983,MR2264927}  but surprisingly they  have never before been 
  used to calculate Kronecker coefficients of symmetric groups.  
 In this work, we  see that the oscillating tableaux defined as paths on the graph given in   \cref{brancher} (which we call Kronecker tableaux) provide
  bases of certain modules for the partition algebra, $P_s(n)$, which is closely related to the symmetric group.
 We hence add a new level of structure to the classical picture --- this extra structure is the key to our main result: the co-Pieri rule for stable Kronecker coefficients.

 \!\!\!\!\!   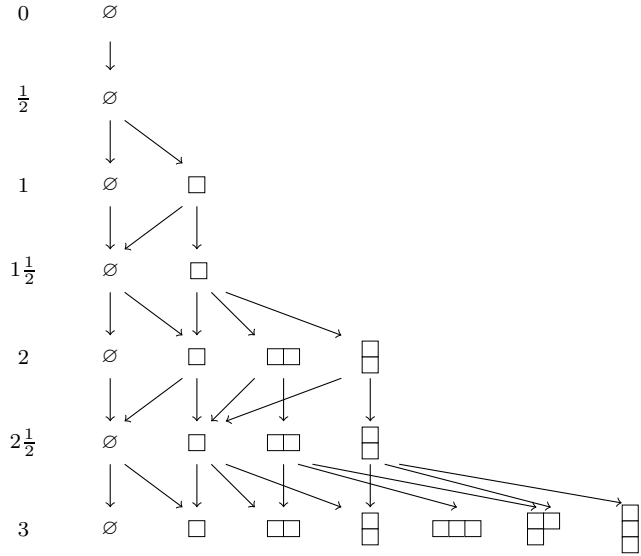
\begin{figure}[ht!]$$  \scalefont{0.8}\begin{tikzpicture}[scale=0.38]
           \begin{scope}     \draw (0,3) node {  $\varnothing$  };   
  \draw (-3,0) node {   $ \half$  };     \draw (-3,3) node {   $ 0$  };   
    \draw (-3,-3) node {   $ 1$  };     \draw (-3,-6) node {   $ 1\half$  };   
    \draw (-3,-9) node {   $ 2$  };         \draw (-3,-12) node {   $ 2\half$  };     
    \draw (-3,-15) node {   $ 3$  };     
    %%%%%%%%%
              \draw (0,0) node { $\varnothing$  };   
    \draw (0,-3) node   {   \text{	$\varnothing$	}}		;    \draw (+3,-3) node   {  
     $ \Yboxdim{6pt}\scalefont{0.9}
  \Yvcentermath1  \gyoung(;) $	
    }		;
      \draw (0,-6) node   {   \text{	$\varnothing$	}};
     \draw (3,-6) node   {  $       \,\Yboxdim{6pt}\scalefont{0.9}
  \Yvcentermath1  \gyoung(;) $	 };
    \draw[<-] (0.0,1) -- (0,2);     \draw[->] (0.0,-0.75) -- (0,-2.25);        
                \draw[->] (0.5,-0.75) -- (2.5,-2.25);    \draw[->] (2.5,-3.75) -- (0.5,-5.25);       \draw[->] (0,-3.75) -- (0,-5.25);   \draw[->] (3,-3.75) -- (3,-5.25);
  \draw[->] (0,-6.75) -- (0,-8.25);   \draw[->] (03,-6.75) -- (3,-8.25); 
    \draw[->] (0.5,-6.75) -- (2.5,-8.25);   \draw[->] (4,-6.75) -- (8,-8.25);   \draw[->] (3.5,-6.75) -- (5,-8.25); 
     \draw (+0,-9) node   {   	  $\varnothing$  };                 
          \draw (+3,-9) node   {  $ \Yboxdim{6pt}\scalefont{0.9}
  \Yvcentermath1  \gyoung(;)  $	 }		;
             \draw (+6,-9) node              {    $\Yboxdim{6pt}\scalefont{0.9}
  \Yvcentermath1  \gyoung(;;)$	 	 }		;
                          \draw (+9,-9) node            {    $\Yboxdim{6pt}\scalefont{0.9}
  \Yvcentermath1  \gyoung(;,;)$	 	 }			;
  %%%
  %%%
   \draw[<-] (0,-11.25) -- (0,-9.75);   \draw[<-] (03,-11.25) -- (3,-9.75);  \draw[<-] (06,-11.25) -- (6,-9.75); \draw[<-] (9,-11.25) -- (9,-9.75); 
    \draw[<-] (0.5,-11.25) -- (2.5,-9.75);   \draw[<-] (4,-11.25) -- (8,-9.75);   \draw[<-] (3.5,-11.25) -- (5,-9.75); 
     \draw (+0,-12) node   {   	$\varnothing$  };                 
          \draw (+3,-12) node   {  $ \Yboxdim{6pt}\scalefont{0.9}
  \Yvcentermath1  \gyoung(;) $	 }		;
             \draw (+6,-12) node              {    $\Yboxdim{6pt}\scalefont{0.9}
  \Yvcentermath1  \gyoung(;;)$	 	 }		;
                          \draw (+9,-12) node            {    $\Yboxdim{6pt}\scalefont{0.9}
  \Yvcentermath1  \gyoung(;,;)$	 	 }			;
    %%%
  %%%
   \draw[->] (0,-12.75) -- (0,-14.25);   \draw[->] (03,-12.75) -- (3,-14.25);  \draw[->] (06,-12.75) -- (6,-14.25); \draw[->] (9,-12.75) -- (9,-14.25); 
    \draw[->] (0.5,-12.75) -- (2.5,-14.25);   \draw[->] (4,-12.75) -- (8,-14.25);   \draw[->] (3.5,-12.75) -- (5,-14.25); 
     \draw (+0,-15) node   {   	$\varnothing$  };                 
          \draw (+3,-15) node   {  $ \Yboxdim{6pt}\scalefont{0.9}
  \Yvcentermath1  \gyoung(;) $	 }		;
             \draw (+6,-15) node              {    $\Yboxdim{6pt}\scalefont{0.9}
  \Yvcentermath1  \gyoung(;;)$	 	 }		;
                          \draw (+9,-15) node            {    $\Yboxdim{6pt}\scalefont{0.9}
  \Yvcentermath1  \gyoung(;,;)$	 	 }			;
  %%%
  %%%
  \draw[->] (6.5,-12.75) -- (12,-14.25);    \draw[->] (7,-12.75) -- (14.75,-14.25);  
    \draw[->] (9.5,-12.75) -- (15.25,-14.25);    \draw[->] (10,-12.75) -- (17.7,-14.1);  
     \draw (12,-15) node            {    $\Yboxdim{6pt}\scalefont{0.9}
  \Yvcentermath1  \gyoung(;;;)$	 	 }			;
       \draw (15,-15) node            {    $\Yboxdim{6pt}\scalefont{0.9}
  \Yvcentermath1  \gyoung(;;,;)$	 	 }			;
       \draw (18,-15) node            {    $\Yboxdim{6pt}\scalefont{0.9}
  \Yvcentermath1  \gyoung(;,;,;)$	 	 }			;
    \end{scope}\end{tikzpicture} 
     $$
  
\!\!\!\!\!\!     \caption{The first three layers of the branching graph $\mathcal{Y}$}
     \label{oscillate}\label{brancher}
\end{figure}

\!\!\!
A momentary glance at the   graph given in \cref{oscillate}
  reveals a very familiar subgraph: namely Young's graph (with each level doubled up).   
 The stable Kronecker coefficients labelled by triples from  this subgraph are  well-understood  --- 
  the values of these coefficients
    can be calculated via a tableaux counting algorithm known as the Littlewood--Richardson rule \cite{Littlewood} (see \cref{LRclassic,LitMurn}).   
  This rule    has long served as the hallmark for our understanding (or lack thereof)  of Kronecker coefficients.  
The Littlewood--Richardson  rule was discovered as a rule of two halves (as we explain below).  
In this  paper  we succeed in  generalising one half of  this rule to all Kronecker tableaux,  and thus solve one half of the stable Kronecker problem.  
 Our main result unifies and vastly generalises the work of Littlewood--Richardson \cite{LR34} and many other authors \cite{RW94,Rosas01,ROSAANDCO,BWZ10,MR2550164}.  
% Most promisingly, our main result explicitly constructs the corresponding  homomorphisms 
%and thus  works on a structural level above any description of a family of Kronecker coefficients   since those first considered by Littlewood--Richardson over eighty years ago \cite{LR34}.  
 Most promisingly, our   result counts  explicit   homomorphisms 
and thus  works on a structural level above any description of a family of Kronecker coefficients   since those first considered by Littlewood--Richardson over eighty years ago  \cite{LR34}.

In more detail,  given a triple of partitions $(\lambda,\nu, \mu)$  and with $|\mu|=s$, we have an associated 
  skew $P_s(n)$-module spanned by the Kronecker  tableaux  from $\lambda$ to $\nu$  of length $s$, which
  we denote by $\Delta_s(\nu \setminus\lambda )   $.  
  For $\lambda=\varnothing$ and $n\geq 2s$  these 
  modules provide a complete set of non-isomorphic $P_s(n)$-modules (and we drop the 
  partition $\varnothing$ from the notation).    
  The stable Kronecker coefficients are then interpreted as the dimensions, 
\begin{equation}\label{dagger}\tag{$\dagger$}
\overline{g}(\lambda,\nu,\mu)
=  \dim_\CC( \Hom_{  P_{s}(n)}( \Delta_{s}(\mu), \Delta_s(\nu \setminus\lambda )    ) )    
\end{equation}for $n\geq 2s$.  
Restricting to the Young subgraph, or equivalently  to a triple     $(\lambda,\nu,\mu)$  of so-called {\sf maximal depth} such that   $|\lambda| + |\mu| = |\nu|$, 
 these modules specialise to be the usual simple and skew modules for the symmetric group and hence the multiplicities $\overline{g}(\lambda,\nu,\mu)$ are the Littlewood--Richardson coefficients $c(\lambda, \nu,\mu)$. Thus we naturally recover, in this context, the well-known fact that the Littlewood--Richardson coefficients appear as the subfamily of stable Kronecker coefficients labelled by triples of maximal depth.
 The tableaux   counted by the Littlewood--Richardson  rule satisfy two conditions: the {\sf semistandard} and the {\sf lattice permutation} conditions  \cite[(16.4)]{JAMBOOK}. Specialising the triple of partitions so that the latter, respectively former, condition is  satisfied for all tableaux,
  we obtain the two halves of the Littlewood--Richardson rule, namely the Pieri, respectively co-Pieri, rule.  

\begin{clcopieri}
Let $(\lambda, \nu,\mu)$ be a triple of partitions such that $\lambda\subseteq \nu$, $|\mu| = |\nu|-|\lambda|$ and the skew partition $\nu \ominus \lambda$ has no two boxes in the same column. Then the Littlewood--Richardson coefficient $c(\lambda,\nu,\mu)$ is given by the number of Young tableaux of shape $\nu\ominus \lambda$ and weight $\mu$ whose reverse reading word is a lattice permutation.
\end{clcopieri}
 \noindent The main purpose of 
 this article is to generalise the classical co-Pieri rule to 
 the stable  Kronecker coefficients.
\begin{thm*}
Let $(\lambda,\nu,\mu)$ be a    {co-Pieri} triple     or a triple of maximal depth. Then the stable Kronecker coefficient $\overline{g}(\lambda, \nu, \mu)$ is given by the number of  
 semistandard  Kronecker tableaux of shape $\nu\setminus\lambda$ and weight $\mu$ whose reverse reading word is a lattice permutation.   
 \end{thm*}
The observant reader will notice that the statement above describes the Littlewood--Richardson coefficients uniformly as part of a far broader family of stable Kronecker coefficients (and is the first result in the literature to do so). Whilst the classical Pieri rule is elementary, it served as a first step towards 
   understanding the full Littlewood--Richardson rule; indeed 
   Knutson--Tao--Woodward  have shown that  the  Littlewood--Richardson rule follows from the Pieri rule by associativity  \cite{taoandco}.  
    We  hope that our  generalisation of the   co-Pieri rule  will prove equally useful in the study of  stable   Kronecker coefficients. 

 The definition of {\sf semistandard Kronecker tableaux} naturally generalises the classical notion of semistandard Young tableaux as certain ``orbits" of paths on the branching graph given in \cref{brancher} (see Section 1.2 and Definition \ref{semistandard}). The {\sf lattice permutation condition} is identical to the classical case once we generalise the dominance order to all steps in the branching graph $\mathcal{Y}$ to define the reverse reading word of a semistandard Kronecker tableau  (see \cref{ordering} and Section 6).

\medskip

\noindent \textbf{Special cases of co-Pieri triples.} The definition of {\sf co-Pieri triples } is given in \cref{{co-Pieri}triple} and can appear quite technical at first reading and so we present a few special cases here.   We have included a further wealth of  examples of both {\em stable} Kronecker and {\em non-stable} Kronecker coefficients   in Section 7.

\begin{itemize}%[leftmargin=*]
\item[$(i)$] $\lambda$ and $\mu$ are one-row partitions and $\mu$ is arbitrary.
This family has been extensively studied over the past thirty years and 
there are   many 
 distinct  combinatorial descriptions of some or all of these coefficients %\cite{MR3338303,RW94,Rosas01,ROSAANDCO,BWZ10,MR2550164}
 \cite{MR3338303,RW94,Rosas01,ROSAANDCO,BWZ10,MR2550164}, none of which generalises.

\item[$(ii)$] the two skew partitions $\lambda \ominus (\lambda \cap \nu)$ and $\nu \ominus (\lambda \cap \nu)$ have no two boxes in the same column and \newline $|\mu| = \max \{|\lambda \ominus (\lambda \cap \nu)| , |\nu \ominus (\lambda \cap \nu)|\}$. 
 It is easy to see that if, in addition, $(\lambda, \nu, \mu)$ is a triple of maximal depth, then this case specialises to the classical co-Pieri triples.

\item[$(iii)$] $\lambda = \nu =  (dl,d(l-1), \ldots , 2d,d)$ for any $l,d\geq 1$ and $|\mu| \leq d$.
\end{itemize}

As already pointed out, our description covers the family
  of stable Kronecker  coefficients labelled by co-Pieri triples uniformly along with the Littlewood--Richardson coefficients.
   In order to demonstrate the uniformity of our approach, we  now illustrate how to calculate 
 $\overline{g}((2,1),(3,3,2), (2,2,1))=1$ and 
  $\overline{g}((4),(5), (2,2,1))=1$.  The former is an example of a triple of maximal depth (and so is calculated by the Littlewood--Richardson rule) and the latter is an example of a coefficient indexed by two one-row partitions.  In both cases, there is a unique semistandard Kronecker tableau whose reverse reading word is a lattice permutation (under the dominance ordering on Kronecker tableaux). 
Each of these semistandard tableaux is an orbit consisting of four individual 
standard Kronecker tableaux.
These tableaux are pictured in \cref{anewfigforintro}: notice that 
$\lambda$ and $\nu$ appear at the top and bottom of the diagram in
 \cref{anewfigforintro}
and that the partition $\mu$ determines the   orbit --- which we depict as a dashed series of rectangular frames.  
This is explained in detail  
Sections  1, 2, 5 and 6 of the paper (but 
%as there is only a single semistandard Kronecker tableau in each case,  
we hope this lightly sketched   example helps the reader).   
We have    included a third example in \cref{anewfigforintro} of a co-Pieri triple  as in $(ii)$, to help the reader get a more general picture (the corresponding stable Kronecker coefficient is calculated in Section 7).

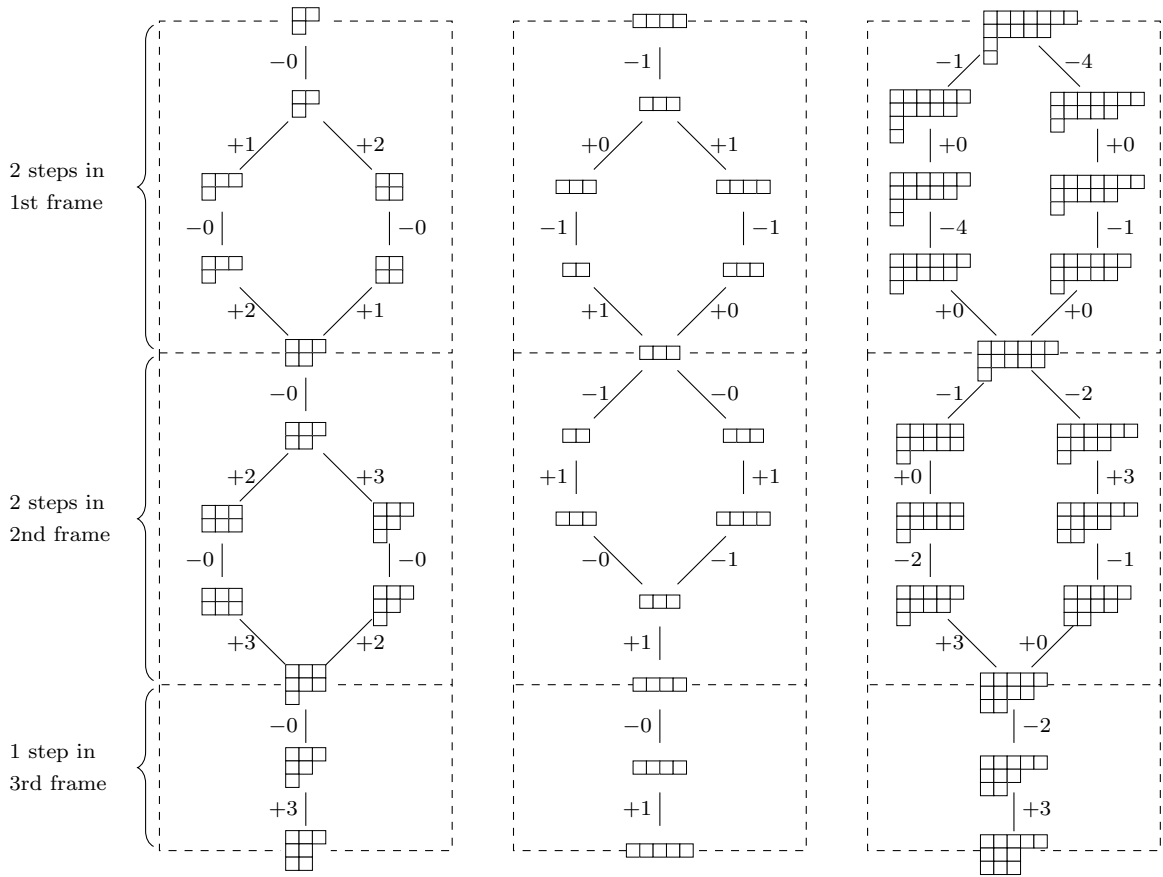
\begin{figure}[ht!]
\scalefont{0.8}  
    %%%%    %%%%    %%%%    %%%%    %%%%    %%%%    %%%%    %%%%    %%%%    %%%%    %%%%    %%%%    %%%%    %%%%    %%%%    %%%%    %%%%    %%%%    %%%%    %%%%    %%%%
    \begin{tikzpicture}[scale=0.55]
  \draw [decorate,decoration={brace,amplitude=6pt},xshift=6pt,yshift=0pt]   (-3.85,-7.9)--   (-3.85,-0.1)  node [right,black,midway,yshift=-0.2cm,xshift=-2cm]{\text{1st frame}} 
  node [right,black,midway,yshift=0.2cm,xshift=-2cm]{\text{2 steps in}} ; %{\footnotesize $P_1$};
  
  \draw [decorate,decoration={brace,amplitude=6pt},xshift=6pt,yshift=0pt]   (-3.85,-7.9-8)--   (-3.85,-0.1-8)  node [right,black,midway,yshift=-0.2cm,xshift=-2cm]{\text{2nd frame}} 
  node [right,black,midway,yshift=0.2cm,xshift=-2cm]{\text{2 steps in}} ; %{\footnotesize $P_1$};

  \draw [decorate,decoration={brace,amplitude=6pt},xshift=6pt,yshift=0pt]   (-3.85,-7.9-8-4)--   (-3.85,-0.1-8-8)  node [right,black,midway,yshift=-0.2cm,xshift=-2cm]{\text{3rd frame}} 
  node [right,black,midway,yshift=0.2cm,xshift=-2cm]{\text{1 step in}} ; %{\footnotesize $P_1$};

  \clip(-3.6,0.5) rectangle (3.6,-20.6);
    \path   (0,0)edge[decorate]  node[left] {$-0$}  (0,-2);
          \path   (0,-2)edge[decorate]  node[right] {$+2$}  (2,-4);
                    \path   (0,-2)edge[decorate]  node[left] {$+1$}  (-2,-4);

                    \path   (-2,-6)edge[decorate]  node[left] {$-0$}  (-2,-4);                         
          \path   (2,-6)edge[decorate]  node[right] {$-0$}  (2,-4);
                                       
                  \path   (-2,-6)edge[decorate]  node[left] {$+2$}  (0,-8);                         
          \path   (2,-6)edge[decorate]  node[right] {$+1$}  (0,-8);

       \path   (0,0-8)edge[decorate]  node[left] {$-0$}  (0,-2-8);
          \path   (0,-2-8)edge[decorate]  node[right] {$+3$}  (2,-4-8);
                    \path   (0,-2-8)edge[decorate]  node[left] {$+2$}  (-2,-4-8);

                    \path   (-2,-6-8)edge[decorate]  node[left] {$-0$}  (-2,-4-8);                         
          \path   (2,-6-8)edge[decorate]  node[right] {$-0$}  (2,-4-8);
                                       
                  \path   (-2,-6-8)edge[decorate]  node[left] {$+3$}  (0,-8-8);                         
          \path   (2,-6-8)edge[decorate]  node[right] {$+2$}  (0,-8-8);

       \path   (0,-16)edge[decorate]  node[left] {$-0$}  (0,-18);
       \path   (0,-20)edge[decorate]  node[left] {$+3$}  (0,-18);

              \draw[dashed] (-3.5,0) rectangle (3.5,-8);                   
                \draw[dashed] (-3.5,-8)--(-3.5,-16)--(3.5,-16)--(3.5,-8) ;  
                
                      \draw[dashed] (-3.5,-16)--(-3.5,-20)--(3.5,-20)--(3.5,-16) ;  
                  \fill[white] (0,0) circle (17pt);   
  \begin{scope}   
 \fill[white] (0,0) circle (17pt);
     \draw (0,0) node {$  \Yboxdim{5pt}\gyoung(;;,;) $  };   
     %%%%%
\fill[white] (0,-2) circle (17pt);    \draw (0,-2) node{$  \Yboxdim{5pt}\gyoung(;;,;)$  };
%%%%%
\fill[white] (-2,-4) circle (17pt);  \draw (-2,-4) node{$  \Yboxdim{5pt}\gyoung(;;;,;) $  };
\fill[white] (2,-4) circle (17pt);  \draw (2,-4) node{$  \Yboxdim{5pt} \gyoung(;;,;;) $  };
%%%%%
  \fill[white] (-2,-6) circle (17pt);  \draw (-2,-6) node{$  \Yboxdim{5pt}\gyoung(;;;,;)$  };
\fill[white] (2,-6) circle (17pt);  \draw (2,-6) node{$  \Yboxdim{5pt}\gyoung(;;,;;) $  };
  \fill[white] (0,-8) circle (17pt);  \draw (0,-8) node{$  \Yboxdim{5pt}\gyoung(;;;,;;)$  };
%%%
%%%
  \fill[white] (0,-16) circle (17pt);  \draw (0,-16) node{$  \Yboxdim{5pt}\gyoung(;;;,;;;,;)$  };
      %%%%%
\fill[white] (0,-2-8) circle (17pt);    \draw (0,-2-8) node{$  \Yboxdim{5pt}\gyoung(;;;,;;)$  };
%%%%%
\fill[white] (-2,-4-8) circle (17pt);  \draw (-2,-4-8) node{$  \Yboxdim{5pt}\gyoung(;;;,;;;) $  };
\fill[white] (2,-4-8) circle (17pt);  \draw (2.1,-4-8.1) node{$  \Yboxdim{5pt} \gyoung(;;;,;;,;) $  };
%%%%%
  \fill[white] (-2,-6-8) circle (17pt);  \draw (-2,-6-8) node{$  \Yboxdim{5pt}\gyoung(;;;,;;;)$  };
\fill[white] (2,-6-8) circle (17pt);  \draw (2.1,-6-8.1) node{$  \Yboxdim{5pt} \gyoung(;;;,;;,;)$  };

%%%%%
  \fill[white] (0,-18) circle (17pt);  \draw (0,-18) node{$  \Yboxdim{5pt}\gyoung(;;;,;;,;)$  };
\fill[white] (0,-20) circle (17pt);  \draw (0,-20) node{$  \Yboxdim{5pt} \gyoung(;;;,;;,;;)$  };

      \end{scope}
   \end{tikzpicture} 
\qquad 
    \begin{tikzpicture}[scale=0.55]
      \clip(-3.6,0.5) rectangle (3.6,-20.6);
    \path   (0,0)edge[decorate]  node[left] {$-1$}  (0,-2);
          \path   (0,-2)edge[decorate]  node[right] {$+1$}  (2,-4);
                    \path   (0,-2)edge[decorate]  node[left] {$+0$}  (-2,-4);

                    \path   (-2,-6)edge[decorate]  node[left] {$-1$}  (-2,-4);                         
          \path   (2,-6)edge[decorate]  node[right] {$-1$}  (2,-4);
                                       
                  \path   (-2,-6)edge[decorate]  node[left] {$+1$}  (0,-8);                         
          \path   (2,-6)edge[decorate]  node[right] {$+0$}  (0,-8);

       \path   (0,0-8)edge[decorate]  node[left] {$-1$}  (-2,-2-8);
              \path   (0,0-8)edge[decorate]  node[right] {$-0$}  (2,-2-8);

          \path   (2,-2-8)edge[decorate]  node[right] {$+1$}  (2,-4-8);
                    \path   (-2,-2-8)edge[decorate]  node[left] {$+1$}  (-2,-4-8);

                    \path   (0,-6-8)edge[decorate]  node[left] {$-0$}  (-2,-4-8);                         
          \path   (0,-6-8)edge[decorate]  node[right] {$-1$}  (2,-4-8);
                                       
                  \path   (0,-6-8)edge[decorate]  node[left] {$+1$}  (0,-8-8);

       \path   (0,-16)edge[decorate]  node[left] {$-0$}  (0,-18);
       \path   (0,-20)edge[decorate]  node[left] {$+1$}  (0,-18);

              \draw[dashed] (-3.5,0) rectangle (3.5,-8);                   
                \draw[dashed] (-3.5,-8)--(-3.5,-16)--(3.5,-16)--(3.5,-8) ;  
                
                      \draw[dashed] (-3.5,-16)--(-3.5,-20)--(3.5,-20)--(3.5,-16) ;  
                  \fill[white] (0,0) circle (17pt);   
  \begin{scope}   
 \fill[white] (0,0) circle (18pt);
     \draw (0,0) node {$  \Yboxdim{5pt}\gyoung(;;;;) $  };   
     %%%%%
\fill[white] (0,-2) circle (17pt);    \draw (0,-2) node{$  \Yboxdim{5pt}\gyoung(;;;)$  };
%%%%%
\fill[white] (-2,-4) circle (17pt);  \draw (-2,-4) node{$  \Yboxdim{5pt}\gyoung(;;;) $  };
\fill[white] (2,-4) circle (17pt);  \draw (2,-4) node{$  \Yboxdim{5pt} \gyoung(;;;;) $  };
%%%%%
  \fill[white] (-2,-6) circle (17pt);  \draw (-2,-6) node{$  \Yboxdim{5pt}\gyoung(;;)$  };
\fill[white] (2,-6) circle (17pt);  \draw (2,-6) node{$  \Yboxdim{5pt}\gyoung(;;;) $  };
  \fill[white] (0,-8) circle (17pt);  \draw (0,-8) node{$  \Yboxdim{5pt}\gyoung(;;;)$  };
%%%
%%%
  \fill[white] (0,-16) circle (17pt);  \draw (0,-16) node{$  \Yboxdim{5pt}\gyoung(;;;;)$  };
      %%%%%
\fill[white] (-2,-2-8) circle (17pt);    \draw (-2,-2-8) node{$  \Yboxdim{5pt}\gyoung(;;)$  };
\fill[white] (2,-2-8) circle (17pt);    \draw (2,-2-8) node{$  \Yboxdim{5pt}\gyoung(;;;)$  };
%%%%%
\fill[white] (-2,-4-8) circle (17pt);  \draw (-2,-4-8) node{$  \Yboxdim{5pt}\gyoung(;;;) $  };
\fill[white] (2,-4-8) circle (17pt);  \draw (2,-4-8) node{$  \Yboxdim{5pt} \gyoung(;;;;) $  };
%%%%%
  \fill[white] (0,-6-8) circle (17pt);  \draw (0,-6-8) node{$  \Yboxdim{5pt}\gyoung(;;;)$  };

%%%%%
  \fill[white] (0,-18) circle (17pt);  \draw (0,-18) node{$  \Yboxdim{5pt}\gyoung(;;;;)$  };
\fill[white] (0,-20) circle (17pt);  \draw (0,-20) node{$  \Yboxdim{5pt} \gyoung(;;;;;)$  };

      \end{scope}
   \end{tikzpicture} 
   \qquad 
    \begin{tikzpicture}[scale=0.55]
      \clip(-3.6,0.5) rectangle (3.6,-20.6);
    \path   (0,0)edge[decorate]  node[left] {$-1$}  (-2,-2);
    \path   (0,0)edge[decorate]  node[right] {$-4$}  (2,-2);

          \path   (2,-2)edge[decorate]  node[right] {$+0$}  (2,-4);
                    \path   (-2,-2)edge[decorate]  node[right] {$+0$}  (-2,-4);

                    \path   (-2,-6)edge[decorate]  node[right] {$-4$}  (-2,-4);                         
          \path   (2,-6)edge[decorate]  node[right] {$-1$}  (2,-4);
                                       
                  \path   (-2,-6)edge[decorate]  node[left] {$+0$}  (0,-8);                         
          \path   (2,-6)edge[decorate]  node[right] {$+0$}  (0,-8);

       \path   (0,0-8)edge[decorate]  node[left] {$-1$}  (-2,-2-8);
              \path   (0,0-8)edge[decorate]  node[right] {$-2$}  (2,-2-8);

          \path   (2,-2-8)edge[decorate]  node[right] {$+3$}  (2,-4-8);
                    \path   (-2,-2-8)edge[decorate]  node[left] {$+0$}  (-2,-4-8);

                    \path   (-2,-6-8)edge[decorate]  node[left] {$-2$}  (-2,-4-8);                         
          \path   (2,-6-8)edge[decorate]  node[right] {$-1$}  (2,-4-8);
                                       
                  \path   (-2,-6-8)edge[decorate]  node[left] {$+3$}  (0,-8-8);                         
                  \path   (2,-6-8)edge[decorate]  node[left] {$+0$}  (0,-8-8);

       \path   (0,-16)edge[decorate]  node[right] {$-2$}  (0,-18);
       \path   (0,-20)edge[decorate]  node[right] {$+3$}  (0,-18);

              \draw[dashed] (-3.5,0) rectangle (3.5,-8);                   
                \draw[dashed] (-3.5,-8)--(-3.5,-16)--(3.5,-16)--(3.5,-8) ;  
                
                      \draw[dashed] (-3.5,-16)--(-3.5,-20)--(3.5,-20)--(3.5,-16) ;  
                  \fill[white] (0,0) circle (17pt);   
  \begin{scope}   
 \fill[white] (0,0.1) circle (26pt); \fill[red] (-.2,-0.1) circle (7pt);
 \fill[white] (-0.6,-0.7) circle (7pt);    \draw (0.4,-0.4) node {$  \Yboxdim{5pt}\gyoung(;;;;;;;,;;;;;,;,;) $  };   
     %%%%%
\fill[white] (2,-2) circle (17pt);    \draw (2,-2.2) node{$  \Yboxdim{5pt}\gyoung(;;;;;;;,;;;;;,;)$  };
\fill[white] (-2,-2) circle (17pt);    \draw (-2,-2.3) node{$  \Yboxdim{5pt}\gyoung(;;;;;;,;;;;;,;,;)$  };
%%%%%
\fill[white] (2,-4) circle (17pt);    \draw (2,-4.2) node{$  \Yboxdim{5pt}\gyoung(;;;;;;;,;;;;;,;)$  };
\fill[white] (-2,-4) circle (17pt);    \draw (-2,-4.3) node{$  \Yboxdim{5pt}\gyoung(;;;;;;,;;;;;,;,;)$  };
%%%%%
  \fill[white] (-2,-6.1) circle (18pt);    \draw (-2,-6.1) node{$  \Yboxdim{5pt}\gyoung(;;;;;;,;;;;;,;)$  };
\fill[white] (2,-6.1) circle (18pt);    \draw (1.84,-6.1) node{$  \Yboxdim{5pt}\gyoung(;;;;;;,;;;;;,;)$  };
  \fill[white] (0,-8) circle (17pt);  
    \fill[white] (0-0.5,-8-0.5) circle (9pt);  \draw (0.1,-8.2) node{$  \Yboxdim{5pt}\gyoung(;;;;;;,;;;;;,;)$  };
%%%
%%%
  \fill[white] (0,-16) circle (17pt);  \draw (0,-16.2) node{$   \Yboxdim{5pt} \gyoung(;;;;;,;;;;,;;)$  };
      %%%%%
\fill[white] (-2,-2-8) circle (17pt);    \draw (-2,-2-8.2) node{$  \Yboxdim{5pt} \gyoung(;;;;;,;;;;;,;)$  };
\fill[white] (2,-2-8) circle (17pt);    \draw (2,-2-8.2) node{$  \Yboxdim{5pt}\gyoung(;;;;;;,;;;;,;)$  };
%%%%%
\fill[white] (-2,-4-8) circle (17pt);  \draw (-2,-4-8.1) node{$  \Yboxdim{5pt} \gyoung(;;;;;,;;;;;,;)$  };
\fill[white] (2,-4-8) circle (17pt);  \draw (2,-4-8.1) node{$    \Yboxdim{5pt}\gyoung(;;;;;;,;;;;,;;) $  };
%%%%%
\fill[white] (-2,-4-8-2) circle (17pt);  \draw (-2,-4-8.1-2) node{$  \Yboxdim{5pt} \gyoung(;;;;;,;;;;,;)$  };
\fill[white] (2,-4-8.1-2) circle (18pt);  \draw (2,-4-8.1-2) node{$    \Yboxdim{5pt}\gyoung(;;;;;,;;;;,;;) $  };

%%%%%
  \fill[white] (0,-18) circle (17pt);  \draw (0,-18.2) node{$    \Yboxdim{5pt}\gyoung(;;;;;,;;;,;;) $  };
\fill[white] (0,-20.2) circle (19pt);  \draw (0,-20.1) node{$    \Yboxdim{5pt}\gyoung(;;;;;,;;;,;;;) $  };

      \end{scope}
   \end{tikzpicture} 
   \qquad  \   \qquad \  \qquad

   \caption{Three examples of semistandard Kronecker tableaux of weight $\mu=(2,2,1)$. 
   The number of steps in the $i$th  frame is   $\mu_i$.  
    The first  is a triple of maximal depth, the latter two are co-Pieri triples.
   }
   \label{anewfigforintro}
\end{figure}

\!\!
 For $\lambda=(2,1)$ and $\nu=(3,2,2)$, the (integral) steps taken in the semistandard tableau on the left of \cref{anewfigforintro} are 
 to add a box in the first row, add two boxes in the second row, and two in the third row
$$a(1)=(-0,+1) \quad a(2)=(-0,+2)  \quad  a(2)=(-0,+2)  \quad  a(3)=(-0,+3)  \quad  a(3)=(-0,+3).$$ 
 We record the  steps according to the dominance ordering for Kronecker tableaux  
 ($a(1)< a(2) < a(3)$)  and then we refine this by recording the  frames   in which these steps  occur in weakly decreasing fashion,  
  as follows 
$$
\left(\begin{array}{c|cc|ccccccccc}
a(1)&a(2)&a(2)&a(3) &a(3) 
%a(1)&\leq&a(2)&\leq&a(2)&\leq&a(3)\leq&a(3)
\\
1& 2& 1 & 3 &2
\end{array}\right).$$
This should be very familiar to experts, who will also recognise that the resulting word
 is a lattice permutation.  
  For $\lambda=(4)$ and $\nu=(5)$, the steps taken in the semistandard Kronecker tableau
  in the middle  of \cref{anewfigforintro} are to remove a box from the first row, do two  ``dummy" steps  in the first row, and add   two boxes  in the first row
$$r(1)= (-1,+0) \quad  d(1)=(-1,+1)\quad  d(1)=(-1,+1)\quad  a(1)=(-0,+1)\quad  a(1)=(-0,+1).$$
  We record the  steps according to the  dominance ordering for Kronecker tableaux 
 ($r(1)< d(1) < a(1)$)  and we  refine  this by recording the  frames   in which these steps  occur backwards,   
$$
\left(\begin{array}{c|cc|ccccccc}
r(1)&d(1)&d(1)&a(1) &a(1) 
%a(1)&\leq&a(2)&\leq&a(2)&\leq&a(3)\leq&a(3)
\\
1& 2& 1 & 3 &2
\end{array}\right)$$
 and notice that the second row is again a lattice permutation (and identical to the previous example!).

\medskip

\noindent\textbf{Structure of the paper.} 
  In Section 1 we recall  the classical tableaux combinatorics
 of the Littlewood--Richardson rule; we re-cast the notion of a semistandard tableau  in a manner 
 which will be generalisable from the symmetric group to the partition algebra setting.  
 We then recall some well-known facts concerning Kronecker coefficients which will be used in what follows.  
  In Section 2, we define a {\sf standard Kronecker  tableau} of shape $\nu \setminus \lambda$ 
to be a path from $\lambda$ to $\nu$  in the branching graph of the partition algebra.  
For triples of maximal depth, our definition specialises to the usual definition of (skew) Young tableaux.

    In Sections  3 and 4 we describe the action of the partition algebra on skew cell modules   of shape $\nu \setminus \lambda$ in the case  of co-Pieri triples.  That we can understand the action of the partition algebra in this case is the crux of this paper.   
 This is definitely the most difficult and technical section of the paper and we  {\em strongly} encourage the  reader to skip these two sections on the first reading.  
 The rest of the paper is entirely readable without this material, if one 
 is willing to either   lift the definition of co-Pieri triple from Theorem 4.12 or 
  temporarily  restrict their attention to the examples of co-Pieri triples $(\lambda,\nu,\mu)$  listed above.  
%  of maximal depth and 
%  those such that $\lambda$ and $\nu$ are one-row partitions.
 
In Section 5, we define a {\sf semistandard Kronecker tableau of shape $\nu \setminus \lambda$ and weight $\mu$} to be an orbit of standard Kronecker tableaux under the action of the corresponding Young subgroups $\mathfrak{S}_\mu$.
 For a triple of partitions of {maximal depth},
  our construction specialises to the usual definition of 
 semistandard Young tableaux.  
    In the case that $(\lambda,\nu,\mu)$ is a   {{co-Pieri} triple}   we are able to provide 
 an elegant combinatorial description of these semistandard Kronecker tableaux.  
 
  In Section 6, using an   ordering 
 on the steps in the branching graph of the partition algebra  
we define the  {\sf reverse reading word}  of  a semistandard Kronecker tableau.  
   We hence extend the classical 
  lattice permutation condition to
  semistandard Kronecker tableaux.   
  When
 $(\lambda,\nu,\mu)$ is 
 a co-Pieri triple of partitions,  we show that the corresponding 
 stable   Kronecker coefficient is equal to the number of semistandard Kronecker  tableaux  whose reverse reading word is a lattice permutation, generalising the Littlewood--Richardson rule to give  the co-Pieri rule for stable Kronecker coefficients.
Section 7 is dedicated to providing examples of Kronecker coefficients which can be calculated using our main theorem.

\section{The Littlewood--Richardson   and Kronecker coefficients}  
     The combinatorics underlying the representation theory of the partition algebras and   symmetric 
     groups  is based on compositions and   partitions.  
A {\sf  composition}    $\lambda $ of $n$, denoted $\lambda \vDash n$, is   a  sequence   of non-negative integers which  sum to $n$.  
  If the sequence is weakly decreasing, we write  $\lambda \vdash n$
  and refer to $\lambda$ as a {\sf  partition} of $n$.  We let 
  $\mathscr{P}_n$ denote the set of all partitions of $n$.
      We let $\varnothing$  denote the unique partition of 0.  
 Given a partition, $\lambda=(\lambda _1,\lambda _2,\dots )$, the  associated   {\sf Young diagram}  is the set of nodes
\[[\lambda]=\left\{(i,j)\in\mathbb{Z}_{>0}^2\ \left|\ j\leq \lambda_i\right.\right\}.\]
We define the length, $\ell(\lambda)$, of a partition $\lambda$, to be the number of non-zero parts.
 Given $\lambda =(\lambda _1,\lambda _2,\dots )$,   
we let $[ \lambda]_a = \sum_{i\geq 1}^a \lambda_i$ for $a \in \ZZ_{>0}$ and 
$|\lambda | = 
\sum_{i \geq 1}\lambda _i$.  
We  formally set $[ \lambda]_0=0$.  
Given  two partitions $\lambda, \mu$  we write $\lambda \trianglerighteq \mu$  
  if  $|\lambda|<|\mu|$ or if
   $|\lambda|=|\mu|$  and  $[\lambda]_a \geq [\mu]_a$ for all $a \in \ZZ_{> 0}$.

Given  $\lambda = (\lambda_1,\lambda_2, \ldots,\lambda_{\ell} )$  a partition and $n$   an integer, define 
  $$\lambda_{[n]}=(n-|\lambda|, \lambda_1,\lambda_2, \ldots,\lambda_{\ell}).$$  
 Given $\lambda_{[n]} $ a partition of $n$, we say that the partition has {\sf depth} equal to $|\lambda|$.  
   Given two compositions  $\lambda$ and $\nu$, we write $\lambda \subseteq \nu$ if $\lambda_i \leq \nu_i$ for all $i\geq 1$.  
For $\lambda$ a partition and $\nu$ a partition  (respectively $\nu$ a composition) 
 such that $\lambda \subseteq \nu$, we define   the {\sf skew partition}
(respectively {\sf skew composition})
   denoted $\nu  \ominus \lambda $,  to be  the set
difference between the Young diagrams of $\lambda $ and $\nu $.
  We write  $\nu\ominus \lambda \vdash s$ if $\nu\ominus \lambda$ is a skew partition of $s$.
  More generally, for two arbitrary compositions  $\lambda$ and $\nu$ we have that 
  $\lambda \cap \nu \subseteq \lambda,\nu$ and so we  let 
   $\nu\ominus \lambda$ denote the union of 
  $\nu \ominus (\lambda \cap \nu)$ and   $\lambda \ominus (\lambda \cap \nu)$.

\subsection{Young tableaux combinatorics and Littlewood--Richardson rule}
\label{LRstatement}

  Given   $\lambda \vdash  {r-s} , \nu \vdash  {r}$    such that  $\lambda \subseteq \nu$ we define a {\sf standard 
  Young tableau} of shape $\nu\ominus \lambda$  to be  a filling of the boxes of the Young diagram, $[\nu \ominus\lambda]$, with the entries $1,\dots, s$
 in such a way that   the entries  are  increasing along the rows  and columns of $[\nu \ominus\lambda]$.  

 \begin{eg}\label{frickingheck}
 The six standard  Young tableaux of shape $(5,3,1) \ominus (4,2)$  are depicted in \cref{hello mister}.

\!\!\!\!\begin{figure}[ht!]
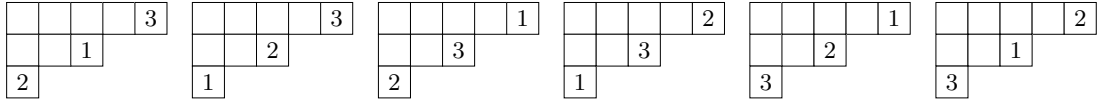
$$
 \Yboxdim{12pt}\scalefont{0.9}
  \Yvcentermath1\gyoung(;;;;;3,;;;1,;2)   \quad
    \Yvcentermath1\gyoung(;;;;;3,;;;2,;1)    \quad
      \Yvcentermath1\gyoung(;;;;;1,;;;3,;2) 
      \quad
      \Yvcentermath1\gyoung(;;;;;2,;;;3,;1)\quad
      \Yvcentermath1\gyoung(;;;;;1,;;;2,;3) 
      \quad
      \Yvcentermath1\gyoung(;;;;;2,;;;1,;3)   $$
      \caption{
The        standard  Young tableaux $\sts_1,\sts_2,\stt_1,\stt_2,\stu_1,\stu_2$  of shape  $(5,3,1) \ominus (4,2)$.
      }
\label{hello mister}
\end{figure}
 \end{eg}
\!\!\!\!

     Given   $\lambda \vdash  {r-s} , \nu \vdash  {r}, \mu = (\mu_1, \mu_2, \ldots , \mu_\ell ) \vDash s $
   such that  $\lambda \subseteq \nu$ we define a {\sf    Young tableau of shape $\nu\ominus \lambda$ and weight $\mu$} to be  a filling of the boxes of 
   $[\nu\ominus \lambda]$  with the entries 
    $$\underbrace{1, \dots, 1}_{\mu_1}, \underbrace{2,\dots, 2}_{\mu_2}, 
  \ldots,   \underbrace{\ell ,\dots, \ell }_{\mu_\ell } $$
   in such a way that    the entries  are weakly increasing along the rows and columns.  
   We say that the  Young tableau is {\sf semistandard} if, in addition, the entries 
   are strictly increasing along the columns of $\nu\ominus \lambda$.    
 In the case that  $\lambda \vdash {r-s}, \nu \vdash {r} $ and   $\mu=(1^s)$, we note that these can be identified with the set of   {\sf standard  Young tableaux} of shape $\nu\ominus\lambda$ in an obvious fashion.

 One should think of a   Young tableau of weight $\mu$ as 
 an $\mathfrak{S}_\mu$-orbit  of   standard  Young tableaux;   we shall now make this idea more  precise.  
  Let $\sts$ be a standard   Young tableau of shape $\nu\ominus \lambda$    and let $\mu$  be a composition. Then define $\mu(\sts)$ to be  the  Young tableau of weight $\mu$
 obtained from $\sts$ by replacing each of the entries
  $[\mu]_{c-1} < i \leq [\mu]_c$ in $\sts$ by the entry $c$  for  $  c \geq  1$.  
We identify a   Young tableau, $\SSTS $, of weight $\mu$ with the set of standard  Young tableaux, $\mu^{-1}(\SSTS)=\{\sts \mid \mu(\sts)=\SSTS\}$. 
  The  set $\mu^{-1}(\SSTS)$ forms the basis of   a cyclic $\mathfrak{S}_\mu$-module with   generator given by any element $\sts \in  \mu^{-1}(\SSTS)$ (see \cite[Chapter 4]{Mathasbook} for more details).

\begin{eg}

The three semistandard  Young tableaux of shape $(5,3,1) \ominus (4,2)$ and weight $(2,1)$ are depicted in \cref{anexampleofLRclassic}.  We have that $\mu(\sts_1)=\mu(\sts_2)=\SSTS$,  
$\mu(\stt_1)=\mu(\stt_2)=\SSTT$, 
and  $\mu(\stu_1)=\mu(\stu_2)=\SSTU$. 
In each case, the non-trivial element $s_1\in \mathfrak{S}_{(2,1)}\leq \mathfrak{S}_3$ acts by permuting these pairs of  Young tableaux (and therefore acts trivially on the orbits sums in each case).  

\!\!\!\!\begin{figure}[ht!]
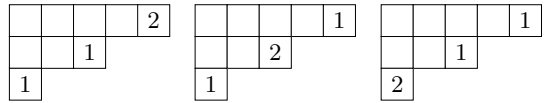
$$
 \Yboxdim{12pt}\scalefont{0.9}
  \Yvcentermath1\gyoung(;;;;;2,;;;1,;1)   \quad
      \Yvcentermath1\gyoung(;;;;;1,;;;2,;1) 
      \quad
      \Yvcentermath1\gyoung(;;;;;1,;;;1,;2) 
      $$
      \caption{
The semistandard  Young tableaux  $\SSTS, \SSTT, \SSTU$ of shape  $(5,3,1) \ominus (4,2)$
and weight $(2,1)$.  
    }
\label{anexampleofLRclassic}
\end{figure}
 \end{eg}

 \begin{eg}\label{readerexample}
 An example of a  semistandard  Young tableaux, $\SSTS$,  of shape $(9,8,6,3) \ominus (6,4,3)$ and weight $(5,5,3)$ 
 is given by the leftmost  Young tableau depicted in \cref{hello mister3}. 
    Two standard Young tableaux,  $\sts$ and  $\stt$, of shape $(9,8,6,3) \ominus (6,4,3)$ are depicted in \cref{hello mister3}.  For $\mu=(5,5,3)$, we have that
 $\mu(\sts)=\mu(\stt)=\SSTS$.  

\!\!\!\!\begin{figure}[ht]
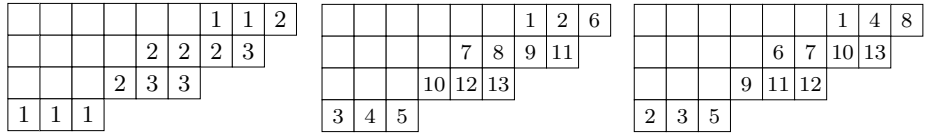

$$
   \Yboxdim{12pt}\scalefont{0.9}  \Yvcentermath1\young(\ \ \ \ \ \ 112,\ \ \ \ 2223,\ \ \ 233,111)     
 \quad 
  \Yboxdim{12pt}\scalefont{0.9}
  \Yvcentermath1\gyoung(;;;;;;;1;2;6,;;;;;7;8;9;{\eleven},;;;;{10};{12};{13},;3;4;5)   \quad
  \Yvcentermath1\gyoung(;;;;;;;1;4;8,;;;;;6;7;{10};{13},;;;;{9};{\eleven};{12},;2;3;5) 
   $$
     \caption{Three semistandard  Young tableau of shape $(9,8,6,3) \ominus (6,4,3)$.
The first, $\SSTS$, is of weight $(5,5,3)$ and the second, $\sts$, and third, $\stt$,    are standard  Young tableaux. 
  }
\label{hello mister3}
\end{figure}
 \end{eg}

\begin{defn}\label{Classical row reading}  Given 
     a  semistandard Young tableau of shape $\nu\ominus \lambda$ and weight $\mu$, we define the {\sf $\mu$-reverse reading word}  to be 
  the sequence of integers  
  obtained by reading the entries of the  Young tableau  from   right-to-left along successive rows (beginning with the first row).
\end{defn}
  
\begin{eg}\label{followreaderexample}
  The  $(1^3)$-reverse reading words  of the standard  Young tableaux in \cref{frickingheck} are
  $$
(3,1,2)\quad (3,2,1)\quad  (1,3,2) \quad (2,3,1) \quad (1,2,3)
\quad (2,1,3)
  $$
  respectively.     The  $(5^2,3)$-reverse reading word  of the semistandard   Young tableau $\SSTS$ in \cref{readerexample} is  
  $
 (2,1,1,3,2,2,2,3,3,2,1,1,1).  
  $

\end{eg} 
 
The representation theory of the symmetric group $\mathfrak{S}_r$ over the rational field $\mathbb{Q}$ is semisimple. 
For each $\nu \vdash r$, we have a corresponding 
{\sf  Specht module}
$\mathbf{S}(\nu )$ which has a basis indexed by all standard Young tableaux of shape $\nu$.
The set $\{\mathbf{S}(\nu ) \mid \nu \in \mathscr{P}_r\}$  forms 
a complete set of non-isomorphic simple   $\CC\mathfrak{S}_{r}$-modules.   
 More generally, for $s\leq r$ and  $ \lambda \vdash r-s$ with $\lambda\subseteq \nu$,  we have a corresponding {\sf skew Specht module}
$\mathbf{S}(\nu\ominus \lambda)$ for $\mathbb{Q}\mathfrak{S}_s$ which has a basis indexed by standard Young tableaux of shape $\nu \ominus \lambda$  \cite{PEEL}.

\begin{thm}\label{LRclassic} \cite{JAM}
Let $\lambda \vdash {r-s}$, $\mu  \vdash {s}$
 and $\nu    \vdash {r}$ and suppose that $\lambda \subseteq \nu$. We define the Littlewood--Richardson coefficients to be the multiplicities,
   \begin{equation} 
  c({\lambda,\nu,\mu})=  
   \dim_\CC\Hom_{    \mathfrak{S}_{r-s}\times    \mathfrak{S}_{s}}
 ( \mathbf{S}(\lambda)\boxtimes \mathbf{S}( \mu) ,
  \mathbf{S}(\nu){\downarrow}_{\mathfrak{S}_{r-s}\times    \mathfrak{S}_{s}}^{\mathfrak{S}_r}     )
  = \dim_\CC\Hom_{    \mathfrak{S}_{s}}
 ( \mathbf{S}(\mu),
  \mathbf{S}(\nu\ominus \lambda)    ). 
  \end{equation}\label{LRcoeff}
   The Littlewood--Richardson coefficient, $c({\lambda,\nu,\mu})$,  is equal to the number of    Young tableaux of shape $\nu\ominus \lambda$ and weight $\mu$ satisfying the following two conditions, 
\begin{enumerate} 
\item    the  Young tableau is  {\sf semistandard};
 \item   the $\mu$-reverse reading word of the  Young tableau is a {\sf lattice permutation}, that is, 
 for each positive integer $j$,
 starting from the first entry of the word 
  to any other place in word, there are at least as many entries equal to   $j$  as there are equal to  $(j+1)$.
\end{enumerate}
\end{thm}

\begin{eg}
The  Young tableau of shape $(9,8,6,3) \ominus (6,4,3)$ and weight $(5,5,3)$   depicted in \cref{hello mister3} is semistandard but its $(5,5,3)$-reverse reading word is not a lattice permutation. 
\end{eg}

 \begin{eg}\label{copierieg}
The  three  Young tableaux of shape $(5,3,1)\ominus (4,2)$ and 
weight $(2,1)$ are  depicted in \cref{anexampleofLRclassic}. 
Only the latter two of these  Young tableaux satisfy condition $(2)$ of \cref{LRclassic}.  
Therefore $c({(5,3,1)},{(4,2), (2,1)})=2$.  
   \end{eg}

A famous pre-cursor to the full Littlewood--Richardson rule was provided by {\sf Pieri's rule}.  
  In this case, we  assume that the weight partition $\mu = (s)$.  This  
 is equivalent to 
 all  Young tableaux of   weight $\mu$ (and any arbitrary fixed shape) 
  satisfying condition $(2)$ of \cref{LRclassic}. 
 Therefore the following rule, while elementary,   serves as a first step towards understanding 
condition $(1)$ of \cref{LRclassic}.

\begin{thm}[The Pieri  rule for Littlewood--Richardson coefficients] \label{PIERITULE}
Let  $\lambda\vdash {r-s}$ and $\nu\vdash r$ be such that  
$\lambda\subseteq \nu$.  We have that 
  $$
\dim_{\CC}\Hom_{\CC\mathfrak{S}_s}(\mathbf{S}((s)), \mathbf{S}(\nu\ominus \lambda))
 $$
is equal to the number of {\sf semistandard  Young tableaux} of shape $\nu\ominus\lambda$ and weight $(s)$.  
 The number of such  Young tableaux is equal to   1 (respectively 0)  if $\nu$ is (respectively is not) obtained from 
  $\lambda$ by adding a total of  $s$ nodes,
   no two of which appear in the same column.  
 \end{thm}

We now consider a dual to the above case, which we refer to as the {\sf co-Pieri} rule.  
Here we assume that the Young diagram of 
$\nu \ominus \lambda$ consists of no two nodes 
  in the same column.  
 This is equivalent to 
 all  Young tableaux of  shape $\nu\ominus \lambda$ (and any arbitrary fixed weight) 
  satisfying condition $(1)$ of \cref{LRclassic}. 
 Therefore the following 
 rule  serves   as a first step towards understanding condition $(2)$ of \cref{LRclassic}.  
 
\begin{thm}[The Co-Pieri rule for Littlewood--Richardson coefficients]\label{piere!}
Suppose that $  \lambda\subseteq \nu $ and that 
$\nu \ominus \lambda$ is a skew partition of $s$ with no two nodes in the same column. We have that 
  $$ c(\lambda, \nu, \mu) = 
\dim_{\CC}\Hom_{\CC\mathfrak{S}_s}( \mathbf{S}(\mu) , \mathbf{S}(\nu\ominus \lambda))
 $$
is equal to the number of {\sf  Young tableaux} of shape $\nu\ominus\lambda$ and weight $\mu$ whose reverse reading word is a {\sf lattice permutation}.
  \end{thm}

To reiterate,  \cref{PIERITULE} describes  precisely the set of Littlewood--Richardson coefficients which can be calculated without 
 mention of  the lattice permutation condition; whilst  
\cref{piere!} describes  precisely the set of Littlewood--Richardson coefficients which can be calculated without 
 mention of  the  semistandardness condition.

\subsection{ Young tableaux combinatorics revisited}
In the next section,  we shall see that the Littlewood--Richardson coefficients appear as a subfamily of the wider class of (stable) Kronecker coefficients.  
The purpose of this paper is to generalise the
   combinatorics  of standard and semistandard  Young tableaux from this subclass 
to the  study of all (stable) Kronecker coefficients.  
In order to illustrate how we shall proceed, we first recast the pictorial  Young tableaux described earlier in the setting of the branching graph of the symmetric groups.

The {\sf branching graph} of the symmetric groups   encodes the induction and restriction of Specht modules for the tower of symmetric groups.
   For $k\in \NN$, the set of vertices on the $k$th   level  are given by 
    the set of partitions of $k$.   There is an edge  $\lambda  \to \mu $ 
if $\mu $ is  obtained from $\lambda $ by adding  a box in the $i$th row for some $i\geq 1$ 
 in which case we write $\mu =\lambda+ \varepsilon_i$.  
The first few levels
of  this graph are given in   Figure \ref{asdhjkdsf}.

\begin{figure}[ht!]
$$\begin{tikzpicture}[scale=0.5]
          \begin{scope}  
     \draw (0,3) node {   $\varnothing $  };   
              \draw (0,0) node { $  \Yboxdim{6pt}\gyoung(;) $  };   
    \draw (-2,-3) node   {  $ \Yboxdim{6pt}\gyoung(;;)	$}		; 
       \draw (+2,-3) node   {   $ \Yboxdim{6pt}\gyoung(;,;)$	}		;
       \draw (0,-6) node    {     \Yboxdim{6pt}\gyoung(;;,;)}		;
      \draw (-4,-6) node    {    \Yboxdim{6pt}\gyoung(;;;)}		;    
        \draw (4,-6) node    {    \Yboxdim{6pt}\gyoung(;,;,;)}		;
     \draw[<-] (0.0,1) -- (0,2);  
                          \draw [<-]  (-1.5,-2) -- (-0.5,-1);   
                                   \draw[->] (0.5,-1)-- (1.5,-2);  
                          \draw[<-] (-0.5,-8.2+3) -- (-1.5,-6.8+3) ;
                         \draw[->] (1.5,-6.8+3) -- (0.5,-8.2+3);
  \draw[<-] (-3.5,-8.2+3)  -- (-2.5,-6.8+3) ;  
     \draw[->] (2.5,-6.8+3) -- (3.5,-8.2+3);   
             \end{scope}\end{tikzpicture}
$$

\!\!\!\!\!\caption{The first few  levels of the branching graph of  the symmetric groups.  }
\label{asdhjkdsf}
\end{figure}
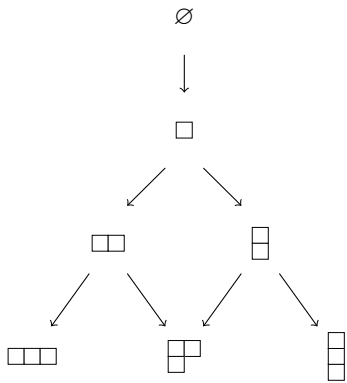

One can then identify any skew standard   Young tableau of shape $\nu\ominus \lambda$  with a  path from $\lambda$ to $\nu$ in the branching graph; this is done simply by adding nodes in the prescribed  order.  This is best illustrated via an example.  

\begin{eg}\label{hello mister2}
 Let   
  $\lambda =(4,2)   $ and  $\nu =(5,3,1)  $.  
    We have six standard   Young tableaux of shape $\nu\ominus \lambda$.  
 Two of   these  Young tableaux  are as  follows:
$$
\sts_1 =   
 \left( \   \Yvcentermath1\Yboxdim{6pt}\gyoung(;;;;,;;) \xrightarrow{\ +\varepsilon_2\ }
   \Yvcentermath1\Yboxdim{6pt}\gyoung(;;;;,;;;)  \xrightarrow{\ +\varepsilon_3\ }
      \Yvcentermath1\Yboxdim{6pt}\gyoung(;;;;,;;;,;)  \xrightarrow{\ +\varepsilon_1\ }
      \Yvcentermath1\Yboxdim{6pt}\gyoung(;;;;;,;;;,;)  \      
 \right)$$
 $$ \sts_2 =   
 \left( \   \Yvcentermath1\Yboxdim{6pt}\gyoung(;;;;,;;)\xrightarrow{\ +\varepsilon_3\ }
   \Yvcentermath1\Yboxdim{6pt}\gyoung(;;;;,;;,;)  \xrightarrow{\ +\varepsilon_2\ }
      \Yvcentermath1\Yboxdim{6pt}\gyoung(;;;;,;;;,;) \xrightarrow{\ +\varepsilon_1\ }
      \Yvcentermath1\Yboxdim{6pt}\gyoung(;;;;;,;;;,;)  \      
 \right)  $$
 These paths correspond with the two leftmost  Young tableaux (also labelled by $\sts_1$ and $\sts_2$)   depicted in  \cref{hello mister}.  
 \end{eg}

We now wish to re-imagine  the notion of a  semistandard  Young tableaux in this setting.  Recall that a 
   Young tableau of weight $\mu$
  is merely a picture which  encodes an $\mathfrak{S}_\mu$-orbit of standard  Young tableaux. 
     We shall  
  picture a   Young tableau, $\SSTS$, of weight $\mu$   simply as the corresponding set of paths $\mu^{-1}(\SSTS)$ in the branching graph.  
 In order to highlight the weight of the    Young tableau, we shall decorate the graph with a corresponding series of frames.  An illustrative example is given in Figure \ref{anexampleofLRclassic2}.  
A  Young tableau is semistandard  (in the classical picture) if and only if the entries are strictly increasing along the columns;  equivalently   the successive differences between partitions on the edges of the frame have no two nodes in the same column.  
 While we have refrained from being too precise here, a more general definition of such a   tableau is made   in \cref{sec:semistandard}.

 \begin{figure}[ht!]
 \scalefont{0.8} 
$$ \begin{tikzpicture}[scale=0.6]
                    \fill[white] (0,0) circle (20pt);   
       \draw[dashed] (-2,0) rectangle (4.5,-7.5);  
                \draw[dashed] (-2,-5) -- (4.5,-5);  
                 \fill[white] (0,0) circle (20pt);   
                                  \fill[white] (0,-5) circle (22pt);   
                  \begin{scope}   
      \draw (0,0) node {$  \Yboxdim{5pt}\gyoung(;;;;,;;) $  };   
     %%%%%
   \draw (0,-2.5) node{$ \Yboxdim{5pt}\gyoung(;;;;,;;;)$  };
      \draw (3,-2.5) node{$ \Yboxdim{5pt}\gyoung(;;;;,;;,;)$  };
%%%%%
  \draw (0,-5) node{$ \Yboxdim{5pt}\gyoung(;;;;,;;;,;)$  };
 %%%%%
  \draw (0,-7.5) node{ \Yboxdim{5pt}\gyoung(;;;;;,;;;,;) };
      \path[->]   (0,-0.6)edge[decorate]  node[left] {$+2$}  (0,-1.9);
       \path[->]   (0.6,-0.6)edge[decorate]  node[auto] {$+3$}  (3,-1.7);
                    \path[->]   (0,-3.1)edge[decorate]  node[left] {$+3$}  (0,-4.4);
     \path[->]   (3,-3.15)edge[decorate]  node[right] {$\ +2$}  (0.6,-4.4);
                    \path[->]   (0,-5.6)edge[decorate]  node[left] {$+1$}  (0,-6.9);
    \end{scope}
   \end{tikzpicture}
   \quad \quad
\begin{tikzpicture}[scale=0.6]
                    \fill[white] (0,0) circle (20pt);   
       \draw[dashed] (-2,0) rectangle (4.5,-7.5);  
                \draw[dashed] (-2,-5) -- (4.5,-5);  
                 \fill[white] (0,0) circle (20pt);   
                                  \fill[white] (0,-5) circle (22pt);   
                  \begin{scope}   
      \draw (0,0) node {$  \Yboxdim{5pt}\gyoung(;;;;,;;) $  };   
     %%%%%
   \draw (0,-2.5) node{$ \Yboxdim{5pt}\gyoung(;;;;;,;;)$  };
      \draw (3,-2.5) node{$ \Yboxdim{5pt}\gyoung(;;;;,;;,;)$  };
%%%%%
  \draw (0,-5) node{$ \Yboxdim{5pt}\gyoung(;;;;;,;;,;)$  };
 %%%%%
  \draw (0,-7.5) node{ \Yboxdim{5pt}\gyoung(;;;;;,;;;,;) };
      \path[->]   (0,-0.6)edge[decorate]  node[left] {$+1$}  (0,-1.9);
       \path[->]   (0.6,-0.6)edge[decorate]  node[auto] {$+3$}  (3,-1.7);
                    \path[->]   (0,-3.1)edge[decorate]  node[left] {$+3$}  (0,-4.4);
     \path[->]   (3,-3.15)edge[decorate]  node[right] {$\ +1$}  (0.6,-4.4);
                    \path[->]   (0,-5.6)edge[decorate]  node[left] {$+2$}  (0,-6.9);
    \end{scope}
   \end{tikzpicture}
      \quad \quad
\begin{tikzpicture}[scale=0.6]
                    \fill[white] (0,0) circle (20pt);   
       \draw[dashed] (-2,0) rectangle (4.5,-7.5);  
                \draw[dashed] (-2,-5) -- (4.5,-5);  
                 \fill[white] (0,0) circle (20pt);   
                                  \fill[white] (0,-5) circle (22pt);   
                  \begin{scope}   
      \draw (0,0) node {$  \Yboxdim{5pt}\gyoung(;;;;,;;) $  };   
     %%%%%
   \draw (0,-2.5) node{$ \Yboxdim{5pt}\gyoung(;;;;;,;;)$  };
      \draw (3,-2.5) node{$ \Yboxdim{5pt}\gyoung(;;;;,;;;)$  };
%%%%%
  \draw (0,-5) node{$ \Yboxdim{5pt}\gyoung(;;;;;,;;;)$  };
 %%%%%
  \draw (0,-7.5) node{ \Yboxdim{5pt}\gyoung(;;;;;,;;;,;) };
      \path[->]   (0,-0.6)edge[decorate]  node[left] {$+1$}  (0,-1.9);
       \path[->]   (0.6,-0.6)edge[decorate]  node[auto] {$+2$}  (3,-1.7);
                    \path[->]   (0,-3.1)edge[decorate]  node[left] {$+2$}  (0,-4.4);
     \path[->]   (3,-3.15)edge[decorate]  node[right] {$\ +1$}  (0.6,-4.4);
                    \path[->]   (0,-5.6)edge[decorate]  node[left] {$+3$}  (0,-6.9);
    \end{scope}
   \end{tikzpicture}
   $$
      \caption{The 3 semistandard   Young tableaux $\SSTS, \SSTT, \SSTU$ of shape $ (4,3,1)\ominus(4,2)$ and weight $(2,1)$. % (compare with \cref{anexampleofLRclassic}). 
      }
   \label{anexampleofLRclassic2}
   \end{figure}
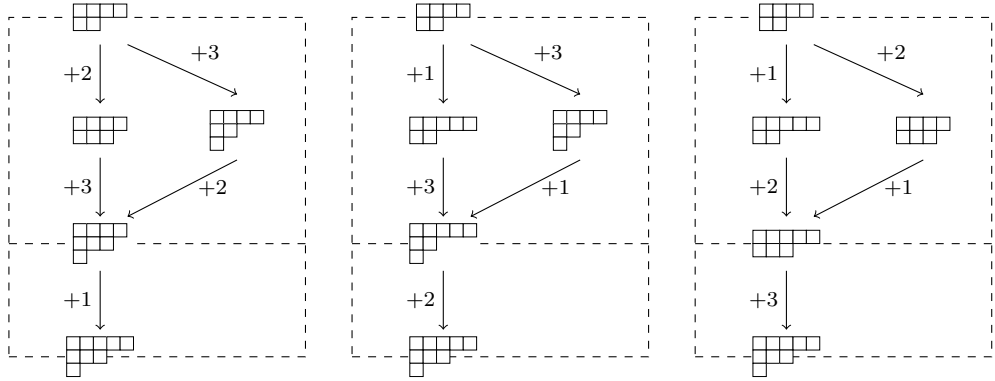

We leave the reinterpretation of the reverse reading word in this setting to Section 6.

\subsection{The Kronecker coefficients}
We  now  introduce the Kronecker coefficients  and illustrate how they generalise the Littlewood--Richardson coefficients discussed above.
Given  $\lambda ,\mu, \nu    \vdash  r$ we define the associated  {\sf Kronecker coefficient} to be the multiplicity
 \begin{align*} 
  g({\lambda,\nu},{\mu})  
&= \dim_\CC(\Hom_{\mathfrak{S}_r}(\Specht(\nu ), \Specht(\lambda )\otimes \Specht(\mu ))).   
 \end{align*} 
  For  $\alpha \subset \lambda$  ,$\beta \subset \mu$ with $|\lambda \ominus \alpha|=|\mu\ominus \beta| = s$ and  $ \nu    \vdash s $  we extend this notation to skew Specht modules in the obvious way,
 \begin{align*} 
  g( \lambda\ominus \alpha  ,\nu,\mu \ominus \beta)  
&= \dim_\CC(\Hom_{\mathfrak{S}_s}(\Specht(\nu ), \Specht(\lambda \ominus \alpha  )\otimes \Specht(\mu \ominus \beta))).  %\\
  \end{align*}
Given $\lambda=(\lambda_1,\lambda_2,\ldots) $ a partition and $n$ sufficiently large, we set $\lambda_{[n]}:=(n-|\lambda|,\lambda_1,\lambda_2,\ldots)$.   
It was discovered  by Murnaghan  in \cite{Mur38} that the sequence  of integers $ \{{g}(\lambda_{[n]} , \mu_{[n]},\nu_{[n]})\}_{n \in \NN} $ 
stabilises as $n\gg 0$    with stable limit $ \overline{g}(\lambda ,  \nu,  \mu)$.  
The  multiplicities 
 \begin{align*} 
   \overline{g}(\lambda ,  \nu,  \mu)   
&= \dim_\CC(\Hom_{\mathfrak{S}_n}(\Specht(\nu_{[n]} ), \Specht(\lambda_{[n]} )\otimes \Specht(\mu_{[n]} ))) 
 \end{align*}
 for $n\gg0$   are known as the    {\sf stable Kronecker coefficients}. Murnaghan also observed that 
\begin{equation}\label{Mbound}
\overline{g}(\lambda, \nu, \mu)\neq 0 \quad \mbox{ implies  $|\mu|\leq |\lambda| + |\nu|$, $|\nu|\leq |\lambda|+|\mu|$  and $|\lambda|\leq |\mu|+|\nu|$.}
\end{equation}
 The (stable) Kronecker coefficients have been studied extensively  (see for example \cite{Mur38,Mur55,brion,ky,MR1725703}).
 Recent work   \cite{ROSAANDCO,stab3,BDO15} has shown that the
 stable  Kronecker coefficients 
 can serve as an important stepping stone towards
  understanding the general case.

  The search for a positive combinatorial formula of the Kronecker coefficients 
 has been described by Richard  Stanley as  `one of the main problems in the combinatorial representation theory of the symmetric group', \cite{Sta99}. 
  While this is a very difficult problem,  there are many useful descriptions of the Kronecker coefficients which \emph{do} involve cancellations; chief among these is the following recursive description.   
%  of the Kronecker coefficients,  due to Dvir.  
% 

\begin{thm}  \cite[2.3]{Dvir}\label{sadfadsffdssdffsdfsdsfdsfafdsfdsasfdfdasdfsdfsafsadfsadfdsafsda}.
Given $  \lambda_{[n]},  \mu_{[n]}, \nu_{[n]} \vdash n$ such that $|\mu|=s$, we have that
\begin{equation}\label{ThDvir2}     g({\lambda_{[n]}, {\nu _{[n]}},\mu_{[n]}})
=
 \sum_{
 \begin{subarray}c \alpha \vdash n-s \\ \alpha \subseteq \lambda_{[n]} \cap \nu_{[n]}
 \end{subarray}}
g({\lambda_{[n]} \ominus \alpha, \nu_{[n]}\ominus\alpha},{ {\mu}})
 -  \sum_{
  \begin{subarray}c 
  \beta \in P(n, \mu) \\ 
  \beta \neq \mu_{[n]}
\end{subarray}
  }
  g({\lambda_{[n]}},{\nu_{[n]}},\beta)   \end{equation}
   where $P(n, \mu) $ is the set of partitions of $n$ obtained by adding 
  a total of $n-s$ boxes   to $\mu$ so that no two of which are in the same column.  In particular, if 
  $s<|\lambda_{[n]}\ominus (\lambda_{[n]}\cap \nu_{[n]})|$ then 
  $ g({\lambda_{[n]}, {\nu _{[n]}},\mu_{[n]}})=0$ 
  and if 
$s=  |\lambda_{[n]}\ominus (\lambda_{[n]}\cap \nu_{[n]})|$ then
\begin{equation} \label{ThDvir1}
 g({\lambda_{[n]}, {\nu _{[n]}},\mu_{[n]}})
=  g(\lambda_{[n]}\ominus (\lambda_{[n]}\cap \nu_{[n]}), \nu_{[n]}\ominus (\lambda_{[n]}\cap \nu_{[n]}),  \mu ).
  \end{equation}
\end{thm}

\begin{cor}\label{boundsons} Let $\lambda, \nu, \mu$ be partitions with $\overline{g}(\lambda, \nu, \mu) \neq 0$. Then we have
$$\max\{ |\lambda \ominus (\lambda \cap \nu)| , |\nu \ominus (\lambda \cap \nu)|\} \leq |\mu| \leq |\lambda| + |\nu|.$$
\end{cor}

\begin{proof}
This follows directly from (\ref{Mbound}) and \cref{sadfadsffdssdffsdfsdsfdsfafdsfdsasfdfdasdfsdfsafsadfsadfdsafsda}, noting that $\max\{ |\lambda \ominus (\lambda \cap \nu)| , |\nu \ominus (\lambda \cap \nu)|\} = |\lambda_{[n]}\ominus (\lambda_{[n]}\cap \nu_{[n]})|$.
\end{proof}
  
  Finally  we conclude this section by realising the Littlewood--Richardson coefficients as a subset of the wider family of stable Kronecker coefficients.

  \begin{defn} Let $\lambda, \nu, \mu$ be partitions. We say that $(\lambda, \nu, \mu)$ is a {\sf triple of partitions of maximal depth} if $|\nu| = |\lambda| + |\mu|$.
      We also call
       $(\lambda,\nu,s)$   a triple of
 of  maximal depth if $|\nu| = |\lambda| + s$.
 \end{defn}
 
\begin{thm}\label{LitMurn} \cite{Littlewood, Mur55}
For  $(\lambda,\nu,\mu)$  a triple of partitions   of maximal depth,  
 $  \overline{g}( \lambda, \nu, \mu)=
c( \lambda, \nu, \mu)$. 
 \end{thm}

   % 
%\begin{eg} Take $\lambda=\mu=(1)$ and $\nu \in \mathscr{P}_{\leq 2}$.  
%We have the following tensor products of Specht modules:
%\begin{align*}   
%\Specht(1^2) \otimes \Specht(1^2) &= \Specht(2) \\
%\Specht(2,1) \otimes \Specht(2,1) &= \Specht(3) \oplus \Specht(2,1) \oplus \Specht(1^3) \\
%\Specht(3,1) \otimes \Specht(3,1) &= \Specht(4) \oplus \Specht(3,1)\oplus \Specht(2,1^2) \oplus \Specht(2^2)  
%\intertext{at which point the product stabilises, in other words  for all $n\geq4$, we have}
%\Specht(n-1,1) \otimes \Specht(n-1,1) &= \Specht(n) \oplus \Specht(n-1,1)\oplus \Specht(n-2,1^2) \oplus \Specht(n-2,2).
% \end{align*}
%  \end{eg}

\section{The partition algebra and   Kronecker tableaux }\label{sec2} We  now define the partition algebra ${P}_r(n)$ for  $r,n \in \mathbb{N}$. Although it can be defined over any field, in this paper we consider $P_r(n)$ over the rational field $\mathbb{Q}$.  As a vector space, it has a basis given by  all set-partitions of  $\{1,2,\dots, r, \overline{1},\overline{2}, \dots, \overline{r}\}.$  
We call a part of a set-partition a {\sf  block}.
For example,
$$d=\{\{\overline1, \overline2, \overline4, {2}, {5}\}, \{\overline3\}, \{\overline5, \overline6, \overline7, {3}, {4}, {6}, {7}\}, \{\overline8, {8}\}, \{{1}\}\},$$
is a set-partition (for $r=8$) with 5  blocks.
To define the multiplication on $P_r(n)$, it is helpful to represent a  set-partition 
 by an  {\sf partition diagram} consisting of a frame with $r$ distinguished points on the northern and southern boundaries, which we call vertices.  We number the northern vertices from left to right by $\overline{1},\overline{2},\dots, \overline{r}$ and the southern vertices similarly by $1,2,\dots, r$  and connect two vertices by an edge if they belong to the same block.  Note that such a diagram is not uniquely defined, two diagrams representing the set-partition $d$ above are given in Figure \ref{2diag}.   
 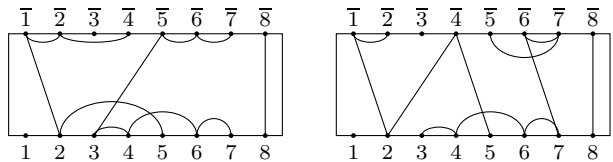
\begin{figure}[ht]\scalefont{0.8}
\begin{tikzpicture}[scale=0.45]
  \draw (0,0) rectangle (8,3);
  \foreach \x in {0.5,1.5,...,7.5}
    {\fill (\x,3) circle (2pt);
     \fill (\x,0) circle (2pt);}
  \begin{scope}%[very thick]
    \draw (0.5,3) -- (1.5,0);
    \draw (7.5,3) -- (7.5,0);
    \draw (4.5,3) -- (2.5,0);
    \draw (0.5,3) arc (180:360:0.5 and 0.25);
    \draw (1.5,3) arc (180:360:1 and 0.25);
%     \draw (1.5,3) arc (180:360:1.5 and 0.75);
    \draw (4.5,0) arc (0:180:1.5 and 1);
    \draw (5.5,0) arc (0:180:1 and .7);
    \draw (3.5,0) arc (0:180:.5 and .25);
    \draw (6.5,0) arc (0:180:0.5 and 0.5);
    \draw (4.5,3) arc (180:360:0.5 and 0.25);
    \draw (5.5,3) arc (180:360:0.5 and 0.25);
      \draw (2.5,-0.49) node {$3$};   \draw (2.5,+3.5) node {$\overline{3}$};
                  \draw (1.5,-0.49) node {$2$};   \draw (1.5,+3.5) node {$\overline{2}$};
                           \draw (0.5,-0.49) node {$1$};   \draw (0.5,+3.5) node {$\overline{1}$};
         \draw (3.5,-0.49) node {$4$};   \draw (3.5,+3.5) node {$\overline{4}$};
                  \draw (4.5,-0.49) node {$5$};   \draw (4.5,+3.5) node {$\overline{5}$};
                           \draw (5.5,-0.49) node {$6$};   \draw (5.5,+3.5) node {$\overline{6}$};
         \draw (6.5,-0.49) node {$7$};   \draw (6.5,+3.5) node {$\overline{7}$};         \draw (7.5,-0.49) node {$8$};   \draw (7.5,+3.5) node {$\overline{8}$};
%    \draw (6.5,3) -- (1.5,0);
%    \draw (7.5,0) arc (0:180:1.5 and 0.75);
%    \draw (5.5,3) -- (6.5,0);
%    \draw (7.5,3) -- (5.5,0);
  \end{scope}
\end{tikzpicture}
\quad \quad 
%%%%%
\begin{tikzpicture}[scale=0.45]
  \draw (0,0) rectangle (8,3);
  \foreach \x in {0.5,1.5,...,7.5}
    {\fill (\x,3) circle (2pt);
     \fill (\x,0) circle (2pt);}
  \begin{scope}%[very thick]
    \draw (0.5,3) -- (1.5,0);
    \draw (7.5,3) -- (7.5,0);
    \draw (5.5,3) -- (6.5,0);  \draw (1.5,0) -- (3.5,3);\draw (3.5,3) -- (4.5,0);
    \draw (0.5,3) arc (180:360:0.5 and 0.25);
  %  \draw (1.5,3) arc (180:360:1 and 0.25);
%     \draw (1.5,3) arc (180:360:1.5 and 0.75);
%    \draw (4.5,0) arc (0:180:1.5 and 1);
    \draw (5.5,0) arc (0:180:1 and .7);
    \draw (3.5,0) arc (0:180:.5 and .25);
    \draw (6.5,0) arc (0:180:0.5 and 0.5);
    \draw (4.5,3) arc (180:360:1 and 0.7);
    \draw (5.5,3) arc (180:360:0.5 and 0.25);
%    \draw (6.5,3) -- (1.5,0);
%    \draw (7.5,0) arc (0:180:1.5 and 0.75);
%    \draw (5.5,3) -- (6.5,0);
%    \draw (7.5,3) -- (5.5,0);
         \draw (2.5,-0.49) node {$3$};   \draw (2.5,+3.5) node {$\overline{3}$};
                  \draw (1.5,-0.49) node {$2$};   \draw (1.5,+3.5) node {$\overline{2}$};
                           \draw (0.5,-0.49) node {$1$};   \draw (0.5,+3.5) node {$\overline{1}$};
         \draw (3.5,-0.49) node {$4$};   \draw (3.5,+3.5) node {$\overline{4}$};
                  \draw (4.5,-0.49) node {$5$};   \draw (4.5,+3.5) node {$\overline{5}$};
                           \draw (5.5,-0.49) node {$6$};   \draw (5.5,+3.5) node {$\overline{6}$};
         \draw (6.5,-0.49) node {$7$};   \draw (6.5,+3.5) node {$\overline{7}$};         \draw (7.5,-0.49) node {$8$};   \draw (7.5,+3.5) node {$\overline{8}$};
  \end{scope}
\end{tikzpicture}
  \caption{Two representatives of the set-partition $d$.}
\label{2diag}
\end{figure}

We define the product $x \cdot y$ of two diagrams $x$
and $y$ using the concatenation of $x$ above $y$, where we identify
the southern vertices of $x$ with the northern vertices of $y$.   
If there are $t$ connected components consisting only of  middle vertices, then the product is set equal to $n^t$ times the diagram  
with the middle components removed. Extending this by linearity defines the multiplication on ${P}_r(n)$.
It is easy to see that $P_r(n)$ is generated (as an algebra) by the elements $s_{k,k+1}$, $p_{k+\half}$ ($1\leq k\leq r-1$) and $p_k$ ($1\leq k\leq r$) depicted in Figure \ref{generators}.

\begin{figure}[ht!]
$$\begin{array}{ll}
% \begin{align*}
{s}_{k,k+1}=
\begin{minipage}{34mm}\scalefont{0.8}\begin{tikzpicture}[scale=0.45]
  \draw (0,0) rectangle (6,3);
  \foreach \x in {0.5,1.5,...,5.5}
    {\fill (\x,3) circle (2pt);
     \fill (\x,0) circle (2pt);}
%    \draw (3.5,-0.49) node {$j$};
%    \draw (4.5,+3.5) node {$\overline{j}$};
    \draw (2.5,-0.49) node {$k$};
    \draw (2.5,+3.5) node {$\overline{k}$};
  \begin{scope}%[thick]
    \draw (0.5,3) -- (0.5,0);
    \draw (5.5,3) -- (5.5,0);
        \draw (2.5,3) -- (3.5,0);
            \draw (4.5,3) -- (4.5,0);
    \draw (3.5,3) -- (2.5,0);
    \draw (1.5,3) -- (1.5,0);
   \end{scope}
\end{tikzpicture}\end{minipage}    
%{p}_{i-\half}
p_k 
=\begin{minipage}{34mm}\scalefont{0.8}\begin{tikzpicture}[scale=0.45]
  \draw (0,0) rectangle (6,3);
  \foreach \x in {0.5,1.5,...,5.5}
    {\fill (\x,3) circle (2pt);
     \fill (\x,0) circle (2pt);}
   
    \draw (2.5,-0.49) node {$k$};
    \draw (2.5,+3.5) node {$\overline{k}$};
  \begin{scope}%[thick]
    \draw (0.5,3) -- (0.5,0);
    \draw (5.5,3) -- (5.5,0);
     \draw (4.5,0) -- (4.5,3); 
       
    \draw (1.5,3) -- (1.5,0);
    \draw (3.5,3) -- (3.5,0);
   \end{scope}
\end{tikzpicture}\end{minipage}  
%{p}_{i,i+1}
p_{k+\half}
=\begin{minipage}{34mm}\scalefont{0.8}\begin{tikzpicture}[scale=0.45]
  \draw (0,0) rectangle (6,3);
  \foreach \x in {0.5,1.5,...,5.5}
    {\fill (\x,3) circle (2pt);
     \fill (\x,0) circle (2pt);}
%    \draw (4.5,-0.49) node {$j$};
%    \draw (4.5,+3.5) node {$\overline{j}$};
    \draw (2.5,-0.49) node {$k$};
    \draw (2.5,+3.5) node {$\overline{k}$};
  \begin{scope}%[thick]
    \draw (0.5,3) -- (0.5,0);
    \draw (5.5,3) -- (5.5,0);
     \draw (1.5,3) -- (1.5,0);
\draw (3.5,0) arc (0:180:.5 and 0.5);
\draw (2.5,3) arc (180:360:.5 and 0.5);
    \draw (2.5,3) -- (2.5,0);
    \draw (4.5,3) -- (4.5,0);
   \end{scope}
\end{tikzpicture}\end{minipage}  
\end{array}$$

\!\!\!\!\!\caption{Generators of ${P}_r(n)$}
\label{generators}
\end{figure}
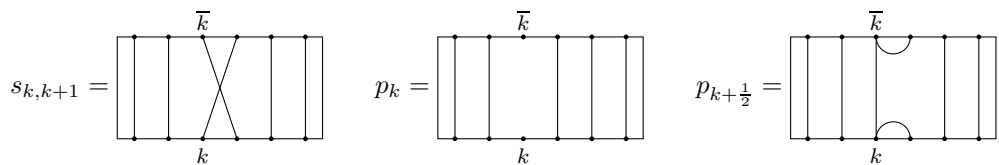

  \subsection{Standard Kronecker tableaux %and the Murphy basis
  }
  \label{sec:standard}

The {\sf branching graph}, $\mathcal{Y}$, of the partition algebras  encodes the induction and restriction of cell modules for the tower of partition algebras. We will construct the cell modules explicitely later in this section.

   For $k\in \NN$, we denote by $\mathscr{P}_{\leq k}$ the set of partitions of degree less or equal to $k$. Now the set of vertices on the $k$th and $(k+\half)$th levels of $\mathcal{Y}$ are given by 
    $$  
   {\mathcal{Y}}_{k}=  \{ (\lambda,k-|\lambda|) \mid \lambda \in \mathscr{P}_{\leq k}\}
   \qquad
   {\mathcal{Y}}_{k+\half} = 
\{ (\lambda,k-|\lambda|) \mid \lambda \in \mathscr{P}_{\leq k}\}.$$   The edges of $\mathcal{Y}$ are   as follows,
\begin{itemize}%[leftmargin=*] 
\item for $(\lambda,l) \in {\mathcal{Y}}_k$ and $(\mu,m) \in {\mathcal{Y}}_{k+\half}$ there is an edge $(\lambda,l) \to(\mu,m)$ 
 if $\mu = \lambda$,  or   
if $\mu $ is  obtained from $\lambda $ by removing a box in the $i$th row for some $i\geq 1$;
  we write $\mu =\lambda- \varepsilon_0$ or $\mu =\lambda- \varepsilon_i$, respectively. 
\item for $(\lambda,l)  \in  {\mathcal{Y}}_{k+\half}$ and $(\mu,m)  \in {\mathcal{Y}}_{k+1}$ there is an edge  $(\lambda,l) \to(\mu,m)$ 
if $\mu = \lambda$,  or   
if    $\mu $ is   obtained from $\lambda $ by adding a box in the $i$th row for some $i\geq 1$;
  we write $\mu =\lambda+ \varepsilon_0$ or $\mu =\lambda+ \varepsilon_i$, respectively. 
 \end{itemize}
When it is convenient, we  decorate each edge with the index of the node that is added or removed when reading down the diagram.  
The first few levels
of $\mathcal{Y}$  are given in   Figure \ref{brancher}.  
 When no confusion is possible, we identify $(\lambda,l) \in \mathcal{Y}_{k}$ with the partition $\lambda$.  

 \begin{defn}
Given   $\lambda  \in \mathscr{P}_{r-s} \subseteq \mathcal{Y}_{r-s}$ and $\nu \in \mathscr{P}_ {\leq r} \subseteq \mathcal{Y}_{r}$, we  
define a  {\sf  standard Kronecker tableau} of shape $ \nu \setminus  \lambda $ and degree $s$  to be a path $\stt$  of the form 
\begin{equation}\label{genericpath} 
\lambda = \stt(0) \to \stt(\tfrac{1}{2}) \to    \stt(1)\to  \dots \to  \stt(s-\tfrac{1}{2})\to \stt(s) = \nu, 
 \end{equation}
 in other words $\stt$ is  a path in the branching graph which begins at $\lambda$ and terminates at $\nu$.  
We let $\Std_s(\nu \setminus  \lambda)$ denote the set of all such paths.  If $\lambda = \emptyset \in \mathcal{Y}_0$ then we write $\Std_r(\nu)$ instead of $\Std_r(\nu \setminus \emptyset)$.

For $\lambda \in  \mathcal{Y}_{r-s}$, 
 $\nu  \in  \mathcal{Y}_r$, $\sts \in \Std_{r-s}(\lambda)$ and $\stt \in \Std_s(\nu\setminus\lambda)$, we denote the composition of these paths by  $\sts\circ \stt \in \Std_{r }(\nu)$. 
Also, for $\stt\in \Std_s(\nu\setminus \lambda)$ as in (\ref{genericpath}) and $0\leq m<m'\leq s$ we denote by $\stt[m,m']$ the truncation $\stt(m) \rightarrow \stt(m+\half) \rightarrow \dots \rightarrow \stt(m')$.

\end{defn}

Note that we have used the notation $\nu\setminus \lambda$, instead of $\nu \ominus \lambda$, as we do not have $\lambda \subseteq \nu$ in general.

\begin{rmk}
For  $(\lambda,\nu,s)$  a triple of maximal depth, the 
set $\Std_s(\nu\setminus\lambda)$ can be identified with the set of  standard skew Young tableau of shape $\nu \ominus \lambda$  for the symmetric group (see \cref{murphyexam1}  below).  
\end{rmk}

We now extend the dominance order on partitions to the set of standard Kronecker tableaux.

 \begin{defn}  
For $\sts,\stt \in \Std_s(\nu\setminus \lambda)$,  we write $\sts \trianglerighteq\stt$ if $ \sts(k) \trianglerighteq \stt(k) $  for $k=1,\ldots, s$.   
\end{defn}

\begin{eg}\label{murphyexam1}
 Let   
  $\lambda =(4,2)   $ and  $\nu =(5,3,1)  $.  
    We have six standard  Kronecker tableaux of shape $\nu\setminus \lambda$ and degree 3.   
 Two of   these tableaux  are as  follows:
$$
\sts_1 =   
 \left( \   \Yvcentermath1\Yboxdim{6pt}\gyoung(;;;;,;;) \xrightarrow{\ -\varepsilon_0\ }
  \Yvcentermath1\Yboxdim{6pt}\gyoung(;;;;,;;) 
  \xrightarrow{\ +\varepsilon_2\ }
   \Yvcentermath1\Yboxdim{6pt}\gyoung(;;;;,;;;) \xrightarrow{\ -\varepsilon_0\ }
   \Yvcentermath1\Yboxdim{6pt}\gyoung(;;;;,;;;) 
     \xrightarrow{\ +\varepsilon_3\ }
      \Yvcentermath1\Yboxdim{6pt}\gyoung(;;;;,;;;,;) \xrightarrow{\ -\varepsilon_0\ }
      \Yvcentermath1\Yboxdim{6pt}\gyoung(;;;;,;;;,;)    \xrightarrow{\ +\varepsilon_1\ }
      \Yvcentermath1\Yboxdim{6pt}\gyoung(;;;;;,;;;,;)  \      
 \right) $$
$$\sts_2 =   
 \left( \   \Yvcentermath1\Yboxdim{6pt}\gyoung(;;;;,;;) 
  \xrightarrow{\ -\varepsilon_0\ }
  \Yvcentermath1\Yboxdim{6pt}\gyoung(;;;;,;;)
   \xrightarrow{\ +\varepsilon_3\ }
   \Yvcentermath1\Yboxdim{6pt}\gyoung(;;;;,;;,;)  \xrightarrow{\ -\varepsilon_0\ }
   \Yvcentermath1\Yboxdim{6pt}\gyoung(;;;;,;;,;) 
      \xrightarrow{\ +\varepsilon_2\ }
      \Yvcentermath1\Yboxdim{6pt}\gyoung(;;;;,;;;,;)   \xrightarrow{\ -\varepsilon_0\ }
      \Yvcentermath1\Yboxdim{6pt}\gyoung(;;;;,;;;,;)  
         \xrightarrow{\ +\varepsilon_1\ }
      \Yvcentermath1\Yboxdim{6pt}\gyoung(;;;;;,;;;,;)  \      
 \right)  $$
%$$
%\sts_1 =   
% \left( \   \Yvcentermath1\Yboxdim{6pt}\gyoung(;;;;,;;) \  , \ 
%  \Yvcentermath1\Yboxdim{6pt}\gyoung(;;;;,;;) \ , \ 
%   \Yvcentermath1\Yboxdim{6pt}\gyoung(;;;;,;;;)  \   , \
%   \Yvcentermath1\Yboxdim{6pt}\gyoung(;;;;,;;;) \   , \
%      \Yvcentermath1\Yboxdim{6pt}\gyoung(;;;;,;;;,;)  \   , \   
%      \Yvcentermath1\Yboxdim{6pt}\gyoung(;;;;,;;;,;)   \   , \
%      \Yvcentermath1\Yboxdim{6pt}\gyoung(;;;;;,;;;,;)  \      
% \right) $$
%$$\sts_2 =   
% \left( \   \Yvcentermath1\Yboxdim{6pt}\gyoung(;;;;,;;) \  , \ 
%  \Yvcentermath1\Yboxdim{6pt}\gyoung(;;;;,;;) \ , \ 
%   \Yvcentermath1\Yboxdim{6pt}\gyoung(;;;;,;;,;)  \   , \
%   \Yvcentermath1\Yboxdim{6pt}\gyoung(;;;;,;;,;) \   , \
%      \Yvcentermath1\Yboxdim{6pt}\gyoung(;;;;,;;;,;)  \   , \   
%      \Yvcentermath1\Yboxdim{6pt}\gyoung(;;;;,;;;,;)   \   , \
%      \Yvcentermath1\Yboxdim{6pt}\gyoung(;;;;;,;;;,;)  \      
% \right)  $$
We remark that $\sts_1 \vartriangleright \sts_2  $.
 These paths correspond with the two leftmost Young tableaux (also labelled by $\sts_1$ and $\sts_2$)   depicted in  \cref{hello mister,hello mister2}.  
 \end{eg}

One can think of a   path $\stt \in \Std_s(\nu\setminus\lambda)$ as a sequence of partitions; or equivalently, as the 
  sequence of  
    boxes added 
 and removed.   We shall   refer to a  pair of steps, $(-\varepsilon_a,+\varepsilon_b)$,   between consecutive  integral levels of the branching graph    as an {\sf integral step} in the branching graph.
% The following ordering on integral steps is related to  the dominance order on paths.  
We place an ordering  on integral steps as follows.  
\begin{defn}\label{ordering}    
  We define {\sf  types} of integral step  (move-up, dummy, move-down) 
 in the branching graph of $P_r(n)$ and order them as follows, 
 $$\begin{array}{ccccccccc}
 	 &\text{move-up }			&  		 &\text{dummy }		&  		 &\text{move-down }	
&  		 
									\\
									
  		& (-\varepsilon_p, + \varepsilon_q)&< 		& (-\varepsilon_t, + \varepsilon_t)
&< 		&(-\varepsilon_u, + \varepsilon_v)  		 
\end{array}$$
for $p>q$ and $u< v$;   we refine this to a total  order as follows,
\begin{itemize}
\item[$({\up })$]  we order $(-\varepsilon_p, + \varepsilon_q)< (-\varepsilon_{p'}, + \varepsilon_{q'}) $ if $q<q'$ or $q=q'$ and $p>p'$; 
 \item[$(d)$]  we order $(-\varepsilon_t, + \varepsilon_t) < (-\varepsilon_{t'}, + \varepsilon_{t'}) $ if $t>t'$;
\item[$({\down })$]  we order $(-\varepsilon_u, + \varepsilon_v)< (-\varepsilon_{u'}, + \varepsilon_{v'})$  if $u>u'$ or $u=u'$ and $v<v'$. %if $i<k$ or $i=k$ and $j<l$;
\end{itemize}
    We sometimes let $a(i):=\down(0,i)$ (respectively $r(i):=\up(i,0)$) and think of this as  {\sf adding} (respectively  {\sf removing}) a box.  

\end{defn}

  \subsection{The Murphy basis}  We shall now recall from \cite{EG}  the construction of an integral   basis of the partition algebra   indexed by (pairs of) paths in the branching graph.  
  This basis captures much of the representation theoretic structure of $P_r(n)$ and  naturally generalises  Murphy's basis of $\mathbb{Z}\mathfrak{S}_r$ \cite{Murphy}.
  
\begin{defn}\label{coefficientsforthepartitionalgebra1}For $1\leq l \leq k\leq r$, we define elements of $P_r(n)$ as follows
\[
 {e}_{k}^{(l)}=\underbrace{{p}_{k-l+1}\cdots {p}_{k-1}{p}_{k}}_{\text{$l$ factors}}
 \qquad
{e}_{k+\half}^{(l)}=\underbrace{{p}_{k-l+\frac{3}{2}} \cdots {p}_{k-\half} \  {p}_{k+\half}}_{\text{$l$ factors}}
\qquad
s_{l,k}=  \underbrace{ {s}_{l } \cdots   {s}_{k-1 } }_{\text{$k-l$ factors}}.
\]
 If $k=0$ or $l= 0$, we let $e_{k}^{(l)}=e_{k-\half}^{(l)}=1$.   

  For $1\leq k< l \leq r$ we  let $s_{l,k}=s_{k,l}^{-1}$.  
  If $l=0$ or $k=0$, we let $s_{l,k}=1$ and if $l<0$ or $k<0$ we let $s_{l,k}=0$.
These elements are depicted  in Figure \ref{idem}.    \end{defn}

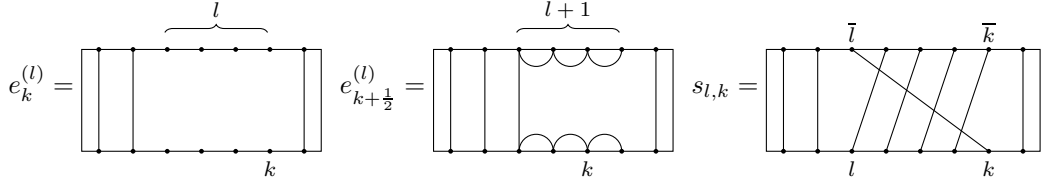
\begin{figure}[h!]  $$ 
 {e}_{k}^{(l)}
={\scalefont{0.8}
\begin{minipage}{34mm}\begin{tikzpicture}[scale=0.45]
\draw [decorate,decoration={brace,amplitude=3pt},xshift=-2pt,yshift=0pt] (2.5,3.5) -- (5.5,3.5) node [black,midway,yshift=0.3cm]{$l$} ; %{\footnotesize $P_1$};
  \draw[white] (0,-1) rectangle (7,4.8);
  \draw (0,0) rectangle (7,3);
  \foreach \x in {0.5,1.5,...,6.5}
    {\fill (\x,3) circle (2pt);
     \fill (\x,0) circle (2pt);}
%     \draw (5.5,3.5) node {$n$};
    \draw (5.5,-0.5) node { $ {k}$};
  \begin{scope}%[thick]
    \draw (0.5,3) -- (0.5,0);
          \draw (6.5,3) -- (6.5,0);
%     \draw (4.5,0) -- (4.5,3); 
    \draw (1.5,3) -- (1.5,0);
%    \draw (3.5,3) -- (3.5,0);
   \end{scope}
\end{tikzpicture}\end{minipage}  }
%{e}_{2k}^l
{e}_{k+\half}^{(l)}
={\scalefont{0.8}
\begin{minipage}{34mm}\begin{tikzpicture}[scale=0.45]
\draw [decorate,decoration={brace,amplitude=3pt},xshift=-2pt,yshift=0pt] (2.5,3.5) -- (5.5,3.5) node [black,midway,yshift=0.3cm]{$l+1$} ; %{\footnotesize $P_1$};
  \draw[white] (0,-1) rectangle (7,4.8);  \draw (0,0) rectangle (7,3);
  \foreach \x in {0.5,1.5,...,6.5}
    {\fill (\x,3) circle (2pt);
     \fill (\x,0) circle (2pt);}
%     \draw (5.5,3.5) node {$n$};
    \draw (4.5,-0.5) node { $ {k}$};
  \begin{scope}%[thick]
    \draw (0.5,3) -- (0.5,0);
    \draw (6.5,3) -- (6.5,0);
%     \draw (4.5,0) -- (4.5,3); 
    \draw (1.5,3) -- (1.5,0);
    \draw (2.5,3) -- (2.5,0);
 \draw (2.5,3) arc (180:360:0.5 and 0.5); \draw (3.5,3) arc (180:360:0.5 and 0.5);
 \draw (4.5,3) arc (180:360:0.5 and 0.5);
 \draw (2.5,0) arc (180:360:0.5 and -0.5);
 \draw (3.5,0) arc (180:360:0.5 and -0.5);
  \draw (4.5,0) arc (180:360:0.5 and -0.5);
   \end{scope}
\end{tikzpicture}\end{minipage}}
s_{l,k}=
{
\scalefont{0.8}
\begin{minipage}{34mm}\begin{tikzpicture}
[scale=0.45]
  \draw[white] (0,-1) rectangle (7,4.8);  \draw (0,0) rectangle (8,3);
  \foreach \x in {0.5,1.5,...,7.5}
    {\fill (\x,3) circle (2pt);
     \fill (\x,0) circle (2pt);}
  \begin{scope}%[very thick]
    \draw (0.5,0) -- (0.5,3);
    \draw (1.5,0) -- (1.5,3);   
     \draw (7.5,0) -- (7.5,3);   
      \draw (5.5,3) -- (4.5,0);  
      \draw (3.5,3) -- (2.5,0); 
       \draw (4.5,3) -- (3.5,0);  
       \draw (6.5,3) -- (5.5,0);  
    \draw (2.5,3) -- (6.5,0);   
         \draw (2.5,3.5) node {$\overline l$};   \draw (2.5,-0.5) node {$ {l}$};
                  \draw (6.5,3.5) node {$\overline k$};   \draw (6.5,-0.5) node {$ {k}$};
  \end{scope} 
\end{tikzpicture}\end{minipage}}
 $$

 \!\!\!\!\!
 \caption{The elements  $e_k^{(l)}$ and $e_{k+\frac{1}{2}}^{(l)}$ and  $s_{l,k}$.}
\label{idem}\label{sij}
\end{figure}

\begin{defn}\label{coefficientsforthepartitionalgebra}
Let $1 \leq k \leq r$ and  $\stt$ be  a standard Kronecker tableau of degree $s$  such that  
$$\stt(k) \xrightarrow{-a} \stt(k+\thalf)\xrightarrow{+b} \stt(k+1).$$
We set   $\stt(k)=\lambda$,  $\stt(k+\thalf)=\mu$, $\stt(k+1)=\nu$ and we  define
 the up  branching coefficients, 
$$\begin{array}{lll}
 u_{ \stt(k) \to  \stt(k+\half)} =
e_{k+\half}^{(k-|\mu|)}
s_{    |\lambda| , [\lambda]_a}
& \mbox{and} &
u_{ \stt(k+\half) \to  \stt(k+1)} = 
  e_{k+1}^{(k+1-|\nu|)}  \displaystyle \left(   \sum_{i=0}^{ \nu_b-1}s_{[\nu]_{b}-i, [\nu]_{b}}\right)
  s_{	 [\nu]_{b},|\nu|  	} 
\end{array}$$
and the down branching coefficients, 

$$\begin{array}{lll}
  d_{ \stt(k) \to  \stt(k+\half)} =
 e_{k}^{(k-|\lambda|)} 
\displaystyle 
\left(\sum_{i=0}^{ \lambda_a-1}s_{[\lambda]_{a}-i, [\lambda]_{a}}\right) s_{ [\lambda]_a,|\lambda|}
& \mbox{and} & \displaystyle 
d_{ \stt(k+\half) \to  \stt(k+1)} =e_{k+\half}^{(k-|\mu|)} 
 s_{|\nu|, [\nu]_b         }
.
\end{array}$$
 
\end{defn}

 \begin{defn}
 Given   $\nu \in \mathcal{Y}_r$   and 
   $\stt \in \Std_r(    \nu) $
 we    let 
\begin{align*}
d_\stt= 
d_{\stt(0) \to \stt(\half)} d_{\stt(\half) \to \stt(1)} 
\cdots
d_{ \stt(r-\half)\to \stt(r) }
\quad \mbox{and} \quad
u_\stt= 
u_{ \stt(r-\half)\to\stt(r)}
\cdots
u_{\stt(\half) \to \stt(1)} 
u_{\stt(0)\to\stt(\half)} . 
\end{align*}
\end{defn}

 \begin{thm}\cite{EG}\label{partition algebra basis}
  The algebra   $P_r(n)$   has an integral  basis  $$  \left\{
d_{\sts} u_{\stt} \mid 
 \sts, \stt \in \Std_r(   \nu),  
 \nu \in \mathscr{P}_{\leq r}
\right\} .$$ 
Moreover, if $\sts,\stt \in \Std_r(\nu)$  for some
      $\nu  \in\mathscr{P}_{\leq r}$, and $a\in  P_r(n) $ then 
    there exist scalars $r_{\stt\stu}(a)$, which do not depend on
    $\sts$, such that 
\begin{equation}\label{constants}
d_{\sts} u_{\stt}  a =\sum_{\stu\in
      \Std_r(\nu)}r_{\stt\stu}(a)d_{\sts} u_{\stu} \pmod 
      { P^{\vartriangleright \nu }_r(n) 	},
      \end{equation}
      where $P^{\vartriangleright \nu }_r(n) $ is the $\CC$-submodule of $P_r(n)$ spanned by
      \[\{d_{\sf q} u_{\sf r}\mid\mu \vartriangleright  \nu\text{ and }{\sf q , r}\in \Std_r(\mu )\}.\]
Finally,  we have that 
      $(d_{\sts} u_{\stt} )^*=d_{\stt} u_{\sts} $, for all $ \nu \in\mathscr{P}_{\leq r}$ and
      all $\sts,\stt\in\Std_r(\nu)$.
 Therefore the algebra is cellular, in the sense of \cite{GL96}.  
 \end{thm}
 
  \begin{rmk}
 The subalgebra spanned by $\left\{
d_{\sts} u_{\stt} \mid 
 \sts, \stt \in \Std_r(   \alpha),  
 \alpha \in \mathscr{P}_{\leq r-1} \subset   \mathscr{P}_{\leq r}
\right\} $ is equal to the 2-sided ideal generated by the element 
$p_{r}\in P_r(n)$ depicted in \cref{generators}.  
The resulting integral cellular structure on the quotient 
 $\mathbb{Q} \mathfrak{S}_r \cong  P_r(n) / P_r(n) p_r P_r(n)$  
 is the   basis of \cite{Murphy}.  
 \end{rmk}
  
  \begin{lem}\label{partition algebra basis2}
For any $\nu =(\nu_1,\ldots, \nu_\ell) \in \mathscr{P}_{\leq r}$, if we  take $\sts$ to be the Kronecker tableau of the form 
$$
 \underbrace{
a(1)
\circ 
\dots\circ 
a(1)}_{\nu_1}\circ
 \underbrace{
a(2)
\circ 
\dots\circ 
a(2)}_{\nu_2}\circ 
\dots \circ
\underbrace{
a(\ell)
\circ 
\dots\circ 
a(\ell)}_{\nu_\ell}\circ 
 \circ \underbrace{
d(0) \circ d(0) \circ 
\dots
\circ d(0) }_{r-|\nu|}
$$
then  for any $\stt \in \Std_r(\nu)$, we have that 
$$d_\sts u_\stt  = x_{(\lambda,r)} d_\stt^\ast =u_\stt 
$$ where    $ x_{(\lambda,r)}=
e^{(r-|\nu|)}_{r}  \sum_
{
g\in 
\mathfrak{S}_{\nu
%\{1,\dots,\nu_1\}  \times \{1,\dots,\nu_2\}  \times \dots
}
}g.$ 
%where $x_{(\lambda,r)}=
%e^{(r-|\nu|)}_{r}  \sum_
%{
%g\in 
%\mathfrak{S}_{
%\{1,\dots,\nu_1\}  \times \{1,\dots,\nu_2\}  \times \dots
%}
%}g$.
\end{lem}
\begin{proof}
We have that  $d_\sts= e_{r-1}^{(r-1-|\nu|)}
e_{r-\half}^{(r-1-|\nu|)}$.  Now, for any $\stt \in \Std_r(\nu)$, we have  
$$u_\stt = e_r^{r-|\nu|} 
  \sum_{g\in 
\mathfrak{S}_{\nu}}
g  d_\stt^\ast$$
by \cite[Lemma A.1]{EG} and \cite[Section 6]{EG}.  So we have
\begin{align*}
d_\sts u_\stt = e_{r-1}^{(r-1-|\nu|)}
e_{r-\half}^{(r-1-|\nu|)}
 e_r^{r-|\nu|} 
 \sum_{g\in 
\mathfrak{S}_{\nu}}
 g  d_\sts^\ast
 =
  e_r^{r-|\nu|} 
 \sum_{g\in 
\mathfrak{S}_{\nu}}
 g  d_\sts^\ast
 = u_\stt
\end{align*}
as required.  \end{proof}

 Thus, using  \cref{partition algebra basis,partition algebra basis2}  we can make the following definition.  
 \begin{defn}\label{murphytypebasisdefn}
%  This   cellular structure allows us to immediately define a natural family of so-called   \emph{standard modules} as follows.
     Given any $\nu\in\mathscr{P}_{\leq r}$, the    {\sf cell module} $\Delta_r (\nu)$ is the right $P_r(n)$-module
  with basis
    $\{ m_\stt = u_{\stt } + P^{\rhd \nu}_r(n) \mid \stt\in \Std_r(\nu )\}$.
    The action of $P_r(n)$ on $\Delta_r(\nu)$ is given by
    \[  m_{ \stt  }a  =\sum_{ \stu \in \Std_r(\nu)}r_{\stt\stu}(a) m_{ \stu  },\]
    where the scalars $r_{\stt\stu}(a)$ are the scalars appearing in  \cref{constants}.

 \end{defn}
 
 \begin{rmk}
 For $\nu \in \mathscr{P}_r \subseteq \mathcal{Y}_r$ the module $\Delta_r(\nu)$ is isomorphic to the Specht module $\mathbf{S}(\nu)$ of $\mathfrak{S}_r$ lifted to $P_r(n)$ via the isomorphism  $\mathbb{Q} \mathfrak{S}_r \cong  P_r(n) / P_r(n) p_r P_r(n)$.
 \end{rmk}
 
 \subsection{Skew cell modules} \label{fdasjkhfdsakjhsadfkjhdfsakjhdfs}  In what follows, we view  $P_s(n)$ as a subalgebra of $P_r(n)$ via the embedding 
$$P_s(n)\cong \mathbb{Q}\otimes P_s(n) \hookrightarrow P_{r-s}(n) \otimes  P_s(n) \hookrightarrow  P_r(n).  $$  
We now recall the definition of skew modules for  $P_s(n)$.  
This  family of   modules   were first  introduced (in the more general context of diagram algebras) in \cite{BE}.   
Given $ \nu  \in   \mathscr{P}_{\leq r}$, we  let $\stt^\nu\in\Std_r(\nu)$  denote the Kronecker tableau of the form 
$$
\underbrace{
d(0) \circ d(0) \circ 
\dots
\circ d(0) }_{r-|\nu|}
\circ 
\underbrace{
a(1)
\circ 
\dots\circ 
a(1)}_{\nu_1}\circ 
\underbrace{
a(2)
\circ 
\dots\circ 
a(2)}_{\nu_2}\circ 
\cdots
 $$ 
which is maximal in the dominance ordering on $\Std_r(\nu)$.

\begin{eg}For $\nu=(2,1)  \in \mathscr{P}_{\leq 5}\subseteq \mathcal{Y}_5$,  the
Kronecker tableau $\stt^\nu$ is equal to
$$
\left(\Yvcentermath1 \varnothing,  \varnothing,  \varnothing,  \varnothing,
\Yboxdim{6pt}\gyoung(;) \; , \; \Yboxdim{6pt}\gyoung(;)\; , \; \Yboxdim{6pt}\gyoung(;;)\; , \; \Yboxdim{6pt}\gyoung(;;)\; , \; \Yboxdim{6pt}\gyoung(;;,;)\; , \;  \Yboxdim{6pt}\gyoung(;;,;) \right)
$$

 \end{eg}

\begin{defn}
 Given $\lambda
\in\mathscr{P}_{r-s} \subseteq \mathcal{Y}_{r-s}$ and  $\nu \in\mathscr{P}_{\leq r} \subseteq \mathcal{Y}_r$, define
 $$
\Delta_r(\nu; \rhd \lambda) = {\rm span}_\CC\{ m_{\stt }  \mid \stt(r-s) \rhd \lambda\}
\qquad \Delta_r(\nu;  \stt^\lambda) = {\rm span}_\CC\{ m_{\stt }  \mid \stt\in \Std_r(\nu), \stt[0,r-s]=\stt^ \lambda\}
$$
then $\Delta_r(\nu; \rhd \lambda)$ and $\Delta_r(\nu;\stt^ \lambda)+ \Delta_r(\nu; \rhd \lambda)$ are $P_s(n)$-submodules of 
$\Delta_r(\nu){\downarrow}_{P_s(n)}$.  We define  the {\sf skew cell module} 
 $$
\Delta_s(\nu\setminus\lambda) = 
(\Delta_r(\nu; \stt^ \lambda) + \Delta_r(\nu; \rhd \lambda) )/ \Delta_r(\nu; \rhd \lambda).
$$
  \end{defn}
 
\begin{rmk}      It follows from \cref{murphytypebasisdefn} that we can realise the skew cell module as a subquotient of the algebra $P_r(n)$ as follows.  Define
$$P_{r,s}^{\rhd \nu \setminus\lambda}
=
P_r(\nu) + {\rm span}_\mathbb{Q}\{	u_\stt \mid \stt \in \Std_r(\nu), \stt(r-s)\rhd \lambda	\},$$
then 
 $$
\Delta_s(\nu\setminus\lambda) = 
 {\rm span}_\mathbb{Q}\{	u_{\stt^\lambda\circ \sts}  + P_{r,s}^{\rhd \nu \setminus\lambda}\mid \sts \in \Std_s(\nu\setminus\lambda) 	\}.
 $$
\end{rmk}
\begin{rmk}\label{specialsym} The  basis of   $\Delta_s(\nu\setminus\lambda)$ 
is  indexed by the elements of  $\Std_s(\nu\setminus\lambda)$ and 
if $(\lambda,\nu, s)$ is triple of maximal depth, this module is isomorphic to $\mathbf{S}(\nu\ominus\lambda)$, the    skew Specht module  for $\mathfrak{S}_s$, lifted to $P_s(n)$.
 \end{rmk}

We can now reinterpret of stable Kronecker coefficients in the context of the partition algebra as follows.

 \begin{thm}\cite{BDO15,BE} Let $\lambda\in \mathscr{P}_{r-s}$, $\mu\in \mathscr{P}_s$ and $\nu \in \mathscr{P}_{\leq r}$. Then we have 
$$
\overline{g}(\lambda,\nu,\mu)
=  \dim_\CC( \Hom_{  P_{r-s}(n)\times P_{s}(n)}( \Delta_{r-s}(\lambda) \boxtimes \Delta_{s}(\mu), \Delta_r(\nu   ){\downarrow}    ) ) 
=  \dim_\CC( \Hom_{  P_{s}(n)}( \Delta_{s}(\mu), \Delta_s(\nu \setminus\lambda )    ) ) 
  $$
  for all $n\gg 0$.
  \end{thm}

\begin{rmk} Using Remark \ref{specialsym} and (\ref{LRcoeff}) we recover Theorem \ref{LitMurn}. So the Littlewood--Richardson coefficients appear naturally as a subclass of the stable Kronecker coefficients in the context of the partition algebra.
\end{rmk}

      \section{The action of the partition algebra on the Murphy basis} \label{ACTION!}
To describe the action of the generators of the partition algebra   on the   Murphy  basis  
is  very difficult in general. 
 In this section, we shall solve this problem for  the Coxeter 
generators   on the basis elements indexed by   a certain class of paths.  
 This section along with Section 4  provide 
 the most difficult and technical calculations  of the paper;   we encourage the  reader 
% whose  principle  interest lies  in understanding Kronecker coefficients 
 to skip these two sections on the first reading and head to Section 5, where   these calculations are used to prove our main results.  
 %begin to reap the rewards of these partition algebra calculations.  

\begin{defn}\label{sdhjkasfghjfahljfahjdflkhmnsv}Fix  $\stt \in   {\Std}_r( \nu   )$ and $1\leq k \leq r$  and suppose that
    $$ \stt{(k-1)}    \xrightarrow{-t}  \stt(k-\thalf)    \xrightarrow{+u}   \stt(k+1) \xrightarrow{-v}      \stt(k+\thalf)  \xrightarrow{+w}  \stt(k+1).$$
We define  $ \stt_{k \leftrightarrow k+1}\in \Std_r(\nu)$ to be the tableau, if it exists, determined by  $  \stt_{k \leftrightarrow k+1}(l) =\stt(l) $ for $l\neq k, k \pm \tfrac{1}{2}  $ and 
 $$  \stt_{k \leftrightarrow k+1} {(k-1)}    \xrightarrow{-v}     \stt_{k \leftrightarrow k+1}{(k-\thalf)}   \xrightarrow{+w}  
   \stt_{k \leftrightarrow k+1}{(k)}  \xrightarrow{-t}       \stt_{k \leftrightarrow k+1}(k+\thalf)  \xrightarrow{+u}  \stt_{k \leftrightarrow k+1}(k+1).$$
  \end{defn}

In this section, we will discuss explicitly the action of $s_{k,k+1}$ on $u_\stt$ for all paths $\stt\in \Std(\nu)$ 
such that the path $\stt_{k\leftrightarrow k+1}$ exists.  
 
\begin{figure}[ht!]
        $$
   \scalefont{0.8} \begin{tikzpicture}[scale=0.5]
  \draw[white] (0,1.2) rectangle (1,-13.4);      \begin{scope}   
\path[<-]   (-3,-2)edge[decorate]  node[auto] {$-2$}  (-0.5,0.5);
 \draw[->] (-0.5,0.5) -- (-3,-2);  \draw[->] (0.5,0.5) -- (3,-2); 
\path[<-] (-3,-6) edge[decorate]  node[auto] {$+5$}  (-3,-2.5); 
\draw[<-] (3,-6) -- (3,-2.5);  
 \path[<-] (-3,-9.5) edge[decorate]  node[auto] {$-3$}  (-3,-6);  %1st
\draw[<-] (3,-9.5) -- (3,-6); 
\draw[<-] (0.5,-11.5)   --(3,-9) ; 
\fill[white] (0,0.2) circle (37pt);
\fill[white] (-3,-2.5) circle (37pt);
\fill[white] (3,-2.5) circle (37pt);
\fill[white] (-3,-6) circle (37pt);
\fill[white] (3,-6) circle (37pt);
 \fill[white] (-3,-9.5) circle (37pt);%2rd
   \path[<-] (-0.5,-11.5) edge[decorate] node[auto] {$+1$}   (-3,-9);  
\fill[white] (-3.5,-9.5) circle (34pt); %%4th 
\fill[white] (3,-9.5) circle (37pt);
 \fill[white] (0,-12) circle (39pt);
           \draw (0,0.2) node {$  \Yboxdim{5pt}\gyoung(;;;;,;;;,;;,;) $  };   
    \draw (-3,-2.5) node   { $\Yboxdim{5pt}\gyoung(;;;;,;;,;;,;)  $ }		;
    \draw (3,-2.5) node   { $\Yboxdim{5pt}\gyoung(;;;;,;;;,;,;)  $ }		;
        \draw (-3,-6) node   { $\Yboxdim{5pt}\gyoung(;;;;,;;,;;,;,;)  $ }		;
    \draw (3,-6) node   { $\Yboxdim{5pt}\gyoung(;;;;;,;;;,;,;)  $ }		;
            \draw (-3,-9.5) node   { $\Yboxdim{5pt}\gyoung(;;;;,;;,;,;,;)  $ }	;
    \draw (3,-9.5) node   { $\Yboxdim{5pt}\gyoung(;;;;;,;;,;,;)  $ }			;
            \draw (0,-12.2) node    { $\Yboxdim{5pt}\gyoung(;;;;;,;;,;,;,;)  $ }		;
             \end{scope}\end{tikzpicture} 
     \quad      \quad      \quad
                           \begin{tikzpicture}[scale=0.5]
      \draw[white] (0,1.2) rectangle (1,-13.4);       \begin{scope}   
\path[<-]   (-3,-2)edge[decorate]  node[auto] {$-2$}  (-0.5,0.5);
 \draw[->] (-0.5,0.5) -- (-3,-2);  \draw[->] (0.5,0.5) -- (3,-2); 
\path[<-] (-3,-6) edge[decorate]  node[auto] {$+5$}  (-3,-2.5); 
\draw[<-] (3,-6) -- (3,-2.5);  
 \path[<-] (-3,-9.5) edge[decorate]  node[auto] {$-0$}  (-3,-6);  %1st
\draw[<-] (3,-9.5) -- (3,-6); 
\draw[<-] (0.5,-11.5)   --(3,-9) ; 
\fill[white] (0,0.2) circle (37pt);
\fill[white] (-3,-2.5) circle (37pt);
\fill[white] (3,-2.5) circle (37pt);
\fill[white] (-3,-6) circle (37pt);
\fill[white] (3,-6) circle (37pt);
 \fill[white] (-3,-9.5) circle (37pt);%2rd
   \path[<-] (-0.5,-11.5) edge[decorate] node[auto] {$+4$}   (-3,-9);  
\fill[white] (-3.5,-9.5) circle (34pt); %%4th 
\fill[white] (3,-9.5) circle (37pt);
 \fill[white] (0,-12) circle (39pt);
           \draw (0,0.2) node {$  \Yboxdim{5pt}\gyoung(;;;;,;;;,;;,;) $  };   
    \draw (-3,-2.5) node   { $ \Yboxdim{5pt}\gyoung(;;;;,;;,;;,;) $ }		;
    \draw (3,-2.5) node   { $ \Yboxdim{5pt}\gyoung(;;;;,;;;,;;,;) $ }		;
        \draw (-3,-6) node   { $ \Yboxdim{5pt}\gyoung(;;;;,;;,;;,;,;) $ }		;
    \draw (3,-6) node   { $\Yboxdim{5pt}  \Yboxdim{5pt}\gyoung(;;;;,;;;,;;,;;)$ }		;
            \draw (-3,-9.5) node   { $\Yboxdim{5pt}\gyoung(;;;;,;;,;;,;,;) $ }	;
    \draw (3,-9.5) node   { $\Yboxdim{5pt} \Yboxdim{5pt}\gyoung(;;;;,;;,;;,;;)  $ }			;
            \draw (0,-12.2) node    { $\Yboxdim{5pt} \gyoung(;;;;,;;,;;,;;,;)  $ }		;
              \end{scope}\end{tikzpicture}  
  \quad      \quad      \quad
                           \begin{tikzpicture}[scale=0.5]
      \draw[white] (0,1.2) rectangle (1,-13.4);       \begin{scope}   
\path[<-]   (-3,-2.5)edge[decorate]  node[auto] {$-0$}  (-3,0.2);
\path[<-] (-3,-6) edge[decorate]  node[auto] {$+2$}  (-3,-2.5); 
\draw[<-] (0.9,-5.8) -- (-3,-3);  
 \path[<-] (-3,-9.5) edge[decorate]  node[auto] {$-0$}  (-3,-6);  %1st
\draw[<-] (1,-9.5) -- (1,-6); 
\draw[<-] (-3,-12.2)   --(1,-9.5) ; 
\fill[white] (-3,0.2) circle (26pt);
\fill[white] (-3,-2.5) circle (26pt);
\fill[white] (1,-2.5) circle (26pt);
\fill[white] (-3,-6) circle (26pt);
\fill[white] (1,-6) circle (26pt);
 \fill[white] (-3,-9.5) circle (26pt);%2rd
   \path[<-] (-3,-12.2) edge[decorate] node[auto] {$+3$}   (-3,-9.5);  
\fill[white] (-3.5,-9.5) circle (26pt); %%4th 
\fill[white] (1,-9.5) circle (26pt);
 \fill[white] (-3,-12) circle (26pt);
           \draw (-3,0.2) node {$  \Yboxdim{5pt}\gyoung(;;;,;;,;) $  };   
    \draw (-3,-2.5) node   { $ \Yboxdim{5pt}\gyoung(;;;,;;,;) $ }		;
        \draw (-3,-6) node   { $ \Yboxdim{5pt}\gyoung(;;;,;;;,;) $ }		;
    \draw (1,-6) node   { $\Yboxdim{5pt}  \Yboxdim{5pt}\gyoung(;;;,;;,;;)$ }		;
            \draw (-3,-9.5) node   { $\Yboxdim{5pt}\gyoung(;;;,;;;,;) $ }	;
    \draw (1,-9.5) node   { $\Yboxdim{5pt} \Yboxdim{5pt}\gyoung(;;;,;;,;;)  $ }			;
            \draw (-3,-12.2) node    { $\Yboxdim{5pt} \gyoung(;;;,;;;,;;)  $ }		;
              \end{scope}\end{tikzpicture}  
 $$

\!\!\!\!\!\caption{ 
  Examples of the pairs of paths  $\stt$ and  $\stt_{k \leftrightarrow k+1}$ in $\mathcal{Y}$.  
   }
 \label{parmoves}
\end{figure}
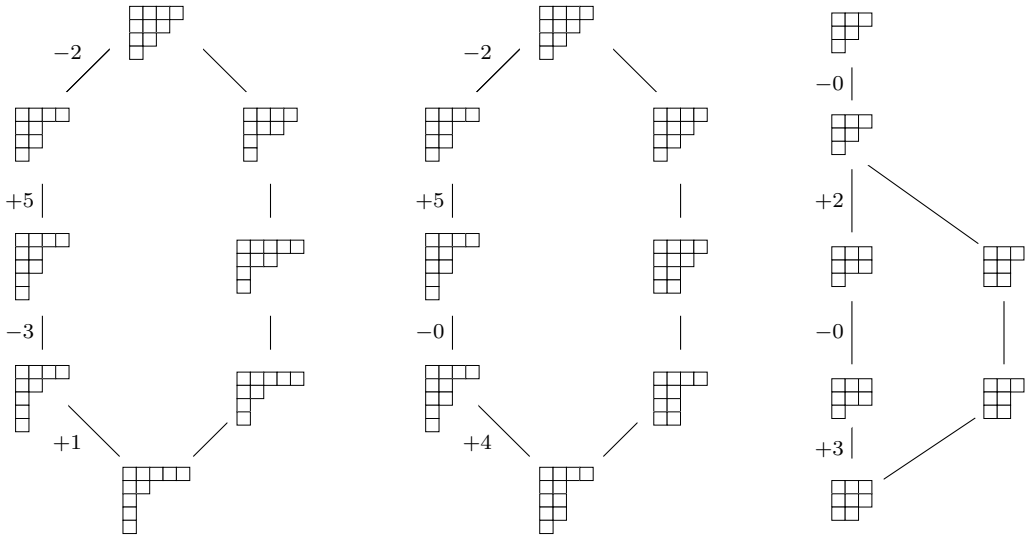

Before stating the main result,  we need one more piece of notation.  
 
\begin{defn}\label{defn:error} For  $\stt \in \Std_r( \nu   )$ and $1\leq k \leq r$ with 
\[
\stt{(k-\thalf)}      \xrightarrow{+u}  \stt{(k )}      \xrightarrow{-u}    \stt(k+\thalf) 
\]
%where $\stt{(k-\thalf)}   = \rho = \stt{(k+\thalf)}   $.
for $u>0$, we define $\sts={\sf e}_k({\stt})\in\Std_r(\nu)$ 
by  $\sts(l)=\stt(l)$ for $l\neq k$ and 
\begin{align}\label{definition of e term}
\sts{(k-\thalf)}   
    \xrightarrow{+L}   \sts{(k)}
    %,c) 
     \xrightarrow{-L}     \sts(k+\thalf)
\end{align}
 where $L=\ell(\stt{(k-\thalf)})+1$.  If  $\stt(k-\thalf)\neq \stt(k+\half)$, then ${\sf e}_k({\stt})$ is  undefined. 
\end{defn}

 \begin{thm}\label{MURPHYPrN}
Fix   $1\leq k \leq r$ and let $\stt  \in \Std_r(\nu)$.  If $\stt _{k \leftrightarrow k+1}$ exists, then 
 \begin{align*}
( u_{\stt})  s_{k,k+1} = 
 u_{\stt_{k\leftrightarrow k+1}} 	+  u_{{\sf e}_k(\stt)}
 - u_{{\sf e}_k(\stt_{k \leftrightarrow k+1})},
 \end{align*}
where we take the convention that $u_{e_k(\mathsf{v})}=0$ whenever the path  ${\sf e}_k({\stv})$ is undefined for $\stv \in \Std_r(\nu)$. 
\end{thm}
   
The remainder of this section is dedicated 
to   proving this result.  
Fix $\stt \in\Std_r(\nu)$ and $1\leq k \leq r$.  First note that we can factorise $u_\stt$ as follows,
$$
u_\stt =u_{\stt[k+1,r]}u_{\stt[k-1,k+1]} u_{\stt[0,k-1]}.
$$
Now as $u_{\stt[0,k-1]}\in P_{k-1}(n)$, it commutes with $s_{k,k+1}$ and so we have 
\begin{equation}\label{passingthrough}
u_\stt =u_{\stt[k+1,r]}u_{\stt[k-1,k+1]}s_{k,k+1} u_{\stt[0,k-1]}.
\end{equation}
So let us first consider $u_{\stt[k-1,k+1]}$.  We fix the following notation. 
Given a fixed  $1\leq k \leq r$ and $\flux  \in \Std_r(\nu)$ for some $\nu \in \mathcal{Y}_{r}$, we set 
$$
\flux (k-1) = (\alpha,a)
\quad
\flux (k-\thalf) = (\beta,b)
\quad
\flux (k) = (\gamma,c)
\quad
\flux (k+\thalf) = (\delta,d)
\quad
\flux (k+1) = (\zeta,z).
$$
As in \cref{defn:error}, if $u=v$ we let $\sts:=e_{k}(\stt)$.  
Given $\nu$ a partition and $u,w>0$ we set
 $$
m_{\nu -\varepsilon_w \to \nu}
=     \sum_{i=0}^{ \nu_w-1}s_{[\nu]_{w}-i, [\nu]_{w}} 
\qquad  m_{\nu,u,w}= 
 \begin{cases}
m_{\nu -\varepsilon_w \to \nu}
m_{\nu -\varepsilon_u \to \nu} &\text{if }u\neq w 
\\
m_{\nu -\varepsilon_w \to \nu}
m_{\nu -2\varepsilon_w \to \nu- \varepsilon_w} &\text{if }u= w.
\end{cases}
$$

(Note that $m_{\nu,u,w}= m_{\nu,w,u}$.)

\begin{prop}\label{4.4}
We have 
$$u_{\stt[k-1,k+1]}
= m_{\zeta,u,w} P_k(\stt)+ (1-\delta_{u,0})\delta_{u,v}u_{\sts[k-1,k+1]}
$$
(for $\sts=e_k(\stt)$ as in \cref{definition of e term}) where  $$
P_k(\stt) =
\begin{cases}
e_{k+1}^{(z)}s_{[\zeta]_w,|\zeta|}e_{k+\frac{1}{2}}^{(d)}
s_{|\gamma|,[\gamma]_u-1}e_{k }^{(c)}
s_{[\gamma]_u, |\gamma| }e_{k-\frac{1}{2}}^{(b)}
s_{|\alpha|, [\alpha]_t}     
&\text{if }u= v>0 
\\
e_{k+1}^{(z)}s_{[\zeta]_w,|\zeta|}
e_{k+\frac{1}{2}}^{(d)}
s_{|\gamma|,[\gamma]_v}
e_{k }^{(c)}
s_{[\gamma]_u, |\gamma| }e_{k-\frac{1}{2}}^{(b)}
s_{|\alpha|, [\alpha]_t}     
&\text{otherwise.}
\end{cases}
$$
\end{prop}

 \begin{proof}
 By definition \ref{coefficientsforthepartitionalgebra}, we have 
 $$
 u_{\stt[k-1,k+1]}
 =
 e_{k+1}^{(z)}
 m_{\delta \to \zeta}
 s_{[\zeta]_w, |\zeta| }
 e_{k+\half}^{(d)} 
 s_{|\gamma|, [\gamma]_v }
 e_{k}^{(c)} 
   m_{\beta \to \gamma}
   s_{[\gamma]_u,|\gamma|}
    e_{k-\half}^{(b)} 
     s_{|\alpha|, [\alpha]_t }.
 $$

\smallskip

 \noindent {\bf Claim A.}   
  If $k\geq 1$, $(\lambda,l)\in \mathcal{Y}_k$ and
   $(\mu,m)\to(\lambda,l)$ is an edge in $\mathcal{Y}$ then we have 
   \begin{itemize}
\item $e_k^{(l)} m_{\mu \to \lambda}=m_{\mu \to \lambda}e_k^{(l)}   $,
\item  $e_k^{(l)}s_{[\lambda]_a,|\lambda|}=
s_{[\lambda]_a,|\lambda|}e_k^{(l)}$ for any $a \geq 0$.
   \end{itemize}
   We have that $|\lambda|=k-l$ and 
   $e_k^{(l)}$ is the identity on the first $k-l$ strands and so commutes with $ m_{\mu \to \lambda}$ and $s_{[\lambda]_a,|\lambda|}$.  
   Therefore Claim A follows.  
 
 \smallskip
   
    \noindent {\bf Claim B.}   
    We have that
    $$
    s_{|\gamma|,[\gamma]_v } m_{\beta \to \gamma}
    =
    \begin{cases}
    m_{\delta-\varepsilon_u\to \delta}s_{|\gamma|,[\gamma]_u-1}+s_{|\gamma|,[\gamma]_u} &\text{if }u=v>0, \\
    m_{\delta-\varepsilon_u\to \delta}s_{|\gamma|,[\gamma]_v}&\text{otherwise.}
    \end{cases} 
    $$
    (We note that $\beta= \gamma-\varepsilon_u$.)  
    If $v=0$, then $     s_{|\gamma|,[\gamma]_v }=1$ and $\delta=\gamma$ and so the result holds trivially.
    If $u=0$, then $m_{\beta \to \gamma}= 1=m_{\delta-\varepsilon_u\to \delta}$ and so the result also holds trivially.  
 We now assume that $u,v>0$.  
If $u<v$ then $\gamma_u=\delta_u$ and $[\delta]_u=[\gamma]_u<[\gamma]_v$ and so 
\begin{align*}
s_{|\gamma|,[\gamma]_v}m_{\gamma-\varepsilon_u \to \gamma}
= m_{\gamma-\varepsilon_u \to \gamma} s_{|\gamma|,[\gamma]_v}=
 m_{\delta-\varepsilon_u \to \gamma} s_{|\gamma|,[\gamma]_v},
\end{align*}as required.  
If $v<u$ then $[\gamma]_v<[\gamma]_u - i\leq |\gamma|$ for all $0\leq i \leq \gamma_u-1$ and so 
\begin{align*}
s_{|\gamma|,[\gamma]_v}m_{\gamma-\varepsilon_u \to \gamma}
=
s_{|\gamma|,[\gamma]_v} \sum_{  i=0}^{  \gamma_u-1}
s_{ [\gamma]_u-i, [\gamma]_u}  =
\left( \sum_{  i=0}^{  \gamma_u-1}
s_{ [\gamma]_u-i-1, [\gamma]_u-1}\right)s_{|\gamma|,[\gamma]_v}
 =
\left( \sum_{  i=0}^{  \delta_u-1}
s_{ [\delta]_u-i, [\delta]_u}\right)s_{|\gamma|,[\gamma]_v}
 \end{align*}
 (where the final equality follows as $\delta_u=\gamma_u$ and $[\delta]_u = [\gamma]_u-1$)  
 and the final term is equal to $m_{\delta-\varepsilon_u\to \delta}s_{|\gamma|,[\gamma]_v}$ by definition.   
Finally if $u=v>0$ then 
\begin{align*}
  s_{|\gamma|,[\gamma]_v } m_{\beta \to \gamma}
=
 s_{|\gamma|,[\gamma]_u } m_{\gamma-\varepsilon_u \to \gamma}
 = s_{|\gamma|,[\gamma]_u } \sum_{i=0}^{\gamma_u-1}
 s_{|\gamma|_u-i,[\gamma]_u }  
 =  s_{|\gamma|,[\gamma]_u }  \left(1+ \sum_{i=1}^{\gamma_u-1}
 s_{|\gamma|_u-i,[\gamma]_u } \right). 
\end{align*}
Expanding out the brackets and shifting the indices, we obtain 
\begin{align*}
  s_{|\gamma|,[\gamma]_u } + \sum_{i=1}^{\gamma_u-1}
 s_{|\gamma|,[\gamma]_u }  s_{|\gamma|_u-i,[\gamma]_u } 
 &=  s_{|\gamma|,[\gamma]_u } + 
  \sum_{i=1}^{\gamma_u-1}
  s_{|\gamma|_u-i,[\gamma]_u-1 } s_{|\gamma|,[\gamma]_u -1} \\ %%
  & =  s_{|\gamma|,[\gamma]_u } + 
  \sum_{i=0}^{\gamma_u-2}
  s_{|\gamma|_u-1-i,[\gamma]_u-1 } s_{|\gamma|,[\gamma]_u -1} 
 \\ & =  s_{|\gamma|,[\gamma]_u } + 
  \sum_{i=1}^{\delta_u-1}
  s_{|\delta|_u-i,[\delta]_u } s_{|\gamma|,[\gamma]_u -1} \\
 &= m_{\delta-\varepsilon_u\to \delta}s_{|\gamma|,[\gamma]_u-1}+s_{|\gamma|,[\gamma]_u} 
\end{align*}
where the penultimate equality follows as  $[\gamma]_u-1=[\delta]_u$ and $\gamma_u-2=\delta_u-1$.  Therefore   Claim B follows.

 \smallskip
   
    \noindent {\bf Claim C.}   
We have that 
$$
s_{[\zeta]_w,|\zeta|}m_{\delta-\varepsilon_u\to \delta} = 
\begin{cases}
m_{\zeta-\varepsilon_u\to \zeta}  
s_{[\zeta]_w,|\zeta|} &\text{if } w\neq u
\\
m_{\zeta-2\varepsilon_w\to \zeta- \varepsilon_w} 
s_{[\zeta]_w,|\zeta|} &\text{otherwise.}
\end{cases}
$$
If $u=0$ or $w=0$ the result holds trivially.  We assume $u,w>0$.  If $u<w$ then 
$[\delta]_u=[\zeta]_u < [\zeta]_w$ so we get
$$
s_{[\zeta]_w,|\zeta|}m_{\delta-\varepsilon_u\to \delta}= m_{\delta-\varepsilon_u \to \delta}
s_{[\zeta]_w,|\zeta|}
=
 m_{\zeta-\varepsilon_u \to \zeta}
s_{[\zeta]_w,|\zeta|},
$$as required.  
If $u>w$ then $[\delta]_u = [\zeta]_u-1\geq [\zeta]_w$ and so we get
\begin{align*}
s_{[\zeta]_w,|\zeta|}m_{\delta-\varepsilon_u\to \delta} 
%&
=
s_{[\zeta]_w,|\zeta|}\sum_{i=0}^{\delta_u-1}s_{[\delta]_u-i,[\delta]_u}
%\\
%&
=
\sum_{i=0}^{\delta_u-1}s_{[\delta]_u-i+1,[\delta]_u+1}s_{[\zeta]_w,|\zeta|}
%\\
%&
=
\sum_{i=0}^{\zeta_u-1}s_{[\zeta]_u-i,[\zeta]_u}s_{[\zeta]_w,|\zeta|} 
%\\
%&
\end{align*} which is equal to $m_{\zeta-\varepsilon_u \to \zeta}s_{[\zeta]_w,|\zeta|}$, as required.  
Finally, if $u=w>0$ then $[\zeta]_w=[\zeta]_u=[\delta]_u+1$ and 
$$
s_{[\zeta]_w,|\zeta|}m_{\delta-\varepsilon_u\to \delta} =
m_{\delta-\varepsilon_u\to \delta} s_{[\zeta]_w,|\zeta|}
=
m_{\zeta-2\varepsilon_w\to \zeta-\varepsilon_w} s_{[\zeta]_w,|\zeta|},
$$as required.  
Therefore Claim C follows.

Applying Claim A and Claim B (and noting that $s_{|\gamma|,[\gamma]_u}s_{[\gamma]_u,|\gamma|}=1$) we deduce that 
$$
u_{\stt[k-1,k+1]}=
\begin{cases}
m_{\delta\to \zeta} e_{k+1}^{(z)}s_{[\zeta]_w,|\zeta|}e_{k+\half}^{(d)}m_{\delta-\varepsilon_u\to \delta}s_{|\gamma|,[\gamma]_u-1}
e_k^{(c)}s_{[\gamma]_u,|\gamma| } e_{k-\half}^{(b)}s_{|\alpha|,[\alpha]_t}		 
\\
\quad + e^{(z)}_{k+1}m_{\delta\to \zeta}s_{[\zeta]_w,|\zeta|}e_{k+\half}^{(d)}e_k^{(c)}e_{k-\half}^{(b)}s_{|\alpha|,[\alpha]_t}&\text{if }u=v>0
\\
m_{\delta\to \zeta} e_{k+1}^{(z)}s_{[\zeta]_w,|\zeta|}e_{k+\half}^{(d)}m_{\delta-\varepsilon_u\to \delta}s_{|\gamma|,[\gamma]_v}
e_k^{(c)}s_{[\gamma]_u,|\gamma| } e_{k-\half}^{(b)}s_{|\alpha|,[\alpha]_t}		&\text{otherwise.}
\end{cases}
$$
Applying Claim A and Claim C to the above equation, we deduce that 
 $$
u_{\stt[k-1,k+1]}=
\begin{cases}
m_{\zeta,u,w} e_{k+1}^{(z)}s_{[\zeta]_w,|\zeta|}e_{k+\half}^{(d)}
%m_{\delta-\varepsilon_u\to \delta}
s_{|\gamma|,[\gamma]_u-1}
e_k^{(c)}
s_{[\gamma]_u,|\gamma| }
 e_{k-\half}^{(b)}s_{|\alpha|,[\alpha]_t}		 
\\
\quad 
+ e^{(z)}_{k+1}
m_{\zeta-\varepsilon_w \to \zeta}  %%%%%%why?
s_{[\zeta]_w,|\zeta|}e_{k+\half}^{(d)}e_k^{(c)}e_{k-\half}^{(b)}s_{|\alpha|,[\alpha]_t} &\text{if }u=v>0
\\
m_{\zeta,u,w}  e_{k+1}^{(z)}s_{[\zeta]_w,|\zeta|}e_{k+\half}^{(d)}
%m_{\delta-\varepsilon_u\to \delta}
s_{|\gamma|,[\gamma]_v}
e_k^{(c)}s_{[\gamma]_u,|\gamma| } e_{k-\half}^{(b)}s_{|\alpha|,[\alpha]_t}		&\text{otherwise.}
\end{cases}
$$
Finally, note that 
$$u_{\sts[k-1,k+1]}
=
 e^{(z)}_{k+1}m_{\zeta-\varepsilon_w\to \zeta}s_{[\zeta]_w,|\zeta|}
 e_{k+\half}^{(d)}
 s_{|\beta+\varepsilon_L|,|\beta+\varepsilon_L|}
 e_k^{(c)}
  m_{ \beta \to \beta+\varepsilon_L}
   s_{|\beta+\varepsilon_L|,|\beta+\varepsilon_L|}
 e_{k-\half}^{(b)}s_{|\alpha|,[\alpha]_t}
$$
and as $ s_{|\beta+\varepsilon_L|,|\beta+\varepsilon_L|}=1=   m_{ \beta \to \beta+\varepsilon_L}$ we have 
\begin{equation}\label{doubledaggerdouble}
u_{\sts[k-1,k+1]}
= e^{(z)}_{k+1}
m_{\zeta-\varepsilon_w \to \zeta}   
s_{[\zeta]_w,|\zeta|}e_{k+\half}^{(d)}e_k^{(c)}e_{k-\half}^{(b)}s_{|\alpha|,[\alpha]_t}.
\end{equation}
This completes the proof of \cref{4.4}. 
\end{proof}

Using \cref{4.4} and   \cref{passingthrough} we have 
\begin{align*}
u_\stt s_{k,k+1}
&=
u_{\stt[k+1,r]}
 m_{\zeta,u,w} P_k(\stt)s_{k,k+1}
u_{\stt[0,k-1]}+ 
 (1-\delta_{u,0})\delta_{u,v}
 u_{\stt[k+1,r]}u_{\sts[k-1,k+1]} s_{k,k+1}
u_{\stt[0,k-1]}.
\end{align*}

\begin{lem}\label{alittlelemma}
For $\sts=e_k(\stt)$ as in \cref{definition of e term}, we have that $u_{\sts[k-1,k+1]}s_{k,k+1}=u_{\sts[k-1,k+1]}$.  
\end{lem} 
 
 \begin{proof}
 As we have seen in \cref{doubledaggerdouble}, 
$$ u_{\sts[k-1,k+1]}
= e^{(z)}_{k+1}
m_{\zeta-\varepsilon_w \to \zeta}   
s_{[\zeta]_w,|\zeta|}e_{k+\half}^{(d)}e_k^{(c)}e_{k-\half}^{(b)}s_{|\alpha|,[\alpha]_t},
$$
with $b=c$ and $d=b+1$.  Now 
 $s_{|\alpha|,[\alpha]_t} \in P_{k-1}(n)$ and so it commutes with $s_{k,k+1}$.  Moreover, we have 
 $$
 e_{k+\half}^{(b+1)}e_k^{(b)}e_{k-\half}^{(b)}
 =
  e_{k+\half}^{(b+1)}
 $$
 and $ e_{k+\half}^{(b+1)}s_{k,k+1}= e_{k+\half}^{(b+1)}$.  Hence  $u_{\sts[k-1,k+1]}s_{k,k+1}=u_{\sts[k-1,k+1]}$ as required.  
 \end{proof}

Applying \cref{alittlelemma} and noting that 
$$
u_{\stt[k+1,r]}u_{\sts[k-1,k+1]}u_{\stt[0,k-1]}= u_\sts
$$
we get 
\begin{align}\label{tripledagger}
u_\stt s_{k,k+1}
&=
u_{\stt[k+1,r]}
 m_{\zeta,u,w} P_k(\stt)s_{k,k+1}
u_{\stt[0,k-1]}+ 
 (1-\delta_{u,0})\delta_{u,v}
 u_{\sts}.
\end{align}
It remains to consider the first term in this sum.  
Note that $P_k(\stt)$ is a single partition diagram and so we should, in theory, be able to describe both this set-partition and  the set-partition $P_k(\stt)s_{k,k+1}$.  This calculation can, however, be much simplified by making the following observation.   Using \cite{EG}, we have 
$$u_{\stt[0,k-1]}=c_{\stt(k-1)}d^\ast_{\stt[0,k-1]}$$
where  $c_{\stt(k-1)  } =  e^{(a)}_{k-1} \sum_{\sigma \in \mathfrak{S}_{\alpha }} \sigma   \in P_{k-1}(n)$.  
So the first term in the sum \cref{tripledagger} can be rewritten as follows, 
 \begin{align} \label{triplestar}
u_{\stt[k+1,r]}
 m_{\zeta,u,w} P_k(\stt)s_{k,k+1}
u_{\stt[0,k-1]}  
&=
u_{\stt[k+1,r]}
 m_{\zeta,u,w} P_k(\stt)s_{k,k+1}
e^{(a)}_{k}  \textstyle \sum_{\sigma \in \mathfrak{S}_{\alpha  }} \sigma d^\ast_{\stt[0,k-1]} \nonumber
\\  
 &=
u_{\stt[k+1,r]}
 m_{\zeta,u,w}\left(  P_k(\stt) 
e^{(a)}_{k} \right)
s_{k,k+1}
 \textstyle \sum_{\sigma \in \mathfrak{S}_{\alpha }} \sigma d^\ast_{\stt[0,k-1]}  \end{align}
Now $  P_k(\stt) 
e^{(a)}_{k} $ is also a single partition diagram and can be described (more simply than $P_k(\stt)$)  as follows.

\begin{defn}
 Let $S=\lbrace S_1,S_2,\ldots, S_j\rbrace$  be a set of pairwise disjoint subsets of $$\big\lbrace 1,\ldots,k+1,\overline{1},\ldots, \overline{k+1}\big\rbrace$$
 such that there is a bijection between the barred and unbarred elements of 
 $$\big\lbrace 1,\ldots,k+1,\overline{1},\ldots, \overline{k+1}\big\rbrace \setminus ( S_1\cup S_2\cup\cdots\cup S_j).$$ 
Write 
 $$\big\lbrace 1,\ldots,k+1,\overline{1},\ldots, \overline{k+1}\big\rbrace \setminus ( S_1\cup S_2\cup\cdots\cup S_j)
 =
 \{i_1 < i_2 \dots < i_\ell\} \cup \{\bar j_1<\bar j_2 < \dots <\bar j_\ell\}.$$ 
   We define $\widehat{S}\in P_{k+1}(n)$ to be the set partition 
 $$\widehat{S}
 =
 S \bigcup_{1\leq m \leq \ell}\{\{i_m,\bar j_m\}\}. 
 $$ 
 In other words, $\widehat{S}$ contains the blocks $S_1, S_2, \dots S_j$ and determined an order preserving bijection between the barred and unbarred elements of 
$ \big\lbrace 1,\ldots,k+1,\overline{1},\ldots, \overline{k+1}\big\rbrace \setminus ( S_1\cup S_2\cup\cdots\cup S_j)$.
  \end{defn}

\begin{eg}\label{maudexample}
Let $k+1=10$ and 
 $$S=\{\{4,9,\bar6\}, \{6,10,\bar4\}, \{\bar9\},\{\bar10\}\},$$ 
then  $$\widehat S=\{\{4,9,\bar6\}, \{6,10,\bar4\}, \{\bar9\},\{\bar10\}, \{1,\bar1\},
 \{2,\bar2\},
  \{3,\bar3\},
   \{5,\bar5\},
    \{7,\bar7\},    \{8,\bar8\}
\}.$$ 
\end{eg}

\begin{prop}\label{thefinal2}
We  have  that $$
P_k(\flux)e_{k-1}^{(a)}= \widehat{S_k(\flux)}$$
where   $S_{k}(\flux)$ is the set of pairwise disjoint subsets
of 
$ \big\lbrace 1,\ldots,k+1,\overline{1},\ldots, \overline{k+1}\big\rbrace$ 
 obtained by omitting all occurrences of $0$ and $\overline 0$ from 
\begin{align*} 
 \left\{\big\lbrace 
\vphantom{\overline{[\alpha]_t}}
[\alpha]_t,k, \overline{[\zeta-\delta_{w,u}\varepsilon_u]_u}
\big\rbrace,
\big\lbrace 
[\alpha-\delta_{t,v}\varepsilon_v]_v,k+1,\overline{[\zeta]_w}
\big\rbrace,
%\right\rbrace
%\bigcup_{i=0}^{a-1}, 
 \big\lbrace 
k-1-i
\big\rbrace _{0\leq i \leq a-1},
 \big\lbrace 
\overline{k+1-j}
\big\rbrace _{0\leq j \leq z-1}\right\}.
 \end{align*}
  
\end{prop}
 
% Before embarking on the proof, we provide an example.  

\begin{eg}\label{maudexample2}
Let $k+1=14$.  Let $\stt$ be any tableau such that 
$$
\stt(12)=(4,2,1^2)  \xrightarrow{\ -1\ }  (3,2,1^2) \xrightarrow{\ +2\ }  (3^2,1^2) \xrightarrow{\ -2\ } (3,2,1^2) \xrightarrow{\ +1\ } (4,2,1^2)=\stt(14).
$$ So that 
$\alpha=\zeta=(4,2,1^2)$,  $\beta=\delta=(3,2,1^2)$, $\gamma=(3^2,1^2)$ (so that $t=w=1$ and $u=v=2$).  
Then $\widehat S_{13}(\stt)=\widehat S$ from \cref{maudexample}.   
\end{eg}

\begin{proof}[Proof of \cref{thefinal2}]
By the definition of $P_k(\stt)$ given  in \cref{4.4}, we have  that 
  $$
P_k(\stt)e_{k-1}^{(a)} = 
e_{k+1}^{(z)}
\left(s_{[\zeta]_w,|\zeta|}e_{k+\frac{1}{2}}^{(d)}
s_{|\gamma|,x }\right)
e_{k }^{(c)}
\left(s_{[\gamma]_u, |\gamma| }e_{k-\frac{1}{2}}^{(b)}
s_{|\alpha|, [\alpha]_t}     \right)e_{k-1}^{(a)} 
$$
where 
$$
x=\begin{cases}
[\gamma]_u -1		&\text{if $u=v>0$ }\\
[\gamma]_v		&\text{otherwise. }
\end{cases}
$$
By concatenating diagrams, it is easy to see that
$$
s_{[\zeta]_w,|\zeta|}e_{k+\half}^{(d)}s_{|\gamma|,x}=\widehat {S_{k+1}}
$$
where 
\begin{equation}\label{eqSk+1}
 {S_{k+1}}
=\left\{
\{k+1,k,\dots, k-d+2, \overline{k+1},  \overline{k}, \dots ,  \overline{k-d+2},  \overline{[\zeta]_w}, x\}
\right\}
\end{equation}
if $v, w >0$.  If $w=0$, $S_{k+1}$ is obtained by replacing $ \overline{[\zeta]_w}$ with $ \overline{k-d+1}$ in \cref{eqSk+1}
 above.
  If $v=0$, $S_{k+1}$ is obtained by replacing $x$ with $ {k-d+1}$ in \cref{eqSk+1} above.
 Similarly, we have 
\begin{equation}\label{eqSk-1}
 s_{[\gamma]_u, |\gamma| }e_{k-\frac{1}{2}}^{(b)}
s_{|\alpha|, [\alpha]_t}      =\widehat{S_{k-1}}
\end{equation}
where
$$
 {S_{k-1}}=\left\{
 \{k,k-1,\dots, k-b+1, \overline{k}, \overline{k-1}, \dots, \overline{k-b+1}, \overline{[\gamma]_u}, [\alpha]_t\}
\right\}
$$
if $u, t >0$.  If $u=0$, then $S_{k-1}$ is obtained by replacing $\overline{[\gamma]_u}$ by $\overline{k-b}$ in 
\cref{eqSk-1} above.
 If $t=0$, then $S_{k-1}$ is obtained by replacing $\overline{[\alpha]_t}$ by ${k-b}$ in 
\cref{eqSk-1} above.
Now we have 
$$
P_k(\stt) e_{k-1}^{(a)}
=
e_{k+1}^{(z)}\widehat{S_{k+1}}
e_{k}^{(c)}\widehat{S_{k-1}}
e_{k-1}^{(a)}, 
$$
%An example of such a product %was given in \cref{maudexample2}.     
%which for $u,v,t,w \neq 0$ can be represented by the  concatenation of diagrams in \cref{BIG-diagram}, below.  
% This continues from \cref{maudexample}.  
 which for $u,v,t,w \neq 0$ can be represented by the  concatenation of diagrams of the form depicted in \cref{BIG-diagram}, below.  
% We have labelled by the points in these diagrams   %This continues from \cref{maudexample}.  
 This diagram is meant to be seen as a generic example of such a concatenation of diagrams; 
 however, it can also be seen to be the diagram obtained from the path $\stt$ in  \cref{maudexample}.

%    \draw (0.5,3) arc (180:360:0.5 and 0.25);
%    \draw (1.5,3) arc (180:360:1 and 0.25);
%     \draw (1.5,3) arc (180:360:1.5 and 0.75);
%    \draw (4.5,0) arc (0:180:1.5 and 1);
%    \draw (5.5,0) arc (0:180:1 and .7);
%    \draw (3.5,0) arc (0:180:.5 and .25);
%    \draw (6.5,0) arc (0:180:0.5 and 0.5);
%    \draw (4.5,3) arc (180:360:0.5 and 0.25);
%    \draw (5.5,3) arc (180:360:0.5 and 0.25);
%      \draw (2.5,-0.49) node {$3$};   \draw (2.5,+3.5) node {$\overline{3}$};
%                  \draw (1.5,-0.49) node {$2$};   \draw (1.5,+3.5) node {$\overline{2}$};
%                           \draw (0.5,-0.49) node {$1$};   \draw (0.5,+3.5) node {$\overline{1}$};
%         \draw (3.5,-0.49) node {$4$};   \draw (3.5,+3.5) node {$\overline{4}$};
%                  \draw (4.5,-0.49) node {$5$};   \draw (4.5,+3.5) node {$\overline{5}$};
%                           \draw (5.5,-0.49) node {$6$};   \draw (5.5,+3.5) node {$\overline{6}$};

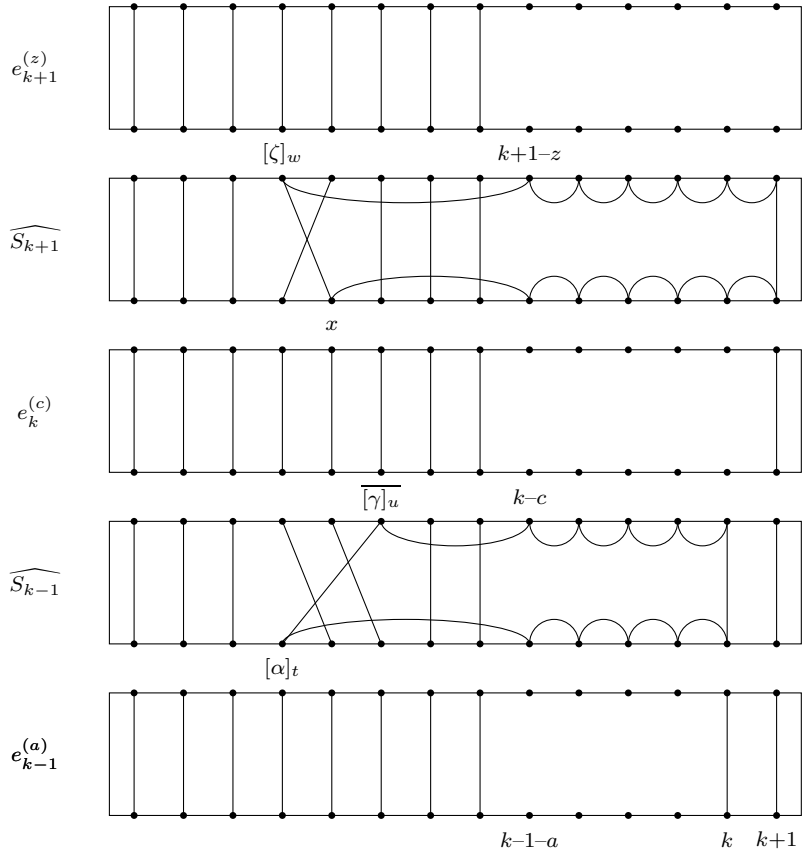
\begin{figure}[ht!]
$$\begin{tikzpicture}[scale=0.65]\scalefont{0.8}
  \begin{scope}  \draw (0,0) rectangle (14,2.5);
  \foreach \x in {0.5,1.5,...,13.5}
    {\fill (\x,2.5) circle (2pt);
     \fill (\x,0) circle (2pt);}
  \foreach \x in {0.5,1.5,...,7.5}{     \draw(\x,0)--(\x,2.5);}
    \foreach \x in {12.5,13.5}{     \draw(\x,0)--(\x,2.5);}
                \draw (-1.5,1.25) node {\scalefont{1}$e_{k-1}^{(a)}$};
 %[very thick]
%    \draw (0.5,2.5) -- (0.5,0);
%        \draw (1.5,2.5) -- (1.5,0);
%    \draw (2.5,2.5) -- (2.5,0);    \draw (7.5,2.5) -- (7.5,0);
%        \draw (8.5,2.5) -- (8.5,0);\draw (13.5,2.5) -- (13.5,0);
%        \draw (12.5,2.5) -- (12.5,0);
%    \draw (4.5,0) -- (3.5,2.5);    \draw (5.5,0) -- (4.5,2.5);
           \draw (12.5,-0.49) node {$k$};            \draw (8.5,-0.49) node {$k$--$1$--$a$}; 
             \draw (8.5,6.5) node {$k$--$c$};              \draw (8.5,13.5) node {$k$+$1$--$z$}; 
     \draw (13.5,-0.49) node {$k$+$1$};      \end{scope}
     \draw (3.5,3) node {$[\alpha]_t$}; %%%%%
  \begin{scope}  \draw (0,3.5) rectangle (14,6);
  \foreach \x in {0.5,1.5,...,13.5}
    {\fill (\x,6) circle (2pt);
     \fill (\x,3.5) circle (2pt);
 }
     \draw (0.5,6) -- (0.5,3.5);
        \draw (1.5,6) -- (1.5,3.5);
    \draw (2.5,6) -- (2.5,3.5);    \draw (7.5,6) -- (7.5,3.5);
        \draw (6.5,6) -- (6.5,3.5);\draw (13.5,6) -- (13.5,3.5);
        \draw (12.5,6) -- (12.5,3.5);
    \draw (4.5,3.5) -- (3.5,6);    \draw (5.5,3.5) -- (4.5,6);
  \draw(5.5,6)--(3.5,3.5); 
    \draw (3.5,3.5) arc (0:180:-2.5 and 0.5);    
        \draw (8.5,3.5) arc (0:180:-0.5 and 0.5);    
        \draw (9.5,3.5) arc (0:180:-0.5 and 0.5);    
        \draw (10.5,3.5) arc (0:180:-0.5 and 0.5); 
        \draw (11.5,3.5) arc (0:180:-0.5 and 0.5);           
 \draw (5.5,6) arc (0:180:-1.5 and -0.5);    
  \draw (8.5,6) arc (0:180:-0.5 and -0.5);    
  \draw (11.5,6) arc (0:180:-0.5 and -0.5);   \draw (10.5,6) arc (0:180:-0.5 and -0.5);   \draw (9.5,6) arc (0:180:-0.5 and -0.5);   
                  \draw (-1.5,4.75) node {\scalefont{1}$\widehat{S_{k-1}}$};   \end{scope}
%%%%%
 \begin{scope}  \draw (0,7) rectangle (14,9.5);
  \foreach \x in {0.5,1.5,...,13.5}
    {\fill (\x,9.5) circle (2pt);
     \fill (\x,7) circle (2pt);}
  \foreach \x in {0.5,1.5,...,7.5}{     \draw(\x,7)--(\x,9.5);}
    \foreach \x in {13.5}{     \draw(\x,7)--(\x,9.5);}
                \draw (-1.5,1.25) node {\scalefont{1}$e_{k-1}^{(a)}$};
                   \draw (-1.5,8.25) node {\scalefont{1}$e_{k}^{(c)}$};
 %[very thick]
%    \draw (0.5,9.5) -- (0.5,7);
%        \draw (1.5,9.5) -- (1.5,7);
%    \draw (2.5,9.5) -- (2.5,7);    \draw (7.5,9.5) -- (7.5,7);
%        \draw (8.5,9.5) -- (8.5,7);\draw (13.5,9.5) -- (13.5,7);
%        \draw (12.5,9.5) -- (12.5,7);
%    \draw (4.5,7) -- (3.5,9.5);    \draw (5.5,7) -- (4.5,9.5);
     \end{scope}
     \draw (5.5,6.5) node {$\overline{[\gamma]_u}$}; %%%%%
      \begin{scope}  \draw (0,13) rectangle (14,10.5);
  \foreach \x in {0.5,1.5,...,13.5}
    {\fill (\x,10.5) circle (2pt);
     \fill (\x,13) circle (2pt);
 }
     \draw (0.5,10.5) -- (0.5,13);
        \draw (1.5,10.5) -- (1.5,13);
    \draw (2.5,10.5) -- (2.5,13);    \draw (7.5,10.5) -- (7.5,13);
        \draw (6.5,10.5) -- (6.5,13);\draw (13.5,10.5) -- (13.5,13);
%        \draw (12.5,10.5) -- (12.5,13);
    \draw (5.5,13) -- (5.5,10.5);  
  \draw(3.5,10.5)--(4.5,13);   \draw(4.5,10.5)--(3.5,13); 
    \draw (3.5,13) arc (0:180:-2.5 and -0.5);    
        \draw (8.5,13) arc (0:180:-0.5 and -0.5);    
        \draw (9.5,13) arc (0:180:-0.5 and -0.5);    
        \draw (10.5,13) arc (0:180:-0.5 and -0.5); 
        \draw (11.5,13) arc (0:180:-0.5 and -0.5);           
 \draw (4.5,10.5) arc (0:180:-2 and 0.5);    
  \draw (8.5,10.5) arc (0:180:-0.5 and 0.5);    
  \draw (11.5,10.5) arc (0:180:-0.5 and 0.5);   \draw (10.5,10.5) arc (0:180:-0.5 and 0.5);   \draw (9.5,10.5) arc (0:180:-0.5 and 0.5);   
\draw (12.5,13) arc (0:180:-0.5 and -0.5);  \draw (12.5,10.5) arc (0:180:-0.5 and 0.5);                    \draw (-1.5,11.75) node {\scalefont{1}$\widehat{S_{k+1}}$}; 
\draw (4.5,10) node {$x$};
\draw (3.5,13.5) node {$[\zeta]_w$};
  \end{scope}
%%%%
 \begin{scope}  \draw (0,14) rectangle (14,16.5);
  \foreach \x in {0.5,1.5,...,13.5}
    {\fill (\x,16.5) circle (2pt);
     \fill (\x,14) circle (2pt);}
  \foreach \x in {0.5,1.5,...,7.5}{     \draw(\x,14)--(\x,16.5);}
%    \foreach \x in {12.5,13.5}{     \draw(\x,14)--(\x,16.5);}
                \draw (-1.5,15.25) node {\scalefont{1}$e_{k+1}^{(z)}$};
      \end{scope}
 \end{tikzpicture}
$$

\!\!\!\!\!\caption{An example of the product $P_k(\stt) e_{k-1}^{(a)}
=
e_{k+1}^{(z)}\widehat{S_{k+1}}
e_{k}^{(c)}\widehat{S_{k-1}}
e_{k-1}^{(a)}.
$}
\label{BIG-diagram}
\end{figure}

For $u,v,t,w \neq 0$  the result would follow if we can show that 
\begin{itemize}%[leftmargin=*]
\item[$(1)$] $\{\overline{x}, [\alpha-\delta_{t,v}\varepsilon_v]_v\}$ is a block of $\widehat{S_{k-1}}$;
\item[$(2)$] $\{[\gamma]_u, \overline{[\zeta-\delta_{u,w}\varepsilon_u]_u}\}$ is a block of $\widehat{S_{k+1}}$.  
\end{itemize}
To prove $(1)$, note that $\alpha-\varepsilon_t=\gamma-\varepsilon_u$  and the propagating lines in $\widehat{S_{k-1}}$  give a bijection between the nodes of these two partitions 
(reading along successive rows starting with the top row).  
So for $v\neq u$ we have that $\{\overline{[\gamma]_v}, [\alpha]_v\}$ is a block of $\widehat{S_{k-1}}$ unless $v=t$, in which case 
$\{\overline{[\gamma]_v}, [\alpha]_v-1\}$ is a block of $\widehat{S_{k-1}}$.  
Similarly, $\{\overline{[\gamma]_u-1},[\alpha]_u\}$ is a a block of $\widehat{S_{k-1}}$ unless $u=t$, 
in which case $[\gamma]_u=[\alpha]_u$ and $\{\overline{[\gamma]_u-1}, [\alpha]_u-1\}$ is a block of $\widehat{S_{k-1}}$.  

The proof of  $(2)$ follows similarly by noting that $\zeta-\varepsilon_w=\gamma-\varepsilon_v$ and that the propagating
 lines in $\widehat{S_{k+1}}$ give a bijection between the nodes of these partitions. 
  So we have that 
$ \{[\gamma]_u,\overline{[\zeta]_u}\} $ is a block of $\widehat{S_{k+1}}$ unless $u=w$, in which case 
$ \{[\gamma]_u,\overline{[\zeta]_u-1}\} $
is a block of 
$\widehat{S_{k+1}}$. 
 For $u=v$, note that $\{[\gamma]_u,\overline{[\gamma]_u}\}$ is a block of $\widehat{S _{k+1}}$ unless $[\gamma]_u \leq [\zeta]_w$, in which case 
$\{[\gamma]_u,\overline{[\gamma]_u-1}\}$ is a block of $\widehat{S _{k+1}}$.     
%Now we have  $\zeta-\varepsilon_w=\gamma-\varepsilon_u$.
If $w<u$ then $[\zeta]_w = [\gamma]_w-1<[\gamma]_u$ and $[\gamma]_u=[\zeta]_u$ 
so $\{[\gamma]_u,\overline{[\zeta]_u}\}$ is a block of $\widehat{S _{k+1}}$, as required.  
If $u<w$ then $[\gamma]_u-1=[\zeta]_u<[\zeta]_w$ so $[\gamma]_u\leq [\zeta]_w$ and 
$\{[\gamma]_u,\overline{[\zeta]_u}\}$ is a block of $\widehat{S _{k+1}}$, as required.  
Finally, if $u=w$ then $\gamma=\zeta$ and $[\gamma]_u=[\zeta]_u=[\zeta]_w$ and 
$\{[\gamma]_u,\overline{[\zeta]_u-1}\}$ is a block of $\widehat{S _{k+1}}$, as required.  
This completes the proof for $t,u,v,w \neq 0$. 

We now consider the cases in which some of  $t,u,v,w $ are equal to zero.  We treat these as degenerate versions  of the above.
%In the above we saw that $S_{k-1}$ established a bijection between the nodes of the partitions $\alpha-\varepsilon_t$ and 
%The key in the above proof was that 
%\begin{itemize}[leftmargin=*]
%\item 
% $\alpha-\varepsilon_t=\gamma-\varepsilon_u$  and the propagating lines in $\widehat{S_{k-1}}$  give a bijection between the nodes of these two partitions;  
%\item $\zeta-\varepsilon_w=\gamma-\varepsilon_v$ and  the propagating
% lines in $\widehat{S_{k+1}}$ give a bijection between the nodes of these partitions.  
%\end{itemize}

Let $w=0$.   This is the simplest degenerate case to describe, however the other cases only differ by superficial book-keeping.  
  If $w=0$, then $z=d+1$ and $\gamma-\varepsilon_v=\zeta$.  We replace the top two diagrams in \cref{BIG-diagram} by the two diagrams in \cref{BIG-diagram1} (which 
  establish the bijection between the nodes of  $\gamma-\varepsilon_v$ and $\zeta$).  
% The values of $[\gamma]_u$, $x$, and $[\zeta]_w$ go through unchanged.  
 The values of $a,b,c,d$, $[\alpha]_t$, $[\gamma]_u$ and $x$   go through unchanged.  
    Thus the block containing $k+1$ in $\widehat{S_{k}{(\stt)}}$ collapses to 
  $\{[\alpha-\delta_{t,v}\varepsilon_v]_v, k+1\}$ as required.

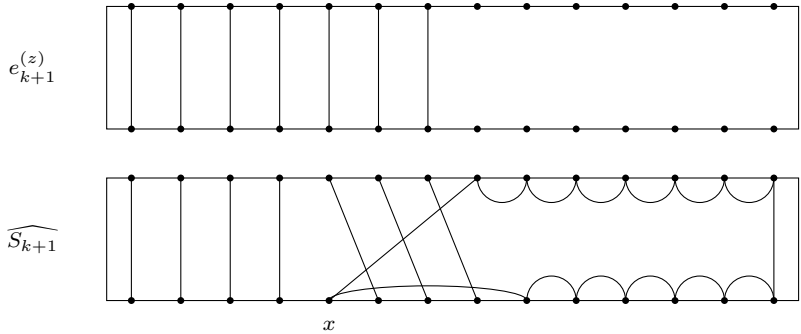
\begin{figure}[ht!]
  $$\begin{tikzpicture}[scale=0.65]\scalefont{0.8}
      \begin{scope}  \draw (0,13) rectangle (14,10.5);
  \foreach \x in {0.5,1.5,...,13.5}
    {\fill (\x,10.5) circle (2pt);
     \fill (\x,13) circle (2pt);
 }
     \draw (0.5,10.5) -- (0.5,13);
        \draw (1.5,10.5) -- (1.5,13);
    \draw (2.5,10.5) -- (2.5,13);    \draw (7.5,10.5) -- (6.5,13);
        \draw (6.5,10.5) -- (5.5,13);\draw (13.5,10.5) -- (13.5,13);
%        \draw (12.5,10.5) -- (12.5,13);
    \draw (4.5,13) -- (5.5,10.5);  
  \draw(3.5,10.5)--(3.5,13);    \draw(4.5,10.5)--(7.5,13); 
%    \draw (3.5,13) arc (0:180:-2.5 and -0.5);    
\draw (7.5,13) arc (0:180:-0.5 and -0.5);        \draw (8.5,13) arc (0:180:-0.5 and -0.5);    
        \draw (9.5,13) arc (0:180:-0.5 and -0.5);    
        \draw (10.5,13) arc (0:180:-0.5 and -0.5); 
        \draw (11.5,13) arc (0:180:-0.5 and -0.5);           
 \draw (4.5,10.5) arc (0:180:-2 and 0.3);    
  \draw (8.5,10.5) arc (0:180:-0.5 and 0.5);    
  \draw (11.5,10.5) arc (0:180:-0.5 and 0.5);   \draw (10.5,10.5) arc (0:180:-0.5 and 0.5);   \draw (9.5,10.5) arc (0:180:-0.5 and 0.5);   
\draw (12.5,13) arc (0:180:-0.5 and -0.5);  \draw (12.5,10.5) arc (0:180:-0.5 and 0.5);                    \draw (-1.5,11.75) node {\scalefont{1}$\widehat{S_{k+1}}$}; 
\draw (4.5,10) node {$x$};
%\draw (3.5,13.5) node {$[\zeta]_w$};
  \end{scope}
%%%%
 \begin{scope}  \draw (0,14) rectangle (14,16.5);
  \foreach \x in {0.5,1.5,...,13.5}
    {\fill (\x,16.5) circle (2pt);
     \fill (\x,14) circle (2pt);}
  \foreach \x in {0.5,1.5,...,6.5}{     \draw(\x,14)--(\x,16.5);}
%    \foreach \x in {12.5,13.5}{     \draw(\x,14)--(\x,16.5);}
                \draw (-1.5,15.25) node {\scalefont{1}$e_{k+1}^{(z)}$};
      \end{scope}
 \end{tikzpicture}
$$

\!\!\!\!\!  

\caption{The $w=0$ case.}
\label{BIG-diagram1}
\end{figure}

If $v=0$, then $c=d$ and $\gamma=\zeta-\varepsilon_w$.  
We replace $\widehat{S_{k+1}}e_k^{(c)}$ in \cref{BIG-diagram} by the two diagrams in \cref{BIG-diagram3.5} (which 
  establish the bijection between $\gamma $ and $\zeta-\varepsilon_w$).  
  The values of $a,b,c$, $[\alpha]_t$, and $[\gamma]_u$  go through unchanged and so the bottom two diagrams of  \cref{BIG-diagram} go through unchanged.  
 Therefore the block containing $k+1$ collapses to 
$\{k+1,\overline{[\zeta]_w}\}$ as required.  
The value of   $[\zeta]_w$ will either decrease by 1 (if   $ w \geq v$)
  or go through unchanged (if   $ w < t$ as in the case depicted in \cref{BIG-diagram3.5}).   
This results in the necessary superficial edits  to the propagating lines in $ \widehat{S_{k+1}} $ in order to obtain  the required 
   bijection between the nodes of   $\gamma $ and $\zeta-\varepsilon_w$; hence all the blocks of $\widehat{S_k(\stt)}$ which do not contain
    $k+1$ remain unchanged.

% 
%
%If $v=0$ then $c=d$ and we have $\widehat{S_{k+1}}e_k^{(c)}$ is given  
%by the   diagram in \cref{BIG-diagram3.5}.  Arguing as above, the block containing $k+1$ collapses to 
%$\{k+1,\overline{[\zeta]_w}\}$ as required.  

\begin{figure}[ht!]
 $$\begin{tikzpicture}[scale=0.65]\scalefont{0.8}
  \begin{scope}  \draw (0,7) rectangle (14,9.5);
  \foreach \x in {0.5,1.5,...,13.5}
    {\fill (\x,9.5) circle (2pt);
     \fill (\x,7) circle (2pt);}
  \foreach \x in {0.5,1.5,...,6.5}{     \draw(\x,7)--(\x,9.5);}
    \foreach \x in {13.5}{     \draw(\x,7)--(\x,9.5);}
      \end{scope}
        \begin{scope}  \draw (0,13) rectangle (14,10.5);
  \foreach \x in {0.5,1.5,...,13.5}
    {\fill (\x,10.5) circle (2pt);
     \fill (\x,13) circle (2pt);
 }
     \draw (0.5,10.5) -- (0.5,13);
        \draw (1.5,10.5) -- (1.5,13);
    \draw (2.5,10.5) -- (2.5,13);             \draw (3.5,10.5) -- (4.5,13);        \draw (6.5,10.5) -- (7.5,13);
        \draw (5.5,10.5) -- (6.5,13);\draw (13.5,10.5) -- (13.5,13);
%        \draw (12.5,10.5) -- (12.5,13);
    \draw (5.5,13) -- (4.5,10.5);  
  \draw(7.5,10.5)--(3.5,13);   
    \draw (3.5,13) arc (0:180:-2.5 and -0.3);    
        \draw (8.5,13) arc (0:180:-0.5 and -0.5);    
        \draw (9.5,13) arc (0:180:-0.5 and -0.5);    
        \draw (10.5,13) arc (0:180:-0.5 and -0.5); 
        \draw (11.5,13) arc (0:180:-0.5 and -0.5);           
% \draw (4.5,10.5) arc (0:180:-2 and 0.5);    
  \draw (8.5,10.5) arc (0:180:-0.5 and 0.5);      \draw (7.5,10.5) arc (0:180:-0.5 and 0.5);    
  \draw (11.5,10.5) arc (0:180:-0.5 and 0.5);   \draw (10.5,10.5) arc (0:180:-0.5 and 0.5);   \draw (9.5,10.5) arc (0:180:-0.5 and 0.5);   
   \end{scope} \draw (-1.5,8.25) node {\scalefont{1}$e_{k}^{(c)}$};       \draw (-1.5,11.75) node {\scalefont{1}$\widehat{S_{k+1}}$}; 
\end{tikzpicture}
$$

\!\!\!\!\!

\caption{The $v=0$ case.}
\label{BIG-diagram3.5}
\end{figure}
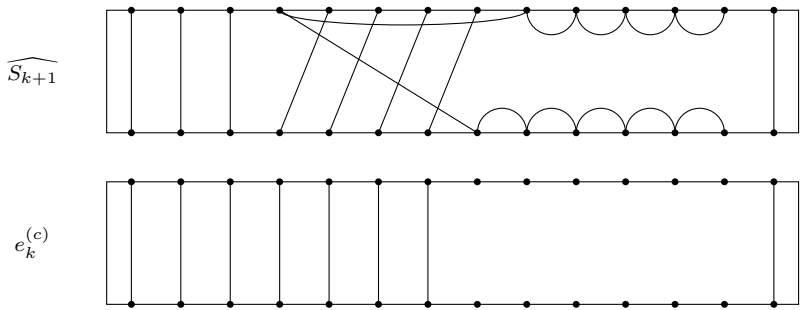

Similarly if $t=0$, then $a=b$ and $\alpha= \gamma-\varepsilon_u$.   
We replace the bottom two diagrams in \cref{BIG-diagram} by the two diagrams in \cref{BIG-diagram2} which 
  establish the bijection between $\alpha$ and $\gamma-\varepsilon_u$.   
         Thus the block containing $k$ in $\widehat{S_{k}{(\stt)}}$ collapses to 
  $\{k,[\zeta-\delta_{w,u}\varepsilon_u]_u\}$ as required.  
 As above, one can verify that all other blocks of $S_k(\stt)$ remain the same, as required.  
%   The values of $c,d,z$ all decrease by 1.  
% The values of    $[\gamma]_u$, $x$, and $[\zeta]_w$ will either increase by 1 (if   $u,v,w \geq t$ respectively)
%  or go through unchanged (if   $u,v,w < t$ respectively).  
% This results in the necessary superficial edits  to the propagating lines in $e_{k+1}^{(z)}\widehat{S_{k+1}}e_k^{(c)}$ in order to obtain  the required 
%   bijection between the nodes of   $\gamma-\varepsilon_u$ and $\zeta-\varepsilon_w$ and hence all the blocks of $\widehat{S_k(\stt)}$ which do not contain
%    $k$ remain unchanged.  
%

\begin{figure}[ht!]
$$\begin{tikzpicture}[scale=0.65]\scalefont{0.8}
  \begin{scope}  \draw (0,0) rectangle (14,2.5);
  \foreach \x in {0.5,1.5,...,13.5}
    {\fill (\x,2.5) circle (2pt);
     \fill (\x,0) circle (2pt);}
  \foreach \x in {0.5,1.5,...,7.5}{     \draw(\x,0)--(\x,2.5);}
    \foreach \x in {12.5,13.5}{     \draw(\x,0)--(\x,2.5);}
                \draw (-1.5,1.25) node {\scalefont{1}$e_{k-1}^{(a)}$};
             \draw (12.5,-0.49) node {$k$}; 
     \draw (13.5,-0.49) node {$k$+$1$};      \end{scope}
   \begin{scope}  \draw (0,3.5) rectangle (14,6);
  \foreach \x in {0.5,1.5,...,13.5}
    {\fill (\x,6) circle (2pt);
     \fill (\x,3.5) circle (2pt);
 }
     \draw (0.5,6) -- (0.5,3.5);
        \draw (1.5,6) -- (1.5,3.5);    \draw (3.5,6) -- (3.5,3.5);     \draw (4.5,6) -- (4.5,3.5);  \draw (5.5,6) -- (5.5,3.5); 
    \draw (2.5,6) -- (2.5,3.5);    \draw (8.5,6) -- (7.5,3.5);
        \draw (7.5,6) -- (6.5,3.5);\draw (13.5,6) -- (13.5,3.5);
        \draw (12.5,6) -- (12.5,3.5);
%    \draw (4.5,3.5) -- (3.5,6);    \draw (5.5,3.5) -- (4.5,6);
   \draw(6.5,6)--(8.5,3.5); 
%    \draw (3.5,3.5) arc (0:180:-2.5 and 0.5);    
        \draw (8.5,3.5) arc (0:180:-0.5 and 0.5);    
        \draw (9.5,3.5) arc (0:180:-0.5 and 0.5);    
        \draw (10.5,3.5) arc (0:180:-0.5 and 0.5); 
        \draw (11.5,3.5) arc (0:180:-0.5 and 0.5);           
 \draw (6.5,6) arc (0:180:-1.5 and -0.3);    
%  \draw (8.5,6) arc (0:180:-0.5 and -0.5);    
  \draw (11.5,6) arc (0:180:-0.5 and -0.5);   \draw (10.5,6) arc (0:180:-0.5 and -0.5);   \draw (9.5,6) arc (0:180:-0.5 and -0.5);   
                  \draw (-1.5,4.75) node {\scalefont{1}$\widehat{S_{k-1}}$};   \end{scope}
%%%%%
\end{tikzpicture}
$$

\!\!\!\!\!

\caption{The $t=0$ case.}
\label{BIG-diagram2}
\end{figure}
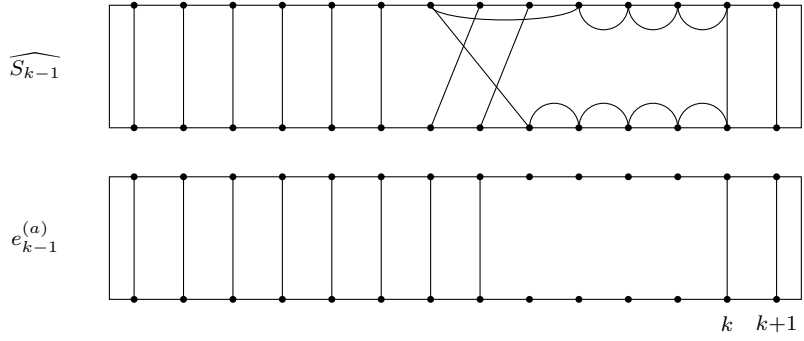
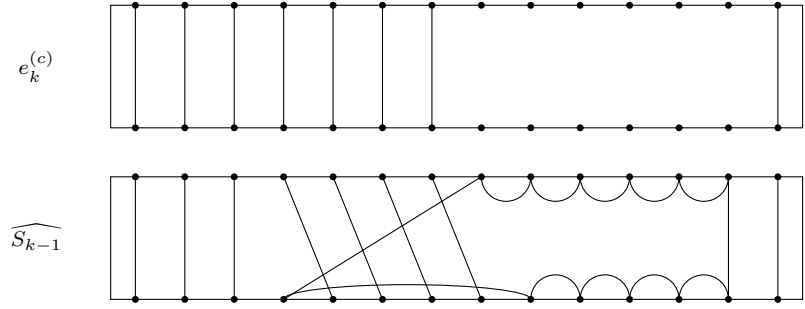
\begin{figure}[ht!]
  $$\begin{tikzpicture}[scale=0.65]\scalefont{0.8}
%        \begin{scope}  \draw (0,0) rectangle (14,2.5);
%  \foreach \x in {0.5,1.5,...,13.5}
%    {\fill (\x,2.5) circle (2pt);
%     \fill (\x,0) circle (2pt);}
%  \foreach \x in {0.5,1.5,...,7.5}{     \draw(\x,0)--(\x,2.5);}
%    \foreach \x in {12.5,13.5}{     \draw(\x,0)--(\x,2.5);}
%                \draw (-1.5,1.25) node {\scalefont{1}$e_{k-1}^{(a)}$};
%            \draw (12.5,-0.49) node {$k$}; 
%     \draw (13.5,-0.49) node {$k$+$1$};      \end{scope}
%     \draw (3.5,3) node {$[\alpha]_t$}; %%%%%
  \draw (-1.5,8.25) node {\scalefont{1}$e_{k}^{(c)}$};                    \draw (-1.5,4.75) node {\scalefont{1}$\widehat{S_{k-1}}$};   \begin{scope}  \draw (0,3.5) rectangle (14,6);
  \foreach \x in {0.5,1.5,...,13.5}
    {\fill (\x,6) circle (2pt);
     \fill (\x,3.5) circle (2pt);
 }
     \draw (0.5,6) -- (0.5,3.5);
        \draw (1.5,6) -- (1.5,3.5);
    \draw (2.5,6) -- (2.5,3.5);    \draw (6.5,6) -- (7.5,3.5);
        \draw (5.5,6) -- (6.5,3.5);\draw (13.5,6) -- (13.5,3.5);
        \draw (12.5,6) -- (12.5,3.5);        \draw (7.5,6) -- (3.5,3.5);
    \draw (4.5,3.5) -- (3.5,6);    \draw (5.5,3.5) -- (4.5,6);
     \draw (3.5,3.5) arc (0:180:-2.5 and 0.3);    
        \draw (8.5,3.5) arc (0:180:-0.5 and 0.5);    
        \draw (9.5,3.5) arc (0:180:-0.5 and 0.5);    
        \draw (10.5,3.5) arc (0:180:-0.5 and 0.5); 
        \draw (11.5,3.5) arc (0:180:-0.5 and 0.5);           
   \draw (8.5,6) arc (0:180:-0.5 and -0.5);       \draw (7.5,6) arc (0:180:-0.5 and -0.5);  
  \draw (11.5,6) arc (0:180:-0.5 and -0.5);   \draw (10.5,6) arc (0:180:-0.5 and -0.5);   \draw (9.5,6) arc (0:180:-0.5 and -0.5);   
  \end{scope}
%%%%%
 \begin{scope}  \draw (0,7) rectangle (14,9.5);
  \foreach \x in {0.5,1.5,...,13.5}
    {\fill (\x,9.5) circle (2pt);
     \fill (\x,7) circle (2pt);}
  \foreach \x in {0.5,1.5,...,6.5}{     \draw(\x,7)--(\x,9.5);}
    \foreach \x in {13.5}{     \draw(\x,7)--(\x,9.5);}
           %[very thick]
%    \draw (0.5,9.5) -- (0.5,7);
%        \draw (1.5,9.5) -- (1.5,7);
%    \draw (2.5,9.5) -- (2.5,7);    \draw (7.5,9.5) -- (7.5,7);
%        \draw (8.5,9.5) -- (8.5,7);\draw (13.5,9.5) -- (13.5,7);
%        \draw (12.5,9.5) -- (12.5,7);
%    \draw (4.5,7) -- (3.5,9.5);    \draw (5.5,7) -- (4.5,9.5);
     \end{scope}
 \end{tikzpicture} $$

\!\!\!\!\!\caption{The $u=0$  case.}
\label{lgshdfhjkls}
\end{figure}
Finally, if $u=0$ then $c=b+1$ and we have $e_k^{(c)}\widehat{S_{k-1}}$ is given  
by the leftmost diagram in \cref{lgshdfhjkls}.  Arguing as above, the block of $S_k(\stt)$ containing $k$ collapses to 
$\{[\alpha]_t,k\}$ and all other blocks in $S_k(\stt)$  remain the same, as required.   
 \end{proof}

\begin{prop}\label{proposition8}
Assume that $  \stt_{k \leftrightarrow k+1}$ exists.  Then we have 
$$
m_{\zeta,u,w}P_k(\stt) e_{k-1}^{(a)}s_{k,k+1}\sum_{\sigma \in \mathfrak{S}_\alpha }\sigma
=
m_{\zeta,u,w}P_k(\stt_{k \leftrightarrow k+1}) e_{k-1}^{(a)} \sum_{\sigma \in \mathfrak{S}_\alpha }\sigma.  
$$
\end{prop}

\begin{proof}
Using \cref{thefinal2} and the fact that $s_{k,k+1}$ swaps $k$ and $k+1$, we have 
$$
P_k(\stt) e_{k-1}^{(a)}s_{k,k+1}= \widehat{S_k(\stt)}s_{k,k+1}= \widehat{S'_{k}(\stt)}
$$
where $ {S'_{k}(\stt)}$ is 
obtained by omitting all occurrences of $0$ and $\overline 0$ from 
\begin{align*} 
 \left\{\big\lbrace 
\vphantom{\overline{[\alpha]_t}}
[\alpha]_t,k+1, \overline{[\zeta-\delta_{w,u}\varepsilon_u]_u}
\big\rbrace,
\big\lbrace 
[\alpha-\delta_{t,v}\varepsilon_v]_v,k,\overline{[\zeta]_w}
\big\rbrace,
%\right\rbrace
%\bigcup_{i=0}^{a-1}, 
 \big\lbrace 
k-1-i
\big\rbrace _{0\leq i \leq a-1},
 \big\lbrace 
\overline{k+1-j}
\big\rbrace _{0\leq j \leq z-1}\right\}.
 \end{align*}
Now, we observe (simply by   definition) that 
 $ {S_k(\stt_{k\leftrightarrow k+1})}$ is obtained by  omitting all occurrences of $0$ and $\overline 0$ from 
\begin{align*} 
 \left\{\big\lbrace 
\vphantom{\overline{[\alpha]_v}}
[\alpha]_v,k+1, \overline{[\zeta-\delta_{w,u}\varepsilon_w]_w}
\big\rbrace,
\big\lbrace 
[\alpha-\delta_{t,v}\varepsilon_t]_t,k,\overline{[\zeta]_u}
\big\rbrace,
%\right\rbrace
%\bigcup_{i=0}^{a-1}, 
 \big\lbrace 
k-1-i
\big\rbrace _{0\leq i \leq a-1},
 \big\lbrace 
\overline{k+1-j}
\big\rbrace _{0\leq j \leq z-1}\right\}.
 \end{align*}
 So we get
 $$
\widehat{ S_k'(\stt)}
=
s_{[\zeta-\delta_{u,w}\varepsilon_w]_w, [\zeta]_w	}
\widehat{ S_k(\stt_{k\leftrightarrow k+1})} 
s_{[\alpha-\delta_{t,v}\varepsilon_t]_t,[\alpha]_t}.$$
If     $t\neq v$, then $s_{[\alpha-\delta_{t,v}\varepsilon_t]_t,[\alpha]_t}=1$ and if 
$t=v$, then 
$$
s_{[\alpha-\varepsilon_t]_t,[\alpha]_t}  \sum_{\sigma \in \mathfrak{S}_\alpha }\sigma = \sum_{\sigma \in \mathfrak{S}_\alpha }\sigma,  
$$ as required.  
 If $u\neq w$ then 
$s_{[\zeta-\varepsilon_w]_w, [\zeta]_w	}=1$.     Finally, if $u=w$, then 
%$$(m_{\zeta,w,w} )s_{[\zeta-\varepsilon_w]_w, [\zeta]_w	}
%=
%\left(\sum_{1\leq i < j \leq n} s_{i, [\zeta-\varepsilon_w]_w}s_{j, [\zeta-\varepsilon_w]_w}
%(1+s_{[\zeta-\varepsilon_w]_w, [\zeta]_w	})
%\right)s_{[\zeta-\varepsilon_w]_w, [\zeta]_w	}
%=
%m_{\zeta,w,w} 
%$$
$$%(
m_{\zeta,w,w} %)s_{[\zeta-\varepsilon_w]_w, [\zeta]_w	}
=
%\left(
\sum_{1\leq j < i \leq  [\zeta]_w}
s_{[\zeta]_w-j, [\zeta]_w}
 s_{ [\zeta-\varepsilon_w]_w-i, [\zeta-\varepsilon_w]_w}
\left(1+s_{[\zeta-\varepsilon_w]_w, [\zeta]_w	}\right).
%\right)%s_{[\zeta-\varepsilon_w]_w, [\zeta]_w	}
$$
Clearly, we have that  $$\left(1+s_{[\zeta-\varepsilon_w]_w, [\zeta]_w	}\right)
 s_{[\zeta-\varepsilon_w]_w, [\zeta]_w	}= \left(1+s_{[\zeta-\varepsilon_w]_w, [\zeta]_w	}\right)$$  and therefore  $(m_{\zeta,w,w} )s_{[\zeta-\varepsilon_w]_w, [\zeta]_w	}   =  m_{\zeta,w,w} $.  
The result follows.  
 \end{proof}

\noindent Finally, we let $\stt':=\stt_{k\leftrightarrow k+1}$ and $\sts'=e_k(\stt')$.    
Combining \cref{triplestar,tripledagger}  and \cref{proposition8}, we get 
\begin{align*}
u_\stt s_{k,k+1}
&=u_{\stt[k+1,r]}m_{\zeta,u,w} P_k(\stt') e_{k-1}^{(a)}
 (\textstyle   \sum_{\sigma \in \mathfrak{S}_\alpha }\sigma) d_{\stt[0,k-1]}^\ast  + (1-\delta_{u,0})\delta_{u,v}  u_\sts \\
 &=  u_{\stt[k+1,r]}m_{\zeta,u,w} P_k(\stt')u_{\stt[0,k-1]} + (1-\delta_{u,0})\delta_{u,v} u_\sts \\
 &=u_{\stt[k+1,r]}
 \left(
 u_{\stt'_{[k-1,k+1]}} - (1- \delta_{w,0})\delta_{w,t}u_{\sts'[k-1,k+1]}\right)
 u_{\stt[0,k-1]}
 + (1-\delta_{u,0})\delta_{u,v})u_\sts \\
 &= u_{\stt'} +
 (1- \delta_{u,0})\delta_{u,v} u_\sts 
 - 
 (1- \delta_{w,0})\delta_{w,t}u_{\sts'}
\end{align*}
which completes the proof of \cref{MURPHYPrN}.

\section{ Skew cell modules for co-Pieri triples   }\label{QUOTIENT!}

We continue with the in-depth partition algebra calculation necessary for our main proofs in Sections 5 and 6.  
 As before, we identify $P_s(n) $ as a subalgebra of $P_r(n)$ via the embedding $P_s(n) \cong \CC \otimes P_s(n)\subseteq P_{r-s}(n) \otimes P_s(n) \subseteq P_r(n)$, that is we view each partition diagram in $P_s(n)$ as a set-partition of $\{r-s+1,\dots, r,\overline{r-s+1},\dots, \overline{r}\} $.   We also assume throughout this section that $n\gg r$.
 We have seen in \cref{sec2} that 
 $$
 g(\lambda_{[n]},\nu_{[n]}, \mu_{[n]}) = \overline{ g}( \lambda, \nu,\mu)
 =
 \dim_\CC(\Hom_{P_s(n)}(\Delta_s(\mu), \Delta_s(\nu\setminus\lambda)))
 $$
for any triple of partitions $(\lambda,\nu,\mu) \in \mathscr{P}_{r-s}\times \mathscr{P}_{\leq r}\times 
\mathscr{P}_{s}$.  Now, as $|\mu|=s$ we have that the ideal $P_s(n)p_r P_s(n)\subset P_s(n)$  annihilates $\Delta_s(\mu)$ and so 
$$
 \overline{ g}( \lambda, \nu,\mu)
 =
 \dim_\CC(\Hom_{P_s(n)}(\Delta_s(\mu), \Delta_s(\nu\setminus\lambda)
 / 
( \Delta_s(\nu\setminus\lambda)P_s(n)p_rP_s(n) ) 
)).
 $$
 
 \begin{defn} We define the {\sf Dvir radical} of the  skew module $\Delta_s(\nu\setminus\lambda)$ by 
$$  {\sf DR}_s(\nu\setminus\lambda)  =  \Delta_s(\nu\setminus\lambda)P_s(n)p_rP_s(n) 
\subseteq \Delta_s(\nu\setminus\lambda) $$
and set $$\Delta^0_s(\nu\setminus\lambda)= 
 \Delta_s(\nu\setminus\lambda) /{\sf DR}_s(\nu\setminus\lambda). $$
 
 \end{defn}

\noindent By definition, we have that  
\begin{equation}\label{depthlayer2}
\overline{ g}( \lambda, \nu,\mu)
 =
 \dim_\CC(\Hom_{P_s(n)}(\Delta_s(\mu), \Delta_s(\nu\setminus\lambda)))
 =
  \dim_\CC(\Hom_{P_s(n)}(\Delta_s(\mu), \Delta^0_s(\nu\setminus\lambda)))  
\end{equation} 
 for any  $\mu \in \mathscr{P}_{s} $.
Thus, in order to understand  the coefficients $\overline{ g}( \lambda, \nu,\mu)
$, we need to construct  a basis for the  modules $\Delta_s^0(\nu\setminus\lambda)$ and to describe the $P_s(n)$-action on this basis.  
    Towards that end, we make the following defintion.  
    
    \begin{defn}For $(\lambda, \nu, s)\in  \mathscr{P}_{r-s} \times \mathscr{P}_{\leq r} \times \mathbb{Z}_{>0}$ we define 
   $$ {\rm  DR}^0\mhyphen\Std_s(\nu\setminus\lambda)   = \{ \stt\in  \Std_s(\nu\setminus\lambda)
\mid
 \sharp \{ \text{integral steps of the form
 $(-\varepsilon_0,+\varepsilon_0)$ in } \stt\} \geq 1\},
 $$
and for   $i \geq1$, we define 
$$
 {\rm  DR}^i\mhyphen\Std_s(\nu\setminus\lambda)  = \{ \stt\in \Std_s(\nu\setminus\lambda)
\mid
\sharp\{ \text{steps of the form $-\varepsilon_i$ in $\stt$} \}  > \lambda_i \}.
$$
 and we set    ${\rm  DR}\mhyphen\Std_s(\nu\setminus\lambda)= \bigcup_{i\geq0}{\rm  DR}^i\mhyphen\Std_s(\nu\setminus\lambda)$. 
 \end{defn} 
 
Note that for $i\geq 1$ we can also define ${\rm DR}^i\mhyphen \Std_s(\nu\setminus\lambda)$ as  $${\rm DR}^i\mhyphen\Std_s(\nu\setminus \lambda) = \{\stt \in \Std_s(\nu \setminus \lambda)\mid
\sharp\{ \text{steps of the form $+\varepsilon_i$ in $\stt$} \}  > \nu_i \}.$$ This follows from the fact that $\lambda_i - \sharp\{ \text{steps of the form $-\varepsilon_i$ in $\stt$}\} + \sharp\{ \text{steps of the form $+\varepsilon_i$ in $\stt$}\} = \nu_i$.  
 
We will prove the following result.   
\begin{prop}\label{Dvirinclusion}

If $\stt \in {\rm  DR}\mhyphen\Std_s(\nu\setminus\lambda)$ then $u_{\stt^\lambda\circ \stt}+P_{r,s}^{\rhd\nu\setminus\lambda}(n)\in {\sf DR}_s(\nu\setminus\lambda)$.
\end{prop}

We can write $u_{\stt^\lambda\circ \stt}$ as a sum of partition  diagrams in $P_r(n)$.  In order to prove the above proposition we need to understand some properties of the diagrams that can occur in this sum.

\begin{lem} \label{MaudLemmaA}Let  $\stt=(-\varepsilon_{i_1}, + \varepsilon_{j_1},  \ldots ,-\varepsilon_{i_{r}}, + \varepsilon_{j_{r}})\in \Std_r(\nu).$    Write 
$u_\stt= u_{\stt[r-1,r]} u_{\stt[0,r-1]} $
where \newline $\stt[0,r-1] \in \Std_{r-1}(\nu')$ with $\stt(r-1)=\nu'$.  
We have that 
\begin{itemize}
\item[$(i)$] if $i_r, j_r \neq 0$ then 
$u_{\stt[r-1,r]}=\sum_{\ik=0}^{\nu_{j_r}-1}d_{\ik}$ with $d_{\ik}$ as in the first diagram in \cref{dis}.
\item[$(ii)$] if $i_r=0$, $j_r \neq 0$ then 
$u_{\stt[r-1,r]}=\sum_{\ik=0}^{\nu_{j_r}-1}d_{\ik}$ with $d_{\ik}$ as in the second diagram in \cref{dis}.
\item[$(iii)$] if $i_r\neq 0$, $j_r =0$ then 
$u_{\stt[r-1,r]}= d_0$   as in the third  diagram in \cref{dis}.
\item[$(iv)$]  if $i_r = j_r =0$ then 
$u_{\stt[r-1,r]}= d_0= e_{r}^{(1)}$ depicted in \cref{sij}.
\end{itemize}
\end{lem}
\begin{figure}[ht!]
$$
d_k=\begin{minipage}{105mm}\scalefont{0.7}\begin{tikzpicture}[scale=0.7]
  \draw[thick] (-1,0) rectangle (13,3);
  \foreach \x in {-0.5,0.5,2.5,3.5,4.5,7.5,7.5+1,8.5+1,11.5,12.5}
    {\fill (\x,3) circle (3pt);
     \fill (\x,0) circle (3pt);}
   \fill (   5.5+1,0) circle (3pt);   
    \fill (   4.5+1,3) circle (3pt);   
   \begin{scope}%[thick]
   \draw[thick] (-0.5,3) -- (-0.5,0);
   \draw[thick] (0.5,3) -- (0.5,0);
    \draw[thick] (2.5,3) -- (2.5,0);
        \draw[thick] (4.5,3) -- (3.5,0);
    \draw[thick] (5.5,3) -- (4.5,0);
     \draw[thick] (5.5+2,3) -- (4.5+2,0);
     \draw[thick] (5.5+3,3) -- (4.5+4,0);
     \draw[thick] (5.5+3+1,3) -- (4.5+4+1,0);   
          \draw[thick] (11.5,3) -- (11.5,0);   
\draw[thick] (7.5,0) to [out=30,in=150] (12.5,0);             
\draw[thick] (3.5,3) to [out=-90,in=90] (12.5,0);     
\draw[thick][white,fill](10.1,3.2) rectangle (10.9,-0.1)        ;
\draw[thick][white,fill](1.1,3.2) rectangle (1.9,-0.1)        ;
\draw[thick][white,fill](7.1-1,3.2) -- (7.9-1,3.2) -- (6.9-1,-0.1) -- (6.1-1,-0.1) --(7.1-1,3.2)       ;
  \draw[densely dotted,thick] (-1,0) rectangle (13,3);
\draw[densely dotted,thick] (7.5,0) to [out=30,in=150] (12.5,0);             
\draw[densely dotted,thick] (3.5,3) to [out=-90,in=90] (12.5,0);     
    \end{scope}
    \draw[below](12.5,0)   node {$r$};
        \draw[below](7.5,0)   node {$[\nu']_{i_r}$};
                \draw[above](3.5,3.1)   node {$[\nu]_{j_r}$--$k$};
\end{tikzpicture}\end{minipage} 
$$

$$
d_k=\begin{minipage}{105mm}\scalefont{0.7}\begin{tikzpicture}[scale=0.7]
  \draw[thick] (-1,0) rectangle (13,3);
  \foreach \x in {-0.5,0.5,2.5,3.5,4.5,7.5,7.5+1,8.5+1,11.5,12.5}
    {\fill (\x,3) circle (3pt);
     \fill (\x,0) circle (3pt);}
   \fill (   5.5+1,0) circle (3pt);   
    \fill (   4.5+1,3) circle (3pt);   
    \fill (   10.5,0) circle (3pt);   
   \begin{scope}%[thick]
   \draw[thick] (-0.5,3) -- (-0.5,0);
   \draw[thick] (0.5,3) -- (0.5,0);
    \draw[thick] (2.5,3) -- (2.5,0);
        \draw[thick] (4.5,3) -- (3.5,0);
    \draw[thick] (5.5,3) -- (4.5,0);
     \draw[thick] (5.5+2,3) -- (4.5+2,0);
     \draw[thick] (5.5+3,3) -- (4.5+3,0);
     \draw[thick] (5.5+3+1,3) -- (4.5+4,0);   
          \draw[thick] (12.5,3) -- (11.5,0);   
                    \draw[thick] (11.5,3) -- (10.5,0);   
                                     \draw[thick] (10.5,3) -- (9.5,0);   
                    
%\draw[thick] (7.5,0) to [out=30,in=150] (12.5,0);             

\draw[thick] (3.5,3) to [out=-90,in=90] (12.5,0);     

%\draw[thick][white,fill](10,1.5) circle (8pt)       ; 
%\draw[thick][white,fill](11,1.5) circle (8pt)       ;
%                    \draw[thick,densely dotted] (11.5,3) -- (10.5,0);   
%                                     \draw[thick,densely dotted] (10.5,3) -- (9.5,0);   

%\draw[thick][white,fill](10.1,3.2) rectangle (10.9,-0.1)        ;
\draw[thick][white,fill](1.1,3.2) rectangle (1.9,-0.1)        ;
\draw[thick][white,fill](7.1-1,3.2) -- (7.9-1,3.2) -- (6.9-1,-0.1) -- (6.1-1,-0.1) --(7.1-1,3.2)       ;

\draw[thick][white,fill](4+7.1-1,3.2) -- (4+7.9-1,3.2) -- (4+6.9-1,-0.1) -- (4+6.1-1,-0.1) --(4+7.1-1,3.2)       ;

  \draw[densely dotted,thick] (-1,0) rectangle (13,3);

 \draw[densely dotted,thick] (3.5,3) to [out=-90,in=90] (12.5,0);     
    \end{scope}
    \draw[below](12.5,0)   node {$r$};
                 \draw[above](3.5,3.1)   node {$[\nu]_{j_r}$--$k$};

\end{tikzpicture}\end{minipage} 
$$
$$
d_0=\begin{minipage}{105mm}\scalefont{0.7}\begin{tikzpicture}[scale=0.7]
  \draw[thick] (-1,0) rectangle (-1+ (14,3);
  \foreach \x in {-0.5,0.5,1.5,2.5,3.5,4.5,5.5,6.5}
         \draw[thick] (\x,3) -- (\x,0);   

  \foreach \x in {-0.5,0.5,1.5,2.5,3.5,4.5,...,12.5}
    {\fill (\x,3) circle (3pt);
     \fill (\x,0) circle (3pt);}
   \fill (   5.5+1,0) circle (3pt);   
    \fill (   4.5+2,3) circle (3pt);   
    \fill (   10.5,0) circle (3pt);   
   \begin{scope}%[thick]
         \draw[thick] (3.5,3) -- (3.5,0);        \draw[thick] (4.5,3) -- (4.5,0);
     \draw[thick] (4.5+2,3) -- (4.5+2,0);
                    \draw[thick] (7.5,3) -- (8.5,0);
                          \draw[thick] (8.5,3) -- (9.5,0);

                            \draw[thick] (10.5,3) -- (11.5,0);

 \draw[thick] (7.5,0) to [out=30,in=150] (12.5,0);     

\draw[thick][white,fill](1.1,3.2) rectangle (1.9,-0.1)        ;

    \draw[thick][white,fill](1.1+3+1,3.2) (-1+ (1.9+3+1,-0.1)        ;
\draw[thick][white,fill](-1+4+7.1-1,3.2) -- (-1+4+7.9-1,3.2) -- (-1+4+6.9-1+2,-0.1) -- (-1+4+6.1-1+2,-0.1) --(-1+4+7.1-1,3.2)       ;
 \draw[thick,densely dotted] (7.5,0) to [out=30,in=150] (12.5,0);     

  \draw[thick,densely dotted] (-1,0) rectangle (-1+ (14,3);

\end{scope}
    \draw[below](12.5,0)   node {$r$};
                 
        \draw[below](7.5,0)   node {$[\nu']_{i_r}$};
 
\end{tikzpicture}\end{minipage} 
$$

\!\!\!\!\!\caption{The diagrams $d_k$ of parts $(i)$ to $(iii)$ of \cref{MaudLemmaA} respectively.  The first diagram is drawn under the assumption that $[\nu]_{j_r}-k<[\nu']_{i_r}$, the   cases $[\nu]_{j_r}-k=[\nu']_{i_r}$ and $[\nu]_{j_r}-k>[\nu']_{i_r}$ are similar.}
\label{dis}
\end{figure}
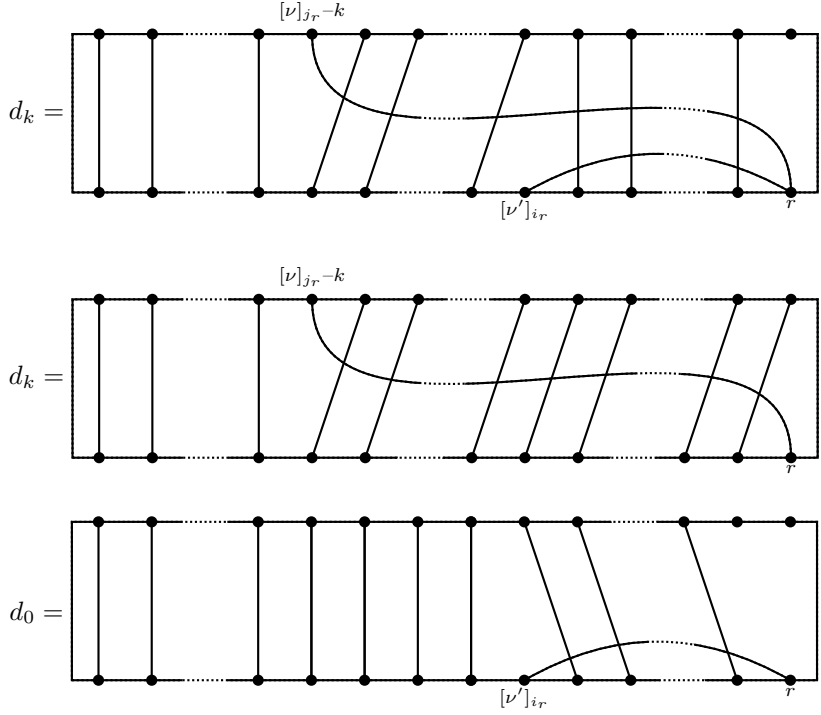
\begin{proof}
By definition, we have \begin{align*}
u_{\stt [r-1,r] } 
&=u_{\stt(r-\half) \to \stt(r ) } u_{\stt(r-1) \to \stt(r-\half ) } 
\\
&=  \sum_{\ik=0}^{   \nu_{j_r}-1}  
  e_{r}^{(r-|\nu|)}     s_{[\nu]_{j_r}-\ik, [\nu]_{j_r}  	}
s_{[\nu]_{j_r}, |\nu|  	}
  e_{r-\half}^{(r-1-|{\stt(r-\half)}|)}
s_{    |\nu'| , [\nu']_{i_r}} \\ 
&=  \sum_{\ik=0}^{   \nu_{j_r}-1}  
  e_{r}^{(r-|\nu|)}   
    s_{[\nu]_{j_r}-\ik,   |\nu|  	}
  e_{r-\half}^{(r-1-|{\stt(r-\half)}|)}
s_{    |\nu'| , [\nu']_{i_r}}. 
 \end{align*}
The result follows by concatenating the four diagrams in each case.  
\end{proof}
\begin{rmk}
Note that in each of cases $(i)$ to $(iv)$ of \cref{MaudLemmaA}, the diagrams in \cref{dis,sij} provide the natural bijection between the nodes of $\nu-(j_r,\nu_{j_r}-\ik)$ and the nodes of $\nu'-(i_r,\nu_{i_r})$.
\end{rmk}

\begin{lem}\label{MaudLemmaB}
Let  $\stt=(-\varepsilon_{i_1}, + \varepsilon_{j_1}, \ldots  ,-\varepsilon_{i_{r}}, + \varepsilon_{j_{r}})\in \Std_r(\nu).$    
Write 
\begin{equation}\label{abigsumofstuff}
u_\stt = \sum_{d} \alpha_{d,\stt}d
\end{equation}
with $\alpha_{d,\stt}\in \ZZ_{>0}$ and $d$  partition diagrams in $P_r(n)$.  Then, for any $d$ appearing in this sum,   we have 
\begin{itemize}[leftmargin=*]
\item[$(1)$] the northern nodes $\{\overline{r}\}$, 
$\{\overline{r-1}\}, \dots,  \{\overline{r-|\nu|}\}$ are singleton blocks of $d$;
\item[$(2)$] for each $1\leq i \leq \ell(\nu)$, any northern nodes in the set 
$\{\overline{[\nu]_{i-1}+1}, \overline{[\nu]_{i-1}+2}, \dots, \overline{[\nu]_{i}} \}$ is connected to some southern node  $k$ satisfying $j_k=i$.  
\end{itemize}
\end{lem}

\begin{proof}
Part $(1)$ follows directly from the fact that $u_\stt=e_r^{(r-|\nu|)}x$ for some $x\in P_r(n)$.   
We prove $(2)$ by induction on $r$.  If $r=1$, then either $\stt=(-\varepsilon_0,+\varepsilon_0)$ or 
$\stt=(-\varepsilon_0,+\varepsilon_1)$.  In the first case, there is nothing to prove.  In the second case, we have that $\overline{1}$ is connected to $1$, which satisfies $j_1=1$, as required.  Now assume the result holds for $r-1$.  Write 
$u_\stt= u_{\stt[r-1,r]} u_{\stt'} $
where $\stt'=\stt[0,r-1] \in \Std_{r-1}(\nu')$ with $\stt(r-1)=\nu'$.  
By induction, we write
$$
u_{\stt'} = \sum_{d'} \alpha_{d',{\stt'}}d'
$$
with $\alpha_{d',{\stt'}}\in \ZZ_{>0}$.  For any $d'$ appearing in this sum and any $1\leq k \leq \ell(\nu')$, we have that any northern node in the set 
$$
\{\overline{	[\nu']_{k-1}+1},
\overline{	[\nu']_{k-1}+2},
\dots,
\overline{	[\nu']_{k}}\}
$$
is connected to some southern nodes $l$ satisfying $j_l=i$.  Now any diagram, $d$, appearing in \cref{abigsumofstuff} is of the form $d= d_{\ik}d'$ 
(for cases $(i)$ and $(ii)$) or $d_0 d'$  
(for cases $(iii)$ and $(iv)$) as in \cref{MaudLemmaA}.  
If $d_0$ is as in case $(iii)$ and $(iv)$, then 
the diagram $d_0$  provides the natural bijection between the nodes of $\nu$ and $\nu'-\varepsilon_{i_r}$ and the result follows.  If $d_{\ik}$ is as in 
case  $(i)$ or  $(ii)$ we must show that $\overline{[\nu]_{j_r}-i}$ is connected to a southern nodes of the required form. (That any other northern node in $d_{\ik}$ is connected to a southern node of the required form is immediate, as in cases $(iii)$ and $(iv)$ above.)  
Now, as $\{\overline{r}, r\}$ is a block of $d'$, we have that 
 $\overline{[\nu]_{j_r}-i}$ is connected to $r$ in $d_{\ik} d'=d$ as required.  
\end{proof}

 \begin{lem}\label{MaudLemmaC}
Let  $\stt=(-\varepsilon_{i_1}, + \varepsilon_{j_1}, \ldots  ,-\varepsilon_{i_{r}}, + \varepsilon_{j_{r}})\in \Std_r(\nu).$    
Write 
\begin{equation}\label{abigsumofstuffC}
u_\stt = \sum_{d} \alpha_{d,\stt}d
\end{equation}
with $\alpha_{d,\stt}\in \ZZ_{>0}$ and $d$   partition diagrams  in $P_r(n)$.  For any diagram $d$ appearing in this sum and any $1\leq k \leq r$,  we have that
\begin{enumerate}[leftmargin=*]
\item[(a)] if $i_k = j_k =0$ then the southern node $k$ in   $d$ is a singleton;
\item[(b)] if $i_k\neq 0$ then the southern node $k$ in  $d$ is connected to a southern node $l<k$ with $j_l = i_k$.
\end{enumerate} 
 
 \end{lem}

\begin{proof}
We prove this lemma by induction on $r$.  
If $r=1$ then $k=1$ and $i_k=0$. The only path to consider is $\stt = (-\varepsilon_0, +\varepsilon_0)$. In this case we have $u_\stt = d$ with $d = \{\{\bar{1}\}, \{ 1 \} \}$, so the result holds. % in this case.

We shall ssume that the result holds for $r-1$ and prove it for $r$. As in Lemma \ref{MaudLemmaB}, we write $u_\stt = u_{\stt[r-1,r]} u_{\stt'}$ with $\stt' = \stt[0,r-1] \in \Std_{r-1}(\nu')$ and $\nu' = \stt'(r-1)$.  Write 
$$u_{\stt'} = \sum_{d'} \alpha_{d',\stt'}d'$$
with $\alpha_{d',\stt'}\in \mathbb{Z}_{>0}$ and $d'$ a partition diagram  in $P_{r-1}(n)\subset P_r(n)$.
By induction, the result holds for all $d'$ in this sum and all $1\leq k\leq r-1$. As any diagram $d$ appearing in \cref{abigsumofstuffC} has the form $d_{\ik}d'$ where the $d_{\ik}$'s are given in Lemma \ref{MaudLemmaA}, we have that the result holds for $d$ and any $1\leq k\leq r-1$. It remains to prove it for $k=r$.

For part $(a)$, note that $d_0$ is as in Lemma \ref{MaudLemmaA}$(iv)$. Now using Lemma \ref{MaudLemmaB}$(1)$ we know that $\{\overline{r-1}\}$, $\{\overline{r-2}\}$, \ldots , $\{\overline{r-1-|\nu'|}\}$ are all singleton blocks in $d'$. As $\{\overline{r},r\}$ is a block in $d'$ we deduce that $\{r\}$ is a singleton block in $d=d_0d'$. 

For part $(b)$, note that $d_{\ik}$ is as in Lemma \ref{MaudLemmaA}$(i)$ or $(iii)$. Thus the southern nodes
 $r$  and  $[\nu']_{i_r}$ are connected  in $d_{\ik}$. But now, using Lemma \ref{MaudLemmaB}$(2)$ we have that  in $d'$ the northern node $\overline{[\nu']_{i_r}}$ is connected to some southern node $k\leq r-1$ with $j_k = i_r$. Moreover, $\{\overline{r}, r\}$ is a block of $d'$. Concatenating $d_{\ik}$ with $d'$ we deduce that in $d$ the node $r$ is connected to some $k<r$ with $j_k = i_r$ as required.
\end{proof}

\begin{proof}[Proof of Proposition \ref{Dvirinclusion}]
Recall that ${\sf DR}_s(\nu\setminus \lambda)=\Delta_s(\nu\setminus \lambda) P_s(n)p_rP_s(n)$. So if $m + P_r^{\rhd \nu\setminus \lambda}(n)\in \Delta_s(\nu\setminus \lambda)$ and $m\in P_r(n)p_rP_s(n)$ then $m+P_r^{\rhd \nu\setminus \lambda}(n)\in {\sf DR}_s(\nu \setminus \lambda)$. Now $P_r(n)p_rP_s(n)$ is spanned by all partition diagrams in $P_r(n)$ having at most $s-1$ distinct blocks containing both an element of the set $\{r-s+1, \ldots r\}$ and an element of the set $\{\bar{1}, \ldots , \bar{r}, 1, \ldots , r-s\}$. We claim that  $u_{\stt^{\lambda} \circ \stt}$ is a sum of such diagrams for any $\stt\in {\rm DR}\mhyphen\Std_s(\nu \setminus \lambda)$.   Thus  $u_{\stt^{\lambda}\circ \stt}\in P_r(n)p_rP_s(n)$ as required.

We now set about proving this claim.   Write $\stt^\lambda \circ \stt = (-\varepsilon_{i_1}, + \varepsilon_{j_1}, \ldots , -\varepsilon_{i_r} , + \varepsilon_{j_r})$ and 
\begin{equation}\label{bigsumprop}
u_{\stt^{\lambda}\circ \stt} = \sum_{d} \alpha_{\stt, d} d
\end{equation}
with $\alpha_{\stt, d}\in \mathbb{Z}_{>0}$ and $d$ a  partition diagram  in $P_r(n)$.   
 First suppose that $\stt\in {\rm DR}^0\mhyphen\Std_s(\nu\setminus \lambda)$. Then there exists $k\geq r-s+1$ such that the $k$-th integral step of $\stt^{\lambda}\circ \stt$ has the form $(-\varepsilon_0, + \varepsilon_0)$. Using Lemma \ref{MaudLemmaC}(a), we deduce that $k$ is a singleton in {\em any} diagram $d$ appearing in \cref{bigsumprop} and hence   $d\in P_r(n)p_rP_s(n)$.  

Now suppose that $\stt\in {\rm DR}^x\mhyphen\Std_s(\nu\setminus \lambda)$ for some $x>0$. Then $M=\{k\, | \, k\geq r-s+1 \, \mbox{and} \, i_k=x\}$ satisfies $|M|>\lambda_x$.  By Lemma \ref{MaudLemmaC}(b)  for any $k\in M$  and any diagram $d$ appearing in \cref{bigsumprop}, we have that the  southern node  $k$  is connected to a  southern  node $l<k$ satisfying $j_l = x$. Now, by definition of $\stt^{\lambda}$, there are precisely $\lambda_x$ such $l$ with $l\leq r-s$. We conclude that there must be at least one $k\in M$ such that the southern node $k$ in $d$ is connected to a southern node from the set $\{r-s+1, \ldots , r\}$. This proves that $d\in P_r(n)p_rP_s(n)$ as required. 
\end{proof}

%    We now provide an example    of the basis elements  of  the Dvir radical which  illustrate this proposition.  

\begin{eg}\label{GHGHGHG}
Let $\nu =\lambda=(2,1)$ and $s=3$.
The path $\stt \in \Std_3(\nu\setminus\lambda) $   given by
  $$
   \Yvcentermath1\Yboxdim{6pt}\gyoung(;;,;) \xrightarrow{\ -\varepsilon_2 \ }
  \Yvcentermath1\Yboxdim{6pt}\gyoung(;;)  \xrightarrow{\ +\varepsilon_2 \ }
      \Yvcentermath1\Yboxdim{6pt}\gyoung(;;,;) \xrightarrow{\ -\varepsilon_0 \ }
      \Yvcentermath1\Yboxdim{6pt}\gyoung(;;,;) \xrightarrow{\ +\varepsilon_2 \ }
         \Yvcentermath1\Yboxdim{6pt}\gyoung(;;,;;)   \xrightarrow{\ -\varepsilon_2 \ }
   \Yvcentermath1\Yboxdim{6pt}\gyoung(;;,;)   \xrightarrow{\ +\varepsilon_0 \ }
         \Yvcentermath1\Yboxdim{6pt}\gyoung(;;,;) 
  $$ belongs to $ {\rm  DR}^2\mhyphen\Std_3(\nu\setminus\lambda)$.  
To see this note that 
$$  \sharp \{ \text{steps of the form $-\varepsilon_2$ in $\stt$} \}  = 2> 1 = \lambda_2.  
$$
 The element $u_{\stt^{\lambda}\circ \stt}$ is depicted in \cref{newonw2}, below.  
We see that every elementary diagram
  in this sum 
  has  at most 2 blocks  with both an element from $\{4,5,6\}$ and 
  an element from $\{\overline1,\overline2,\dots ,\overline6\}\cup\{1,2,3\}$.  
Therefore   $u_{\stt^{\lambda}\circ\stt}\in  {\sf DR}_s(\nu\setminus\lambda)$.

 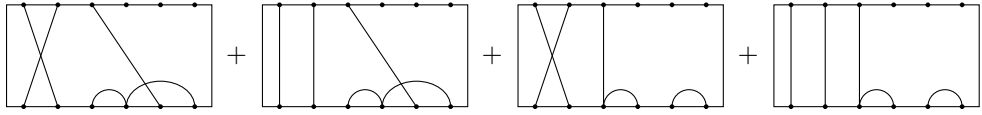
\begin{figure}[ht!]
$$
\begin{minipage}{27mm}\scalefont{0.8}\begin{tikzpicture}[scale=0.45]
  \draw (0,0) rectangle (6,3);
  \foreach \x in {0.5,1.5,...,5.5}
    {\fill (\x,3) circle (2pt);
     \fill (\x,0) circle (2pt);}
   \begin{scope}%[thick]
   \draw (0.5,3) -- (1.5,0);
  \draw (2.5,0) arc (180:360:0.5 and -0.5);
  \draw (4.5,0) --  (2.5,3);
   \draw (3.5,0) arc (180:360:1 and -0.75);
 %    \draw (3.5,3) -- (3.5,0);
    \draw (1.5,3) -- (0.5,0);
   \end{scope}
\end{tikzpicture}\end{minipage} \  + \ 
\begin{minipage}{27mm}\scalefont{0.8}\begin{tikzpicture}[scale=0.45]
  \draw (0,0) rectangle (6,3);
  \foreach \x in {0.5,1.5,...,5.5}
    {\fill (\x,3) circle (2pt);
     \fill (\x,0) circle (2pt);}
   \begin{scope}%[thick]
    \draw (1.5,3) -- (1.5,0);
  \draw (2.5,0) arc (180:360:0.5 and -0.5);
  \draw (4.5,0) --  (2.5,3);
   \draw (3.5,0) arc (180:360:1 and -0.75);
 %    \draw (3.5,3) -- (3.5,0);
    \draw (0.5,3) -- (0.5,0);
   \end{scope}
\end{tikzpicture}\end{minipage}  \ + \
\begin{minipage}{27mm}\scalefont{0.8}\begin{tikzpicture}[scale=0.45]
  \draw (0,0) rectangle (6,3);
  \foreach \x in {0.5,1.5,...,5.5}
    {\fill (\x,3) circle (2pt);
     \fill (\x,0) circle (2pt);}
   \begin{scope}%[thick]
   \draw (0.5,3) -- (1.5,0);
  \draw (2.5,0) arc (180:360:0.5 and -0.5);
  \draw (2.5,0) --  (2.5,3);
   \draw (4.5,0) arc (180:360:0.5 and -0.5);
 %    \draw (3.5,3) -- (3.5,0);
    \draw (1.5,3) -- (0.5,0);
   \end{scope}
\end{tikzpicture}\end{minipage} \  + \ 
\begin{minipage}{27mm}\scalefont{0.8}\begin{tikzpicture}[scale=0.45]
  \draw (0,0) rectangle (6,3);
  \foreach \x in {0.5,1.5,...,5.5}
    {\fill (\x,3) circle (2pt);
     \fill (\x,0) circle (2pt);}
   \begin{scope}%[thick]
    \draw (1.5,3) -- (1.5,0);
  \draw (2.5,0) arc (180:360:0.5 and -0.5);
  \draw (2.5,0) --  (2.5,3);
   \draw (4.5,0) arc (180:360:0.5 and -0.5);
 %    \draw (3.5,3) -- (3.5,0);
    \draw (0.5,3) -- (0.5,0);
   \end{scope}
\end{tikzpicture}\end{minipage}    
$$

\!\!\!\!\!\caption{The element $u_\stt$ from \cref{GHGHGHG}.}
\label{newonw2}
\end{figure}
 \end{eg}

The following definition is 
motivated by \cref{Dvirinclusion}.
\begin{defn}
Given $(\lambda,\nu,s) \in \mathscr{P}_ { r-s} \times \mathscr{P}_{\leq r} \times \NN$,   
 we let  
 $$  \Std^0_s(\nu\setminus \lambda) = \Std_s(\nu\setminus \lambda)  \setminus  {\rm  DR}\mhyphen\Std_s(\nu\setminus\lambda)  
$$ 
\end{defn}

 \begin{rmk} 
%\begin{enumerate}
%\item 
If $(\lambda, \nu,s)$ is a triple of maximal depth then $\Std_s^0(\nu\setminus \lambda) = \Std_s (\nu\setminus \lambda)$.
 \end{rmk}

 \begin{rmk}\label{akjhfdsajhjkhfdsakjhdfasasfdbmnzxvbmnvcbmncvxznmsfadhof}

   By \cref{Dvirinclusion}, we know that $ \Std^0_s(\nu\setminus \lambda)$ indexes a spanning set for the  module $\Delta_s^0(\nu\setminus\lambda)$ for any $(\lambda,\nu,s) \in \mathscr{P}_ { r-s} \times \mathscr{P}_{\leq r} \times \NN$.   
In particular, if  $s>|\lambda|+|\nu|$ or $s<{\rm max}\{|\lambda \ominus (\lambda \cap \nu)|, |\nu \ominus (\lambda \cap \nu)|\}$ then 
 $\Std^0_s(\nu\setminus\lambda)=\emptyset$ 
 and $\overline{g}(\lambda,\nu,\mu)=0$ for all $\mu \vdash s$.  
%\end{enumerate}
 \end{rmk}

  \begin{thm} \label{{co-Pieri}triple}\label{hjkldhjklfadshlkjfdshfasd}
Let   $(\lambda,\nu,s) \in \mathscr{P}_{r-s}   \times \mathscr{P}_{\leq r}  \times \NN$  
 be   such that $\Std_s^0(\nu\setminus\lambda)\neq \emptyset$.
 We have that 
 \begin{itemize}
 \item [ {\rm (C1)} ]   $\sts_{k\leftrightarrow k+1}$ exists for all 
     $\sts \in \Std_{s}^0(\nu\setminus\lambda)$ and   $1\leq k \leq s-1 $  
    and 
 \item [ {\rm (C2)} ]  $ \{ u_{\stt^\lambda \circ \stt}+P_{r,s}^{\rhd\nu\setminus\lambda}(n) \mid    \stt \in  {\rm  DR}\mhyphen\Std_s(\nu\setminus \lambda)  \} $     is a basis for ${\sf DR}_s(\nu\setminus \lambda)$
 \end{itemize}
 if and only if

{\rm (coP)}  $\left\{ \begin{array}{l}
  s=1, \, \mbox{or}\\
 s>1\,\, \mbox{and if} \,\,  \max \{\ell(\lambda), \ell(\nu)\} \geq 2 \,\,  \mbox{then}\, \,
 s \leq 
    \max\{|\lambda\ominus (\lambda \cap \nu)|,|\nu\ominus (\lambda \cap \nu)|\} + {\rm minmax}(\lambda,\nu)
\end{array} \right.$

\noindent where $${\rm minmax}(\lambda, \nu) =  \min \{ \min\{\lambda_{i-1},\nu_{i-1}\} - \max\{\lambda_{i},\nu_{i}\} 
  \mid  2\leq  i \leq  \max\{\ell(\lambda), \ell(\nu) \} \}.$$ 
 We refer to such triples,  $(\lambda,\nu,s)$, as {\sf co-Pieri triples}. In this case, we will also refer to any triple of the form $(\lambda, \nu,\mu)$ with $\mu \vdash s$ as a co-Pieri triple.
 \end{thm}

\begin{rmk}
Note that if $(\lambda, \nu,s)$ satisfies $\Std_s^0(\nu\setminus \lambda)\neq \emptyset$ and {\rm (coP)} then the skew partitions $\lambda \ominus (\lambda \cap \nu)$ and $\nu \ominus (\lambda \cap \nu)$ contain no two nodes in the same column. To see this, observe that ${\rm minmax}(\lambda, \nu) <0$ precisely when one of these skew partitions has two nodes in the same column. On the other hand, $\Std_s^0(\nu\setminus \lambda)\neq \emptyset$ implies that $s\geq  \max\{|\lambda\ominus (\lambda \cap \nu)|,|\nu\ominus (\lambda \cap \nu)|\}$. Thus we must have ${\rm minmax}(\lambda, \nu) \geq 0$.
\end{rmk}

%\begin{rmk}
%The conditions 
%\end{rmk}

We will prove this theorem in the rest of the section but first we note that for co-Pieri triples  we are able to completely understand the action of the partition algebra on $\Delta_s^0(\nu\setminus\lambda)$.  
To simplify the notation for the basis elements of the skew module $\Delta_s(\nu\setminus \lambda)$ we set
$$m_{\stt} := u_{\stt^{\lambda} \circ \stt} + P_{r,s}^{\rhd \nu \setminus \lambda}(n)$$
 for all $\stt\in \Std_s(\nu\setminus \lambda)$.

 \begin{cor}
 Let   $(\lambda,\nu,s) \in \mathscr{P}_{r-s}   \times \mathscr{P}_{\leq r}  \times \NN$  be a co-Pieri triple. Then
 we have that 
$$ \{ m_\stt  + {\sf DR}_s(\nu \setminus \lambda)  \mid \stt \in \Std^0_s(\nu\setminus \lambda)  \} 
$$
is a basis for $\Delta_s^0(\nu \setminus \lambda)$
 and the $P_s(n)$-action on $\Delta_s^0(\nu\setminus\lambda)$ is   as follows:
\begin{equation}\label{co-case}
(m_\stt + {\sf DR_s(\nu\setminus\lambda)} ) s_{k,k+1}  =
 m_{\stt_{k\leftrightarrow k+1}}  + {\sf DR_s(\nu\setminus\lambda)}
\end{equation}
for $1\leq k < s$,
$$ (m_\stt + {\sf DR_s(\nu\setminus\lambda)} ) p_{k,k+1}  =
0 \quad \mbox{and} \quad
 (m_\stt + {\sf DR_s(\nu\setminus\lambda)} ) p_{k}  =
0
$$
for all $1\leq k < s$ and $1\leq k \leq s$, respectively.  
    \end{cor}
    
    \begin{proof}
    This follows immediately from \cref{hjkldhjklfadshlkjfdshfasd,MURPHYPrN}
    \end{proof}

\begin{eg}
Note that any triple $(\lambda, \nu, s)$ with $\ell(\lambda) = \ell(\nu) = 1$ is a co-Pieri triple.  We calculate the corresponding  Kronecker  coefficients labelled by two two-line partitions in Section 7.
\end{eg}

\begin{eg} 
For $d,\ell,m\geq 0$, we define the partition
 $$ 
\rho = d(\ell, \ell-1, \dots ,2,1) + (m^l)
$$ 
As ${\rm minmax}(\rho,\rho) = d$ we have that    $(\rho ,\rho ,s)$ with any $s\leq d$ 
is a co-Pieri triple.  
\end{eg}

\begin{eg}
%The triple $((10,5,2),(8,3,3,2), 4)$ is a co-Pieri triple. Note that here $\minmax(\lambda, \nu) =0$ and $\max \{|\lambda \ominus (\lambda \cap \nu)|, |\nu \ominus (\lambda \cap \nu)|\} = 4$.
Let $\lambda$ and $\nu$ be any pair of partitions such that 
$\lambda\ominus (\lambda \cap \nu)$ and $\nu\ominus (\lambda \cap \nu)$ are skew partitions with no two nodes in the same column and 
let $s= \max\{|\lambda\ominus (\lambda \cap \nu)|,|\nu\ominus (\lambda \cap \nu)|\}$.  Then $(\lambda,\nu,s)$ is a co-Pieri triple.  
This clearly includes  the triples of \cref{piere!} as a subcase.  For another example,  $((10,5,2),(8,3,3,2), 4)$ is such a co-Pieri triple.
\end{eg}

 \begin{eg}
Let    $\lambda=(4,2) $ and $\nu=(4,3,1)$.  
We have that  $\max \{|\lambda \ominus (\lambda \cap \nu)|, |\nu \ominus (\lambda \cap \nu)|\} = 2$ 
 and   $\minmax(\lambda,\nu)=1$. 
Therefore 
$(\lambda,\nu,s)$ is a co-Pieri triple for 
 $s=2$ or $3$.  
  \end{eg}

 \begin{lem}\label{LemmaB}
 Let $(\lambda,\nu,s)\in \mathscr{P}_ { r-s} \times \mathscr{P}_{\leq r} \times \NN$   with $\Std_s^0(\nu\setminus \lambda)\neq \emptyset$. Assume that $(\lambda, \nu,s)$ satisfies {\rm (coP)}.   Let $n\gg r$
 and $\alpha \subseteq \lambda_{[n]}\cap \nu_{[n]}$ be any composition of $n-s$, say
 $$
 \alpha= (\alpha_1, \alpha_2, \ldots ) = \lambda_{[n]}-\varepsilon_{i_1}-\varepsilon_{i_2}- \dots-  \varepsilon_{i_s}
 = \nu_{[n]}-\varepsilon_{j_1}-\varepsilon_{j_2}- \dots-  \varepsilon_{j_s}.
 $$
 Define the composition 
  $$
\beta= (\beta_1, \beta_2, \ldots ) = \lambda_{[n]}+\varepsilon_{j_1}+\varepsilon_{j_2}+ \dots+  \varepsilon_{j_s}
 = \nu_{[n]}+\varepsilon_{i_1}+\varepsilon_{i_2}+ \dots+  \varepsilon_{i_s}.
 $$
Then for all $c\geq 1$ we have 
$$
\alpha_c \geq \beta_{c+1}.
$$
In particular, $\alpha\subseteq \lambda_{[n]}\cap \nu_{[n]}$  is a partition and $\lambda_{[n]}\ominus \alpha$ and $\nu_{[n]} \ominus \alpha$ have no two nodes in the same column.
 \end{lem}
   \begin{proof}

 First note that as $n\gg r$, $\alpha_1 \geq \beta_2$.  If $\ell(\lambda) = \ell(\nu)=1$ then $\alpha_2 \geq \beta_3 = 0$ and for $c \geq 3$ we have $\alpha_c = \beta_{c+1} = 0$ so we are done. Now assume $\max \{\ell(\lambda), \ell(\nu)\} \geq 2$. Define multi-sets
 $$
 I=\{i_1,i_2,\dots, i_s\} \quad \mbox{and} \quad  J=\{j_1,j_2,\dots, j_s\}.
 $$
 For $c\geq 2$, define 
  $|I|_c=\sharp\{i_k \in I \mid i_k=c\}$ and define $|J|_c$ and $|I\cap J|_c$ similarly.  
Now, 
$$
\alpha_c = \lambda_{c-1} - |I|_c = \lambda_{c-1} - |I \setminus (I \cap J)|_c - |I\cap J|_c
$$
$$
\beta_{c+1} = \lambda_{c} + |J|_{c+1} = \lambda_{c} + |J \setminus (I \cap J)|_c + |I\cap J|_{c+1}.
$$
Note that 
$$
|I \setminus I \cap J|_c=
\begin{cases}
 \lambda_{c-1}- \nu_{c-1} &\text{if }\lambda_{c-1} - \nu_{c-1} \geq 0  \\
 0				&\text{otherwise,}
\end{cases}
\quad\quad
|J \setminus I \cap J|_{c+1}=
\begin{cases}
  \nu_{c} -  \lambda_{c} &\text{if } \nu_{c}  - \lambda_{c}\geq 0  \\
 0				&\text{otherwise.}
\end{cases}
$$
Hence 
$$
\lambda_c - |I\setminus I \cap J|_c = \min \{\lambda_{c-1},\nu_{c-1}\},
\qquad\qquad
\lambda_{c}+
 |J\setminus I \cap J|_{c+1} = \max \{\lambda_{c},\nu_{c}\},
$$
and we get 
\begin{align*}\alpha_c-\beta_{c+1}	
&= \min \{\lambda_{c-1},\nu_{c-1}\}- \max\{\lambda_{c},\nu_{c}\}
-|I \cap J|_c - |I \cap J|_{c+1} \\
&\geq 
\min \{\lambda_{c-1},\nu_{c-1}\}- \max\{\lambda_{c},\nu_{c}\}
-| I \cap J| .
\end{align*}
Now, $$| I \cap J|= s-
    \max\{|\lambda\ominus (\lambda \cap \nu)|,|\nu\ominus (\lambda \cap \nu)|\}.  $$  
 So we get 
$$\alpha_c -\beta_{c+1}\geq \min\{\lambda_{c-1},\nu_{c-1}\} - \max \{\lambda_{c},\nu_{c}\} 
 +    \max\{|\lambda\ominus (\lambda \cap \nu)|,|\nu\ominus (\lambda \cap \nu)|\}  -s.$$
Using {\rm (coP)}, we get that $\alpha_c - \beta_{c+1}\geq 0 $ for $2\leq c\leq \max\{\ell(\lambda),\ell(\nu)\}$. Now, if $c>\max\{\ell(\lambda), \ell(\nu)\}$ then $\beta_{c+1} = 0$ and so $\alpha_c\geq \beta_{c+1} = 0$ as required. 
   \end{proof}

 We define $\Std_s^+(\nu\setminus\lambda)=  \Std_s(\nu\setminus\lambda) \setminus (\cup_{i\geq 1}{\rm DR}^i(\nu\setminus\lambda))$.   

\begin{lem} \label{combbijection}
Let  $(\lambda,\nu,s)\in  \mathscr{P}_{r-s}\times\mathscr{P}_{\leq r}\times \ZZ_{\geq 0}$ 
with $\Std^0_s(\nu\setminus\lambda)\neq \emptyset$. Assume that $(\lambda, \nu, s)$ satisfies {\rm (coP)}. Then we have a bijective map
\begin{equation}\label{anequationweshalluse}
\varphi_s\, : \,  \bigsqcup_{\begin{subarray}c \alpha  \vdash n-s \\ \alpha\subseteq  \lambda_{[n]} \cap \nu_{[n]}
\end{subarray}} \Std_s (\nu_{[n]}\setminus \alpha ) \times  \Std_s (\alpha  \setminus \lambda_{[n]})
\rightarrow
 \Std_s^+(\nu\setminus\lambda)
 \end{equation}
 where a given  pair on the lefthand-side is necessarily of the form
\begin{equation*}
 (\sts,\stt)=((-\varepsilon_{0}, + \varepsilon_{j_1},-\varepsilon_{0}, + \varepsilon_{j_2}, \ldots ,-\varepsilon_{0}, + \varepsilon_{j_{s}}),
(-\varepsilon_{i_1}, + \varepsilon_{0}, -\varepsilon_{i_2}, + \varepsilon_{0}, \ldots ,-\varepsilon_{i_{s}}, +\varepsilon_{0})),
\end{equation*}
with $i_l,j_l\neq 0$ for all $1\leq l\leq s$,
and such a pair of tableaux is sent to 
$$\varphi_s(\sts , \stt)=(-\varepsilon_{i_1-1}, + \varepsilon_{j_1-1},-\varepsilon_{i_2-1}, + \varepsilon_{j_2-1}, \ldots ,-\varepsilon_{i_s-1}, + \varepsilon_{j_s-1})
\in   \Std_s^+(\nu\setminus\lambda).
 $$
Moreover, given any
 $\varphi_s(\sts,\stt)=\stu \in \Std^+_s(\nu\setminus\lambda)$ and any $1\leq k\leq s-1$ we have that
  $\varphi(\sts_{k\leftrightarrow k+1},\stt_{k\leftrightarrow k+1})=  
\stu_{k\leftrightarrow k+1}\in \Std_s^+(\nu\setminus \lambda)$ and hence {\rm (C1)} holds.   
 \end{lem}
\begin{proof}
We first show that for any $\alpha\vdash n-s$ with $\alpha \subseteq \lambda_{[n]} \cap \nu_{[n]}$ and $(\sts, \stt)\in \Std_s(\nu_{[n]}\setminus \alpha ) \times \Std_s(\alpha \setminus \lambda_{[n]})$ we have $\varphi_s(\sts, \stt)\in \Std_s(\nu\setminus \lambda)$. Write 
%\begin{eqnarray*}
%&& \sts = (-\varepsilon_{0}, + \varepsilon_{j_1},-\varepsilon_{0}, + \varepsilon_{j_2}, \ldots ,-\varepsilon_{0}, + \varepsilon_{j_{s}}) \\
%&& \stt = (-\varepsilon_{i_1}, + \varepsilon_{0}, -\varepsilon_{i_2}, + \varepsilon_{0}, \ldots ,-\varepsilon_{i_{s}}, +\varepsilon_{0}).
%\end{eqnarray*}
\begin{align*}
  \sts = (-\varepsilon_{0}, + \varepsilon_{j_1},-\varepsilon_{0}, + \varepsilon_{j_2}, \ldots ,-\varepsilon_{0}, + \varepsilon_{j_{s}}) 
\qquad  \stt = (-\varepsilon_{i_1}, + \varepsilon_{0}, -\varepsilon_{i_2}, + \varepsilon_{0}, \ldots ,-\varepsilon_{i_{s}}, +\varepsilon_{0}).
\end{align*}So we have 
$$\alpha = \lambda_{[n]} -\varepsilon_{i_1} - \varepsilon_{i_2} -\ldots - \varepsilon_{i_s} = \nu_{[n]} -\varepsilon_{j_1} - \varepsilon_{j_2} -\ldots - \varepsilon_{j_s}.$$
Setting 
$$\beta = \lambda_{[n]} +\varepsilon_{j_1} + \varepsilon_{j_2} + \ldots + \varepsilon_{j_s}$$
and using Lemma \ref{LemmaB} we get
$$\alpha_i \geq \beta_{i+1} \quad \forall i\geq 1.$$
In order to prove that $\stu  = \varphi_s(\sts, \stt) \in \Std_s(\nu\setminus \lambda)$ we need to show that for all $1\leq l\leq s-1$ we have that $\gamma(l) := \lambda_{[n]} + \sum_{k=1}^l (-\varepsilon_{i_k} + \varepsilon_{j_k})$ and $\gamma'(l) = \lambda_{[n]} + \sum_{k=1}^{l-1} (-\varepsilon_{i_k} + \varepsilon_{j_k}) - \varepsilon_{i_l}$ are partitions. But for $\gamma = \gamma(l)$ or $\gamma'(l)$ we have
$$\gamma_i \geq \alpha_i \geq \beta_{i+1} \geq \gamma_{i+1} \quad \forall i\geq 1.$$
So we are done. Now $\varphi_s(\sts, \stt)\in \Std^+_s(\nu\setminus \lambda)$ follows directly from the fact that $\alpha$ is a partition. Moreover, it is clear that the map $\varphi_s$ is injective and that $\varphi_s(\sts_{k\leftrightarrow k+1}, \stt_{\leftrightarrow k+1}) = \stu_{k \leftrightarrow k+1}$ by definition.

It remains to show that $\varphi_s$ is surjective.
Given  $$\stu =(-\varepsilon_{i_1}, + \varepsilon_{j_1},-\varepsilon_{i_2}, + \varepsilon_{j_2}, \ldots ,-\varepsilon_{i_s}, + \varepsilon_{j_s}) 
 \in \Std^+_s(\nu\setminus\lambda),$$ 
we set  $\alpha = \min_n(\stu):=  \lambda_{[n]} - \varepsilon_{i_1+1}-\varepsilon_{i_2+1} - \dots - \varepsilon_{i_s+1} = \nu_{[n]} - \varepsilon_{j_1+1}-\varepsilon_{j_2+1} - \dots - \varepsilon_{j_s+1}$. As $\stu\in \Std_s^+(\nu\setminus \lambda)$ we have that $\alpha$ must be a composition of $n-s$. 
 Using \cref{LemmaB}, we know that 
 $\alpha \subseteq \lambda_{[n]}\cap \nu_{[n]}$  is in fact a partition and that
$\lambda_{[n]} \ominus  \alpha$ and $\nu_{[n]} \ominus \alpha$ contain no two boxes in the same column.  It follows that 
 $$
\sts := (-\varepsilon_{0}, + \varepsilon_{j_1+1},-\varepsilon_{0}, + \varepsilon_{j_2+1}, \ldots ,-\varepsilon_{0}, + \varepsilon_{j_{s}+1})
\in \Std_s (\nu_{[n]}\setminus \alpha ) \quad \mbox{and} 
$$
$$
\stt := (-\varepsilon_{i_1+1}, + \varepsilon_{0}, -\varepsilon_{i_2+1}, + \varepsilon_{0}, \ldots ,-\varepsilon_{i_{s}+1}, +\varepsilon_{0})
\in 
\Std_s (\alpha  \setminus \lambda_{[n]}) 
 $$
satisfy $\varphi_s(\sts, \stt) = \stu$ as required.  
 \end{proof}

  The next proposition   gives a representation theoretic interpretation (for co-Pieri triples) of Dvir's recursive formula for calculating Kronecker coefficients (and hence justifies the name \lq Dvir radical').

\begin{prop} Let
 $(\lambda,\nu,s)\in \mathscr{P}_{r-s} \times \mathscr{P}_{\leq r} \times \mathbb{Z}_{>0}$ with $\Std_s^0(\nu\setminus \lambda) \neq \emptyset$. Assume that $(\lambda, \nu,s)$ satisfies {\rm (coP)}.   Then  there is a surjective $P_s(n)$-homomorphism
\begin{equation}\label{anequationweshalluse2}
\overline{\varphi}_s \, : \, \bigoplus_{\begin{subarray}c \alpha  \vdash n-s \\ \alpha \subseteq  \lambda_{[n]} \cap \nu_{[n]}
\end{subarray}} \Delta_s(\nu_{[n]}\setminus \alpha ) \otimes  \Delta_s (\alpha \setminus \lambda_{[n]})
\rightarrow
 \Delta_s^0(\nu\setminus\lambda)
 \end{equation}
 given by 
 $$\overline{\varphi}_s(
  m_{\sts} 
\otimes  m_{\stt}
) = m_{\varphi_s(\sts ,\stt)}+{\sf DR}_s(\nu\setminus\lambda)$$ for all $\sts \in \Std_s(\nu_{[n]}\setminus \alpha)$ and $\stt \in \Std_s(\alpha \setminus \lambda_{[n]})$ (where $P_s(n)$ acts diagonally on the   module on the lefthand-side).  
  Furthermore, the kernel of this homomorphism is spanned by
\begin{equation}
\{m_\sts \otimes m_\stt \, | \, \varphi_s(\sts , \stt) \in {\rm DR}^0\mhyphen\Std_s(\nu \setminus \lambda)\}.
\end{equation}
and hence the set
$$\{m_\stu + {\sf DR}_s(\nu \setminus \lambda) \, | \, \stu \in \Std_s^0(\nu \setminus \lambda)\}$$
form a basis for $\Delta_s^0(\nu \setminus \lambda)$, i.e. {\rm (C2)} holds.  

 \end{prop}

 \begin{proof}

By \cref{combbijection} and \cref{Dvirinclusion}, it is clear that $\overline{\varphi}_s$ is a surjective map. The generators $p_k$ and $p_{k,k+1}$ act as zero on both modules. Using Section 3, the action of $\mathfrak{S}_s$ on skew cell modules and \cref{combbijection} we have that the action of $s_{k,k+1}$ also coincide under the map $\overline{\varphi}_s$. Thus $\overline{\varphi}_s$ is a surjective $P_s(n)$-homomorphism. It remains to show that its kernel has the required form. As $p_k$ and $p_{k,k+1}$ act as zero,   we can view $\overline{\varphi}_s$ as a homorphism of $\mathfrak{S}_s$-modules. As such we have
\begin{equation}\label{411}
\Delta^+_s(\nu\setminus \lambda) := \bigoplus_{\begin{subarray}c \alpha  \vdash n-s \\ \alpha \subseteq  \lambda_{[n]} \cap \nu_{[n]}
\end{subarray}} \Delta_s(\nu_{[n]}\setminus \alpha ) \otimes  \Delta_s (\alpha \setminus \lambda_{[n]}) \cong \bigoplus_{\begin{subarray}{c} \alpha \vdash n-s \\ \alpha \subseteq \lambda_{[n]} \cap \nu_{[n]} \\ \mu \vdash s \end{subarray}} g(\lambda_{[n]} \ominus \alpha , \nu_{[n]} \ominus \alpha , \mu) \, {\sf S}(\mu).
\end{equation}
On the other hand, recall that we have
\begin{equation}\label{412}
\Delta^0_s(\nu \setminus \lambda) = \bigoplus_{\mu \vdash s} g(\lambda_{[n]}, \nu_{[n]}, \mu_{[n]}) \, {\sf S}(\mu).
\end{equation}
Now, note that $\Delta^+_s(\nu\setminus \lambda)$ decomposes as 
\begin{equation}\label{413}
\Delta_s^+(\nu\setminus \lambda) = \bigoplus_{0\leq m\leq s} \Vec_s^m
\end{equation}
where
$\Vec_s^m$ is spanned by all $m_\sts \otimes m_\stt$ such that $\varphi_s(\sts, \stt)$ has precisely $m$ integral steps of the form $(-\varepsilon_0 , +\varepsilon_0)$. In particular we have that $\Vec_s^0$ is spanned by all $m_\sts \otimes m_\stt$ with $\varphi_s(\sts, \stt)\in \Std_s^0(\nu \setminus \lambda)$. We claim that
$$\ker  (\overline{\varphi}_s) = \bigoplus_{0<m\leq s}\Vec_s^m.$$
By \cref{Dvirinclusion} we know that 
\begin{equation}\label{414}
\bigoplus_{0<m\leq s}\Vec_s^m \subseteq \ker(\overline{\varphi}_s).
\end{equation} We will prove that in fact we have equality, in other words $\Vec_s^0 \cong \Delta_s^0(\nu \setminus \lambda)$. We proceed by induction on $s$.  
If $s = \max\{ |\lambda \ominus (\lambda \cap \nu)|, |\nu \ominus (\lambda \cap \nu)|\}$ then $\Std^+_s(\nu \setminus \lambda) = \Std^0_s(\nu \setminus \lambda)$ and so $\bigoplus_{1<m\leq s}\Vec^m_s =0$. Moreover, in this case \cref{411} gives
\begin{eqnarray}\label{415}
\Delta_s^+(\nu \setminus \lambda ) &\cong & \sum_{\mu \vdash s} g(\lambda_{[n]} \ominus(\lambda_{[n]}\cap \nu_{[n]}), \nu_{[n]} \ominus (\lambda_{[n]} \cap \nu_{[n]}), \mu_{[n]})  {\sf S}(\mu) \nonumber \\
&=& \sum_{\mu\vdash s} g(\lambda_{[n]}, \nu_{[n]}, \mu_{[n]})  {\sf S}(\mu) \nonumber \\
&\cong &  \Delta_s^0(\nu \setminus \lambda),
\end{eqnarray}
so we are done in this case.
Now let $s>\max\{ |\lambda \ominus (\lambda \cap \nu)|, |\nu \ominus (\lambda \cap \nu)|\}$ and assume that the result holds for all $s'<s$. Note that for $m>0$ we have
\begin{equation}\label{418}
\Vec_s^m \cong (\Vec_{s-m}^0 \boxtimes {\sf S}(m))\uparrow_{\mathfrak{S}_{s-m} \times \mathfrak{S}_m}^{\mathfrak{S}_s},
\end{equation}
and by induction, we have
\begin{equation}\label{419}\Vec_{s-m}^0 \cong \bigoplus_{\beta \vdash s-m} g(\lambda_{[n]}, \nu_{[n]}, \beta_{[n]}) \, {\sf S}(\beta)
\end{equation}
for $m>0$.  
Using the Littlewood--Richardson rule, we have
\begin{eqnarray}\label{420}
\Vec_s^m &\cong& \bigoplus_{\beta\vdash s-m}g(\lambda_{[n]}, \nu_{[n]}, \beta_{[n]}) ({\sf S}(\beta) \boxtimes {\sf S}(m))\uparrow_{\mathfrak{S}_{s-m}\times \mathfrak{S}_m}^{\mathfrak{S}_s} \nonumber \\
&\cong & \bigoplus_{\begin{subarray}{c} \beta\vdash s-m \\  \mu \in P(s,\beta) \end{subarray}} g(\lambda_{[n]}, \nu_{[n]}, \beta_{[n]} ) \, {\sf S}(\mu).
\end{eqnarray}
for $m>0$.  
Note that $\mu\in P(s,\beta)$ if and only if $\beta_{[n]}\in P(n,\mu)$. This follows from the fact that $\mu \in P(s,\beta)$ if and only if $\mu_i \geq \beta_i \geq \mu_{i+1}$ for all $i\geq 1$, the fact that $\beta_{[n]}\in P(n,\mu)$ if and only if $\mu_i \geq (\beta_{[n]})_{i+1} \geq \mu_{i+1}$ for all $i\geq 1$, and noting that $(\beta_{[n]})_{i+1}=\beta_i$.
 Thus we get
\begin{eqnarray} \label{422}
\bigoplus_{0<m\leq s} \Vec^m_s &\cong& \bigoplus_{0<m\leq s} \bigoplus_{\begin{subarray}{c} \beta\vdash s-m \\ \mu \in P(s, \beta)\end{subarray}} g(\lambda_{[n]}, \nu_{[n]}, \beta_{[n]}) \, {\sf S}(\mu) \nonumber \\
&=& 
  \bigoplus_{\begin{subarray}c 0<m\leq s
  \\
  \mu\vdash s
  \end{subarray}} 
 \bigoplus_{\begin{subarray}{c}   \beta \vdash s-m \\ \beta_{[n]}\in P(n,\mu) \end{subarray}} g(\lambda_{[n]}, \nu_{[n]}, \beta_{[n]}) \, {\sf S}(\mu) \nonumber \\
&=& \bigoplus_{\mu\vdash s} \bigoplus_{\begin{subarray}{c} \beta_{[n]}\in P(n,\mu)\\ \beta_{[n]} \neq \mu_{[n]} \end{subarray}} g(\lambda_{[n]}, \nu_{[n]}, \beta_{[n]}) \, {\sf S}(\mu).
\end{eqnarray}
Combining this with \cref{411} we get
\begin{eqnarray*}
\Vec_s^0 &\cong& 
\bigoplus_{\mu \vdash s} \Bigg( [ \Delta^+_s(\nu \setminus \lambda): {\sf S}(\mu)] - 
\sum_{0<m\leq s} [%\bigoplus_{0<m\leq s}
 \Vec^m_s: {\sf S}(\mu)]\Bigg) {\sf S}(\mu)\\
&=& \bigoplus_{\mu \vdash s} 
\Bigg( 
\sum_{\begin{subarray}{c} \alpha \vdash n-s \\ \alpha \subseteq \lambda_{[n]}\cap \nu_{[n]}\end{subarray}} 
%\sum_{ \alpha \vdash n-s } \sum_{\alpha \subseteq \lambda_{[n]}\cap \nu_{[n]}}
%%%
g(\lambda_{[n]} \ominus \alpha, \nu_{[n]} \ominus \alpha, \mu) - \sum_{\begin{subarray}{c} \beta_{[n]}\in P(n,\mu)\\ \beta_{[n]} \neq \mu_{[n]} \end{subarray}} g(\lambda_{[n]}, \mu_{[n]}, \beta_{[n]}) \Bigg) \, {\sf S}(\mu)\\
&=& \bigoplus_{\mu \vdash s} g(\lambda_{[n]}, \nu_{[n]}, \mu_{[n]}) \, {\sf S}(\mu)
\end{eqnarray*}
where the last equality follows by using Dvir's recursive formula.
Finally using \cref{412} we deduce that $\Vec_s^0 \cong \Delta_s^0(\nu \setminus \lambda)$ as required.  
\end{proof}

\begin{lem}\label{diffcolumn}
Suppose that $(\lambda,\nu,s)\in \mathscr{P}_ { r-s} \times \mathscr{P}_{\leq r} \times \NN$ 
satisfies $(C1)$.
Then 
neither of the skew-partitions $ \nu \ominus (\lambda\cap \nu)$ or  $ \lambda \ominus (\lambda\cap \nu)$ contains  two nodes in the same column.
%and that $\Std_s^0(\nu\setminus\lambda) \neq \emptyset$.  
%Then  $(\lambda,\nu,s)$ is as in one of  cases
% $(i)$ or $(ii)$ of  \cref{hjkldhjklfadshlkjfdshfasd}.  
\end{lem}

\begin{proof}

For $s=1$, the result is clear.  We assume $s>1$.
 We assume that one of the skew partitions  $ \nu \ominus (\lambda\cap \nu)$ or  $ \lambda \ominus (\lambda\cap \nu)$ does contain two nodes in the same column.  
 (Recall that   $	\max\{|\lambda \ominus (\lambda  \cap \nu )|, |\nu\ominus (\lambda  \cap \nu )|\}\leq 	s\leq |\lambda|+|\nu|$ by \cref{akjhfdsajhjkhfdsakjhdfasasfdbmnzxvbmnvcbmncvxznmsfadhof} and our assumption that $\Std_s(\nu\setminus\lambda)\neq \emptyset$).  
 We first assume that  
$s'  =\max\{|\lambda \ominus (\lambda  \cap \nu )|, |\nu\ominus (\lambda  \cap \nu )|\} $. % (i.e.,  $s$  is minimal such that  $\Std_s (\nu\setminus\lambda)\neq \emptyset$).  
 We let  $\stu \in \Std^0 _{s'}(\nu\setminus\lambda)$ be any path of the form 
\begin{equation}\label{jahgfljlkdfgjhdslkjghdsjklhglksdjfhg}
\stu = (-\varepsilon_{i_1}, + \varepsilon_{j_1},-\varepsilon_{i_2}, + \varepsilon_{j_2}, \ldots ,-\varepsilon_{i_{s'}}, + \varepsilon_{j_{s'}})
\end{equation}
 such that   the nodes $ -\varepsilon_{i_k}$ and $ -\varepsilon_{i_{k+1}}$ 
 (respectively $ +\varepsilon_{j_k}$ and $+ \varepsilon_{j_{k+1}}$)  are removed  (respectively added) in the same column for some $1\leq k < s$.
  %
%  $(2)$ if $|\nu|=|\lambda|+z$ (respectively $ |\lambda|=|\nu|+z$) 
%  then $j_s=j_{s-1}= \dots j_{s-z}=0$ (respectively 
%   $i_s=i_{s-1}= \dots i_{s-z}=0$).
    Such a pair of nodes  exists by our assumption  on $\lambda$ and $\nu$. Note that we can also assume that the tableau $\stu$ given in \cref{jahgfljlkdfgjhdslkjghdsjklhglksdjfhg} satisfies $i_l,j_l\neq 0$ for all $1\leq l \leq \min\{|\lambda \ominus (\lambda \cap \nu)|, |\nu \ominus (\lambda \cap \nu)|\}$ (we will use this fact later in the proof).  Now the  sequence
\[
(-\varepsilon_{i_1}, + \varepsilon_{j_1}, \ldots  , -\varepsilon_{i_{k+1}}, + \varepsilon_{j_{k+1}}, -\varepsilon_{i_k}, + \varepsilon_{j_k}, \ldots ,-\varepsilon_{i_{s'}}, + \varepsilon_{j_{s'}})
\]
 is not an element of $\Std_{s'}(\nu\setminus\lambda)$, and so $\stu_{k \leftrightarrow k+1}$ does not exist.  Therefore $(\lambda,\nu,s')$ is not a co-Pieri triple, as required.    
 
We shall now consider larger values of $s\in \mathbb N$ by {\em inflating} the tableau in \cref{jahgfljlkdfgjhdslkjghdsjklhglksdjfhg}.   
 For $s$ satisfying
$$s'
 < s \leq 
 |\lambda \ominus (\lambda  \cap \nu )|+  |\nu\ominus (\lambda  \cap \nu )|, $$
we have 
$s-s' \leq \min\{   |\lambda \ominus (\lambda  \cap \nu )|,  |\nu\ominus (\lambda  \cap \nu )| \}$, so we can inflate the tableau $\stu$ given in \cref{jahgfljlkdfgjhdslkjghdsjklhglksdjfhg} to get $\bar{\stu}\in \Std_s(\nu \setminus \lambda)$ by setting $\bar{\stu}$ to be the tableau 
 \begin{equation}\label{firsttbar} 
  (\underbrace{-\varepsilon_{i_1}, + \varepsilon_{0},   \ldots  ,
 - \varepsilon_{i_{s-s'}}, + \varepsilon_{0}}_{2(s-s')},  
  \underbrace{ -\varepsilon_{0}, + \varepsilon_{j_1},   \ldots  ,
 - \varepsilon_{0}, + \varepsilon_{j_{s-s'}}}_{ 2(s-s')},
  -\varepsilon_{i_{s-s'+1}}, + \varepsilon_{j_{s-s'+1}},   \ldots ,-\varepsilon_{i_{s'}}, + \varepsilon_{j_{s'}} 
)
\end{equation}
if the nodes  $ -\varepsilon_{i_k}$ and $ -\varepsilon_{i_{k+1}}$ are {\em removed} from   the same column  or $\bar{\stu} $ to be the tableau 
 \begin{equation}\label{secondtbar}
 ( 
  \underbrace{ -\varepsilon_{0}, + \varepsilon_{j_1},   \ldots  ,
 - \varepsilon_{0}, + \varepsilon_{j_{s-s'}}}_{ 2(s-s')},
  \underbrace{-\varepsilon_{i_1}, + \varepsilon_{0},   \ldots  ,
 - \varepsilon_{i_{s-s'}}, + \varepsilon_{0}}_{2(s-s')}, 
  -\varepsilon_{i_{s-s'+1}}, + \varepsilon_{j_{s-s'+1}},   \ldots ,-\varepsilon_{i_{s'}}, + \varepsilon_{j_{s'}} 
)
\end{equation}
if the nodes $ +\varepsilon_{j_k}$ and $+ \varepsilon_{j_{k+1}}$ are   {\em added} in the same column.  
In either case, we have that $\overline{\stu}_{k\leftrightarrow k+1}$ does not exist, as before.  
Finally, assume  
$$ 
  |\lambda \ominus (\lambda  \cap \nu )|+  |\nu\ominus (\lambda  \cap \nu )|
\leq  s \leq 
 |\lambda|+|\nu|.$$   
%  We assume that $2t:= s-2  |\lambda_{[n]} \ominus (\lambda_{[n]} \cap \nu_{[n]}) | $ is even. (One can treat the odd case similarly.)    
 We let  $\lambda\cap \nu=(\alpha_1,\alpha_2,\dots,\alpha_\ell)$. 
 % and note that $\alpha$ is a partition and that $|\lambda|+|\nu|-s \geq |\alpha|$.  %$2t\leq |\alpha|$.
We let  ${\sf a} $ denote the sequence of steps   obtained from deleting the middle 
% $t=(|\lambda|+|\nu|-\max\{|\lambda \ominus (\lambda  \cap \nu )|, |\nu\ominus (\lambda  \cap \nu )|\} )	$
 $t=(2|\alpha|+  |\lambda \ominus (\lambda  \cap \nu )| + |\nu\ominus (\lambda  \cap \nu )| -s)	$ 
  integral steps from
\begin{equation}\label{72348987924197821478923489724789279824117982437918241792179842798429278}
\underbrace{  a(1)\circ    a(1) \circ \cdots   a(1)}_{\alpha_1}
\circ 
\dots 
\circ 
\underbrace{  a(\ell)\circ    a(\ell) \circ \cdots   a(\ell)}_{\alpha_\ell}
\circ 
\underbrace{  r(\ell)\circ    r(\ell) \circ \cdots   r(\ell)}_{\alpha_\ell}
\circ 
\dots
\circ 
\underbrace{  r(1)\circ    r(1) \circ \cdots   r(1)}_{\alpha_1}
\end{equation}
or 
\begin{equation}\label{723489879241978214789234897247892798241179824379182417921798427984292783}
\underbrace{  a(1)\circ    a(1) \circ \cdots   a(1)}_{\alpha_1}
\circ 
\dots 
\circ 
\underbrace{  a(\ell)\circ    a(\ell) \circ \cdots   a(\ell)}_{\alpha_\ell-1}
\circ d(\ell) \circ 
\underbrace{  r(\ell)\circ    r(\ell) \circ \cdots   r(\ell)}_{\alpha_\ell-1}
\circ 
\dots
\circ 
\underbrace{  r(1)\circ    r(1) \circ \cdots   r(1)}_{\alpha_1}
\end{equation}
for $t$ even or odd respectively.  
  As $\alpha\subseteq \nu$ is a partition,
we have that $ {\sf a}$ is a standard tableau of degree $ s- |\lambda \ominus (\lambda  \cap \nu )| - |\nu\ominus (\lambda  \cap \nu )|$ 
beginning and terminating at $\nu$.  
  Finally if $\overline\stu$ is the tableau of degree $ |\lambda \ominus (\lambda  \cap \nu )|+  |\nu\ominus (\lambda  \cap \nu )|$ 
    as in \cref{firsttbar} or \cref{secondtbar}, then 
$$\stv=\overline\stu \circ {\sf a} \in \Std_{s}^0(\nu\setminus\lambda) \text{ and } 
 \stv_{  { k}\leftrightarrow  { k}+1} \not \in \Std_{s}^0(\nu\setminus\lambda)$$
for   $1\leq  { k} \leq s' $ as before, as required.      \end{proof}

\begin{prop}
Let  $(\lambda,\nu,s)\in \mathscr{P}_ { r-s} \times \mathscr{P}_{\leq r} \times \NN$ with $\Std_s^0(\nu\setminus \lambda)\neq \emptyset$. If $(\lambda, \nu,s)$ satisfies $(C1)$ and $(C2)$, 
then  $(\lambda,\nu,s)$ satisfies (coP).  

\end{prop}

 \begin{proof}
Using Lemma  \ref{diffcolumn} we   can assume that neither  of the skew partitions  $ \nu \ominus (\lambda\cap \nu)$ or  $ \lambda \ominus (\lambda\cap \nu)$   contain two nodes in the same column, i.e. $\minmax(\lambda,\nu)\geq 0$.

 Throughout the proof, we let 
$s'=\max\{|\lambda\ominus (\lambda \cap \nu)|,|\nu\ominus(\lambda \cap \nu)|\} $. 

We will prove this result by contrapositive. Suppose that $(\lambda, \nu,s)$ does not satisfy (coP). Then $s>1$, $\max\{\ell(\lambda), \ell(\nu)\} \geq 2$ and $ s'+\minmax(\lambda,\nu) +1\leq s\leq |\lambda| + |\nu|$.
We pick $c\geq 2$ minimal such that $\minmax(\lambda , \nu) = \min \{ \lambda_{c-1},\nu_{c-1}\} - \max\{\lambda_{c},\nu_{c}\}$.

\textbf{Case I.} $\minmax(\lambda,\nu)=0$.  
By the minimality of $c$ we can find $\stu\in \Std_{s'}^0(\nu\setminus \lambda)$ 
and $0\leq  k \leq s'$ such that the $(c-1)$th and $c$th rows of 
either $\stu(  k  )$ or $\stu (k+1/2)$ have the same length. We choose $k$ minimal with this property. 
 Let $s=s'+1$.  
By   the minimality of $c$, we have that 
$$
\stv=\stu[0,k] \circ
 \underbrace{ (-i_{k+1},+(c-1), -(c-1), + j_{k+1})}_{\text{important}} \circ \stu[k+1,s']
$$
belongs to $\Std_s^0(\nu\setminus\lambda)$.  If we swap the two important integral steps   of $\stv$ we obtain a sequence which does not belong to 
 $\Std_s(\nu\setminus\lambda)$. This violates condition $(C1)$. One can inflate the  tableau $\stv$ as in \cref{firsttbar} or \cref{secondtbar} %(by inflating $\stu$) 
 and/or by concatenating with a path  of the form in \cref{72348987924197821478923489724789279824117982437918241792179842798429278,723489879241978214789234897247892798241179824379182417921798427984292783} to obtain  an element of $ \Std^0_t(\nu\setminus\lambda)$ for any $s \leq t \leq |\lambda|+|\nu|$ which violates (C1).

\textbf{Cases II and III.} For the remainder of the proof  we  set  $k=\max\{0,
 \lambda_{c-1}- \nu_{c-1},    \nu_{c} -  \lambda_{c}  
\}.$   
We let $\stu\in \Std_{s'}^0(\nu\setminus\lambda)$ denote any 
path in which all steps of the form 
$-\varepsilon_{c-1}$ or $+\varepsilon_{c}$ 
occur in the first $k$ integral steps   and all steps of the form
$+\varepsilon_{c-1}$ or $-\varepsilon_{c}$ 
occur in the final  $s'-k$ integral steps.  
 That such a tableau exists follows  from our assumption that $s'$ is minimal such that $\Std_{s'}(\nu\setminus\lambda)\neq \emptyset$ (so no step can be added and removed in the same row).

\textbf{Case II.} $\minmax(\lambda,\nu)>0$ and $c<\max\{\ell(\lambda), \ell(\nu)\}$.  
% We now set  $k=\max\{0,
% \lambda_{c-1}- \nu_{c-1},    \nu_{c} -  \lambda_{c}  
%\}.$   
%We let $\stu\in \Std_{s'}^0(\nu\setminus\lambda)$ denote any 
%path in which all steps of the form 
%$-\varepsilon_{c-1}$ or $+\varepsilon_{c}$ 
%occur in the first $k$ integral steps   and all steps of the form
%$+\varepsilon_{c-1}$ or $-\varepsilon_{c}$ 
%occur in the final  $s'-k$ integral steps.  
% That such a tableau exists follows  from our assumption that $s'$ is minimal such that $\Std_{s'}(\nu\setminus\lambda)\neq \emptyset$ (so no step can be added and removed in the same row).  
  Let 
  $s=s'+\minmax(\lambda,\nu)+1$.                   For $\minmax(\lambda,\nu)$ even, we let $\stv$ denote the following tableau
$$
   \stu [0,k] \circ
 \underbrace{
 m{\downarrow}(c-1,c)\circ \dots \circ 
 m{\downarrow}(c-1,c) }_{\minmax(\lambda,\nu)/2-1}
 \circ 
 \underbrace{ d(c-1)
 \circ 
 m{\downarrow}(c-1,c) }_{\text{important}}
 \circ 
 \underbrace{
   m{\uparrow}(c,c-1) \circ \dots  \circ m{\uparrow}(c,c-1)   }_{\minmax(\lambda,\nu)/2}\circ \stu[k,s'] . $$
We have that
  $\stv \in \Std_s^0(\nu\setminus\lambda)$. For $\minmax(\lambda,\nu)$  odd, we let $\stv$ denote the following tableau
 $$
   \stu [0,k] \circ
 \underbrace{
 m{\downarrow}(c-1,c)\circ \dots 
 \circ 
 m{\downarrow}(c-1,c) }_{(\minmax(\lambda,\nu)-1)/2}
 \circ 
 \underbrace{m{\uparrow}(c,c-1)
 \circ  m{\downarrow}(c-1,c) }_{\text{important}}\circ 
 \underbrace{
   m{\uparrow}(c,c-1) \circ \dots  \circ m{\uparrow}(c,c-1)   }_{(\minmax(\lambda,\nu)-1)/2} 
  \circ  \stu [k,s'] 
 $$
 We have that
  $\stv \in \Std_s^0(\nu\setminus\lambda)$. 
In both cases, if we swap the two important integral steps in the tableau $\stv$ we obtain a sequence which does not belong to $\Std_s(\nu\setminus \lambda)$. This violates (C1).
Again, we can inflate $\stv$ as in Case I to get an element of $ \Std^0_t(\nu\setminus\lambda)$ for any $s \leq t \leq |\lambda|+|\nu|$ which also violates (C1).

\textbf{Case III.} $\minmax(\lambda,\nu)>0$ and $c=\max\{\ell(\lambda), \ell(\nu)\}$.  For
  $s=s'+ 2\minmax(\lambda,\nu)+1$. 
 We let $\stv$ denote the following tableau
$$
   \stu [0,k] \circ
 \underbrace{
r(c-1)\circ \dots \circ 
 r(c-1) }_{\minmax(\lambda,\nu) -1}
 \circ 
 \underbrace{ d(c-1)
 \circ 
r(c-1) }_{\text{important}}
 \circ 
 \underbrace{
 a(c-1)\circ \dots  \circ a (c-1)  }_{\minmax(\lambda,\nu) }\circ \stu[k,s'] .$$
%where $k$ is chosen as in Case II.
We have that
  $\stv \in \Std_s^0(\nu\setminus\lambda)$. 
If we swap the two important integral steps in the tableau  
  $\stv$,
 we obtain a sequence which does not belong to 
 $\Std_s(\nu\setminus\lambda)$.  This violates condition $(C1)$. Moreover we can inflate $\stv$ as in Case I to show that $(\lambda, \nu,s)$ does not satisfy (C1) for any $s'+2\minmax(\lambda, \nu)+1 \leq s\leq |\lambda|+|\nu|$.

It remains to  consider the case 
 $s' +\minmax(\lambda,\nu) +1\leq  s \leq s'+2\minmax(\lambda,\nu)$.  We will show that $(\lambda, \nu, s)$ does not satisfy (C2). 
We begin with the case    $s=s'+ \minmax(\lambda,\nu)+ 1$. 
We shall see that the map of \cref{anequationweshalluse} is well-defined and  injective, but no longer surjective.  

  Let  $\alpha \subset \lambda_{[n]}\cap\nu_{[n]}$ with $\alpha \vdash n-s$. Let $\sts\in \Std_s(\nu_{[n]}\setminus \alpha)$ and $\stt\in \Std_s(\alpha \setminus \lambda_{[n]})$ and write $\stv = \varphi_s(\sts, \stt)$ defined as in \cref{anequationweshalluse}. We need to show that $\stv\in \Std_s(\nu\setminus \lambda)$. Using the same notation as in Lemma 4.19, following its proof, and using the fact that 
$s\leq s'+\min\{\lambda_{i-1}, \nu_{i-1}\}-\max\{\lambda_i,\nu_i\}$ for all $i \neq c$ (by minimality of $c $) we obtain that
$\alpha_i\geq \beta_{i+1}$ for all $i\neq c$ and $\alpha_c\geq \beta_{c+1}$. Now following the proof of Lemma 4.20 this implies that $\stv(l)_{i-1} \geq \stv(l)_i$ for all $i\neq c$  and $\stv(l)_{c-1}\geq \stv(l)_c -1$ for all $1\leq l\leq s$.  Now suppose, for a contradiction that $\stv(k)_{c-1} = \stv(k)_{c}-1$. Then we must have $\stv(k)_{c-1} = \alpha_c = \sts(k)_c$ and $\stv(l)_{c} = \beta_{c+1} = \sts(l)_{c+1}$, contradicting the fact that $\sts$ is a standard tableau.
Thus $\varphi_s$ is well-defined. Injectivity is obvious by definition.We now show that there is some $\bar{\stu}\in \Std_s^0(\nu\setminus \lambda)$ which is not in the image of $\varphi_s$. 
%As in Case II, we let $k=\max\{0,
% \lambda_{c-1}- \nu_{c-1},    \nu_{c} -  \lambda_{c}  
%\},$   
%and
%We  pick $\stu\in \Std_{s'}^0(\nu\setminus\lambda)$ to be any 
%path in which all steps of the form 
%$-\varepsilon_{c-1}$ or $+\varepsilon_{c}$ 
%occur in the first $k$ integral steps   and all steps of the form
%$+\varepsilon_{c-1}$ or $-\varepsilon_{c}$ 
%occur in the final  $s'-k$ integral steps
% this ensures that $\stu(k)_{c-1} - \stu(k)_c = \minmax(\lambda, \nu)$. Now consider the tableau
Recall, we  picked $\stu\in \Std_{s'}^0(\nu\setminus\lambda)$ such that  all steps of the form 
$-\varepsilon_{c-1}$ or $+\varepsilon_{c}$ 
occur in the first $k$ integral steps   and all steps of the form
$+\varepsilon_{c-1}$ or $-\varepsilon_{c}$ 
occur in the final  $s'-k$ integral steps; 
 this ensures that $\stu(k)_{c-1} - \stu(k)_c = \minmax(\lambda, \nu)$. Now consider the tableau
$$
  \bar{\stu} =
  \stu [0,k] 
  \circ
\underbrace{  d(c-1)\circ \dots \circ d(c-1)
}_{\minmax(\lambda,\nu)+1} \circ 
   \stu[k,s']
$$
which belongs to $\Std_s^0(\nu\setminus\lambda)$.   
Suppose for a contradiction that $\bar{\stu} = \varphi_s(\sts, \stt)$ for some standard tableaux $\sts$ and $\stt$.
Now if $\lambda_c = \nu_c$ then $\alpha = \min_n(\stu)$ is not a partition so $\stu$ cannot be in the image of $\varphi_s$. If $\lambda_c>\nu_c$ then $\stt(k+\minmax(\lambda,\nu)+1)$ is not a partition and if $\nu_c>\lambda_c$ then $\sts(k)$ is not a partition. In all cases we see that $\bar{\stu}\in \Std_s^0(\nu\setminus \lambda)$ is not in the image of $\varphi_s$.
 Now we can decompose 
$$\bigoplus_{\begin{subarray}c \alpha  \vdash n-s \\ \alpha \subseteq  \lambda_{[n]} \cap \nu_{[n]}
\end{subarray}} \Delta_s(\nu_{[n]}\setminus \alpha ) \otimes  \Delta_s (\alpha \setminus \lambda_{[n]})  = \bigoplus_{0\leq m\leq s} \Vec_s^m$$
as in (4.11). The fact that the map $\varphi_s$ is not surjective implies that $|\Std_s^0(\nu\setminus \lambda)| > \dim V_s^0$. Now if we follow (4.16) -- (4.22), noting that $(\lambda, \nu, s-m)$ satisfies (coP) for $m>0$, we obtain 
$$|\Std_s^0(\nu\setminus \lambda)|>\dim V_s^0 = \dim \Delta_s^0(\nu\setminus \lambda).$$
This implies that (C2) is not satisfied, as required.

More precisely, we know that there must be some element $\sum_{\stt \in \Std^0_s(\nu\setminus\lambda)} r_\stt u_\stt  \in {\sf DR}_s(\nu\setminus\lambda)$ for $r_\stt \in \CC$.   
 
   We now consider $(\lambda,\nu,s)$ for
    $s' +\minmax(\lambda,\nu)+1+k =  s \leq s'+ 2\minmax(\lambda,\nu)$.  
Let  $\nu'=\nu - k\varepsilon_{c-1}$ and 
$\lambda'=\lambda -k\varepsilon_{c-1}$. Notice that 
 $s-k=
 \max\{|\lambda'\ominus (\lambda' \cap \nu')|,|\nu'\ominus(\lambda' \cap \nu')|\}+
 \minmax(\lambda',\nu')+1$.  
By the above, there exists $a \in P_{s-k}(n) p_{r-k} P_{s-k}(n)$ and 
 $\sts \in \Std^0_{s-k}(\nu'\setminus\lambda')$  such that 
 $$
u_\sts a = \sum_{\stt \in \Std^0_{s-k}(\nu'\setminus\lambda')} r_\stt u_\stt
 \in {\sf DR}_{s-k}(\nu'\setminus\lambda') 
 $$
 with some $r_\stt\neq 0$.  
Now, 
for  any   tableau  $\stv \in  \Std_{s-k}(\nu'\setminus\lambda') $   
  we can inflate the tableau $\stv$ to obtain   $$
\overline{\stv}= \underbrace{r(c-1) \circ \dots \circ r(c-1) }_{k}\circ  
\stv \circ 
\underbrace{a(c-1) \circ \dots \circ a(c-1) }_{k} \in \Std_{s+k}^0(\nu\setminus\lambda).$$   
Similarly, given 
 $a \in P_{s-k}(n)  $ we let 
 $\bar{a}\in P_{s+k}(n)  $ denote the image of $a$ under the embedding 
 $P_{s-k}(n)   \to P_k(n)\times  P_{s-k}(n) \times  P_k(n)$.  
  By \cite[Corollary 3.12]{BE} we have that 
$$
u_{\overline{\sts}} \overline{a} = 
 \sum_{
\begin{subarray}c
\stt \in \Std^0_s(\nu'\setminus\lambda)
\end{subarray}} r_\stt u_{\overline{\stt}} 
+
\sum_{
\begin{subarray}c
\stw \in \Std_{s+k}(\nu\setminus\lambda)
\\
\stw{(s)} \rhd \nu'  
\end{subarray}
} q_\stw u_\stw
 \in {\sf DR}_{s+k}(\nu\setminus\lambda)$$
 which again violates $(C2)$. This completes the proof.  
    \end{proof}

\section{Semistandard Kronecker tableaux}\label{sec:semistandard}

Recall from \cref{depthlayer2} that for any $(\lambda,\nu,s) \in \mptn{r-s}\times  \mptn{\leq r} \times  \ZZ_{> 0} $ and any $\mu \vdash s$ we have
 $$
\overline g( \lambda, \nu,\mu) = \dim_\CC \Hom_{P_s(n)}(\Delta_s(\mu), \Delta_s^0(\nu\setminus\lambda) ) = \dim_{\CC} \Hom_{\CC \mathfrak{S}_s}({\sf S}(\mu), \Delta_s^0(\nu \setminus \lambda)),
$$
where $\CC \mathfrak{S}_s$ is viewed as the quotient of $P_s(n)$ by the ideal generated by $p_r$.  
Now for each   $\mu = (\mu_1, \mu_2, \ldots , \mu_l) \vdash s$ we have an associated   Young permutation module,
 $${\sf M}  (\mu) = \CC\otimes_{\mathfrak{S}_\mu} \CC\mathfrak{S}_s$$
 where $\mathfrak{S}_\mu = \mathfrak{S}_{\mu_1} \times \mathfrak{S}_{\mu_2}\times \dots\times \mathfrak{S}_{\mu_l} \subseteq  \mathfrak{S}_s$.  
It is well known that there is a surjective homomorphism ${\sf M}(\mu)\rightarrow {\sf S}(\mu)$ and moreover, for any $\tau \vdash s$, the multiplicity of ${\sf S}(\tau)$ as a composition factor of ${\sf M}(\mu)$ is given by the number of semistandard Young tableaux of shape $\tau$ and weight $\mu$.
So,  as a first step towards understanding the stable Kronecker coefficients,     it is natural to  consider
$$
\dim_{\CC} \Hom_{\mathfrak{S}_s}({\sf M} (\mu), \Delta^0_s(\nu \setminus \lambda)  ). 
$$ 
In the case of triples of maximal depth, this dimension is given by the number of semistandard Young tableaux of shape $\nu\setminus \lambda$ and weight $\mu$.  
We now extend this result by defining semistandard Kronecker tableaux and show that in the case of co-Pieri triples the number of such tableaux give the required dimension. 
 In fact, we explicitly construct these homomorphisms directly from the associated tableaux.  
 
 We start with a definition of semistandard Kronecker tableaux, generalising the classical definition of semistandard Young  tableaux.

\begin{defn}\label{semistandard}
Let $\mu = (\mu_1, \mu_2, \ldots , \mu_l)\vDash s$, $\lambda \in \mathscr{P}_{r-s}$, $\nu \in \mathscr{P}_{\leq r}$ and
let $\sts, \stt \in \Std^0_s(\nu \setminus \lambda)$.
\begin{enumerate}
\item For $1\leq k <s$ we write $\sts \overset{k}{\sim} \stt$ if $\sts = \stt_{k\leftrightarrow k+1}$.
\item We write $\sts \overset{\mu}{\sim}  \stt$ if there exists a sequence of standard Kronecker tableaux $\stt_1, \stt_2, \ldots , \stt_d \in \Std^0_s(\nu\setminus\lambda)$ 
such that 
$$\sts =  \stt_{1}  \overset{k_1}{\sim}   \stt_{2}  ,  \
 \stt_{2}  \overset{k_2}{\sim}    \stt_{3}  ,   \ \dots \ , 
\stt_{d-1}\overset{k_{d-1}}{\sim}   \stt_{d} 
=\stt  $$  
for some $k_1,\dots, k_{d-1}\in  \{1, \ldots , s-1\} \setminus 
 \{    [\mu]_c \mid  c = 1, \ldots , l-1 \}$.   
We define a {\sf tableau of weight} $\mu$ to be an equivalence class of tableau under $\overset{\mu}{\sim}  $, denoted $[\stt]_\mu = \{ \sts \in \Std^0_s(\nu\setminus \lambda) \, |\, \sts \overset{\mu}{\sim} \stt\}$.
\item We say that a    Kronecker tableau, $[\stt]_\mu$,  of shape $\nu\setminus \lambda$ and weight $\mu$ is {\sf semistandard} if, for all $1\leq c \leq l$, the skew partitions $\stt([\mu]_c) \ominus (\stt([\mu]_{c-1}) \cap \stt([\mu]_c))$ and $\stt([\mu]_{c-1}) \ominus (\stt([\mu]_{c-1}) \cap \stt([\mu]_c))$ have no two boxes in the same column.\\
We denote the set of all semistandard Kronecker tableaux of shape $\nu\setminus \lambda$ and weight $\mu$ by $\SStd_s^0(\nu\setminus \lambda, \mu)$.
\end{enumerate}
\end{defn}

 \begin{rmk} \label{sscopieri2}
 Note that if $\sts, \stt\in \Std_s^0(\nu\setminus \lambda)$ with $\sts \in [\stt]_\mu$ then $\sts([\mu]_c)=\stt([\mu]_c)$ for all $1\leq c\leq l$ hence part (3) is independent of the choice of representative in $[\stt]_\mu$ and hence the notion of semistandard   Kronecker tableau is well-defined.
 \end{rmk}
\begin{rmk}\label{sscopieri}
 If $(\lambda, \nu,\mu)$ is a co-Pieri triple, it follows from \cref{LemmaB} that for any $\stt \in \Std_s^0(\nu\setminus \lambda)$ the class $[\stt]_\mu$ is a semistandard Kronecker tableau.
 \end{rmk} 

\begin{rmk}  If $(\lambda, \nu,\mu)$ is a triple of maximal depth then $\SStd_s^0(\nu\setminus \lambda, \mu)$ coincide with the classical notion of semistandard Young tableaux of shape $\nu\setminus \lambda$ and weight $\mu$ (and similarly for the non-semistandard tableaux of a given weight).  
 \end{rmk}

To represent these semistandard Kronecker tableaux graphically, we will add \lq frames' corresponding to the composition $\mu$ on  the set of paths $\Std_s^0(\nu \setminus \lambda)$ in the branching graph. 
For $\stt =(-\varepsilon_{i_1}, + \varepsilon_{j_1}, \ldots , -\varepsilon_{i_s}, + \varepsilon_{j_s})$ we say that the integral step $(-\varepsilon_{i_k}, + \varepsilon_{j_k})$ belongs to the $c$th frame if $[\mu]_{c-1} < k\leq  [\mu]_c$.
Thus for $\sts, \stt \in \Std_s^0(\nu\setminus \lambda)$ we have that $\sts \overset{\mu}{\sim} \stt$ if and only if $\sts$ is obtained from $\stt$ by permuting integral steps within each frame.

 \begin{eg}\label{semiexam1}
  Let  
  $\lambda =(4,2) $, $\nu =( 5,3,1)$ and $s=3$. Then $(\lambda, \nu,s)$ is a triple of maximal depth. Take $\mu=(2,1) \vDash {3}$.   
  We have three semistandard  tableaux of shape $\nu\setminus \lambda$  and weight $\mu$ given by
\begin{align*}
&  \SSTS_1= \{a(2) \circ a(3) \circ a(1) \ , \ a(3) \circ a(2) \circ a(1) \}
\\
 & \SSTS_2= \{a(1) \circ a(3) \circ a(2) \ , \ a(3) \circ a(1) \circ a(2) \}
\\ & \SSTS_3= \{a(1) \circ a(2) \circ a(3) \ , \ a(2) \circ a(1) \circ a(3) \}.
\end{align*}
 They are depicted in \cref{smiinpartition} and  
 ordered so that one can compare them directly with the 
  tableaux in \cref{copierieg}.

 \begin{figure}[ht!]
 \scalefont{0.8} 
$$ \begin{tikzpicture}[scale=0.6]
         \draw[white] (-2,0.35) rectangle (4.5,-16.3);  
                   \fill[white] (0,0) circle (20pt);   
                                   \fill[white] (0,-15) circle (18pt);                                    
      \draw[dashed] (-2,0) rectangle (4.5,-15);  
                \draw[dashed] (-2,-10) -- (4.5,-10);  
                 \fill[white] (0,0) circle (20pt);   
                                  \fill[white] (0,-10) circle (22pt);   
 \fill[white] (0,-15) circle (24pt);   
                 \begin{scope}   
      \draw (0,0) node {$  \Yboxdim{5pt}\gyoung(;;;;,;;) $  };   
     %%%%%
   \draw (0,-2.5) node{$ \Yboxdim{5pt}\gyoung(;;;;,;;)$  };
%%%%%
  \draw (0,-5) node{$ \Yboxdim{5pt}\gyoung(;;;;,;;;)$  };
  \draw (3,-5) node{$    \Yboxdim{5pt}\gyoung(;;;;,;;,;) $  };
%%%%%
  \draw (0,-7.5) node{ \Yboxdim{5pt}\gyoung(;;;;,;;;) };
   \draw (3,-7.5) node{$   \Yboxdim{5pt}\gyoung(;;;;,;;,;)$  };
  \draw (0,-10) node{$   \Yboxdim{5pt}\gyoung(;;;;,;;;,;)$};
  %%%
   \draw (0,-12.5) node{$  \Yboxdim{5pt}\gyoung(;;;;,;;;,;)$  };
  \draw (0,-15) node{$ \Yboxdim{5pt}\gyoung(;;;;;,;;;,;) $};
     \path   (0,-0.6)edge[decorate]  node[left] {$-0$}  (0,-1.9);
          \path   (0.6,-3.1)edge[decorate]  node[auto] {$+3$}  (3,-4.2);
                    \path   (0,-3.1)edge[decorate]  node[left] {$+2$}  (0,-4.4);
                    \path   (0,-5.6)edge[decorate]  node[left] {$-0$}  (0,-6.9);
                                        \path   (3,-5.6)edge[decorate]  node[auto] {$-0  $}  (3,-6.8);
          \path   (0,-8.1)edge[decorate]  node[left] {$+3$}  (0,-9.3);
                    \path   (2.3,-8.2)edge[decorate]  node[above] {$+2$}  (0.6,-9.3);
%                                        \path   (2.3,-8.2-5)edge[decorate]  node[above] {$+1$}  (0.6,-9.3-5);
%                                         \path   (3,-8.1-2.5)edge[decorate]  node[left] {$-1$}  (3,-9.3-2.5);
                              \path   (0,-10.6)edge[decorate]  node[left] {$-0$}  (0,-11.8);
                                                            \path   (0,-13.1)edge[decorate]  node[left] {$+1$}  (0,-14.3);
   \end{scope}
   \end{tikzpicture}
   \quad   \quad
   \begin{tikzpicture}[scale=0.6]
         \draw[white] (-2,0.35) rectangle (4.5,-16.3);  
                   \fill[white] (0,0) circle (20pt);   
                                   \fill[white] (0,-15) circle (18pt);                                    
      \draw[dashed] (-2,0) rectangle (4.5,-15);  
                \draw[dashed] (-2,-10) -- (4.5,-10);  
                 \fill[white] (0,0) circle (20pt);   
                                  \fill[white] (0,-10) circle (22pt);   
 \fill[white] (0,-15) circle (24pt);   
                 \begin{scope}   
      \draw (0,0) node {$  \Yboxdim{5pt}\gyoung(;;;;,;;) $  };   
     %%%%%
   \draw (0,-2.5) node{$ \Yboxdim{5pt}\gyoung(;;;;,;;)$  };
%%%%%
  \draw (0,-5) node{$ \Yboxdim{5pt}\gyoung(;;;;;,;;)$  };
  \draw (3,-5) node{$    \Yboxdim{5pt}\gyoung(;;;;,;;,;) $  };
%%%%%
  \draw (0,-7.5) node{ \Yboxdim{5pt}\gyoung(;;;;;,;;) };
   \draw (3,-7.5) node{$   \Yboxdim{5pt}\gyoung(;;;;,;;,;)$  };
  \draw (0,-10) node{$   \Yboxdim{5pt}\gyoung(;;;;;,;;,;)$};
  %%%
   \draw (0,-12.5) node{$  \Yboxdim{5pt}\gyoung(;;;;;,;;,;)$  };
  \draw (0,-15) node{$ \Yboxdim{5pt}\gyoung(;;;;;,;;;,;) $};
     \path   (0,-0.6)edge[decorate]  node[left] {$-0$}  (0,-1.9);
          \path   (0.6,-3.1)edge[decorate]  node[auto] {$+3$}  (3,-4.2);
                    \path   (0,-3.1)edge[decorate]  node[left] {$+1$}  (0,-4.4);
                    \path   (0,-5.6)edge[decorate]  node[left] {$-0$}  (0,-6.9);
                                        \path   (3,-5.6)edge[decorate]  node[auto] {$-0  $}  (3,-6.8);
          \path   (0,-8.1)edge[decorate]  node[left] {$+3$}  (0,-9.3);
                    \path   (2.3,-8.2)edge[decorate]  node[above] {$+1$}  (0.6,-9.3);
                               \path   (0,-10.6)edge[decorate]  node[left] {$-0$}  (0,-11.8);
                                                            \path   (0,-13.1)edge[decorate]  node[left] {$+2$}  (0,-14.3);
   \end{scope}
   \end{tikzpicture}
  \quad   \quad
   \begin{tikzpicture}[scale=0.6]
         \draw[white] (-2,0.35) rectangle (4.5,-16.3);  
                   \fill[white] (0,0) circle (20pt);   
                                   \fill[white] (0,-15) circle (18pt);                                    
      \draw[dashed] (-2,0) rectangle (4.5,-15);  
                \draw[dashed] (-2,-10) -- (4.5,-10);  
                 \fill[white] (0,0) circle (20pt);   
                                  \fill[white] (0,-10) circle (22pt);   
 \fill[white] (0,-15) circle (24pt);   
                 \begin{scope}   
      \draw (0,0) node {$  \Yboxdim{5pt}\gyoung(;;;;,;;) $  };   
     %%%%%
   \draw (0,-2.5) node{$ \Yboxdim{5pt}\gyoung(;;;;,;;)$  };
%%%%%
  \draw (0,-5) node{$ \Yboxdim{5pt}\gyoung(;;;;;,;;)$  };
  \draw (3,-5) node{$    \Yboxdim{5pt}\gyoung(;;;;,;;;) $  };
%%%%%
  \draw (0,-7.5) node{ \Yboxdim{5pt}\gyoung(;;;;;,;;) };
   \draw (3,-7.5) node{$   \Yboxdim{5pt}\gyoung(;;;;,;;;)$  };
  \draw (0,-10) node{$   \Yboxdim{5pt}\gyoung(;;;;;,;;;)$};
  %%%
   \draw (0,-12.5) node{$  \Yboxdim{5pt}\gyoung(;;;;;,;;;)$  };
  \draw (0,-15) node{$ \Yboxdim{5pt}\gyoung(;;;;;,;;;,;) $};
     \path   (0,-0.6)edge[decorate]  node[left] {$-0$}  (0,-1.9);
          \path   (0.6,-3.1)edge[decorate]  node[auto] {$+2$}  (3,-4.2);
                    \path   (0,-3.1)edge[decorate]  node[left] {$+1$}  (0,-4.4);
                    \path   (0,-5.6)edge[decorate]  node[left] {$-0$}  (0,-6.9);
                                        \path   (3,-5.6)edge[decorate]  node[auto] {$-0  $}  (3,-6.8);
          \path   (0,-8.1)edge[decorate]  node[left] {$+2$}  (0,-9.3);
                    \path   (2.3,-8.2)edge[decorate]  node[above] {$+1$}  (0.6,-9.3);
                               \path   (0,-10.6)edge[decorate]  node[left] {$-0$}  (0,-11.8);
                                                            \path   (0,-13.1)edge[decorate]  node[left] {$+3$}  (0,-14.3);
   \end{scope}
   \end{tikzpicture}
   $$
   \caption{The three elements of $\SStd^0_3((5,3,1)\setminus(4,2),(2,1))$.  These tableaux are ordered to facilitate comparison with \cref{anexampleofLRclassic,anexampleofLRclassic2}. }
   \label{smiinpartition}
   \end{figure}
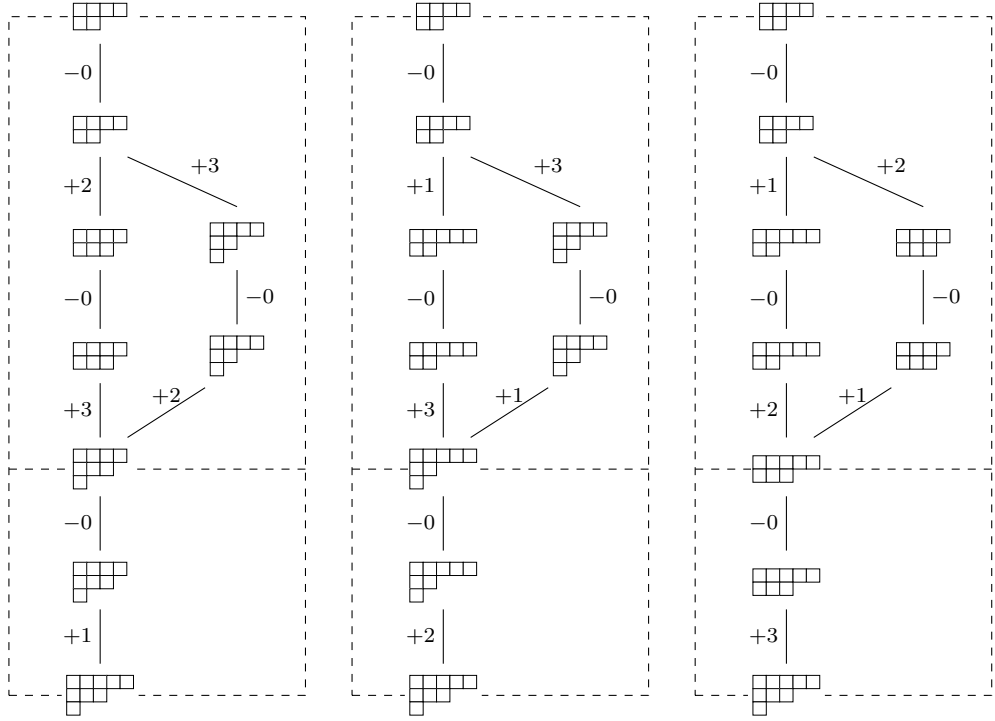
 \end{eg}

\begin{eg}\label{onerowss}
Let $\lambda = (7)$, $\nu = (6)$ and $s=6$. Then $(\lambda, \nu, 6)$ is a co-Pieri triple. We have $|\SStd_6^0(\nu \setminus \lambda, (6))| = 3$ and a representative for each of these semistandard tableaux is given by
\begin{eqnarray*}
&&\stt_1 = r(1)\circ r(1) \circ r(1) \circ d(1) \circ a(1) \circ a(1)\\
&& \stt_2 =   r(1)\circ r(1) \circ d(1) \circ d(1) \circ d(1) \circ a(1)\\
&& \stt_3 =  r(1)\circ d(1) \circ d(1) \circ d(1) \circ d(1) \circ d(1)
\end{eqnarray*}
We have $|\SStd_6^0(\nu\setminus \lambda, (3,2,1))| = 27$. To see this, observe that $[\stt_1]_{(6)}$ and $[\stt_2]_{(6)}$ each splits into 12 semistandard Kronecker tableaux of weight $(3,2,1)$, whereas $[\stt_3]_{(6)}$ splits into 3 semistandard Kronecker tableaux of weight $(3,2,1)$.
\end{eg}

%\begin{thm} \label{YOUNGSRULE}
%Let $(\lambda, \nu, s)$ be a co-Pieri triple and $\mu\vdash s$. Then we have
%$$\dim_{\CC} \Hom_{\mathfrak{S}_s}({\sf M}(\mu), \Delta_s^0(\nu\setminus \lambda)) = |\SStd_s^0(\nu \setminus \lambda, \mu)|.$$
% \end{thm}

\begin{thm} \label{YOUNGSRULE}
Let $(\lambda, \nu, s)$ be a co-Pieri triple and $\mu\vdash s$. Then we define 
$$\varphi_\SSTT(u_{\stt^\mu}) = \sum_{\sts\in  \SSTT}u_\sts.  $$
for $\SSTT \in \SStd_s^0(\nu \setminus \lambda, \mu)$.  We have that 
$$\{\varphi_\SSTT \mid \SSTT\in \SStd_s^0(\nu \setminus \lambda, \mu)\}$$
is a $\ZZ$-basis for 
$   \Hom_{\mathfrak{S}_s}({\sf M}(\mu), \Delta_s^0(\nu\setminus \lambda))$.    
 \end{thm}

\begin{proof} 
    By Frobenius reciprocity,  
\begin{align*}
     \Hom_{\mathfrak{S}_s}({\sf M}(\mu), \Delta_s^0(\nu\setminus\lambda))
   &  \cong
          \Hom_{\mathfrak{S}_\mu}(\CC, \Delta_s^0(\nu\setminus\lambda)\downarrow_{\mathfrak{S}_\mu})
\end{align*} 
 It is clear from \cref{co-case} and \cref{sscopieri2,sscopieri}  that $\Delta_s^0(\nu\setminus\lambda)\downarrow_{\mathfrak{S}_\mu}$ decomposes as
$$\Delta_s^0(\nu\setminus\lambda)\downarrow_{\mathfrak{S}_\mu} = \bigoplus_{\SSTT \in \SStd_s^0(\nu\setminus \lambda,\mu)}V(\SSTT)$$
where $V(\SSTT) = {\rm Span}_\CC \{ m_{\stt} + {\sf DR}(\nu \setminus \lambda) \, | \, [\stt ]_{\mu} = \SSTT \}$.
Moreover, each $V(\SSTT)$ is itself a permutation module of the form $\CC\uparrow_{\mathfrak{S}_\tau}^{\mathfrak{S}_\mu}$ for some composition $\tau \vDash s$ which is a refinement of $\mu$. Thus we have that $\dim_\CC \Hom_{\mathfrak{S}_{\mu}}(\CC, V(\SSTT)) = 1$ for each $\SSTT\in \SStd_s^0(\nu \setminus \lambda, \mu)$ and the result follows.
      \end{proof}

 \begin{eg}\label{semiexam}
Let $\lambda =(8,5,3) $,
 $\nu=(6,5,3,2)$ and $s=3$. Then $(\lambda, \nu,3)$ is a co-Pieri triple.
 We have that  $|\SStd_3^0(\nu\setminus \lambda,(3))|=6$. A representative for each of these semistandard tableaux is given as follows,
$$\begin{array}{cccc}
&d(1) \circ \down(1,4) \circ  \down(1,4)
 & d(2) \circ \down(1,4) \circ  \down(1,4) 
& \down(1,2) \circ \down(1,4) \circ \down(2,4)   \\
%%%
&d(3)\circ \down(1,4) \circ  \down(1,4)
& 
 r(1) \circ \down(1,4) \circ a(4)
&
\down(1,3) \circ \down(1,4) \circ  \down(3,4).
\end{array}$$
The semistandard tableau corresponding to the first  of these 
  tableaux is depicted in  \cref{semiexamfig}.
   We have that   $|\SStd_3^0(\nu\setminus \lambda,(2,1))|=15$. 
Two  examples  of  such tableaux  are  
  depicted in \cref{semiexamfig}.

\begin{figure}[ht!]
$$\scalefont{0.8} \begin{tikzpicture}[scale=0.6]
        \draw[white] (-1.8,0.35) rectangle (4.4,-16.3);  
        \draw[dashed] (-1.8,0) rectangle (4.4,-15);  
                  \fill[white] (0,0) circle (20pt);   
                                  \fill[white] (0,-15) circle (18pt);    \begin{scope}   
 \fill[white] (0,0) circle (37pt);
     \draw (0,0) node {$  \Yboxdim{5pt}\gyoung(;;;;;;;;,;;;;;,;;;) $  };   
     %%%%%
\fill[white] (0,-2.5) circle (37pt);    \draw (0,-2.5) node{$  \Yboxdim{5pt}\gyoung(;;;;;;;,;;;;;,;;;) $  };
%%%%%
\fill[white] (0,-5) circle (37pt);  \draw (0,-5) node{$  \Yboxdim{5pt}\gyoung(;;;;;;;;,;;;;;,;;;) $  };
\fill[white] (3,-5) circle (37pt);  \draw (3,-5) node{$  \Yboxdim{5pt}\gyoung(;;;;;;;,;;;;;,;;;,;) $  };
%%%%%
  \fill[white] (0,-7.5) circle (37pt);  \draw (0,-7.5) node{$  \Yboxdim{5pt}\gyoung(;;;;;;;,;;;;;,;;;) $  };
\fill[white] (3,-7.5) circle (37pt);  \draw (3,-7.5) node{$  \Yboxdim{5pt}\gyoung(;;;;;;,;;;;;,;;;,;) $  };
  \fill[white] (0,-10) circle (37pt);  \draw (0,-10) node{$  \Yboxdim{5pt}\gyoung(;;;;;;;,;;;;;,;;;,;) $};
\fill[white] (3,-10) circle (37pt);  \draw (3,-10) node{$  \Yboxdim{5pt}\gyoung(;;;;;;,;;;;;,;;;,;;) $  };
 %%%
  \fill[white] (0,-12.5) circle (37pt);  \draw (0,-12.5) node{$  \Yboxdim{5pt}\gyoung(;;;;;;,;;;;;,;;;,;) $};
  \fill[white] (3,-12.5) circle (37pt);  \draw (3,-12.5) node{$  \Yboxdim{5pt}\gyoung(;;;;;,;;;;;,;;;,;;) $  };
  \fill[white] (0,-15) circle (37pt);  \draw (0,-15) node{$  \Yboxdim{5pt}\gyoung(;;;;;;,;;;;;,;;;,;;) $};
     \path   (0,-0.6)edge[decorate]  node[left] {$-1$}  (0,-1.9);
          \path   (0.6,-3.1)edge[decorate]  node[auto] {$+4$}  (3,-4.2);
                    \path   (0,-3.1)edge[decorate]  node[left] {$+1$}  (0,-4.4);
                    \path   (0,-5.6)edge[decorate]  node[left] {$-1$}  (0,-6.9);
                                        \path   (3,-5.6)edge[decorate]  node[auto] {$-1  $}  (3,-6.8);
          \path   (0,-8.1)edge[decorate]  node[left] {$+4$}  (0,-9.3);
                    \path   (2.3,-8.2)edge[decorate]  node[above] {$+1$}  (0.6,-9.3);
                                        \path   (2.3,-8.2-5)edge[decorate]  node[above] {$+1$}  (0.6,-9.3-5);
                    \path   (3,-8.1)edge[decorate]  node[right] {$+4$}  (3,-9.3);
                                        \path   (3,-8.1-2.5)edge[decorate]  node[right] {$-1$}  (3,-9.3-2.5);
                              \path   (0,-10.6)edge[decorate]  node[left] {$-1$}  (0,-11.8);
                                                            \path   (0,-13.1)edge[decorate]  node[left] {$+4$}  (0,-14.3);
   \end{scope}
   \end{tikzpicture} 
  \qquad
    \begin{tikzpicture}[scale=0.6]
        \draw[white] (-1.8,0.35) rectangle (4.4,-16.3);  
            \draw[dashed] (-1.8,0) rectangle (4.4,-10); 
                \draw[dashed] (-1.8,-10)--(-1.8,-15)--(4.4,-15)--(4.4,-10) ;  
                  \fill[white] (0,0) circle (20pt);   
                                  \fill[white] (0,-15) circle (18pt);    \begin{scope}   
 \fill[white] (0,0) circle (37pt);
     \draw (0,0) node {$  \Yboxdim{5pt}\gyoung(;;;;;;;;,;;;;;,;;;) $  };   
     %%%%%
\fill[white] (0,-2.5) circle (37pt);    \draw (0,-2.5) node{$  \Yboxdim{5pt}\gyoung(;;;;;;;,;;;;;,;;;) $  };
%%%%%
\fill[white] (0,-5) circle (37pt);  \draw (0,-5) node{$  \Yboxdim{5pt}\gyoung(;;;;;;;;,;;;;;,;;;) $  };
\fill[white] (3,-5) circle (37pt);  \draw (3,-5) node{$  \Yboxdim{5pt}\gyoung(;;;;;;;,;;;;;,;;;,;) $  };
%%%%%
  \fill[white] (0,-7.5) circle (37pt);  \draw (0,-7.5) node{$  \Yboxdim{5pt}\gyoung(;;;;;;;,;;;;;,;;;) $  };
\fill[white] (3,-7.5) circle (37pt);  \draw (3,-7.5) node{$  \Yboxdim{5pt}\gyoung(;;;;;;,;;;;;,;;;,;) $  };
  \fill[white] (0,-10) circle (37pt);  \draw (0,-10) node{$  \Yboxdim{5pt}\gyoung(;;;;;;;,;;;;;,;;;,;) $};
  %%%
  \fill[white] (0,-12.5) circle (37pt);  \draw (0,-12.5) node{$  \Yboxdim{5pt}\gyoung(;;;;;;,;;;;;,;;;,;) $};
    \fill[white] (0,-15) circle (37pt);  \draw (0,-15) node{$  \Yboxdim{5pt}\gyoung(;;;;;;,;;;;;,;;;,;;) $};
     \path   (0,-0.6)edge[decorate]  node[left] {$-1$}  (0,-1.9);
          \path   (0.6,-3.1)edge[decorate]  node[auto] {$+4$}  (3,-4.2);
                    \path   (0,-3.1)edge[decorate]  node[left] {$+1$}  (0,-4.4);
                    \path   (0,-5.6)edge[decorate]  node[left] {$-1$}  (0,-6.9);
                                        \path   (3,-5.6)edge[decorate]  node[auto] {$-1  $}  (3,-6.8);
          \path   (0,-8.1)edge[decorate]  node[left] {$+4$}  (0,-9.3);
                    \path   (2.3,-8.2)edge[decorate]  node[above] {$+1$}  (0.6,-9.3);
                               \path   (0,-10.6)edge[decorate]  node[left] {$-1$}  (0,-11.8);
    \path   (0,-13.1)edge[decorate]  node[left] {$+4$}  (0,-14.3);
   \end{scope}
   \end{tikzpicture} 
     \qquad
    \begin{tikzpicture}[scale=0.6]
        \draw[white] (-2,0.35) rectangle (2,-16.3);  
            \draw[dashed] (-2,0) rectangle (2,-10); 
                \draw[dashed] (-2,-10)--(-2,-15)--(2,-15)--(2,-10) ;  
                  \fill[white] (0,0) circle (20pt);   
                                  \fill[white] (0,-15) circle (18pt);    \begin{scope}   
 \fill[white] (0,0) circle (37pt);
     \draw (0,0) node {$  \Yboxdim{5pt}\gyoung(;;;;;;;;,;;;;;,;;;) $  };   
     %%%%%
\fill[white] (0,-2.5) circle (37pt);    \draw (0,-2.5) node{$  \Yboxdim{5pt}\gyoung(;;;;;;;,;;;;;,;;;) $  };
%%%%%
\fill[white] (0,-5) circle (37pt);  \draw (0,-5) node
{$  \Yboxdim{5pt}\gyoung(;;;;;;;,;;;;;,;;;,;) $  };
 %%%%%
  \fill[white] (0,-7.5) circle (37pt);  \draw (0,-7.5) node{$  \Yboxdim{5pt}\gyoung(;;;;;;,;;;;;,;;;,;) $  };
  \fill[white] (0,-10) circle (37pt);  \draw (0,-10) node{$  \Yboxdim{5pt}\gyoung(;;;;;;,;;;;;,;;;,;;) $};
  %%%
  \fill[white] (0,-12.5) circle (37pt);  \draw (0,-12.5) node{$  \Yboxdim{5pt}\gyoung(;;;;;,;;;;;,;;;,;;) $};
    \fill[white] (0,-15) circle (37pt);  \draw (0,-15) node{$  \Yboxdim{5pt}\gyoung(;;;;;;,;;;;;,;;;,;;) $};
     \path   (0,-0.6)edge[decorate]  node[left] {$-1$}  (0,-1.9);
                     \path   (0,-3.1)edge[decorate]  node[left] {$+4$}  (0,-4.4);
                    \path   (0,-5.6)edge[decorate]  node[left] {$-1$}  (0,-6.9);
          \path   (0,-8.1)edge[decorate]  node[left] {$+4$}  (0,-9.3);
                                \path   (0,-10.6)edge[decorate]  node[left] {$-1$}  (0,-11.8);
    \path   (0,-13.1)edge[decorate]  node[left] {$+1$}  (0,-14.3);
   \end{scope}
   \end{tikzpicture} 
     \qquad
    \begin{tikzpicture}[scale=0.6]
        \draw[white] (-2,0.35) rectangle (4.5,-16.3);  
            \draw[dashed] (-2,0) rectangle (4.8,-10); 
                \draw[dashed] (-2,-10)--(-2,-15)--(4.8,-15)--(4.8,-10) ;  
                  \fill[white] (-0.1,0) circle (20pt);   
                                  \fill[white] (-0.1,-15) circle (18pt);    \begin{scope}   
 \fill[white] (-0.1,0) circle (37pt);
     \draw (-0.1,0) node {$  \Yboxdim{5pt}\gyoung(;;;;;;;;;,;;;;;;,;;;) $  };   
     %%%%%
\fill[white] (-0.1,-2.5) circle (37pt);    
\draw (-0.1,-2.5) node{$  \Yboxdim{5pt}\gyoung(;;;;;;;;;,;;;;;;,;;) $  };
\fill[white] (3,-2.5) circle (37pt);    
\draw (3,-2.5) node{$  \Yboxdim{5pt}\gyoung(;;;;;;;;;,;;;;;,;;;) $  };
%%%%%
\fill[white] (-0.1,-5) circle (37pt);  \draw (-0.1,-5) node{$  \Yboxdim{5pt}\gyoung(;;;;;;;;;,;;;;;;;,;;)$  };
\fill[white] (3,-5) circle (37pt);  \draw (3,-5) node{$  \Yboxdim{5pt}\gyoung(;;;;;;;;;;,;;;;;,;;;)$  };
          \path   (0.6,-0.6)edge[decorate]  node[auto] {$-2$}  (3,-1.9);
%%%%%     \path   (-0.1,-0.6)edge[decorate]  node[left] {$-3$}  (-0.1,-1.9);
  \fill[white] (-0.1,-7.5) circle (37pt);  \draw (-0.1,-7.5) node{$  \Yboxdim{5pt}\gyoung(;;;;;;;;;,;;;;;;,;;)$  };
\fill[white] (3,-7.5) circle (37pt);  \draw (3,-7.5) node{$  \Yboxdim{5pt}\gyoung(;;;;;;;;;;,;;;;;,;;) $  };
  \fill[white] (0.1,-10) circle (37pt);  \draw (0.1,-10) node{$  \Yboxdim{5pt}\gyoung(;;;;;;;;;;,;;;;;;,;;) $};
  %%%
  \fill[white] (-0.1,-12.5) circle (37pt);  \draw (-0.1,-12.5) node{$  \Yboxdim{5pt}\gyoung(;;;;;;;;;,;;;;;;,;;)  $};
    \fill[white] (-0.1,-15) circle (37pt);  \draw (-0.1,-15) node{$  \Yboxdim{5pt}\gyoung(;;;;;;;;;,;;;;;;,;;;)  $};
     \path   (-0.1,-0.6)edge[decorate]  node[left] {$-3$}  (-0.1,-1.9);
          \path   (3,-3.1)edge[decorate]  node[auto] {$+1$}  (3,-4.2);
                    \path   (-0.1,-3.1)edge[decorate]  node[left] {$+2$}  (-0.1,-4.4);
                    \path   (-0.1,-5.6)edge[decorate]  node[left] {$-2$}  (-0.1,-6.9);
                                        \path   (3,-5.6)edge[decorate]  node[auto] {$-3  $}  (3,-6.8);
          \path   (-0.1,-8.1)edge[decorate]  node[left] {$+1$}  (-0.1,-9.3);
                    \path   (2.3,-8.2)edge[decorate]  node[below] {$+2$}  (0.6,-9.3);
                               \path   (-0.1,-10.6)edge[decorate]  node[left] {$-1$}  (-0.1,-11.8);
    \path   (-0.1,-13.1)edge[decorate]  node[left] {$+3$}  (-0.1,-14.3);
   \end{scope}
   \end{tikzpicture} $$
   \caption{Three  semistandard Kronecker tableaux of shape
   $(6,5,3,2) \setminus (8,5,3)  $ and one of shape 
   $(9,6,3)\setminus(9,6,3)$.  The leftmost 
 is   of weight $( 3)$ and %  \cref{semiexam}.
   the latter three are of weight $( 2,1)$. 
 }   \label{semiexamfig}
\end{figure}
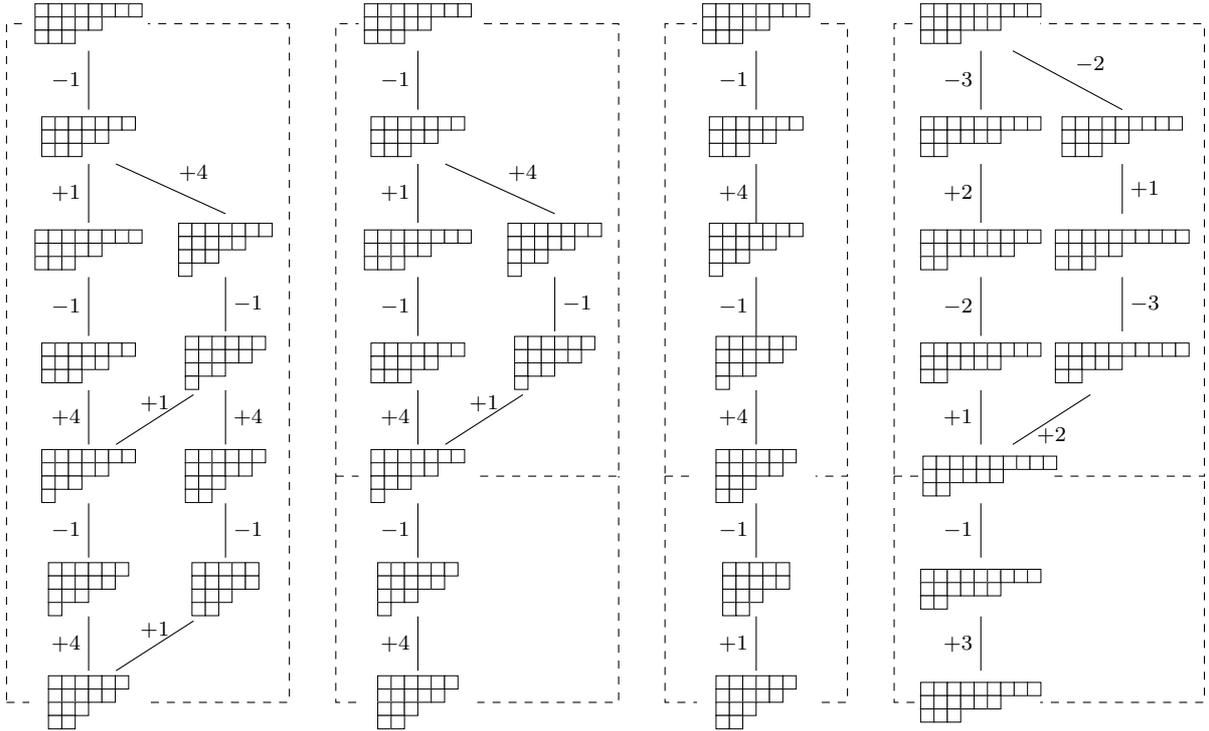
 \end{eg}

  \section{Latticed Kronecker tableaux}\label{sec:latticed}

In this section we prove the main result of the paper, namely we find a combinatorial description for 
$$\overline{g}(\lambda, \nu, \mu) = \dim \Hom_{\mathfrak{S}_s}({\sf S}(\mu), \Delta_s^0(\nu\setminus \lambda))$$
for all co-Pieri triples $(\lambda, \nu, s)$ and all $\mu \vdash s$ which naturally extends the Littlewood--Richardson rule.

In the previous section we saw that the semistandard Kronecker tableaux of shape $\nu \setminus \lambda$ and weight $\mu$ index a basis for $\Hom_{\mathfrak{S}_s}({\sf M}(\mu), \Delta_s^0(\nu \setminus \lambda))$. We will now find which of these index a basis for \newline $\Hom_{\mathfrak{S}_s}({\sf S}(\mu), \Delta_s^0(\nu \setminus \lambda))$.
We follow James' approach \cite{JAM} and extend his notion of latticed semistandard tableaux.

We start with any standard tableau $\sts \in \Std_s^0(\nu \setminus \lambda)$ and any $\mu = (\mu_1, \mu_2, \ldots , \mu_l)\vDash s$. Write 
$$\sts =  (-\varepsilon_{i_1}, 
+\varepsilon_{j_1},
-\varepsilon_{i_2},
+\varepsilon_{j_2},
\dots
, -\varepsilon_{i_s},
+\varepsilon_{j_s}). $$
Recall from the previous section that, to each integral step $(-\varepsilon_{i_k}, + \varepsilon_{j_k})$ in $\sts$, we associate its frame  $c$, that is the unique positive integer such that 
$$[\mu]_{c-1} < k \leq [\mu]_c.$$
 
  Now we encode the integral steps of $\sts$ and their frames in a $2\times s$ array, denoted by $\omega_\mu (\sts)$ and called the $\mu$-reverse reading word of $\sts$ as follows. The first row of $\omega_\mu(\sts)$ contains all the integral steps of $\sts$ and the second row contains their corresponding frames. We order the columns of $\omega_\mu(\sts)$ increasingly using the ordering on integral steps given in Definition 2.5 (and we place a vertical lines between any two integral steps which are not equal). For two equal integral steps we order the columns so that the frame numbers are weakly decreasing (and so between any two vertical lines,    the entries in  $\omega_2(\SSTS)$  are weakly decreasing).   

Note that if $\stt \in [\sts]_\mu$ then $\omega_\mu(\stt) = \omega_\mu(\sts)$. So it makes sense to define the reverse reading word $\omega(\SSTS)$ of a semistandard Kronecker tableau $\SSTS \in \SStd_s^0(\nu\setminus \lambda , \mu)$ by setting $\omega(\SSTS)=\omega_\mu(\sts)$ for some $\sts\in \SSTS$.

\begin{eg}\label{reverse}
We begin with an example of a 
 triple of maximal depth. Let $ \nu  =(9,8,6,3) $, $\lambda=  ( 6,4,3)$ and $s=13$.  
 Let  $\sts \in \Std_s^0( \nu \setminus \lambda )$ be the path       
$$
a(1)\circ a(1) \circ a(4)\circ a(4)\circ a(4)
\circ
 a(1)\circ a(2)\circ a(2)\circ a(2)\circ a(3)
\circ
 a(2)\circ a(3)\circ a(3).
$$ 
Let $\mu = (5,5,3)$, then in classical notation, the semistandard tableau $\SSTS = [\sts]_\mu$ is the leftmost tableaux depicted in \cref{hello mister3}.
The reverse reading word of $\SSTS$ is as follows:
$$ \left( \begin{array}{ccc|cccc|ccc|ccc}
a(1) & a(1) & a(1) & a(2) & a(2) & a(2) & a(2) 
& a(3) & a(3) & a(3) 
& a(4) & a(4) & a(4) \\
2& 1 & 1 & 3 & 2 & 2 & 2 & 3&3&2&1&1&1 \end{array}\right) .$$
% \omega(\SSTS) =  \left(
%   \begin{tikzpicture}[baseline={([yshift=-.7ex]current bounding box.center)},scale=0.55]   \begin{scope}
%        \draw (-1.5,1) node  {$a(1)$};      
%        \draw (-1.5,0) node  {$\tiny {2}$};      
%         \draw (0,1) node  {$a(1)$};      
%        \draw (0,0) node  {$\tiny {1}$};      
%            \draw (1.5,1) node  {$a(1)$};      
%        \draw (1.5,0) node  {$\tiny {1}$};      
%                 \draw (3,1) node  {$a(2)$};      
%        \draw (3,0) node  {$\tiny {3}$};      
%            \draw (4.5,1) node  {$a(2)$};      
%        \draw (4.5,0) node  {$\tiny {2}$};      
%            \draw (6,1) node  {$a(2)$};      
%        \draw (6,0) node  {$\tiny {2}$};      
%            \draw (7.5,1) node  {$a(2)$};      
%        \draw (7.5,0) node  {$\tiny {2}$};      
%            \draw (9,1) node  {$a(3)$};      
%        \draw (9,0) node  {$\tiny {3}$};      
%          \draw (10.5,1) node  {$a(3)$};      
%        \draw (10.5,0) node  {$\tiny {3}$};      
%        \draw (12,1) node  {$a(3)$};      
%        \draw (12,0) node  {$\tiny {2}$};      
%                 \draw (13.5,1) node  {$a(4)$};      
%        \draw (13.5,0) node  {$\tiny {1}$};      
%                          \draw (15,1) node  {$a(4)$};      
%        \draw (15,0) node  {$\tiny {1}$};      
%                    \draw (16.5,1) node  {$a(4)$};      
%        \draw (16.5,0) node  {$\tiny {1}$};       \end{scope} 
%\end{tikzpicture}\right).  
 Compare the second row of the above array 
 with the corresponding word given in \cref{readerexample,followreaderexample}.
\end{eg}

\begin{rmk}Let  $(\lambda,\nu,\mu)$  be  of  
maximal depth and $\SSTS\in \SStd_s^0(\nu\setminus\lambda,\mu)$.  
 The second row of the reverse reading word of  $\SSTS$    coincides with the  classical   reverse reading word given in  \cref{Classical row reading}.  
 \end{rmk}

For $\SSTS\in \SStd_s^0(\nu\setminus \lambda, \mu)$ we write 
$$\omega(\SSTS) = (\omega_1(\SSTS), \omega_2(\SSTS))$$
where $\omega_1(\SSTS)$ (respectively  $\omega_2(\SSTS)$) is the first (respectively  second) row of $\omega(\SSTS)$.
Note that $\omega_2(\SSTS)$ is a sequence of type $\mu$, that is a sequence of positive integers such that $i$ appears precisely $\mu_i$ times, for all $i\geq 1$.

\begin{defn}\label{goodbad}
Given a finite sequence of positive integers we define the quality (good/bad) of each term as follows.
\begin{enumerate}
\item All $1$'s are good.
\item An $i+1$ is good if and only if the number of previous good $i$'s is strictly greater than the number of previous good $i+1$'s.
\end{enumerate}
A sequence of positive integers is called a lattice permutation if every term in the sequence is good.
\end{defn}

\begin{defn}
For $\SSTS\in \SStd_s^0(\nu\setminus \lambda, \mu)$ we say that its reverse reading word $\omega(\SSTS)$ is a lattice permutation if $\omega_2(\SSTS)$ is a lattice permutation. We define $\Latt_s^0(\nu \setminus \lambda, \mu)$ to be the set of all $\SSTS\in \SStd_s^0(\nu\setminus \lambda, \mu)$ such that $\omega(\SSTS)$ is a lattice permutation.
\end{defn}

 \begin{eg}
Continuing from \cref{reverse}, the quality of each term (or step) in the reverse reading word of  $\SSTS$ is as follows
$$ \left( \begin{array}{ccc|cccc|ccc|ccc}
a(1) & a(1) & a(1) & a(2) & a(2) & a(2) & a(2) 
& a(3) & a(3) & a(3) 
& a(4) & a(4) & a(4) \\
\times 
&
\checkmark
&
\checkmark
&
\times 
&
\checkmark
&
\checkmark
&
\times 
&
\checkmark
&
\checkmark
&
\times 
&\checkmark
&
\checkmark&
\checkmark
\\
2& 1 & 1 & 3 & 2 & 2 & 2 & 3&3&2&1&1&1 \end{array}\right) .$$
We have indicated good steps with a $\checkmark$ and each bad step with a $\times$. We see that $\SSTS \not \in \Latt_s^0(
 (9,8,6,3) \setminus  ( 6,4,3), ( 5,5,3) 
)$.   
\end{eg}

\begin{eg}
Of the three semistandard Kronecker tableaux depicted in \cref{smiinpartition}, the reverse reading words of final two are lattice permutations, whereas the first one is not.  
\end{eg}

\begin{eg}\label{semiexamfig2}
Of the two  elements of  $\SStd_3((6,5,3,2) \setminus (8,5,3),(2,1))$ depicted in \cref{semiexamfig}, 
the reverse reading word of the former is a lattice permutation, whereas the latter is not.

\end{eg}

  \begin{eg}\label{a2parteg}
 We continue with \cref{onerowss}. So we take $\lambda = (7), \nu = (6)$ and $s=6$. Let $\SSTS \in \SStd_6^0(\nu\setminus \lambda, \mu)$ for any $\mu \vdash 6$. Then $\omega_1(\SSTS)$ must be one of the following
$$\begin{array}{cccc}
( r(1) \, r(1) \, r(1)\, d(1) \, a(1) \, a(1)), & 
( r(1) \, r(1) \, d(1)\, d(1) \, d(1) \, a(1)) & \mbox{or} &
( r(1) \, d(1) \, d(1)\, d(1) \, d(1) \, d(1)).
\end{array}$$
It is easy to check that for $\mu = (3,2,1)$ we have $\SSTS\in \Latt_6^0(\nu\setminus \lambda , \mu)$ if and only if $\omega(\SSTS)$ is one of the following
$$\begin{array}{ccc}
 \left( \begin{array}{ccc|c|cc}
r(1) & r(1) & r(1) & d(1) & a(1) & a(1) \\
1 & 1 & 1 & 2 & 3 & 2 \end{array}\right) &  \mbox{or} & 
 \left( \begin{array}{cc|ccc|c}
r(1) & r(1) & d(1) & d(1) & d(1) & a(1) \\
1 & 1 & 2 & 2 & 1 & 3 \end{array}\right).
\end{array}$$
Thus $|\Latt_6^0(\nu\setminus \lambda, \mu )|=2$. 
Similarly, for $\tau = (4,2)$ we have that $\SSTS \in \Latt_6^0(\nu\setminus \lambda , \tau)$   if and only if $\omega(\SSTS)$ is one of the following
$$\begin{array}{ccc}
 \left( \begin{array}{ccc|c|cc}
r(1) & r(1) & r(1) & d(1) & a(1) & a(1) \\
1 & 1 & 1 & 1 & 2 & 2 \end{array}\right) &  \mbox{,} & 
 \left( \begin{array}{ccc|c|cc}
r(1) & r(1) & r(1) & d(1) & a(1) & a(1) \\
1 & 1 & 1 & 2 & 2 & 1 \end{array}\right), \\
\left( \begin{array}{cc|ccc|c}
r(1) & r(1) & d(1) & d(1) & d(1) & a(1) \\
1 & 1 & 2 & 1 & 1 & 2 \end{array}\right) &  \mbox{or} & 
 \left( \begin{array}{cc|ccc|c}
r(1) & r(1) & d(1) & d(1) & d(1) & a(1) \\
1 & 1 & 2 & 2 & 1 & 1 \end{array}\right).
\end{array}$$
So we get $|\Latt_6^0(\nu\setminus \lambda, \tau)|=4$.

  \end{eg}

\begin{thm}\label{mainresult}
For any co-Pieri triple $(\lambda, \nu, s)$ and any $\mu\vdash s$ we have 
  that 
%$$\{\varphi_\SSTT \mid \SSTT\in  \Latt_s^0(\nu \setminus \lambda, \mu)\}$$
%is a $\ZZ$-basis for 
%$   \Hom_{\mathfrak{S}_s}({\sf S}(\mu), \Delta_s^0(\nu\setminus \lambda))$.    
%In particular, 
  $$\overline{g}(\lambda, \nu, \mu) = \dim_{\CC}\Hom_{\mathfrak{S}_s}({\sf S}(\mu), \Delta_s^0(\nu\setminus \lambda)) = |\Latt_s^0(\nu\setminus \lambda, \mu)|.$$
\end{thm}

In the rest of this section we will prove this result. The main technique we will  use is James' {\sf pairs of partitions} method which describes how to `turn bad steps into good ones'.  

\begin{defn}
Let $ \mu \vDash s$  and let   $ \mu^\sharp \in \mathscr{P}_{\leq s}  $  
 be such that 
 $  \mu ^{\sharp}_{c}\leq \mu _c,$ 
for all $c\geq 1$.  Then $( \mu ^\sharp, \mu )$ is called a {\sf pair of partitions} for $s$.  
\end{defn}

 We record a pair of partitions diagrammatically by drawing the Young diagram for $ \mu$ and filling all boxes corresponding to $ \mu ^\sharp$ with a  $\times$, for example  we have that $( \mu ^\sharp, \mu )=((2^2,1),(2^4))$    is represented 
 as in the leftmost  diagram in  \cref{updownaroundmoves}.

\begin{defn}
Let $(\mu^{\sharp},\mu)$ be a pair of partitions of $s$. We denote by $s(\mu)$ the set of all sequences of type $\mu$ and by $s(\mu^{\sharp}, \mu)\subseteq s(\mu)$ the set of all sequences of type $\mu$ having at least $\mu_i^{\sharp}$ good $i$'s for all $i$.
\end{defn}

By definition we have  
$$s(\emptyset , \mu) = s((\mu_1), \mu) = s(\mu)$$
and if $\tau^{\sharp} \subseteq \mu^{\sharp}$ then
$$s(\mu^{\sharp}, \mu) \subseteq s(\tau^{\sharp}, \mu).$$

   \begin{defn} \label{movies}
Let $( \mu , \mu ^\sharp)$ be a pair of partitions for $s$ and let  
 $1< c \leq  \ell( \mu )$ be the smallest integer such that $\mu_c^{\sharp} <\mu_c$.
 \begin{itemize}[leftmargin=*]
 \item We let $r^{\mu_c-\mu^{\sharp}_c}_{c}(\mu )$ denote the composition of $s $ obtained by removing the $\mu_c - \mu^{\sharp}_c$ boxes at the end of row $c$ 
  and  adding  them at the end of row $c-1$. 
  \item
We let  $a_{c}(\mu ^\sharp)$ denote the partition obtained by adding  a single box to the end of $\mu ^\sharp_{c}$ if the result is a partition.  If the result is not a partition, then set $(\mu ,a_{c}(\mu ^\sharp))=(\varnothing,\varnothing)$. 
 \end{itemize}
  \end{defn}

\begin{eg}  For example, let  $(\mu^\sharp,\mu)=(( 2^2,1),(2^4))$.  Some of 
the pairs of partitions obtained by applying the  moves in \cref{movies}
to $(\mu^\sharp,\mu)$  are depicted in \cref{updownaroundmoves}.
\begin{figure}[ht!]
$$
(\mu  ,\mu  ^\sharp)={   \Yboxdim{12pt}\scalefont{0.9} \Yvcentermath1 \young(\times\times,\times\times,\times\ ,\ \ ) } \quad 
(r^1_{ 3}(\mu  ),\mu  ^\sharp)= {   \Yboxdim{12pt}\scalefont{0.9} \Yvcentermath1\young(\times\times,\times\times\ ,\times,\ \ ) }
\quad 
(\mu  ,a_{3}( \mu  ^\sharp))= {   \Yboxdim{12pt}\scalefont{0.9} \Yvcentermath1\young(\times\times,\times\times ,\times\times,\ \ ) }
$$

\!\!\!\!\!\caption{Examples for \cref{movies}}
\label{updownaroundmoves}
\end{figure}
\end{eg}

\begin{defn}
Fix $\mu\vdash s$ and define a rooted tree $\mathscr{T}(\mu) = (V(\mathscr{T}(\mu)), E)$ with vertices labelled by pairs of partitions.  Its root vertex is   labelled by  $ ((\mu_1),\mu) $ and given
 a vertex, $(\tau^{\sharp}, \tau) \in V(\mathscr{T}(\mu))$  its descendants   
 are labelled by the pairs of partitions   
$$ (\tau^\sharp, r_c^{\tau_c-\tau_c^\sharp}(\tau)) \quad \text{  and  } \quad  (a_c(\tau^\sharp), \tau) $$ 
where   $1< c \leq \ell(\mu)$ is  minimal  such that  and $\tau_{c}^\sharp < \tau_c$.
If there is no such $1< c \leq \ell(\mu) $, then $\tau^{\sharp}=\tau$ and $(\tau, \tau)$ is a terminal vertex. 
(Note that we identify the labels $(\tau^\sharp, \tau)$ and $(\eta^\sharp, \tau)$ if $\tau^\sharp$ and $\eta^\sharp$ only differ in the first row and usually choose to write the label $(\tau^\sharp , \tau)$ with $\tau^\sharp _1 = \tau_1$.)

   \end{defn}
   
We decorate the edges of the tree with the appropriate operators, $r_c^k$ and $a_c$,  for $k\geq 1$ and $c\geq 2$.  
We let $V_T(\mathscr{T}(\mu))$ denote the set of terminal vertices in $V(\mathscr{T}(\mu))$   
 which are not labeled by pairs  of the form $  (\varnothing, \varnothing)$. 
Given $\mathfrak{t} \in V_T(\mathscr{T}(\mu))$, we associate 
 the ordered sequence of  operators, $r_{\mathfrak{t}}$,  labelling   the edges  path from the root vertex to   the vertex  $\mathfrak{t}$.  
 A pair of partitions $(\tau,\tau)$ will not (in general) label a unique terminal vertex (see for example,   \cref{procedure}).

  \begin{eg}
The tree $\mathscr{T}(\mu)$ for $\mu= ( 3,2,1 )$ is given  in \cref{procedure}.
There are $8$ vertices in $V_T(\mathscr{T}(\mu))$. The rightmost terminal vertex is labelled by $((6),(6))$ and it correspondes to the path $r_2r_3r_2^2 = r_2^1 \circ r_3^1 \circ r_2^2$. Note that we write the composition of operators from right to left and write $r_j$ for $r_j^1$ to simplify the notation.
The sequences of operators labelling terminal vertices are as follows, 
 $$ 
a_3a_2 a_2\quad  
a_2 r_3a_2 a_2 \quad  
  r_2  r_3a_2 a_2\quad  
  a_3 r_2   a_2\quad 
a_2   r_3r_2 a_2 \quad 
r_2  r_3r_2 a_2\quad 
   a_2 r_3r^2_2\quad 
    r_2r_3r^2_2 
 $$
where each of  the paths   above  can be identified from left to right with the terminal nodes in the graph in \cref{procedure}.   

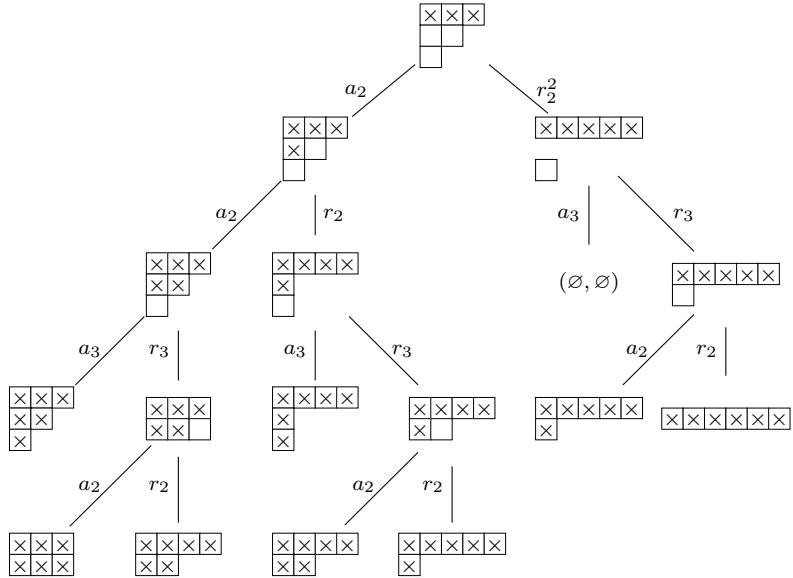
\begin{figure}[ht!]
 $$ \scalefont{0.8} 
 \begin{tikzpicture}[scale=0.6]
 \draw(3,4.5)edge[decorate]  node[left] {$a_2 \ $}(0,2);
 \draw(3,4.5)edge[decorate]  node[right] {$\ r^2_2  $}(6,2);
 \draw(0,2)edge[decorate]  node[left] {$ a_2  $}(-3,-1);
 \draw(0,2)edge[decorate]  node[right] {$ r_2  $}(-0,-1);
  \draw(6,2)edge[decorate]  node[right] {$\ r_3 $}(9,-1) ;
    \draw(6,2)edge[decorate]  node[left] {$a_3  $}(6,-1) ;
\fill[white] (3,4.5) circle (30pt);    
\fill[white] (0,2) circle (30pt);      
    \draw(9,-1)edge[decorate]  node[left] {$a_2  $}(6,-4) ;   
    \draw(9,-1)edge[decorate]  node[left] {$r_2 $}(9,-4) ;   

\fill[white] (9,-1) circle (28pt);      
\fill[white] (6,2.2) circle (30pt); 
    \draw(-3,-1)edge[decorate]  node[left] {$a_3  $}(-6,-4) ;   
    \draw(-3,-1)edge[decorate]  node[left] {$r_3  $}(-3,-4) ;   
    \draw(0,-1)edge[decorate]  node[left] {$a_3  $}(0,-4) ;   
    \draw(0,-1)edge[decorate]  node[right] {$r_3  $}(3,-4) ;

    \draw(-3,-4)edge[decorate]  node[left] {$a_2  $}(-6,-7) ;   
    \draw(-3,-4)edge[decorate]  node[left] {$r_2  $}(-3,-7) ;   

\fill[white] (-6,-7) circle (25pt);      
\fill[white] (-3,-7) circle (20pt);

  \draw(3,-4)edge[decorate]  node[left] {$a_2  $}(0,-7) ;   
    \draw(3,-4)edge[decorate]  node[left] {$r_2  $}(3,-7) ;   

\fill[white] (3,-7) circle (20pt);      
\fill[white] (0,-7) circle (26pt);     
  
 \fill[white] (-3,-1) circle (30pt);    
  \fill[white] (0,-1) circle (30pt);    
\fill[white] (3,-1) circle (30pt);   
\fill[white] (6,-1) circle (24pt);    
      \fill[white] (-6,-4) circle (30pt);    
            \fill[white] (-3,-4) circle (24pt);   
        \fill[white] (0,-4) circle (24pt);    
         \fill[white] (3,-4) circle (30pt);    
         \fill[white] (6,-4) circle (30pt);   
                  \fill[white] (9,-4) circle (27pt);   
           %%%%%
           %%%%%
      \draw (-6,-7) node {$  \Yboxdim{8pt}\gyoung(;\times;\times;\times,;\times;\times;\times) $  };   
    %                                   \fill[white] (0,-15) circle (18pt);                                    
%      \draw[dashed] (-2,0) rectangle (2,-15);  
%                \draw[dashed] (-2,-5) -- (2,-5);  
%                \draw[dashed] (-2,-10) -- (2,-10);  
%                 \fill[white] (0,0) circle (20pt);   
\draw (3,4.5) node {$  \Yboxdim{8pt}\gyoung(;\times;\times;\times,;;,;) $  };   
%                                  \fill[white] (0,-5) circle (17pt);   
\draw (0,2) node {$  \Yboxdim{8pt}\gyoung(;\times;\times;\times,;\times;,;) $  };   
\draw (6,2) node {$  \Yboxdim{8pt}\gyoung(;\times;\times;\times;\times;\times,,;) $  };   
%                                                                    \fill[white] (0,-10) circle (22pt);   
 \draw (-3,-1) node {$ \Yboxdim{8pt}\gyoung(;\times;\times;\times,;\times;\times,;)  $  };   
  \draw (0,-1) node {$  \Yboxdim{8pt}\gyoung(;\times;\times;\times;\times,;\times,;) $  };   
\draw (6,-1) node {$ (\varnothing,\varnothing) $  };   
\draw (9,-1) node {$  \Yboxdim{8pt}\gyoung(;\times;\times;\times;\times;\times,;) $  };   
% \fill[white] (0,-15) circle (24pt);   
      \draw (-6,-4) node {$  \Yboxdim{8pt}\gyoung(;\times;\times;\times,;\times;\times,;\times) $  };   
            \draw (-3,-4) node {$  \Yboxdim{8pt}\gyoung(;\times;\times;\times,;\times;\times;) $  };   
        \draw (0,-4) node {$  \Yboxdim{8pt}\gyoung(;\times;\times;\times;\times,;\times,;\times) $  };   
         \draw (3,-4) node {$  \Yboxdim{8pt}\gyoung(;\times;\times;\times;\times,;\times;) $  };   
         \draw (6,-4) node {$  \Yboxdim{8pt}\gyoung(;\times;\times;\times;\times;\times,;\times) $  };   
                  \draw (9,-4) node {$  \Yboxdim{8pt}\gyoung(;\times;\times;\times;\times;\times;\times) $  }; 
           %%%%%
           %%%%%
                   \draw (-3,-7) node {$  \Yboxdim{8pt}\gyoung(;\times;\times;\times;\times,;\times;\times) $  };   
    \draw (3,-7) node {$  \Yboxdim{8pt}\gyoung(;\times;\times;\times;\times;\times,;\times) $  };          
   \draw (0,-7) node {$  \Yboxdim{8pt}\gyoung(;\times;\times;\times;\times,;\times;\times) $  };   
               \end{tikzpicture}
$$
  \caption{
 The tree, $\mathscr{T}(\mu)$,  for $ \mu= ( 3,2,1)$. 
 }
\label{procedure} \end{figure}

\end{eg}

James proved the following result, see \cite{JAMBOOK}(15.14 Theorem).

\begin{thm}\label{James}
Let $(\mu^{\sharp}, \mu)$ be a pair of partitions of $s$ and let $c>1$ be minimal such that $\mu^{\sharp}_c < \mu_c$. There is a bijection 
$$R_c\, : \, s(\mu^{\sharp}, \mu) \setminus s((a_c(\mu^{\sharp}),\mu))\rightarrow s(\mu^\sharp , r_c^{\mu_c-\mu_c^\sharp}(\mu))$$
defined by changing all bad $c$'s into $c-1$'s.
\end{thm}

The next Lemma shows that we can extend this bijection to sets of semistandard Kronecker tableaux for co-Pieri triples. 
The corresponding result for triples of maximal depth is given in \cite{JAMBOOK}(16.3 Lemma). 

Define $\SStd_s^0(\nu\setminus \lambda ,(\mu^{\sharp} ,\mu))\subseteq \SStd_s^0(\nu\setminus \lambda, \mu)$ to be the subset of all semistandard Kronecker tableaux $\SSTS$ whose reverse reading word satisfies $\omega_2(\SSTS)\in s(\mu^\sharp , \mu)$. 

\begin{lem}\label{bijectionSS} Let $(\lambda, \nu, s)$ be a co-Pieri triple and let $(\mu^{\sharp}, \mu)$ be a pair of partitions of $s$.
Take $c>1$ to be minimal such that $\mu^\sharp _c < \mu_c$. The map 
$$\mathscr{R}_c \, : \, \SStd_s^0(\nu\setminus \lambda, (\mu^\sharp, \mu))\setminus \SStd_s^0(\nu \setminus \lambda , (a_c(\mu^\sharp), \mu)) \rightarrow \SStd_s^0(\nu\setminus \lambda , (\mu^\sharp , r_c^{\mu_c - \mu_c^\sharp}(\mu))$$
defined by taking
$$
 \omega_1(\mathscr{R}_c(\SSTS)) = \omega_1(\SSTS) \quad \mbox{and} \quad \omega_2(\mathscr{R}_c(\SSTS)) = R_c(\omega_2(\SSTS))$$
for all $\SSTS\in \SStd_s^0(\nu\setminus \lambda, (\mu^\sharp, \mu))\setminus \SStd_s^0(\nu \setminus \lambda , (a_c(\mu^\sharp), \mu))$ (where the map $R_c$ is given in Theorem \ref{James}) is a bijection.

\end{lem} 

\begin{proof}
Note that each semistandard Kronecker tableau $\SSTS$ is completely determined by the multisets $X_i(\SSTS)$ containing the integral steps in frame $i$ for each $i$. Hence, the reverse reading word $\omega(\SSTS)$ completely determines $\SSTS$. Moreover, as $\stt_{k\leftrightarrow k+1}\in \Std_s^0(\nu \setminus \lambda)$ for all $\stt \in \Std_s^0(\nu \setminus \lambda)$ and all $k$, if we move some integral steps from one frame of $\SSTS$ to another the result will still be a semistandard tableau of the same shape and the appropriate weight. 
So, using Theorem \ref{James}, the only thing we need to prove the bijection is that $(\omega_1(\SSTS), R_c(\omega_2(\SSTS))$ is the reverse reading word of a semistandard tableau if and only if so is $(\omega_1(\SSTS), \omega_2(\SSTS))$.

Write $\omega_1(\SSTS) = (x_1, x_2, \ldots , x_s)$ where the $x_i$'s are integral steps, $\omega_2(\SSTS)=(u_1, u_2, \ldots , u_s)$ and $R_c(\omega_2(\SSTS))=(v_1, v_2, \ldots , v_s)$. We need to show that for $x_j = x_{j+1}$ we have $u_j \geq u_{j+1}$ if and only if $v_j \geq v_{j+1}$. Assume first that $u_j \geq u_{j+1}$ and $v_j < v_{j+1}$. By definition of the map $R_c$ we must have $u_j = u_{j+1}=c$, $v_j = c-1$ and $v_{j+1} = c$. This means that $u_j$ is a bad $c$ and $u_{j+1}$ is a good $c$ but this is impossible by \cref{goodbad}.

Conversely, assume that $v_j \geq v_{j+1}$ and $u_j < u_{j+1}$. By definition of $R_c$ we must have $u_j = c-1$, $u_{j+1} = c$ and $v_j = v_{j-1} = c-1$. This means that  $u_{j+1}$ is a bad $c$ but it is preceeded by $u_j = c-1$ so $u_{j+1}$ has to be a good $c$ by \cref{goodbad}. So again this case cannot occur. 
\end{proof}

Starting at the root vertex of $\mathscr{T}(\mu)$ and working our way down the edges, \cref{bijectionSS} allows us to partition the set $\SStd_s^0(\nu\setminus \lambda , \mu)$ into subsets corresponding to $\Latt_s^0(\nu\setminus \lambda , \tau)$ for each terminal vertex labelled by $(\tau,\tau)$ for $\tau \vdash s$.
The next lemma describes the terminal vertices of the $\mathscr{T}(\mu)$. 

\begin{lem}\label{bijectionTV}
Let $\mu, \tau \vdash s$. There is a bijective correspondence between the set of  terminal vertices in $\mathscr{T}(\mu)$ labelled by $(\tau, \tau)$ and the set $\SStd_s(\tau, \mu)$ of semistandard Young tableaux of shape $\tau$ and weight $\mu$.
\end{lem}

\begin{proof}
For this proof, it is easier to view the set $\SStd_s(\tau, \mu)$ in the classical way, as Young diagrams of shape $\tau$ with boxes filled with $\mu_1$ $1$'s, $\mu_2$ $2$'s , $\ldots$.
For each edge $a_c$ and $r_c^{\tau_c - \tau_c^\sharp}$ in the tree $\mathscr{T}(\mu)$, we define corresponding maps
\begin{eqnarray*}
a_c &:& \SStd_s(\tau, \mu) \rightarrow \SStd_s(\tau ,\mu) \, : \, \SSTT \mapsto a_c(\SSTT)=\SSTT, \\
r_c^{\tau_c-\tau_c^\sharp} &:& \SStd_s(\tau, \mu) \rightarrow \SStd_s(r_c^{\tau_c - \tau_c^\sharp}(\tau), \mu) \, : \, \SSTT \mapsto r_c^{\tau_c - \tau_c^\sharp}(\SSTT) 
\end{eqnarray*}
where $r_c^{\tau_c-\tau_c^\sharp}(\SSTT)$ is obtained from $\SSTT$ by moving the last $\tau_c-\tau_c^\sharp$ boxes at the end of row $c$ to the end of row $c-1$ together with their content.

Now each terminal vertex in $\mathscr{T}(\mu)$ correspond to a unique path $\mathfrak{t}$ starting at the root vertex and ending at a vertex labelled by $(\tau, \tau)$ for some $\tau \vdash s$. Let $\SSTT^\mu$ be the unique element in $\SStd_s(\mu,\mu)$ and denote by $r_{\mathfrak{t}}(\SSTT^\mu)$ the  tableau obtained by applying the operators along the edges of $\mathfrak{t}$ to $\SSTT^\mu$. We claim that the map $\mathfrak{t}\mapsto r_\mathfrak{t}(\SSTT^\mu)$ for each terminal vertex labelled by $(\tau, \tau)$ gives a bijection between these terminal vertices and $\SStd_s(\tau, \mu)$. 

As the operator $r_{\mathfrak{t}}$  moves up the boxes of content $2$ first, then the boxes of content $3$, then $4$, and so on, it is clear that the result will be a semistandard tableau of shape $\tau$ and weight $\mu$, and moreover, different paths will lead to different semistandard tableaux.

It remains to show that this map is surjective. First note that if $\SStd_s(\tau, \mu)\neq \emptyset$  then $\tau \trianglerighteq \mu$. Now let $\SSTT\in \SStd_s(\tau, \mu)$ for some $\tau \vartriangleright \mu$. Assume that $\SSTT$ has precisely $k_d^c$ boxes of content $c$ in row $d$. (Note that of $k_d^c\neq 0$ then $d\leq c$.) For each $2\leq c\leq \ell (\mu)$ define
$$r^{(c)} = r_2^{k_1^c} \circ \ldots \circ (a_{c-2})^{k_{c-2}^c} \circ r_{c-1}^{\sum_{d=1}^{c-2} k_d^c} \circ (a_c)^{k_{c-1}^c} \circ r_c^{\sum_{d=1}^{c-1}k_d^c} \circ (a_c)^{k_c^c}.$$
By construction, we have $r^{(\ell (\mu))} \ldots r^{(3)} r^{(2)} (\SSTT^{\mu}) = \SSTT$ and $r^{(\ell (\mu))} \ldots r^{(3)} r^{(2)}$ is a path in $\mathscr{T}(\mu)$ starting at the root vertex and ending at a vertex labelled with $(\tau, \tau)$. Thus the map is surjective as required.
\end{proof}

\begin{eg}\label{a2partegcontuinued further} 
Given   $ \mu = ( 3,2,1)$, we have that the sequences 
 \begin{align}\label{procsdfjhkfsdajkhsafd}
 a_3a_2 a_2\quad  
a_2 r_3a_2 a_2 \quad  
  r_2  r_3a_2 a_2\quad  
  a_3 r_2   a_2\quad 
a_2   r_3r_2 a_2 \quad 
r_2  r_3r_2 a_2\quad 
   a_2 r_3r^2_2\quad 
    r_2r_3r^2_2 
\end{align}
label the terminal vertices in $\mathscr{T}(\mu)$.  Applying these operators to $\SSTT^{\mu}$ we obtain all semistandard Young tableaux of weight $\mu$. 
This procedure is illustrated in \cref{helothe!}.

\begin{figure}[ht!]
 $$ \scalefont{0.8} 
 \begin{tikzpicture}[scale=0.6]
 \draw(3,4.5)edge[decorate]  node[left] {$a_2 \ $}(0,2);
 \draw(3,4.5)edge[decorate]  node[right] {$\ r^2_2  $}(6,2);
 \draw(0,2)edge[decorate]  node[left] {$ a_2  $}(-3,-1);
 \draw(0,2)edge[decorate]  node[right] {$ r^1_2  $}(-0,-1);
  \draw(6,2)edge[decorate]  node[right] {$\ r_3^1 $}(9,-1) ;
    \draw(6,2)edge[decorate]  node[left] {$a_3  $}(6,-1) ;
\fill[white] (3,4.5) circle (30pt);    
\fill[white] (0,2) circle (30pt);      
    \draw(9,-1)edge[decorate]  node[left] {$a_2  $}(6,-4) ;   
    \draw(9,-1)edge[decorate]  node[left] {$r_2 $}(9,-4) ;   

\fill[white] (9,-1) circle (28pt);      
\fill[white] (6,2.2) circle (30pt); 
    \draw(-3,-1)edge[decorate]  node[left] {$a_3  $}(-6,-4) ;   
    \draw(-3,-1)edge[decorate]  node[left] {$r_3  $}(-3,-4) ;   
    \draw(0,-1)edge[decorate]  node[left] {$a_3  $}(0,-4) ;   
    \draw(0,-1)edge[decorate]  node[right] {$r_3  $}(3,-4) ;

    \draw(-3,-4)edge[decorate]  node[left] {$a_2  $}(-6,-7) ;   
    \draw(-3,-4)edge[decorate]  node[left] {$r_2  $}(-3,-7) ;   

\fill[white] (-6,-7) circle (25pt);      
\fill[white] (-3,-7) circle (20pt);

  \draw(3,-4)edge[decorate]  node[left] {$a_2  $}(0,-7) ;   
    \draw(3,-4)edge[decorate]  node[left] {$r_2  $}(3,-7) ;   

\fill[white] (3,-7) circle (20pt);      
\fill[white] (0,-7) circle (26pt);     
  
 \fill[white] (-3,-1) circle (30pt);    
  \fill[white] (0,-1) circle (30pt);    
\fill[white] (3,-1) circle (30pt);   
\fill[white] (6,-1) circle (24pt);    
      \fill[white] (-6,-4) circle (30pt);    
            \fill[white] (-3,-4) circle (24pt);   
        \fill[white] (0,-4) circle (24pt);    
         \fill[white] (3,-4) circle (30pt);    
         \fill[white] (6,-4) circle (30pt);   
                  \fill[white] (9,-4) circle (27pt);   
         \draw (-6,-7) node {$  \Yboxdim{8pt}\gyoung(;1;1;1,;2;2;3) $  };   
  \draw (3,4.5) node {$  \Yboxdim{8pt}\gyoung(;1;1;1,;2;2,;3) $  };   
 \draw (0,2) node {$  \Yboxdim{8pt}\gyoung(;1;1;1,;2;2,;3)$  };   
\draw (6,2) node {$  \Yboxdim{8pt}\gyoung(;1;1;1;2;2,,;3) $  };   
  \draw (-3,-1) node {$ \Yboxdim{8pt}\gyoung(;1;1;1,;2;2,;3)  $  };   
   \draw (0,-1) node {$  \Yboxdim{8pt}\gyoung(;1;1;1;2,;2,;3) $  };   
\draw (6,-1) node {$ (\varnothing,\varnothing) $  };   
\draw (9,-1) node {$  \Yboxdim{8pt}\gyoung(;1;1;1;2;2,;3) $  };   
      \draw (-6,-4) node {$  \Yboxdim{8pt}\gyoung(;1;1;1,;2;2,;3)$  };   
             \draw (-3,-4) node {$  \Yboxdim{8pt}\gyoung(;1;1;1,;2;2;3) $  };   
      \draw (0,-4) node {$  \Yboxdim{8pt}\gyoung(;1;1;1;2,;2,;3) $  };   
         \draw (3,-4) node {$  \Yboxdim{8pt}\gyoung(;1;1;1;2,;2;3) $  };   
         \draw (6,-4) node {$  \Yboxdim{8pt}\gyoung(;1;1;1;2;2,;3) $  };   
                  \draw (9,-4) node {$  \Yboxdim{8pt}\gyoung(;1;1;1;2;2;3) $  }; 
           %%%%%
           %%%%%
                   \draw (-3,-7) node {$  \Yboxdim{8pt}\gyoung(;1;1;1;3,;2;2) $  };   
    \draw (3,-7) node {$  \Yboxdim{8pt}\gyoung(;1;1;1;2;3,;2) $  };          
   \draw (0,-7) node {$  \Yboxdim{8pt}\gyoung(;1;1;1;2,;2;3) $  };   
               \end{tikzpicture}
$$
  \caption{
The set of terminal vertices of this graph gives precisely the set of all semistandard Young tableaux of weight $(3,2,1)$. 
 }
\label{helothe!} \end{figure}
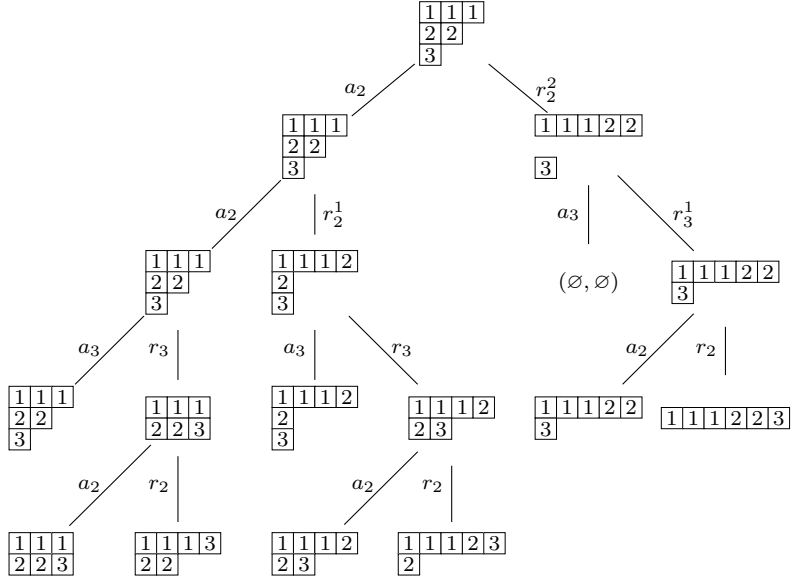

\end{eg}

\begin{cor}\label{mainbijection}
Let $(\lambda, \nu, s)$ be a co-Pieri triple and $\mu \vdash s$. There is one-to-one correspondence 
$$\SStd_s^0(\nu \setminus \lambda, \mu) \overset{1-1}\longleftrightarrow \bigsqcup_{\tau \vdash s} \SStd_s(\tau, \mu) \times \Latt_s^0(\nu \setminus \lambda , \tau).$$
\end{cor}

\begin{proof}
By repeated applications of  Lemma \ref{bijectionSS} we have a bijection between $\SStd_s^0(\nu \setminus \lambda , \mu)$ and the disjoint union over all terminal vertices of $\mathscr{T}(\mu)$ of the sets $\SStd_s^0(\nu \setminus \lambda , (\tau, \tau))$ where $(\tau,\tau)$ is the label of the corresponding terminal vertex.
Now, by Lemma \ref{bijectionTV}, we have that for each $\tau\vdash s$, the number of terminal vertices labelled by $(\tau,\tau)$ is precisely the cardinality of $\SStd_s(\tau, \mu)$. Moreover, by definition we have that $\SStd_s^0(\nu\setminus \lambda, (\tau,\tau))=\Latt_s^0(\nu\setminus \lambda , \tau)$. Hence the result follows.
\end{proof}

   \begin{eg}
    Let $\lambda=(7)$, $\nu=(6)$, $\mu=(3,2,1)$ and   $\tau=(4,2)$. 
We have that   $$|{\rm Latt}_6^0(\nu\setminus\lambda,\tau)\times \SStd_s( \tau,\mu)|= 4\times 2 = 8$$  
and the tableaux
 are listed explicitly in  \cref{a2parteg,a2partegcontuinued further}.   
 We shall now list the 8 elements of $
 \SStd_6^0(\nu\setminus\lambda,\mu)$ which correspond to these pairs of  tableaux under  the bijection given in \cref{mainbijection}.  
 The two terminal vertices  labelled by
 $(\tau,\tau)$ 
 are determined by the paths  $  r_2  r_3a_2 a_2$ and $  a_2   r_3r_2 a_2 $.  

 First consider the  path 
 $  r_2  r_3a_2 a_2$.  Using \cref{bijectionSS} we apply  
$\mathscr{R}_3^{-1} \circ \mathscr{R}_2^{-1}$ to the tableaux in \newline $\Latt_6^0(\nu\setminus \lambda , \tau)$ to get
$$\begin{array}{ccc}
 \left( \begin{array}{ccc|c|cc}
r(1) & r(1) & r(1) & d(1) & a(1) & a(1) \\
1 & 1 & 1 & 1 & 2 & 2 \end{array}\right) 
& \overset{\mathscr{R}_3^{-1}\circ \mathscr{R}_2^{-1}}\longrightarrow &
 \left( \begin{array}{ccc|c|cc}
r(1) & r(1) & r(1) & d(1) & a(1) & a(1) \\
3 & 1 & 1 & 1 & 2 & 2 \end{array}\right) 
\end{array}$$
$$\begin{array}{ccc}
 \left( \begin{array}{ccc|c|cc}
r(1) & r(1) & r(1) & d(1) & a(1) & a(1) \\
1 & 1 & 1 & 2 & 2 & 1 \end{array}\right) 
& \overset{\mathscr{R}_3^{-1}\circ \mathscr{R}_2^{-1}}\longrightarrow &
 \left( \begin{array}{ccc|c|cc}
r(1) & r(1) & r(1) & d(1) & a(1) & a(1) \\
3 & 1 & 1 & 2 & 2 & 1 \end{array}\right) 
\end{array}$$
$$\begin{array}{ccc}
 \left( \begin{array}{cc|ccc|c}
r(1) & r(1) & d(1) & d(1) & d(1) & a(1) \\
1 & 1 & 2 & 1 & 1 & 2 \end{array}\right) 
& \overset{\mathscr{R}_3^{-1}\circ \mathscr{R}_2^{-1}}\longrightarrow &
 \left( \begin{array}{cc|ccc|c}
r(1) & r(1) & d(1) & d(1) & d(1) & a(1) \\
3 & 1 & 2 & 1 & 1 & 2 \end{array}\right) 
\end{array}$$
$$\begin{array}{ccc}
 \left( \begin{array}{cc|ccc|c}
r(1) & r(1) & d(1) & d(1) & d(1) & a(1) \\
1 & 1 & 2 & 2 & 1 & 1 \end{array}\right) 
& \overset{\mathscr{R}_3^{-1}\circ \mathscr{R}_2^{-1}}\longrightarrow &
 \left( \begin{array}{cc|ccc|c}
r(1) & r(1) & d(1) & d(1) & d(1) & a(1) \\
1 & 1 & 3 & 2 & 2 & 1 \end{array}\right) 
\end{array}$$

Now consider the path 
$  a_2   r_3r_2 a_2 $.  
Using \cref{bijectionSS} we apply $\mathscr{R}_2^{-1}\circ \mathscr{R}_3^{-1}$ to the tableaux in $\Latt_6^0(\nu\setminus \lambda , \tau)$ to get  the following four elements of $\SStd_6^0(\nu\setminus\lambda,\mu)$.
$$\begin{array}{ccc}
 \left( \begin{array}{ccc|c|cc}
r(1) & r(1) & r(1) & d(1) & a(1) & a(1) \\
1 & 1 & 1 & 1 & 2 & 2 \end{array}\right) 
& \overset{\mathscr{R}_2^{-1}\circ \mathscr{R}_3^{-1}}\longrightarrow &
 \left( \begin{array}{ccc|c|cc}
r(1) & r(1) & r(1) & d(1) & a(1) & a(1) \\
2 & 1 & 1 & 1 & 3 & 2 \end{array}\right) 
\end{array}$$
$$\begin{array}{ccc}
 \left( \begin{array}{ccc|c|cc}
r(1) & r(1) & r(1) & d(1) & a(1) & a(1) \\
1 & 1 & 1 & 2 & 2 & 1 \end{array}\right) 
& \overset{\mathscr{R}_2^{-1}\circ \mathscr{R}_3^{-1}}\longrightarrow &
 \left( \begin{array}{ccc|c|cc}
r(1) & r(1) & r(1) & d(1) & a(1) & a(1) \\
2 & 1 & 1 & 3 & 2 & 1 \end{array}\right) 
\end{array}$$
$$\begin{array}{ccc}
 \left( \begin{array}{cc|ccc|c}
r(1) & r(1) & d(1) & d(1) & d(1) & a(1) \\
1 & 1 & 2 & 1 & 1 & 2 \end{array}\right) 
& \overset{\mathscr{R}_2^{-1}\circ \mathscr{R}_3^{-1}}\longrightarrow &
 \left( \begin{array}{cc|ccc|c}
r(1) & r(1) & d(1) & d(1) & d(1) & a(1) \\
2 & 1 & 3 & 1 & 1 & 2 \end{array}\right) 
\end{array}$$
$$\begin{array}{ccc}
 \left( \begin{array}{cc|ccc|c}
r(1) & r(1) & d(1) & d(1) & d(1) & a(1) \\
1 & 1 & 2 & 2 & 1 & 1 \end{array}\right) 
& \overset{\mathscr{R}_2^{-1}\circ \mathscr{R}_3^{-1}}\longrightarrow &
 \left( \begin{array}{cc|ccc|c}
r(1) & r(1) & d(1) & d(1) & d(1) & a(1) \\
2 & 1 & 3 & 2 & 1 & 1 \end{array}\right) 
\end{array}$$

   \end{eg}

 We are now ready to prove our main theorem.

    \begin{proof}[Proof of \cref{mainresult}]
Recall that 
$$\overline{g}(\lambda, \nu, \mu) = \dim_{\CC} \Hom_{\mathfrak{S}_s}({\sf S}(\mu), \Delta_s^0(\nu\setminus \lambda)).$$
We prove the result by downwards induction on $\mu$ (using the dominance order $\trianglerighteq$). If $\mu$ is maximal then $\mu = (s)$ and ${\sf S}(\mu) = {\sf M}(\mu)$. Moreover $\Latt_s^0(\nu\setminus \lambda , \mu) = \SStd_s^0(\nu \setminus \lambda, \mu)$. Thus the result follows from \cref{YOUNGSRULE}. 
 We now assume that the result holds for all partitions $\tau \rhd \mu$.  
 We have
\begin{equation}\label{decompositionperm}
{\sf M}(\mu) = \bigoplus_{\tau \trianglerighteq s} |\SStd_s(\tau, \mu)| {\sf S}(\tau).
\end{equation}
By induction we have
\begin{equation}\label{mainthmind}
\dim_\CC \Hom_{\mathfrak{S}_s}({\sf S}(\tau), \Delta_s^0(\nu\setminus\lambda))=|\Latt_s^0(\nu\setminus \lambda, \tau)| \quad \forall \tau \vartriangleright \mu.
\end{equation}
By \cref{YOUNGSRULE},  (\ref{decompositionperm}) and (\ref{mainthmind}) we have
\begin{eqnarray*}
|\SStd_s^0(\nu\setminus \lambda, \mu) | &=& \dim_\CC \Hom_{\mathfrak{S}_s}({\sf M}(\mu), \Delta_s^0(\nu \setminus \lambda)) \\
&=& \dim_{\CC}\Hom_{\mathfrak{S}_s}({\sf S}(\mu), \Delta_s^0(\nu\setminus \lambda)) + \sum_{\tau \vartriangleright \mu} |\SStd_s(\tau, \mu)| \dim_\CC \Hom_{\mathfrak{S}_s} ({\sf S}(\tau), \Delta_s^0(\nu\setminus \lambda))\\
&=& \dim_{\CC}\Hom_{\mathfrak{S}_s}({\sf S}(\mu), \Delta_s^0(\nu\setminus \lambda)) + \sum_{\tau \vartriangleright \mu} |\SStd_s(\tau, \mu)| \, |\Latt_s^0(\nu\setminus \lambda, \tau)|.
\end{eqnarray*}
Comparing this equality with \cref{mainbijection} and noting that $|\SStd_s(\mu,\mu)|=1$ we get
$$\overline{g}(\lambda, \nu, \mu) =\dim_\CC \Hom_{\mathfrak{S}_s}({\sf S}(\mu), \Delta_s^0(\nu\setminus \lambda)) = |\Latt_s^0(\nu\setminus \lambda, \mu)|$$
as required.
      \end{proof}

\section{Examples}\label{examplesattheend}

In this section we provide several illustrative examples of
how to calculate  Kronecker coefficients in terms of latticed Kronecker tableaux.  
%The first few examples
%have been chosen for their simplicity
%(in order to help the reader  and   illustrate how we unify earlier work) and they then  increase in difficulty. 
%%examples are more complicated and have been chosen to .  
%
%
%\subsection{Tensor products with the standard representation of the symmetric group}
As a warm up exercise, we first consider the decomposition of tensor products of the form ${\sf S}(\lambda_{[n]})\otimes {\sf S}(n-1,1)$. 
 These coefficients are  trivial to calculate (even for advanced undergraduates) but they provided our initial motivation for this paper and they illustrate some of the basic ideas  very well.  
We have
\begin{align*}
g(\nu_{[n]},\lambda_{[n]}, (n-1,1))
&=
\dim_\CC(\Hom_{\mathfrak{S}_n}
({\sf S}(\lambda_{[n]}) \otimes {\sf S}(n-1,1) , {\sf S}(\nu_{[n]})))\\
&=
\dim_\CC(\Hom_{\mathfrak{S}_1}
(  {\sf S}((1)) , \Delta^0_1(\nu\setminus \lambda)))\\
&= \dim_\CC(\Hom_{\mathfrak{S}_1}
(  {\sf M}((1)) , \Delta^0_1(\nu\setminus \lambda)))\\
&= |\SStd_1^0(\nu\setminus\lambda , (1))|.
\end{align*}
Note that as $s=1$ we have $\SStd_1^0(\nu\setminus \lambda ,(1)) = \Std_1^0(\nu\setminus \lambda)$. Moreover we have $\Std_1^0(\nu\setminus \lambda) = \Std_1(\nu\setminus \lambda)$ unless $\lambda = \nu$ in which case we have $\Std_1^0 (\lambda \setminus \lambda ) = \Std_1(\lambda \setminus \lambda) \setminus \{ (-\varepsilon_0 , + \varepsilon_0)\}$.
The coefficient $g(\nu_{[n]},\lambda_{[n]}, (n-1,1))$ is therefore equal to the number of paths of length 1 from $\lambda$ to $\nu$ for $\lambda\neq \nu$ and 
is equal to one fewer for $\lambda=\nu$.  
In the former (respectively latter) case the number of such paths is  equal to 1 (respectively  equal to the number of removable nodes of $\lambda$).  
 Compare  with  \cite[Exercise 7.81]{Sta99}.

 \begin{eg}
For example,  the coefficients stabilise for $n\geq 7$ and we have  that
\begin{align*}\mathbf{S}(n-3,2,1) \otimes \mathbf{S}(n-1,1) = &\mathbf{S}(n-2,2) \oplus \mathbf{S}(n-2,1^2) \oplus 
 \mathbf{S}(n-3,3) \oplus 
 2 \mathbf{S}(n-3,2,1) \oplus 
  \mathbf{S}(n-3,1^3) \\  &\oplus  
   \mathbf{S}(n-4,3,1) \oplus 
      \mathbf{S}(n-4,2^2) \oplus 
         \mathbf{S}(n-4,2,1^2). 
\end{align*}
 The only coefficient not equal to 0 or 1 is $g( (n-3,2,1), (n-3,2,1), (n-1,1))=2$ for $n\geq 7$.    
 See     \cref{sdfasafasfdasdfa11111} for the paths from $(2,1) \in \mathcal{Y}_3$ to points in $\mathcal{Y}_4$.  
 
    \end{eg}

\begin{figure}[ht!]$$
  \scalefont{0.8}
    \begin{tikzpicture}[scale=0.36]
          \begin{scope}

       \draw[->] (0,5)--(-9,1);
       \draw(0,5)--(-6,1);       
              \draw(0,5)--(0,1);
              \draw(-9,1)--(-9,-3);       
              \draw(-6,1)--(-6,-3);              
              \draw(-6,1)--(0,-3);              
              \draw(-6,1)--(3,-3);                            
              \draw(-9,1)--(0,-3);                            
              \draw(-9,1)--(-3,-3);     
              \draw(0,1)--(9,-3);     
              \draw(0,1)--(6,-3);                   \draw(0,1)--(0,-3);     
                            \draw(0,1)--(12,-3);     
            \fill[white] (0,5) circle (25pt);	 % {  $\Yboxdim{6pt}\gyoung(;;,;)$  };   
      \fill[white] (-9,1) circle (25pt);	 % { $\Yboxdim{6pt}\gyoung(;;)$ };
      \fill[white] (-6,1) circle (25pt);	 % { $\Yboxdim{6pt}\gyoung(;,;)$ };
     %%
%          \fill[white] (-1,2) -- (-2,1);    \fill[white] (-2,2.5) -- (-5,1);   
     %%
                \fill[white] (0,1) circle (25pt);	 % { $\Yboxdim{6pt}\gyoung(;;,;)$ };   
     \fill[white] (-9,-3) circle (25pt);	 %   {$\Yboxdim{6pt}\gyoung(;;)$ }	;   
     \fill[white] (-6,-3) circle (25pt);	 %   {$\Yboxdim{6pt}\gyoung(;,;)$ }	;   
     \fill[white] (-3,-3) circle (25pt);	 %   {$\Yboxdim{6pt}\gyoung(;;;)$ }	;   
     \fill[white] (0,-3) circle (25pt);	 %   {$\Yboxdim{6pt}\gyoung(;;,;)$ }	;   
      \fill[white] (+3,-3) circle (25pt);	 %   { $\Yboxdim{6pt}\gyoung(;,;,;)$    }		;
      \fill[white] (+6,-3) circle (25pt);	 %   { $\Yboxdim{6pt}\gyoung(;;;,;)$    }		;     
      \fill[white] (+9,-3) circle (25pt);	 
           \fill[white] (+12,-3)    circle (25pt);

           \draw (0,5) node {  $\Yboxdim{6pt}\gyoung(;;,;)$  };   
     \draw (-9,1) node { $\Yboxdim{6pt}\gyoung(;;)$ };
     \draw (-6,1) node { $\Yboxdim{6pt}\gyoung(;,;)$ };
     %%
%         \draw (-1,2) -- (-2,1);   \draw (-2,2.5) -- (-5,1);   
     %%
               \draw (0,1) node { $\Yboxdim{6pt}\gyoung(;;,;)$ };   
    \draw (-9,-3) node   {$\Yboxdim{6pt}\gyoung(;;)$ }	;   
    \draw (-6,-3) node   {$\Yboxdim{6pt}\gyoung(;,;)$ }	;   
    \draw (-3,-3) node   {$\Yboxdim{6pt}\gyoung(;;;)$ }	;   
    \draw (0,-3) node   {$\Yboxdim{6pt}\gyoung(;;,;)$ }	;   
     \draw (+3,-3) node   { $\Yboxdim{6pt}\gyoung(;,;,;)$    }		;
     \draw (+6,-3) node   { $\Yboxdim{6pt}\gyoung(;;;,;)$    }		;     
     \draw (+9,-3) node   { $\Yboxdim{6pt}\gyoung(;;,;;)$    }		;          
          \draw (+12,-3) node   { $\Yboxdim{6pt}\gyoung(;;,;,;)$    }		;          
%    \draw (0.0,1) -- (0,2);     \draw (0.0,-1) -- (0,-2);      

    \end{scope}\end{tikzpicture}$$

\!\!\!\!
    \caption{Paths of degree 1 beginning at $(2,1)$ in $\mathcal{Y}_3$.}
\label{sdfasafasfdasdfa11111}
\end{figure}
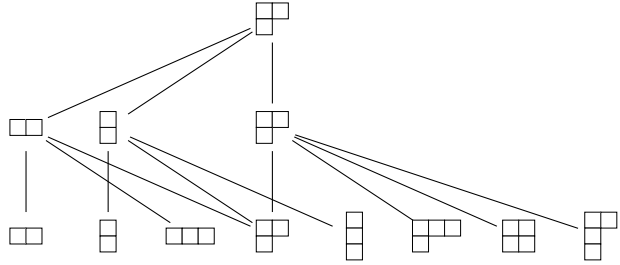

%\subsection{Examples from earlier in the paper}
We now revisit some of the earlier examples in the paper.

\begin{eg}
Consider the rightmost example in \cref{anewfigforintro} from the introduction.  We have that 
$((5,3,3), (7,5,1,1), \mu)$ is a co-Pieri triple for $\mu \vdash 5$.  
  We have that
$$
g ((n-11,5,3,3), (n-14,7,5,1,1) , (n-5,2,2,1))=
\overline{g} ((5,3,3), (7,5,1,1) , (2,2,1))
=11
$$ for all $n\geq 21$ and an example of an element of $\Latt^0_5(((5,3,3) \setminus (7,5,1,1), (2,2,1)))$ is depicted in \cref{anewfigforintro}. 
\end{eg}

\begin{eg}
We have that $(( 6,5,3,2), (12,8,5,3), \mu)$ is a co-Pieri triple for $\mu \vdash 3$.  Some of the corresponding semistandard and latticed tableaux 
are depicted  in \cref{semiexamfig} and discussed in \cref{semiexamfig2}.  
We have that
$$
\overline{g}(( 6,5,3,2), (12,8,5,3), (3))=6
\quad
\overline{g}(( 6,5,3,2), (12,8,5,3), (2,1))=9
\quad
\overline{g}(( 6,5,3,2), (12,8,5,3), (1^3))=3
$$
the six latticed tableaux of weight $(3)$ are given in \cref{semiexam}.  We leave constructing those of weight $(2,1)$ and $(1^3)$ as an exercise for the reader.  
\end{eg}

\begin{eg}
We have that $((9,6,3),(9,6,3),(2,1))$ is a co-Pieri triple and that 
$$
\overline{g}((9,6,3),(9,6,3),(2,1))=
\Latt^0_3(((9,6,3)\setminus(9,6,3),(2,1)))
=60.
$$
The dedicated reader might wish to attempt this calculation themselves once they have digested the other examples in this section. 
The rightmost tableau in \cref{semiexamfig} is an example of a latticed tableau for this triple.  
\end{eg}

\subsection{Kronecker coefficients labelled by two 2-row partitions}
In this section we provide examples of our  tableaux combinatorics for coefficients $g(\lambda_{[n]},\nu_{[n]},\mu_{[n]}) $ in which $\lambda_{[n]}$ and $\nu_{[n]}$ are two-part partitions but $\mu_{[n]}$ is arbitrary.
   These coefficients have been described in many ways and received the attention of many authors \cite{MR3338303,RW94,Rosas01,ROSAANDCO,BWZ10,MR2550164};   Hilbert series related to these coefficients have been linked to problems in quantum information theory \cite{1,2}. 
  The advantage of our description over previous work is that it covers these coefficients as a simple  example 
   in a far broader class of Kronecker coefficients.

    \begin{prop}\label{2partersmallthings}
    If  $\lambda_{[n]}$ and $\nu_{[n]}$ are 2-part partitions and
     $ {g}(  \lambda_{[n]},\nu_{[n]}, \mu_{[n]})  \neq 0 $, then 
  $\ell(\mu_{[n]})\leq 4$.
     \end{prop}
    \begin{proof}
First note that $g(\lambda_{[n]},\nu_{[n]}, \mu_{[n]})\neq 0$ implies that $g(\lambda , \nu, \mu)=|\Latt_s^0(\nu\setminus \lambda, \mu)| \neq 0$.
  Now the only possible steps in semistandard Kornecker tableaux in $\Latt_s^0(\nu\setminus \lambda, \mu)$ are  
 $r(1)$, $d(1)$, or $a(1)$ and the ordering on these steps  is 
 $ 
 r(1) < d(1) < a(1)
 $.  
 Now by definition of a lattice permutation, for any $\SSTS\in \Latt_s^0(\nu\setminus\lambda, \mu)$, the frame number of a step of type $r(1)$ in $\SSTS$ is equal to 1, the frame number of a step of type $d(1)$ is less or equal to $2$ and the frame number of a step of type $a(1)$ is less or equal to $3$. Thus if $\Latt_s^0(\nu\setminus \lambda , \mu)\neq \emptyset$ then  $\ell (\mu) \leq 3$ and hence $\ell (\mu_{[n]}) \leq 4$ as required.
\end{proof}

\begin{prop}\label{fourterms}
Let $\lambda_{[n]}$ and $\nu_{[n]}$ be 2-part partitions. Let $\mu_{[n]}$ be an arbitrary partition.  Then we have that
$$g({\lambda_{[n]},\nu_{[n]}},{\mu_{[n]}}) = \sum_{i=0}^3  (-1)^i|\Latt_{s_i}^0(\nu\setminus \lambda, \mu^{(i)})|$$  
where $\mu^{(0)} = \mu$ and for $i\geq 1$ the partition $\mu^{(i)}$ is obtained from $\mu^{(i-1)}$ by adding a single row of boxes in the $i$th row, the last of which having content $n-|\mu^{(i-1)}|$, and $s_i = |\mu^{(i)}|$.
 \end{prop}
 \begin{proof}
By \cite[Theorem 3.7]{BDO15} we can write
$$  g(\lambda_{[n]},\nu_{[n]},\mu_{[n]}) = \sum_{i\geq 0}  (-1)^{i} \overline{g}(\lambda ,\nu, \mu^{(i)}).$$  
Using \cref{2partersmallthings} we have that if $\ell(\mu^{(i)})>3$ then $\overline{g}(\lambda, \nu, \mu^{(i)}) = 0$. Now the result follows from \cref{mainresult}. 
  \end{proof}

\begin{rmk} 
Note also that if $|\mu^{(i)}|> |\lambda| + |\nu|$ then $\Latt_{s_i}^0(\nu\setminus\lambda, \mu^{(i)}) = \emptyset$. Thus as $n$ gets larger the sum in \cref{fourterms} has fewer than 4 terms. In fact when $n>|\lambda|+|\nu| +\mu_1 -1$ then we have $|\mu^{(1)}|>|\lambda|+|\nu|$ and so (letting  $s = |\mu|$) we have that 
$$g(\lambda_{[n]}, \nu_{[n]}, \mu_{[n]}) = |\Latt_s^0(\nu\setminus\lambda , \mu)|.$$
\end{rmk}

 \begin{eg}
Let $\lambda =(  7)$ and $\nu =(  6)$ and $\mu =(4,3,1)$.  
 Then $\omega_1(\SSTS)$ must be one of the following
$$
(r(1) \; r(1) \; r(1) \; | \; d(1) \; d(1) \; d(1) \; | \; a(1) \; a(1)  ) \quad 
(r(1) \; r(1) \; r(1) \; r(1) \; | \; d(1) \; | \; a(1) \; a(1) \; a(1)).
$$
Is is easy to check that $\SSTS \in \Latt_8^0(\nu\setminus \lambda, \mu)$ if and only if $\omega(\SSTS)$ is one of the following
$$ 
 \left( \begin{array}{ccc|ccc|cc}
r(1) & r(1) & r(1) & d(1)& d(1) & d(1) & a(1)& a(1) \\
1 & 1 & 1 		& 2 	&	2	&	2	& 3 & 1 \end{array}\right)   
$$
$$ \left( \begin{array}{ccc|ccc|cc}
r(1) & r(1) & r(1) & d(1)& d(1) & d(1) & a(1)& a(1) \\
1 & 1 & 1 		& 2 	&	2	&	1	& 3 & 2 \end{array}\right)   $$
$$ \left( \begin{array}{cccc|c|ccc}
r(1) & r(1) & r(1) & r(1)& d(1) & a(1) & a(1)& a(1) \\
1 & 1 & 1 		& 1 	&	2	&	3	& 2 & 2 \end{array}\right)  
$$ 
%
% The corresponding semistandard tableaux are as follows, 
%$$
%r(1) \circ r(1) \circ r(1) \circ a(1) \mid d(1) \circ d(1) \circ d(1) \mid a(1)
%$$
%$$
%r(1) \circ r(1) \circ r(1) \circ d(1) \mid d(1) \circ d(1) \circ a(1) \mid a(1)
%$$
%$$
%r(1) \circ r(1) \circ r(1) \circ r(1) \mid d(1) \circ a(1) \circ a(1) \mid a(1)
%$$
Therefore $g(( n-7,7), ( n-6,6), (n-8,4,3,1))=3$  for $n\geq 15$.  
We leave it as an exercise for the reader to verify that these semistandard Kronecker tableaux are orbits of size 12, 3, and 1 respectively.

 \end{eg}

\subsection{A Kronecker product labelled by two three-row partitions}
 \label{yetanother}
We now consider the next simplest case: namely a pair of 3-row partitions.  
 Let $\lambda=( 6,1)$ and  $\nu=( 4,3)$, we have $|\SStd_3^0(\nu\setminus \lambda, ( 3))|=
|\Latt_3^0(\nu\setminus \lambda, ( 3))|=3$. 
The corresponding reading words are as follows, 
%$$ d(1)  \circ \down(1,2) \circ   \down(1,2) \quad\quad
%  d(2)  \circ \down(1,2) \circ   \down(1,2) \quad\quad  r(1)  \circ \down(1,2) \circ  a(1).$$
%$$  
%   r(1)  \circ \down(1,2) \circ  a(1).$$
$$ \left( \begin{array}{c|cc} d(1)  &\down(1,2) &  \down(1,2) \\ 1&1&1
\end{array}
\right)
\quad
 \left( \begin{array}{c|cc} d(2)  &\down(1,2) &  \down(1,2) \\ 1&1&1
\end{array}\right)
\quad
\left( \begin{array}{c|c|c}  r(1)  &\down(1,2) & a(2) \\ 1&1&1
\end{array}\right)
$$
It is easy to check that any 
$\SSTS \in \SStd_3^0(\nu\setminus \lambda, ( 2,1)) $ must have $\omega_1(\SSTS)$ as follows,
$$
 (d(1) \; |\; \down(1,2) \;   \down(1,2)) \quad\quad
 (d(2) \; |\; \down(1,2) \;   \down(1,2)) \quad\quad
 (r(1) \; |\; \down(1,2) \; |\;  a(2)) 
$$
and  $|\SStd_3^0(\nu\setminus \lambda, ( 2,1))|=7$.  
We have  $\SSTS \in \Latt_3^0(\nu\setminus \lambda, ( 2,1))$ if and only if 
 $\omega(\SSTS)$ is one of the following,
 $$ \left( \begin{array}{c|cc} d(1) &  \down(1,2) &   \down(1,2) \\ 1&2&1
\end{array}
\right)
\quad
 \left( \begin{array}{c|c|c} r(1) &  \down(1,2) &   a(2) \\ 1&2&1
\end{array}
\right)
 $$ 
 $$
 \left( \begin{array}{c|cc} d(2) &  \down(1,2) &   \down(1,2) \\ 1&2&1
\end{array}
\right)
\quad
 \left( \begin{array}{c|c|c} r(1) &  \down(1,2) &   a(2) \\ 1&1&2
\end{array}
\right)
$$ 
 We have that $|\SStd_3^0(\nu\setminus \lambda, ( 1^3))|=12$.  
The  unique element  $\SSTS \in \Latt_3^0(\nu\setminus \lambda, ( 1^3))$ has $\omega(\SSTS)$ equal to
 $$
  \left( \begin{array}{c|c|c}  r(1)   &   \down(1,2) &  a(2) \\ 1&2&3
\end{array} \right)
 $$
 We therefore conclude that 
 $$
 \overline{g} ((6,1),(4,3),(3))=3 \quad \quad
 \overline{g}((6,1),(4,3),(2,1))=4 \quad \quad
 \overline{g}((6,1),(4,3),(1^3))=1.
 $$
The Kronecker coefficients quickly stabilise in this case, for example
$$
g((6^2,1),(6,4,3),(10,3))=3
\qquad
g((7,6,1),(7,4,3),(11,3))=4
$$
and $ g((n-7,6,1),(n-7,4,3),(n-3,3))=4$ for $n\geq 14$.

\subsection{A larger example}
 
 Let $\lambda=( 6,2)$, $\nu=(  7,4)$.  
We have that $(\lambda,\nu,s)$ is a co-Pieri triple for $s\leq 5$.   
Let $s=4$ and $\mu \vdash s$.  Given 
$\SSTS \in \SStd_4^0(\nu\setminus \lambda,\mu)$, we have that $\omega_1(\SSTS)$ is equal to one of 
$$
(d(1) \;  |\;  a(1) \;   |\;  a(2) \;   a(2) )
\quad
(d(2) \;  |\;   a(1) \;  |\;   a(2) \;   a(2) )
\quad
(\down(1,2) 
 \;   |\; 
a(1)
 \;   
a(1) 
 \;   |\; 
a(2))
\quad
(
\up (2,1) 
 \;   |\; 
a(2)
 \;  
a(2)
 \;   
a(2) 
).
$$
 We now consider the 
 semistandard and latticed tableaux for each weight $\mu$ for $\mu \vdash 4$. 
    We have that $|\SStd_4^0(\nu\setminus \lambda), ( 4))|=|\Latt_4^0(\nu\setminus \lambda, ( 4))|=4$.  
  The corresponding $\omega_1(\SSTS)$ are as follows: 
$$ 
  \left( \begin{array}{c|c|cc} 
d(1) 
&
a(1)
&
a(2) 
&
a(2)
\\
1 &1&1&1
\end{array}
\right)
\quad
  \left( \begin{array}{c|cc|c} 
\down(1,2) 
& a(1)
& 
a(1) 
&
a(2)
\\
1 &1&1&1
\end{array}
\right)
 $$
   $$\left( \begin{array}{c|c|cc} 
  d(2) 
&
 a(1) 
&
a(2)
&  
a(2)\\
1 &1&1&1
\end{array}
\right)\quad 
%%%%%%%%%%
  \left( \begin{array}{c|ccc}  
\up (2,1)   &
a(2)
&
a(2)
&  
a(2)
\\
1 &1&1&1
\end{array}
\right)
.$$
 Given $\SSTS \in \Latt_4^0(\nu\setminus \lambda, ( 3,1)) $, we have that $\omega_1(\SSTS)$ is one of the following,
 $$
   \left( \begin{array}{c|c|cc} 
 d(1) &    a(1) &    a(2) &    a(2) \\
  1&1&2&1
\end{array}\right)
\quad
 \left( \begin{array}{c|cc|c} 
 \down(1,2)  &    a(1) &    a(1) &    a(2) 
 \\
   1&2&1&1
\end{array}\right)
\quad
  \left( \begin{array}{c|c|cc} 
 d(2) &    a(1) &    a(2) &    a(2) 
 \\
   1&2&1&1
\end{array}\right)
%   \left( \begin{array}{cccc} 
% d(1) &    a(1) &    a(2) &    a(2) \\
%  1&2&1&1
%\end{array}\right)
%\quad
% \left( \begin{array}{cccc} 
% \down(1,2)  &    a(1) &    a(1) &    a(2) 
% \\
%   1&2&1&1
%\end{array}\right)
 $$
$$
   \left( \begin{array}{c|c|cc} 
 d(1) &    a(1) &    a(2) &    a(2) \\
  1&2&1&1
\end{array}\right)
\quad
 \left( \begin{array}{c|cc|c} 
 \down(1,2)  &    a(1) &    a(1) &    a(2) 
 \\
   1&1&1&2
\end{array}\right)\quad
 \left( \begin{array}{c|c|cc} 
 d(2) &    a(1) &    a(2) &    a(2) 
 \\
   1&1&2&1
\end{array}\right)
%\quad
% \left( \begin{array}{cccc} 
% \down(1,2)  &    a(1) &    a(1) &    a(2) 
% \\
%   1&1&1&2
%\end{array}\right)
$$
$$
 \left( \begin{array}{c|ccc} 
 \up(2,1)  &    a(2) &    a(2) &    a(2) 
 \\
   1&2&1&1
\end{array}\right)$$   
% 
%  The corresponding $(3,1)$-admissible tableaux  are as follows: 
%$$\begin{array}{rrcccccc}
% &
%d(1) 
%\circ
%a(2)
%\circ
%a(2) 
%\mid 
%a(1)
% %%%%%%%
%&
%\down(1,2) 
%\circ
%a(1)
%\circ 
%a(1) 
%\mid 
%a(2)
%\\
%&
%\down(1,2) 
%\circ
%a(1)
%\circ 
%a(2) 
%\mid 
%a(1)
% %%%%%%%%
%& d(2) 
%\circ
% a(1) 
%\circ
%a(2)
%\mid 
%a(2)
%&\up (2,1) 
%\circ
%a(2)
%\circ
%a(2)
%\mid  
%a(2) 
%%%%%%%%%%%
% %%%%%
%\\ &
%&  d(2) 
%\circ
% a(2) 
%\circ
%a(2)
%\mid 
%a(1)
% \end{array}$$
% We have that $ |\Latt_4^0(\nu\setminus \lambda, ( 2,2))|=3$.  
 Given $\SSTS \in  \Latt_4^0(\nu\setminus \lambda, ( 2,2)) $, we have that $\omega_1(\SSTS)$ is one of the following,
 $$
   \left( \begin{array}{c|c|cc} 
 d(1) &    a(1) &    a(2) &    a(2) \\
  1&1&2&2
\end{array}\right)
\quad
 \left( \begin{array}{c|cc|c} 
 \down(1,2)  &    a(1) &    a(1) &    a(2) 
 \\
   1&2&1&2
\end{array}\right)
\quad
  \left( \begin{array}{c|c|cc} 
 d(2) &    a(1) &    a(2) &    a(2) 
 \\
   1&1&2&2
\end{array}\right) 
 $$
% 
%
%
%$$
%d(1) 
%\circ
%a(1)
%\mid 
%a(2) 
%\circ
%a(2)
%\quad
%\down(1,2) 
%\circ
%a(1)
%\mid  
%a(1) 
%\circ
%a(2)
%\quad
% d(2) 
%\circ
% a(1) 
%\mid
%a(2)
%\circ  
%a(2)
%$$
% We have that $ |\Latt_4^0(\nu\setminus \lambda, ( 2,1^2))|=3$.  
  Given $\SSTS \in  \Latt_4^0(\nu\setminus \lambda, ( 2,1^2)) $, we have that $\omega_1(\SSTS)$ is one of the following,
 $$
   \left( \begin{array}{c|c|cc} 
 d(1) &    a(1) &    a(2) &    a(2) \\
  1&2&1&3
\end{array}\right)
\quad
 \left( \begin{array}{c|cc|c} 
 \down(1,2)  &    a(1) &    a(1) &    a(2) 
 \\
   1&2&1&3
\end{array}\right)
\quad
  \left( \begin{array}{c|c|cc} 
 d(2) &    a(1) &    a(2) &    a(2) 
 \\
   1&2&1&3
\end{array}\right) 
 $$
%$$
%d(1) 
%\circ
%a(2)
%\mid 
%a(1) 
%\mid 
%a(2)
%\quad
%\down(1,2) 
%\circ
%a(1)
%\mid  
%a(1) 
%\mid 
%a(2)
%\quad
%d(2) 
%\circ
%a(2)
%\mid  
%a(1) 
%\mid 
%a(2)
%$$
 Finally, we have that $ |\Latt_4^0(\nu\setminus \lambda, ( 1^4))|=0$ and 
therefore 
$$\begin{array}{rlllll}
&\overline{g}((6,2),(7,4),(4))=4
&
\overline{g}((6,2),(7,4),(2,2))=3 
&
\overline{g}((6,2),(7,4),(1^4))=0
\\
&\overline{g}((6,2),(7,4),(3,1))=7 
&\overline{g}((6,2),(7,4),(2,1^2))=3 
\end{array}.$$
 We do not calculate all the coefficients $\overline{g}(\lambda,\nu,\mu)$ for $\mu\vdash 5$ and instead  only calculate the $\mu=(2^2,1)$ case. % (the remaining cases are left as an exercise for the reader).  
%We now calculate   $\overline{g}(\lambda,\nu,(2^2,1))$.  
%We have that 
%$ 
%|\Latt_5^0( \nu\setminus\lambda ,(2^2,1)) | = 11
%$. 
 Given $\SSTS \in  \Latt_5^0(\nu\setminus \lambda, ( 2^2,1)) $, we have that $\omega_1(\SSTS)$ is one of the following,
\begin{align*}
   \left( \begin{array}{c|cc|cc} 
 r(1) &    a(1)   &a(1)  	& a(2) &    a(2) \\
  1&2&1&3	&2
\end{array}\right) \quad
&
   \left( \begin{array}{cc|c|cc} 
 d(1) &    d(1)   &a(1)  	& a(2) &    a(2) \\
  1&1&2&3	&2
\end{array}\right)		\\
   \left( \begin{array}{c|c|cc|c} 
 d(1) &   \down(1,2)   &a(1)  	& a(1) &    a(2) \\
  1&1&2&2	&3
\end{array}\right) \quad
&
   \left( \begin{array}{c|c|cc|c} 
 d(1) &   \down(1,2)   &a(1)  	& a(1) &    a(2) \\
  1&2&3&1	&2
\end{array}\right)		\\
%%%%%%%
   \left( \begin{array}{c|c|c|cc} 
 d(2) &   d(1)   &a(1)  	& a(2) &    a(2) \\
  1&1&2&3	&2
\end{array}\right) \quad
&
   \left( \begin{array}{c|c|cc|c} 
 d(2) &   \down(1,2)   &a(1)  	& a(1) &    a(2) \\
  1&2&3&1	&2
\end{array}\right)		\\
%%%%%%%
   \left( \begin{array}{cc|c|cc} 
 d(2) &   d(2)   &a(1)  	& a(2) &    a(2) \\
  1&1&2&3	&2
\end{array}\right) \quad
&
   \left( \begin{array}{c|c|c|cc} 
\up(2,1) &   \down(1,2)   &a(1)  	& a(2) &    a(2) \\
  1&2&1&3	&2
\end{array}\right)		\\
%%%%%%%
   \left( \begin{array}{c|c|cc|c} 
 d(2) &   \down(1,2)   &a(1)  	& a(1) &    a(2) \\
  1&1&2&2	&3
\end{array}\right) \quad
&
   \left( \begin{array}{c|c|c|cc} 
 \up(2,1) &   \down(1,2)   &a(1)  	& a(2) &    a(2) \\
  1&1&2&3	&2
\end{array}\right)		\\
%%%%%%%
   \left( \begin{array}{c|c|c|cc} 
 d(2) &   d(1)   &a(1)  	& a(2) &    a(2) \\
  1&2&1&3	&2
\end{array}\right) \quad
\end{align*}
% The corresponding $(2^2,1)$-admissible tableaux  are as follows: 
%$$\begin{array}{lrllll}
%&
%r(1) 
%\circ
%a(1)
%\mid 
%a(1) 
%\circ
%a(2)
%\mid  
%a(2)  
%&
%d(1) 
%\circ
%d(1)
%\mid 
%a(1) 
%\circ
%a(2)
%\mid  
%a(2)
%\\
%& r(1) 
%\circ
%\down(1,2)
%\mid 
%a(1) 
%\circ
%a(1)
%\mid  
%a(2)
% & 
%d(1) 
%\circ
%a(1)
%\mid 
%\down(1,2) 
%\circ
%a(2)
%\mid  
%a(1)
% \\
%&
%d(2) 
%\circ
%d(1)
%\mid 
%a(1) 
%\circ
%a(2)
%\mid  
%a(2) 
%&d(2) 
%\circ
%a(1)
%\mid 
%\down(1,2) 
%\circ
%a(2)
%\mid  
%a(1)
%\\
%&
%d(2) 
%\circ
%d(2)
%\mid 
%a(1) 
%\circ
%a(2)
%\mid  
%a(2) 
%&
%\up(2,1) 
%\circ
%a(1)
%\mid 
%\down(1,2) 
%\circ
%a(2)
%\mid  
%a(2)
%\\
%&
%\down(1,2) 
%\circ
%d(2)
%\mid 
%a(1) 
%\circ
%a(1)
%\mid  
%a(2)
%&
%\up(2,1) 
%\circ
%\down(1,2)
%\mid 
%a(1) 
%\circ
%a(2)
%\mid  
%a(2) 
%\\
%&d(2) 
%\circ
%a(1)
%\mid 
%d(1) 
%\circ
%a(2)
%\mid  
%a(2)
%\end{array}
%$$
and therefore $\overline{g}(\lambda,\nu, (2,2,1) 	)   = 11$.

\begin{eg}
We now consider an example  which is not a co-Pieri triple.  
 We let $\lambda=\nu=(1^2)$.  We have that
 $ 
 {\sf DR}_2(\Delta_2(\nu\setminus\lambda)) $  is 5-dimensional and  
  is isomorphic to $\Delta_2(1)\oplus \Delta_2(\varnothing)$.  
The former summand is spanned by the basis elements 
indexed by the Kronecker tableaux $$
  d(2)\circ d(2)
   \qquad
 d(0)\circ d(2)
    \qquad
   d(2)\circ d(0)
  $$
 and the latter summand is spanned by the basis elements indexed by the Kronecker tableaux
$$
 a(3)\circ r(3)
   \qquad
 d(0)\circ d(0).
$$
One can show that the quotient $\Delta_2^0(\nu\setminus\lambda)$ decomposes as a direct sum of two transitive permutation modules
$$
\CC\{u_\stt \mid \stt =m{\uparrow}(2,1) \circ m{\downarrow}(1,2)\}
\oplus 
\CC\{u_\sts \mid \sts \in \{a(1)\circ r(1), r(2)\circ a(2)\}\}.
$$ 
Note that   $\stt_{1\leftrightarrow  2}$ is not a standard Kronecker tableau and hence we cannot use the results of this paper to understand $\Delta_2^0(\nu\setminus\lambda)$.  
However, one can see that the former summand is isomorphic to
 $\Delta_2(2)$ via the isomorphism  $\Delta_2(2) \cong \Delta_2(1^2)\otimes \Delta_2(1^2)$.  
The latter summand is isomorphic to  $\Delta_2(2)\oplus   \Delta_2(1^2)$ as one might expect.

\end{eg}

\begin{Acknowledgements*}
The authors are grateful for the financial support received from the Royal Commission for the Exhibition of 1851  and  EPSRC grant EP/L01078X/1.
\end{Acknowledgements*}

\bibliographystyle{amsplain}   \bibliography{bib}

 \end{document}